\documentclass[a4paper]{article}
\usepackage{amsfonts}
\usepackage{amsmath}
\usepackage{amssymb, amsthm, amsopn, amsfonts,makeidx}
\usepackage{amstext}
\usepackage[a4paper]{geometry}
\usepackage{graphicx}
\usepackage{appendix}
\numberwithin{equation}{section}
\providecommand{\U}[1]{\protect\rule{.1in}{.1in}}

\newtheorem{theorem}{Theorem}[section]

\newtheorem{corollary}[theorem]{Corollary}

\newtheorem{definition}[theorem]{Definition}

\newtheorem{lemma}[theorem]{Lemma}

\newtheorem{proposition}[theorem]{Proposition}
\newtheorem{remark}[theorem]{Remark}

\renewenvironment{proof}[1][Proof]{\noindent\textbf{#1:~}}{\hfill \qed}
\newcommand{\field}[1]{\mathbb{#1}}
\newcommand{\N}{\field{N}}
\newcommand{\R}{\field{R}}
\newcommand{\C}{\field{C}}
\newcommand{\Z}{\field{Z}}

\def\preuve{\noindent\textbf{Proof:~}}

\def\diag{\mathop{\mathrm{diag}}}

\def\Real{\mathop{\mathrm{Re}}~}
\def\Imag{\mathop{\mathrm{Im}}~}
\def\Tr{\mathop{\mathrm{Tr}\,}}
\def\Id{\mathop{\mathrm{Id}\,}}
\def\supp{\mathrm{supp}\,}
\def\Ker{\mathrm{Ker}}
\def\diag{\mathrm{diag}\,}
\begin{document}
\bibliographystyle{plain}

\title{Adiabatic evolution of 1D shape resonances: an artificial interface conditions approach.}
\author{A. Faraj$^{*}$, A. Mantile$^{*}$, F. Nier\footnote{IRMAR, UMR - CNRS 6625, Universit\'{e} Rennes 1, Campus de
Beaulieu, 35042 Rennes Cedex, France}}
\date{}
\maketitle
\begin{abstract}
  Artificial interface conditions parametrized by a complex number
  $\theta_{0}$ are
  introduced for 1D-Schr{\"o}dinger operators. When this complex parameter
  equals the parameter $\theta\in i\R$ of the complex deformation which unveils the
  shape resonances, the Hamiltonian becomes dissipative. This makes
  possible an adiabatic theory for the time evolution of  resonant
  states for arbitrarily large  time scales. The effect of 
  the artificial interface conditions on the important stationary
  quantities involved in quantum transport models is also checked to
  be as small as wanted, in the polynomial scale $(h^N)_{N\in \N}$
as $h\to 0$, according to $\theta_{0}$.
\end{abstract}
\section{Introduction}
\label{se.intro}

The adiabatic evolution of resonances is an old problem which has
received various answers in the last twenty years. The most effective
results where obtained by remaining on the real spectrum and by
considering the evolution of quasi-resonant states (see for example
\cite{Per}\cite{SoWe}\cite{ASFr}). Motivated by nonlinear problems coming
from the modelling of quantum electronic transport, we reconsider this problem and propose a new approach which rely on a
modification of the initial kinetic energy operator
$-h^{2}\Delta$ into $-h^{2}\Delta_{\theta_{0}}^{h}$ where
$\theta_{0}$ parametrizes  artificial interface conditions. 
With this analysis, we aim at developing reduced models for the
nonlinear dynamics of transverse quantum transport in resonant
tunneling diodes or possibly more complex structures. A functional
framework for such a model has been proposed in \cite{Ni1} and
implements a dynamically nonlinear version of the Landauer-B{\"u}ttiker
approach based on Mourre's theory and Sigal-Soffer propagation estimates
(see \cite{BDM} and \cite{BeA} for an alternative presentation of 
the stationary problem). The
derivation of reduced models for the steady state problem has been
developed in \cite{BNP1}\cite{BNP2}\cite{Ni2} on the basis of
Helffer-Sj{\"o}strand analysis of resonances in \cite{HeSj}. This
asymptotic model elucidated the influence of the
geometry of the potential on the feasibility of hysteresis phenomena
already studied in  \cite{JLPS}\cite{PrSj}.
Numerical
applications have been carried out in realistic 
$Ga-As$ or $Si-SiO_{2}$ structure in \cite{BNP} and \cite{BFN}, showing
a good agreement with previous numerical simulations in \cite{BAPi} or
\cite{LKF} and finally predicting the possibility of exotic
bifurcation diagrams. From the modelling point
of view a difficulty comes from the phase-space description of the tunnel
effect which can be summarized with the importance in the asymptotic
nonlinear system of the asymptotic
value of the branching ratio
\begin{equation}
  \label{eq.branching}
t_{j}=\lim_{h\to 0}\frac{|\langle W^{h} \widetilde{\psi}^{h}_{-}(+k,.)\,,\,
  \Phi_{j}^{h}\rangle|^{2}}{4hk\Gamma_{j}^{h}}\,.
\end{equation}
In the above formula $z_{j}^{h}=E_{j}^{h}-i\Gamma_{j}^{h}$ is a resonance for the Hamiltonian
$H^{h}=-h^{2}\Delta+V-W^{h}$ with a semiclassical island $V$ and a
quantum well $W^{h}$, $\tilde{\psi}_{-}(\pm k, .)$ are the generalized
eigenfunctions for the filled well Hamiltonian
$\widetilde{H}^{h}=-h^{2}\Delta+V$ with a momentum $\pm k$ such that
$k^{2}\sim E_{j}^{h}$ and $\Phi_{j}^{h}$ is the $j$-th
eigenfunction of the Dirichlet Hamiltonian $H^{h}_{D}=-h^{2}\Delta+
V-W^{h}$ on some finite interval $(a,b)$.
 The imaginary part of the
resonance is given by the Fermi Golden rule proved in \cite{BNP2}
\begin{equation}
  \label{eq.FGint}
\Gamma_{j}^{h}(1+o(1))=
\frac{|\langle W^h\tilde\psi^h_-(+k,\cdot),\phi^h_j\rangle|^2}{4hk}
  +\frac{|\langle W^h\tilde\psi^h_-(-k,\cdot),\phi^h_j
    \rangle|^2}{4hk}\,,
\end{equation}
where the interaction with the continuous spectrum leading to a
resonance
contains two contributions from the left-hand side with $+k$ and
from the right-hand side with $-k$.\\

Our purpose is the derivation of reduced models for the dynamics of 
quantum nonlinear systems like it has been done in the stationary case
in  \cite{BNP1}\cite{BNP2}\cite{Ni2} with the following motto: The
(nonlinear) phenomena are governed by a finite number of resonant
states. As this was already explained in \cite{BNP1}, such a remark
dictates the scaling of the potential which leads to the small
parameter analysis and in the end to effective reduced models even
when  $h\sim 0.1$ or $0.3$ (see \cite{BNP}).\\
 For the dynamical problem, the time evolution of resonant states have
 to be considered possibly with a time-dependent potential. And it is
 known that this is a rather subtle point. 
Quantum resonances follow \underline{almost} but not exactly the
general intuition (see \cite{SimR}) and remain an inexhaustible
playground for mathematical analysis.
For example,  the exponential decay law is an approximation
which has a physical interpretation in terms of the
evolution of quasi-resonant states
(truncated resonant states which lie on the real $L^{2}$-space) and writes as
$$
e^{-itH^{h}}\psi_{qr,j}=e^{-itE_{j}}e^{-t\Gamma_{j}}\psi_{qr,j}+ R(t)\,,
$$
where the remainder term $R(t)$ is small only in the range of times scaled as $\frac{1}{\Gamma_{j}}$. A very accurate analysis of this
has been done in
\cite{GeSi}\cite{Ski1}\cite{Ski2}\cite{Ski3}\cite{KRW}\cite{KlRa}
and adiabatic results for slowly
varying potentials and for quasi-resonant states have been
obtained in \cite{Per}(see also \cite{SoWe} and \cite{ASFr} in a
similar spirit) on this range of time-scales.
On the other side the relation
$$
e^{-itH^{h}(\theta)}\psi_{j}=e^{-itE_{j}}e^{-t\Gamma_{j}}\psi_{j}
$$
holds without remainder terms when $\theta\in i\R_{+}$ parametrizes a
complex deformation of $H^{h}$ according to the general approach to
resonances (see
\cite{AgCo}\cite{BaCo}\cite{HiSi1}\cite{CFKS}\cite{HeSj}\cite{HeMa}). 
However, and this is well known within the analysis of
resonances, the deformed generator $iH^{h}(\theta)$ is not maximal
accretive although its spectrum lies in $\left\{\Real z\geq 0\right\}$
and no uniform estimates are available on $e^{-itH^{h}(\theta)}$\,. 
One of the two next strategies have to be chosen:
\begin{itemize}
\item Stay on the real space with quasi-resonant states, with uniform
  estimates of the semigroups, groups or dynamical systems (they are
  unitary) but with remainder term which can be neglected only on some
  parameter dependent range of time.
\item Consider the complex deformed situation and try to solve or
  bypass the defect of accretivity.
\end{itemize}
Because the remainder terms seemed hard to handle within the original
nonlinear problem and also because there may be multiple time scales to
handle, due to the nonlinearity or due to several resonances
involved in the nonlinear process, we chose the second one.\\

In a one dimensional problem the simplest approach is the complex
dilation method according to \cite{AgCo}\cite{BaCo}\cite{CFKS}. 
Since the possibly nonlinear potential with a compact
support has a limited regularity inside the interval $(a,b)$, this
deformation is done only outside this interval following an approach
already presented in \cite{SimE}. The dilation is defined according
to
\begin{equation}
U_{\theta}\psi(x)=\left\{
\begin{array}
[c]{ll}
e^{\frac{\theta}{2}}\psi(\smallskip e^{\theta}(x-b)+b),& x>b\\
\smallskip\psi(x),& x\in(a,b)\\
e^{\frac{\theta}{2}}\psi(\smallskip e^{\theta}(x-a)+a),& x<a\,,
\end{array}
\right.\label{U_theta.intro}
\end{equation}
and finally handled with $\theta\in i\R_{+}$. The conjugated Laplacian
is 
$$
U_{\theta}(-h^{2}\Delta)U_{\theta}^{-1}=-h^{2}e^{- 2\theta\,1_{\R\setminus(a,b)}(x)}\Delta\,,
$$
with the domain made of functions $u\in H^{2}(\R\setminus
\left\{a,b\right\})$ with the interface conditions
\begin{equation}
\left[
\begin{array}
[c]{l}
e^{-\frac{\theta}{2}}u(b^{+})=u(b^{-}
);\quad\medskip e^{-\frac{3\theta}{2}}u^{\prime
}(b^{+})=u^{\prime}(b^{-})\\
e^{-\frac{\theta}{2}}u(a^{-})=u(a^{+});\quad
e^{-\frac{3\theta}{2}}u^{\prime}(a^{-})=u^{\prime
}(a^{+})\,.
\end{array}
\right.  \label{BC_1.intro}
\end{equation}
This can be viewed as a singular version of the black-box formalism of
\cite{SjZw}. Additionally to the fact that this singular deformation
is convenient for the original model with potential barriers presented
with a discontinuous potential and with a nonlinear part inside
$(a,b)$, 
the obstruction to the accretivity
of $U_{\theta}(-h^{2}\Delta)U_{\theta}^{-1}$ is concentrated in two
boundary terms at $x=a$ and $x=b$ in
\begin{multline}
  \label{eq.boundterm}
 \Real \langle u\,,\, i U_{\theta}(-h^{2}\Delta)U_{\theta}^{-1}
  u\rangle_{L^{2}(\R)}
= \Real \left[ih^{2} (\bar u
    u')\big|_{a^{-}}^{b^{+}}(e^{-2\theta}-e^{-\frac{\bar
      \theta+3\theta}{2}})
 \right]
\\
 + h^{2}e^{-2\Real \theta}\sin(2\Imag \theta)\int_{\R\setminus[a,b]}|u'|^{2}~dx\,.
\end{multline}
Our strategy then relies on the introduction of artificial interface
conditions parametrized with $\theta_{0}$ which modify the operator
$-h^{2}\Delta$. The parameter $\theta_{0}$ is then 
chosen so that the above boundary term vanishes
when $\theta=\theta_{0}=i\tau$. The modified and deformed Hamiltonian 
$H^{h}_{\theta_{0}=i\tau}(\theta=i\tau)$ then generates a contraction semigroup
and uniform estimates are available for
$e^{-itH_{i\tau}(i\tau)}$ or for the dynamical system
$(U^{h}(t,s))_{t\geq s}$ for time-dependent potentials\,.\\
Hopefully this modification has a little effect on the Hamiltonian
$H^{h}=-h^{2}\Delta+ V-W^{h}$ and all the quantities involved in the
nonlinear
problem, with explicit estimates with respect to $\theta_{0}$ and
$h$. Indeed all the quantities and even the exponentially small ones
like $\Gamma_{j}$ or the ones appearing in the branching ratio
\eqref{eq.branching} experiment small relative variation with respect
to $\theta_{0}$ when $\theta_{0}=ih^{N_{0}}$ with $N_{0}\geq 5$.\\
In comparison with the modelling of artificial dissipative boundary
conditions in \cite{BKNR1}\cite{BKNR2}\cite{BNR}\cite{BMN}, 
our approach has the advantage of remaining
close to the initial quantum model. Such a comparison is valuable and
ensures the validity of numerical applications when the non-linear
bifurcation phenomena are very sensitive to small variations of the data.\\
Once the above comparison is done, it is checked that adiabatic
evolution for a slowly varying potential or equivalently for the
$\varepsilon$-dependent Cauchy problem
$$
i\varepsilon\partial_{t}u= H_{\theta_{0}}^{h}(\theta_0, t)u\,,\quad u_{t=0}=u_{0}\,,
$$
with some exponentially large time scale
$\frac{1}{\varepsilon}=e^{\frac{C}{h}}$,
is adapted from the general approach for the adiabatic evolution of
bound states of self-adjoint generators in
\cite{ASY}\cite{Nen}\cite{JoPf}.\\
Adiabatic dynamics have already been considered within the modelling of
out-of-equilibrium quantum transport in
\cite{CDNP}\cite{AEGS}\cite{AEGSS} playing with the continuous
spectrum with self-adjoint techniques.
Only partial results are known with non self-adjoint generators: in
\cite{Loc} only small time results are valid for resonances, in
\cite{NeRa} bounded generators are considered and in \cite{Sjo} a
general scheme for the the higher order construction of the adapted
projector is done but without time propagation estimates.
In \cite{Joy}, A.~Joye considered a general time-adiabatic evolution
for semigroups in Banach spaces under a fixed gap condition and with
analyticity assumptions: The exponential growth of the dynamical system
$\|S_{\varepsilon}(t,0)\|\leq e^{\frac{Ct}{\varepsilon}}$ is compensated by the ${\cal
  O}(e^{-\frac{c}{\varepsilon}})$ error of the adiabatic approximation
under analyticity assumptions (see \cite{Nen}\cite{JoPf}) and lead to a
total error $\mathcal{O}(\frac{Ct-c}{\varepsilon})$ which is small
when $t< c/C$.
 Our adaptation combines the uniform estimates due to the accretivity
of the modified and deformed Hamiltonian with the accurate resolvent
estimate provided by the accurate comparison with self-adjoint
problems (shape resonances result from the coupling of some Dirichlet
eigenvalues with a continuous spectrum).
Although, the error associated with the adiabatic evolution is estimated at the first order as an 
$\mathcal{O}(\varepsilon^{1-\delta})$, with $\delta>0$ as small as wanted,
 it is necessary to reconsider the higher-order method
 in \cite{Nen} or \cite{JoPf}, because
we work with small gaps (vanishing as $h\to 0$) and with
non self-adjoint operators. Finally note that the exponential time
scale is not necessarily related with the imaginary parts of resonances
and several resonant states with various life-time scales are taken
into account in our application.\\

The outline of the article is the following. 
\begin{itemize}
\item The artificial interface
conditions parametrized by $\theta_{0}$ are
 introduced in Section~\ref{Modified_laplacian}. 
With these new interface conditions $-\Delta$ is transformed into a
non self-adjoint operator conjugated with $-\Delta$, 
$W(\theta_{0})(-\Delta)W(\theta_{0})^{-1}$ with
$W(\theta_{0})=\Id_{L^{2}}+\mathcal{O}(\theta_{0})$. The case with a
potential is illustrated with numerical computations.
\item The functional analysis of the complex deformation parametrized
  with $\theta$ is done in Section~\ref{se.extcompl}. After
  introducing a Krein formula associated with the
  $(\theta_{0},\theta)$-dependent interface
  conditions, it mimics the
  standard approach to resonances summarized in \cite{CFKS}\cite{HeMa}
  but things have to be reconsidered for we start from an already non
  self-adjoint operator when $\theta_{0}\neq 0$. Assumptions on the
  time-dependent variations of the potential  which
  ensure the well-posedness of the dynamical systems are specified in
  the end of this section.
\item The small parameter  problem modelling quantum wells in a
  semiclassical island is introduced  in
  Section~\ref{se.expdecay}. Accurate exponential decay estimates are
  presented for the spectral problems reduced to $(a,b)$ making use of
  the fact that our operators are proportional to $-h^{2}\Delta$
  outside  $[a,b]$.
\item The Grushin problem leading to an accurate theory of resonances
  is presented in Section~\ref{se.grush}. There it is checked that the
  imaginary parts of the resonances, which are exponentially small,
 are little perturbed by the introduction of $\theta_{0}$-dependent
 artificial conditions. The conclusion of this section is that all the
 quantities involved in the Fermi Golden Rule \eqref{eq.FGint} have
 little relative variations with respect to $\theta_{0}$, even the
 exponentially small ones.
\item Accurate parameter-dependent resolvent estimates for the whole
  space problem are done in 
Section~\ref{se.wholeline}. Again the Krein formula for the
resolvent associated with the $(\theta_{0},\theta)$-dependent
interface conditions is especially useful.
\item The adiabatic evolution of resonances is really introduced in
  \ref{se.adiabev}. After specifying all the assumptions, the main
  result about this is stated in Theorem~\ref{th.adiabappl}.
\item The appendix contains various preliminary and sometimes
  well-known estimates, plus a variation of 
the general adiabatic theory concerned with non self-adjoint maximal
accretive operators in Section~\ref{se.adiabvar}
\end{itemize}

\section{\label{Modified_laplacian}Artificial interface conditions}

Modified Hamiltonians with artificial interface conditions are
introduced. It is checked that the effect of these artificial
conditions parametrized by $\theta_{0}\in \C$
is of order $|\theta_{0}|$ for all the quantities associated with 
the free Schr{\"o}dinger Hamiltonian $-h^{2}\Delta$ on the whole line. 
Instead of pursuing this analysis for
general Schr{\"o}dinger operators $-h^{2}\Delta +V$ we simply give a
numerical evidence of this stability with respect $\theta_{0}$\,.
\subsection{The modified Laplacian}
\label{se.modlap}
We consider a class of singular perturbations of the 1-D Laplacian, defined
through non self-adjoint boundary conditions in the extrema of a bounded
interval. For $\theta_{0}\in\mathbb{C}$ and $a,b\in\mathbb{R}$, with $a<b$ and
$b-a=L$, the Hamiltonian $H_{\theta_{0},0}^{h}$ is defined by
\begin{equation}
\left\{
\begin{array}
[c]{l}
\medskip D(H_{\theta_{0},0}^{h})=\left\{  u\in H^{2}(\mathbb{R}\backslash
\left\{  a,b\right\}  ):\left[
\begin{array}
[c]{l}
\smallskip e^{-\frac{\theta_{0}}{2}}u(b^{+})=u(b^{-});\ e^{-\frac{3}{2}
\theta_{0}}u^{\prime}(b^{+})=u^{\prime}(b^{-})\\
e^{-\frac{\theta_{0}}{2}}u(a^{-})=u(a^{+});\ e^{-\frac{3}{2}\theta_{0}
}u^{\prime}(a^{-})=u^{\prime}(a^{+})
\end{array}
\right.  \right\}\,, \\
H_{\theta_{0},0}^{h}u=-h^{2}\partial_{x}^{2}\,,
\end{array}
\right. \label{H}
\end{equation}
where $u(x^{+})$ and $u(x^{-})$ respectively denote the right and the left
limits of $u$ in $x$, while the notation $u^{\prime}$ is used for the first
derivative. When  $\left\vert \theta_{0}\right\vert $ is small, the
analysis of $H_{\theta_{0},0}^{h}$ follows by a direct comparison with
$H_{0,0}^{h}$ (coinciding with the usual Laplacian: $-h^{2}\Delta_{\mathbb{R}
}$). To fix this point, let us introduce the intertwining operator
$W_{\theta_{0}}$ defined through the integral kernel
\begin{equation}
W_{\theta_{0}}(x,y)=\int_{-\infty}^{+\infty}\psi
_{-}(k,x)e^{-i\frac{k}{h}y}~\frac{dk}{2\pi h}\,, \label{W}
\end{equation}
with $\psi_{-}(k,x)$ denoting the generalized eigenfunctions associated with our
model. These are described by the plane wave solutions to the equation
\begin{equation}
\left(  H_{\theta_{0},0}^{h}-k^{2}\right)  \psi_{-}(k,\cdot)=0\,.
\end{equation}
For $k>0$, one has
\begin{equation}
1_{(0,+\infty)}(k)\,\psi_{-}(k,x)=\left\{
\begin{array}
[c]{ll}
\medskip e^{i\frac{k}{h}x}+R(k)e^{-i\frac{k}{h}x}, & x<a\\
\medskip A(k)e^{i\frac{k}{h}x}+B(k)e^{-i\frac{k}{h}x}, & x\in(a,b)\\
T(k)e^{i\frac{k}{h}x}, & x>b \,.
\end{array}
\right. \label{psi_k}
\end{equation}
Since $\psi_{-}(k,x)$ fulfills the boundary conditions in (\ref{H}), the
explicit expression of the coefficients in (\ref{psi_k}) are
\begin{equation}
\left\{
\begin{array}
[c]{l}
\medskip A(k)=2\frac{c_{+}(\theta_{0})e^{-i\frac{k}{h}L}}{d(\theta_{0}
,k)}\,,\qquad B(k)=-2\frac{c_{-}(\theta_{0})e^{2i\frac{k}{h}a}e^{i\frac{k}{h}L}}{d(\theta
_{0},k)}\,,\\
T(k)=\frac{e^{-i\frac{k}{h}L}}{d(\theta_{0},k)}\left(c_{+}(\theta_{0})^2-c_{-}(\theta_{0})^2\right)\,,\qquad R(k)=\frac{-2ic_{+}(\theta_{0})c_{-}(\theta_{0})}{d(\theta
_{0},k)}e^{2i\frac{k}{h}a}\sin\left(\frac{kL}{h}\right)\,,
\end{array}
\right.  \label{coefficients1}
\end{equation}
with: $c_{+}(\theta_{0})=e^{\frac{\theta_{0}}{2}}+e^{\frac{3}{2}\theta_{0}}$,
$c_{-}(\theta_{0})=e^{\frac{\theta_{0}}{2}}-e^{\frac{3}{2}\theta_{0}}$ and
\begin{equation}
d(\theta_{0},k)=\det
\begin{pmatrix}
c_{+}(\theta_{0})e^{i\frac{k}{h}a} & c_{-}(\theta_{0})e^{-i\frac{k}{h}a}\\
c_{-}(\theta_{0})e^{i\frac{k}{h}b} & c_{+}(\theta_{0})e^{-i\frac{k}{h}b}
\end{pmatrix}
=c_{+}(\theta_{0})^2e^{-i\frac{k}{h}L}-c_{-}(\theta_{0})^2e^{i\frac{k}{h}L}\,. \label{d_k}
\end{equation}
For $k<0$, an analogous computation gives
\begin{equation}
1_{(-\infty,0)}(k)\,\psi_{-}(k,x)=\left\{
\begin{array}
[c]{ll}
\medskip\tilde{T}(k)e^{i\frac{k}{h}x}, & x<a\\
\medskip\tilde{A}(k)e^{i\frac{k}{h}x}+\tilde{B}(k)e^{-i\frac{k}{h}x}, & x\in(a,b)\\
e^{i\frac{k}{h}x}+\tilde{R}(k)e^{-i\frac{k}{h}x}, & x>b\,,
\end{array}
\right. \label{psi_(-k)}
\end{equation}
with
\begin{equation}
\left\{
\begin{array}
[c]{ll}
\medskip\tilde{A}(k)=A(-k)\,, & \tilde{B}(k)=e^{4i\frac{k}{h}a}e^{2i\frac{k}{h}L}B(-k)\,,\\
\tilde{T}(k)=T(-k)\,,&\tilde{R}(k)=e^{4i\frac{k}{h}a}e^{2i\frac{k}{h}L}R(-k)\,.
\end{array}
\right.  \label{coefficients2}
\end{equation}
In what follows we adopt the simplified notation
\begin{equation}
A(k,\theta_{0})=\left\{
\begin{array}
[c]{ll}
A(k),& k\geq 0\\
\tilde{A}(k),& k<0\,,
\end{array}
\right. \quad B(k,\theta_{0})=\left\{
\begin{array}
[c]{ll}
B(k),& k\geq 0\\
\tilde{B}(k)\,& k<0\,,
\end{array}
\right.
\end{equation}
and
\begin{equation}
T(k,\theta_{0})=\left\{
\begin{array}
[c]{ll}
T(k), & k\geq0\\
\tilde{T}(k),& k<0\,,
\end{array}
\right. \quad R(k,\theta_{0})=\left\{
\begin{array}
[c]{ll}
R(k), & k\geq0\\
\tilde{R}(k), & k<0\,.
\end{array}
\right.  \label{coefficients T,R}
\end{equation}

\begin{lemma}
\label{Lemma W} The operator $W_{\theta_0}$ defined by \eqref{W} verifies the expansion 
\begin{equation}
W_{\theta_{0}}-Id=\mathcal{O}(|\theta_{0}|) \label{W1}
\end{equation}
in operator norm.
\end{lemma}

\begin{proof}
According to (\ref{psi_k}), (\ref{psi_(-k)}) and to the definition (\ref{W}),
one can express the integral kernel $W_{\theta_{0}}(x,y)$ as follows
\begin{multline*}
W_{\theta_{0}}(x,y)=\int_{-\infty}^{+\infty}e^{i\frac{k}{h}(x-y)}\frac{dk}{2\pi h}\\
+1_{(a,b)}(x)\int_{-\infty}^{+\infty}\left(A(k,\theta_{0})-1\right)e^{i\frac{k}{h}(x-y)}\frac{dk}{2\pi h}+1_{(a,b)}(x)\int_{-\infty}^{+\infty}B(k,\theta_{0})e^{-i\frac{k}{h}(x+y)}\frac{dk}{2\pi h}\\
+1_{(-\infty,a)}(x)\int_{-\infty}^{0}\left(T(k,\theta_{0})-1\right)e^{i\frac{k}{h}(x-y)}\frac{dk}{2\pi h}+1_{(b,+\infty)}(x)\int_{0}^{+\infty}\left(T(k,\theta_{0})-1\right)e^{i\frac{k}{h}(x-y)}\frac{dk}{2\pi h}\\
+1_{(b,+\infty)}(x)\int_{-\infty}^{0}R(k,\theta_{0})e^{-i\frac{k}{h}(x+y)}\frac{dk}{2\pi h}+1_{(-\infty,a)}(x)\int_{0}^{+\infty}R(k,\theta_{0})e^{-i\frac{k}{h}(x+y)}\frac{dk}{2\pi h}\,.
\end{multline*}
The previous expression is rewritten in terms of operators:
\begin{multline*}
W_{\theta_{0}}-Id=1_{(a,b)}\mathcal{F}^{-1}\left(A(k,\theta_{0})-1\right)\mathcal{F}+1_{(a,b)}P
\mathcal{F}^{-1}B(k,\theta_{0})\mathcal{F}\\
+1_{(-\infty,a)}\mathcal{F}^{-1}1_{(-\infty,0)}\left(T(k,\theta_{0})-1\right)\mathcal{F}+1_{(b,+\infty)}\mathcal{F}^{-1}1_{(0,+\infty)}\left(T(k,\theta_{0})-1\right)\mathcal{F}\\
+1_{(b,+\infty)}P\mathcal{F}^{-1}1_{(-\infty,0)}R(k,\theta_{0})\mathcal{F}+1_{(-\infty,a)}P\mathcal{F}^{-1}1_{(0,+\infty)}R(k,\theta_{0})\mathcal{F}\,,
\end{multline*}
where $\mathcal{F}$ denotes the Fourier transform normalized as
$\mathcal{F}u(k)=\int_{\R}u(x)e^{-i\frac{k}{h}x}\frac{dx}{(2\pi h)^{1/2}}$, and $P$ denotes the parity operator: $Pu(x)=u(-x)$.\\
Since the operators $P$, $\mathcal{F}$, $\mathcal{F}^{-1}$ and multiplication by the characteristic function of a set are uniformly bounded with respect to $\theta_0$, we get:
\[
\|W_{\theta_{0}}-Id\|\leq C\left( \|A(k,\theta_{0})-1\|_{L^{\infty}(\R)} + \|B(k,\theta_{0})\|_{L^{\infty}(\R)}+\|T(k,\theta_{0})-1\|_{L^{\infty}(\R)} + \|R(k,\theta_{0})\|_{L^{\infty}(\R)}  \right)\,.
\]
From \eqref{coefficients2}, it is enough to estimate for $k>0$ the terms at the r.h.s of the inequality above to get the $L^{\infty}(\R)$ bounds. Moreover, from the definition of the coefficients $c_{-}(\theta_{0})$, $c_{+}(\theta_{0})$ and $d(\theta_{0},k)$, we have:
\[
c_{-}(\theta_{0})=\mathcal{O}(\vert\theta_{0}\vert), \quad c_{+}(\theta_{0})=2+\mathcal{O}(\vert\theta_{0}\vert), \quad d(\theta_{0},k)=4e^{-\frac{ikL}{h}}+\mathcal{O}(\vert\theta_{0}\vert)\,,
\]
where the upper bound of $\mathcal{O}(\vert\theta_{0}\vert)$ holds with a universal constant. The previous equation gives $\vert d(\theta_{0},k)\vert \geq 1$ when $\vert\theta_{0}\vert$ is small enough, and using \eqref{coefficients1}, this implies:
\[
\vert A(k,\theta_0)-1 \vert = \frac{\left\vert 2c_{+}(\theta_{0})e^{-\frac{ikL}{h}}-d(\theta_{0},k) \right\vert}{\left\vert d(\theta_{0},k) \right\vert}= \frac{\left\vert \mathcal{O}(\vert\theta_{0}\vert) \right\vert}{\left\vert d(\theta_{0},k) \right\vert} \leq C\vert\theta_{0}\vert\,,
\]
\[
\vert T(k,\theta_0)-1 \vert = \frac{\left\vert e^{-\frac{ikL}{h}}\left(c_{+}(\theta_{0})^2-c_{-}(\theta_{0})^2\right)-d(\theta_{0},k) \right\vert}{\left\vert d(\theta_{0},k) \right\vert}
= \frac{\left\vert 2c_{-}(\theta_{0})^2\sin\left(\frac{kL}{h}\right) \right\vert}{\left\vert d(\theta_{0},k) \right\vert}=\frac{\left\vert \mathcal{O}(\vert\theta_{0}\vert^2) \right\vert}{\left\vert d(\theta_{0},k) \right\vert} \leq C\vert\theta_{0}\vert\,,
\]
\[
\vert B(k,\theta_0)\vert =\frac{\left\vert 2c_{-}(\theta_{0}) \right\vert}{\left\vert d(\theta_{0},k) \right\vert} =\frac{\left\vert \mathcal{O}(\vert\theta_{0}\vert) \right\vert}{\left\vert d(\theta_{0},k) \right\vert} \leq C\vert\theta_{0}\vert\,,
\]
\[
\vert R(k,\theta_0)\vert =\frac{\left\vert 2c_{+}(\theta_{0})c_{-}(\theta_{0})\sin\left(\frac{kL}{h}\right)\right\vert}{\left\vert d(\theta_{0},k) \right\vert} =\frac{\left\vert\left(2+\mathcal{O}(\vert\theta_{0}\vert)\right) \mathcal{O}(\vert\theta_{0}\vert) \right\vert}{\left\vert d(\theta_{0},k) \right\vert} \leq C\vert\theta_{0}\vert\,.
\]
\end{proof}

According to the result of Lemma \ref{Lemma W}, for $\left\vert \theta
_{0}\right\vert $ small enough, $W_{\theta_{0}}$ is invertible; in particular
one has
\begin{equation}
W_{\theta_{0}}=1+\mathcal{O}(\vert\theta_{0}\vert)\,,\qquad W_{\theta_{0}}^{-1}
=1+\mathcal{O}(\vert\theta_{0}\vert)\,.\label{expansion_0}
\end{equation}
Then, it follows from the definition (\ref{W}) that $H_{\theta_{0},0}^{h}$ are
$H_{0,0}^{h}$ conjugated operators with
\begin{equation}
H_{\theta_{0},0}^{h}=W_{\theta_{0}}H_{0,0}^{h}W_{\theta_{0}}^{-1}\,.\label{conjugation}
\end{equation}
This relation allows to discuss the spectral and the dynamical properties
related to $H_{\theta_{0},0}^{h}$ for small values of the parameter $\theta_0$.

\begin{proposition}
\label{Prop.1}There exists $c>0$ such that: for any $\theta_{0}$ with
$\left\vert \theta_{0}\right\vert\leq c$, the following property holds:
\medskip\newline1) The operator
$H_{\theta_{0},0}^{h}$ has only essential spectrum defined by $\sigma
_{ess}(H_{\theta_{0},0}^{h})=\mathbb{R}_{+}$.\newline2) The semigroup
associated with $H_{\theta_{0},0}^{h}$ is uniformly bounded in time and the expansion
\begin{equation}
e^{-itH_{\theta_{0},0}^{h}}=e^{-itH_{0,0}^{h}}+\mathcal{O}(\vert\theta_{0}\vert)\label{expansion}
\end{equation}
holds uniformly in $t\in\R$.
\end{proposition}

\begin{proof}
1) According to (\ref{conjugation}), one has
\[
\left(  H_{\theta_{0},0}^{h}-z\right)  ^{-1}=W_{\theta_{0}}\left(  H_{0,0}
^{h}-z\right)  ^{-1}W_{\theta_{0}}^{-1},
\]
with $W_{\theta_{0}},W_{\theta_{0}}^{-1}$ bounded in $L^{2}(\mathbb{R})$. This
implies: $\sigma(H_{\theta_{0},0}^{h})=\sigma_{ess}(H_{0,0}^{h})=\mathbb{R}
_{+}$.

2) Since $H_{\theta_{0},0}^{h}$ are $H_{0,0}^{h}$ conjugated by $W_{\theta
_{0}}$, we have
\[
e^{-itH_{\theta_{0},0}^{h}}W_{\theta_{0}}=W_{\theta_{0}}e^{-itH_{0,0}^{h}}.
\]
For $\left\vert \theta_{0}\right\vert $ small, one can use the expansions
(\ref{expansion_0}) to write
\[
e^{-itH_{\theta_{0},0}^{h}}=W_{\theta_{0}}e^{-itH_{0,0}^{h}}W_{\theta_{0}
}^{-1}=\left(  1+\mathcal{O}(\vert\theta_{0}\vert)\right)e^{-itH_{0,0}^{h}}\left(  1+\mathcal{O}(\vert\theta_{0}\vert)\right)  .
\]
Recalling that $H_{0,0}^{h}$ is self-adjoint (it coincides with the usual
Laplacian), and $\left(e^{-itH_{0,0}^{h}}\right)_{t\in\R}$ is a unitary group, one has
\[
e^{-itH_{\theta_{0},0}^{h}}=e^{-itH_{0,0}^{h}}+\mathcal{O}(\vert\theta_{0}\vert).
\]
\end{proof}

\begin{remark}
It is worthwhile to notice that $H_{\theta_{0},0}^{h}$ is not self-adjoint
(excepting for $\theta_{0}=0$) neither accretive, since
\[
\operatorname{Re}\left\langle u,iH_{\theta_{0},0}^{h}u\right\rangle
=h^2\operatorname{Im}\left\{  \left[  \bar{u}(a^{-})u^{\prime}(a^{-})-\bar
{u}(b^{+})u^{\prime}(b^{+})\right]  \left(  1-e^{-\frac{3\theta_{0}+\bar
{\theta}_{0}}{2}}\right)  \right\}  ,\qquad u\in D(H_{\theta_{0},0}^{h})
\]
has not a fixed sign. Thus, it is not possible to use standard arguments to
state that $e^{-itH_{\theta_{0},0}^{h}}$ is a contraction. Nevertheless, for
small values of the parameter $\theta_{0}$, the result of Proposition
\ref{Prop.1} allows to control the operator norm of $e^{-itH_{\theta_{0}
,0}^{h}}$ uniformly in time, and states that the time evolution generated by
$H_{\theta_{0},0}^{h}$ is close to the one associated with the usual
Laplacian $H_{0,0}^{h}$.
\end{remark}

Spectral properties of Hamiltonians obtained as singular perturbations of
$H_{\theta_{0},0}^{h}$ can be discussed using standard results in spectral
analysis adapted to this non self-adjoint case.

\begin{lemma}
\label{Lemma 2}Let $\mathcal{V}=\mathcal{V}_{1}+\mathcal{V}_{2}$, with $\mathcal{V}_{1}\in L^{\infty}((a,b))$
and $\mathcal{V}_{2}$ a bounded measure supported in $U\subset\subset\left(  a,b\right)
$. Then: $\sigma_{ess}(H_{\theta_{0},0}^{h}+\mathcal{V})=\sigma_{ess}
(H_{\theta_{0},0}^{h})$.
\end{lemma}

\begin{proof}
The proof follows from the first point of Corollary~\ref{cor_krein} below in the case $\theta=0$.
\end{proof}

\subsection{Numerical computation of the time propagators for small $|\theta_0|$}

This part is devoted to the numerical comparison of the propagators
$e^{-itH_{\theta_{0},V}^{h}}$ and $e^{-itH_{0,V}^{h}}$ where $V$ is a locally supported perturbation of $H_{\theta_{0},0}^{h}$ with
\[
H_{\theta_{0},V}^{h}=H_{\theta_{0},0}^{h}+V, \quad \textrm{and} \quad V\in L^{\infty}((a,b))\,.
\]
Using discrete time dependent transparent boundary
conditions for the Schr{\"o}dinger equation, it is possible to compute the
propagator $e^{-itH_{0,V}^{h}}$ with a Crank-Nicolson scheme, see 
\cite{ErAr}\cite{Arn}\cite{Pin}\cite{AABCMS}. To compute the propagator $e^{-itH_{\theta_{0},V}
^{h}}$, the key point is to integrate the boundary conditions in \eqref{H} in
the resolution in a way which preserves the stability. This is performed by
integrating the boundary conditions in the finite difference discretization of
the Laplacian at the points $a$ and $b$.
\newline To simplify the presentation,
we suppose temporarily that the interface conditions occur only at $0$. So we want
to write a modified discretization of the operator $\frac{d^{2}}{dx^{2}}$ with
the condition
\begin{equation}
\label{BC_0}\left\{
\begin{array}
[c]{l}
e^{-\frac{\theta_{0}}{2}}u(0^{+})=u(0^{-})\\
e^{-\frac{3}{2}\theta_{0}}u^{\prime}(0^{+})=u^{\prime}(0^{-})\,.
\end{array}
\right.
\end{equation}
For a given mesh size $\Delta x$, we introduce the discretization of
$\mathbb{R}$: $x_{j}=j\Delta x$ with $j\in\mathbb{Z}$. For $j\neq0$, the
number $u_{j}$ will denote the approximation of $u(x_{j})$, while $u_{0}$ will
denote the approximation of $u(0^{-})$ and $u_{0}^{+}$ will denote the
approximation of $u(0^{+})$. If the function $u$ is regular on $\mathbb{R}
^{*}$, we can use the usual finite difference approximation
\begin{equation}
\label{eq.diff}\left(  \frac{d^{2}}{dx^{2}}u\right)  _{j}=\frac{u_{j-1}
-2u_{j}+u_{j+1}}{\Delta x^{2}}\,,
\end{equation}
for $j\notin\{-1,0,1\}$. Due to the regularity constraint, this approximation
is written correctly for $j=-1$, and respectively for $j=1$, only when
considering the continuous extension of the function from the left, and
respectively from the right, which leads to
\[
\left(  \frac{d^{2}}{dx^{2}}u\right)  _{-1}=\frac{u_{-2}-2u_{-1}+u_{0}}{\Delta
x^{2}}\,, \quad\quad\left(  \frac{d^{2}}{dx^{2}}u\right)  _{1}=\frac{u_{0}
^{+}-2u_{1}+u_{2}}{\Delta x^{2}}\,.
\]
With the first relation in \eqref{BC_0}, the approximation at $j=1$ is
\begin{equation}
\label{eq.diff_mod1}\left(  \frac{d^{2}}{dx^{2}}u\right)  _{1}=\frac
{e^{\frac{\theta_{0}}{2}}u_{0}-2u_{1}+u_{2}}{\Delta x^{2}}\,.
\end{equation}
At $j=0$, due to the possible discontinuity of a function $u$ verifying
\eqref{BC_0}, we define $u^{-}$ on $\mathbb{R}$ as a regular continuation of
$u\vert_{(-\infty,0)}$. More precisely $u^{-}\in C^{2}(\mathbb{R})$ is such
that $u^{-}=u$ on $(-\infty,0)$ and we get the following approximation at
$j=0$
\begin{equation}
\label{eq.diff_mod0_temp}\left(  \frac{d^{2}}{dx^{2}}u\right)  _{0}
=\frac{u_{-1}-2u_{0}+u_{1}^{-}}{\Delta x^{2}}\,.
\end{equation}
This method corresponds to the introduction of a fictive point $u_{1}^{-}$
which allows to write the finite difference approximation for $\frac{d^{2}
}{dx^{2}}$ and to calculate $u^{\prime}(0^{-})$ and $u^{\prime}(0^{+})$ in
\eqref{BC_0} by using the same points of the space grid
\[
\left\{
\begin{array}
[c]{l}
e^{-\frac{\theta_{0}}{2}}u_{0}^{+}=u_{0}\\
e^{-\frac{3}{2}\theta_{0}}(u_{1}-u_{0}^{+})=u_{1}^{-}-u_{0}\,.
\end{array}
\right.
\]
The resolution of the system above gives: $u_{1}^{-}=(1-e^{-\theta_{0}}
)u_{0}+e^{-\frac{3}{2}\theta_{0}}u_{1}$ and \eqref{eq.diff_mod0_temp} becomes
\begin{equation}
\label{eq.diff_mod0}\left(  \frac{d^{2}}{dx^{2}}u\right)  _{0}=\frac
{u_{-1}-(1+e^{-\theta_{0}})u_{0}+e^{-\frac{3}{2}\theta_{0}}u_{1}}{\Delta
x^{2}}\,.
\end{equation}
Therefore, from the boundary conditions in \eqref{H}, the scheme to compute
the propagator $e^{-itH_{\theta_{0},V}^{h}}$ is obtained by using the modified
Laplacian corresponding to the application of \eqref{eq.diff_mod1} and
\eqref{eq.diff_mod0} at $x=a$ and $x=b$. When $\theta_{0}$ is small, the
equations \eqref{eq.diff_mod1} and \eqref{eq.diff_mod0} approximate well the
usual finite difference equation \eqref{eq.diff}, therefore the solution
$e^{-itH_{\theta_{0},V}^{h}}$ will be close to the solution $e^{-itH_{0,V}^{h}
}$, as expected.\newline After the change of variable $x^{\prime}=\frac
{x-a}{\ell}-1$, where $\ell=\frac{b-a}{2}$, the problem for $e^{-itH_{\theta_{0}
,V}^{h}}$ is the following
\begin{equation}
\label{eq.pb_resnum}\left\{
\begin{array}
[c]{ll}
i\partial_{t}u(t,x)=\left[  -\frac{h^{2}}{\ell^{2}}\partial_{x}^{2}+\tilde
{V}(x)\right]  u(t,x), & t>0,x\in\mathbb{R}\setminus\{-1,1\}\\
\ e^{\frac{\theta_{0}}{2}}u(t,-1^{+})=u(t,-1^{-}),\quad e^{\frac{3}{2}
\theta_{0}}\partial_{x}u(t,-1^{+})=\partial_{x}u(t,-1^{-}), & t>0\\
e^{-\frac{\theta_{0}}{2}}u(t,1^{+})=u(t,1^{-}),\quad e^{-\frac{3}{2}\theta
_{0}}\partial_{x}u(t,1^{+})=\partial_{x}u(t,1^{-}), & t>0\\
u(0,x) = u_{I}(x), & x\in\mathbb{R}\,,
\end{array}
\right.
\end{equation}
where $\tilde{V}(x)=V((x+1)\ell+a)$ and $u_{I}\in C^{\infty}(\mathbb{R})$ is the
initial data. The resolution will be performed on the bounded interval
$[-5,5]$ using homogeneous transparent boundary conditions valid when
$\supp u_{I}\subset\subset(-5,5)$.\newline Set $\Delta x=\frac{1}
{J}$ and consider the uniform grid points $x_{j}=j\Delta x$ for $j\in
\{-5J,\ldots,5J\}$. Then using \eqref{eq.diff_mod1} and \eqref{eq.diff_mod0},
the Crank-Nicolson scheme for the modified Hamiltonian is obtained from the
Crank-Nicolson scheme in \cite{Arn}\cite{ErAr} by replacing the usual discrete Laplacian
by the modified discrete Laplacian defined below
\[
\Delta_{\theta_0}u_{j} = \frac{1}{\Delta x^{2}}\left\{
\begin{array}
[c]{ll}
u_{j-1}-(1+e^{\mp\theta_{0}})u_{j}+e^{\mp\frac{3}{2}\theta_{0}}u_{j+1}, &
\text{ if }~j=\pm J\\
e^{\pm\frac{\theta_{0}}{2}}u_{j-1}-2u_{j}+u_{j+1}, & \text{ if }~j=\pm J+1\\
u_{j-1}-2u_{j}+u_{j+1}, & \text{ else}\,.
\end{array}
\right.
\]
The discrete transparent boundary conditions at $x=-5$ and $x=5$ are those
used for $e^{-itH_{0,V}^{h}}$ in \cite{Arn}\cite{ErAr}.

For a given time step $\Delta t$ and for $\theta_{0}
=i\mathop{\mathrm{Im}}~\theta_{0}$, we present some comparison of the
numerical solution $u_{\theta_{0}}^{n}$ to the system \eqref{eq.pb_resnum},
given at time $t^{n}=n\Delta t$ by the scheme described above, with the
numerical solution $u^{n}$ to the reference problem, computed by taking
$\theta_{0}=0$. In particular, the numerical parameters are the following:
$\ell=1$, $h=0.03$, $J=30$, $\Delta t = 0.8$, and the comparison is realized with
the initial condition equal to the wave packet
\[
u_{I}(x) = \exp\left(  -\frac{(x-x_{0})^{2}}{2\sigma^{2}} + ik(x-x_{0})
\right)\,,
\]
where $\sigma=0.2$, $k=\frac{2\pi}{8\Delta x}$ and the center $x_{0}$ will be specified in each simulation.

Three simulations were performed corresponding to different values of the
potential $V$ and of the center $x_{0}$. The first test was realized with
$V=0$ and $x_{0}=-3$. Although the comparison presented here can be extented
to more general potentials, the two other tests were realized in the case
where $V$ is a non trivial barrier potential
\[
V=V_{0}I_{(a,b)}\,,
\]
where $V_{0}=0.8$: for this potential one test was realized with an initial
condition localized at the left of $(-1,1)$ by taking $x_{0}=-3$, and the
second with an initial condition localized in $(-1,1)$ by taking $x_{0}=0$.
The solution $u_{\theta_{0}}^{n}$ is represented, at different time $t$, next
to the reference solution $u^{n}$ in the Figures \ref{fig.evlibre},
\ref{fig.evbarg} and \ref{fig.evbarm}, for the fixed small value $\theta
_{0}=0.09i$. We remark that $u_{\theta_{0}}^{n}$ has the same qualitative
behaviour than the reference solution.

In the case $V=0$, the solution $u_{\theta_{0}}^{n}$ corresponds to an
incoming function from the left which goes near the domain $(-1,1)$. When time
grows, it crosses the interface points and leaves the domain.

\begin{figure}[ptb]
\begin{center}
\begin{tabular}{cc}
\includegraphics[width=0.40\linewidth]{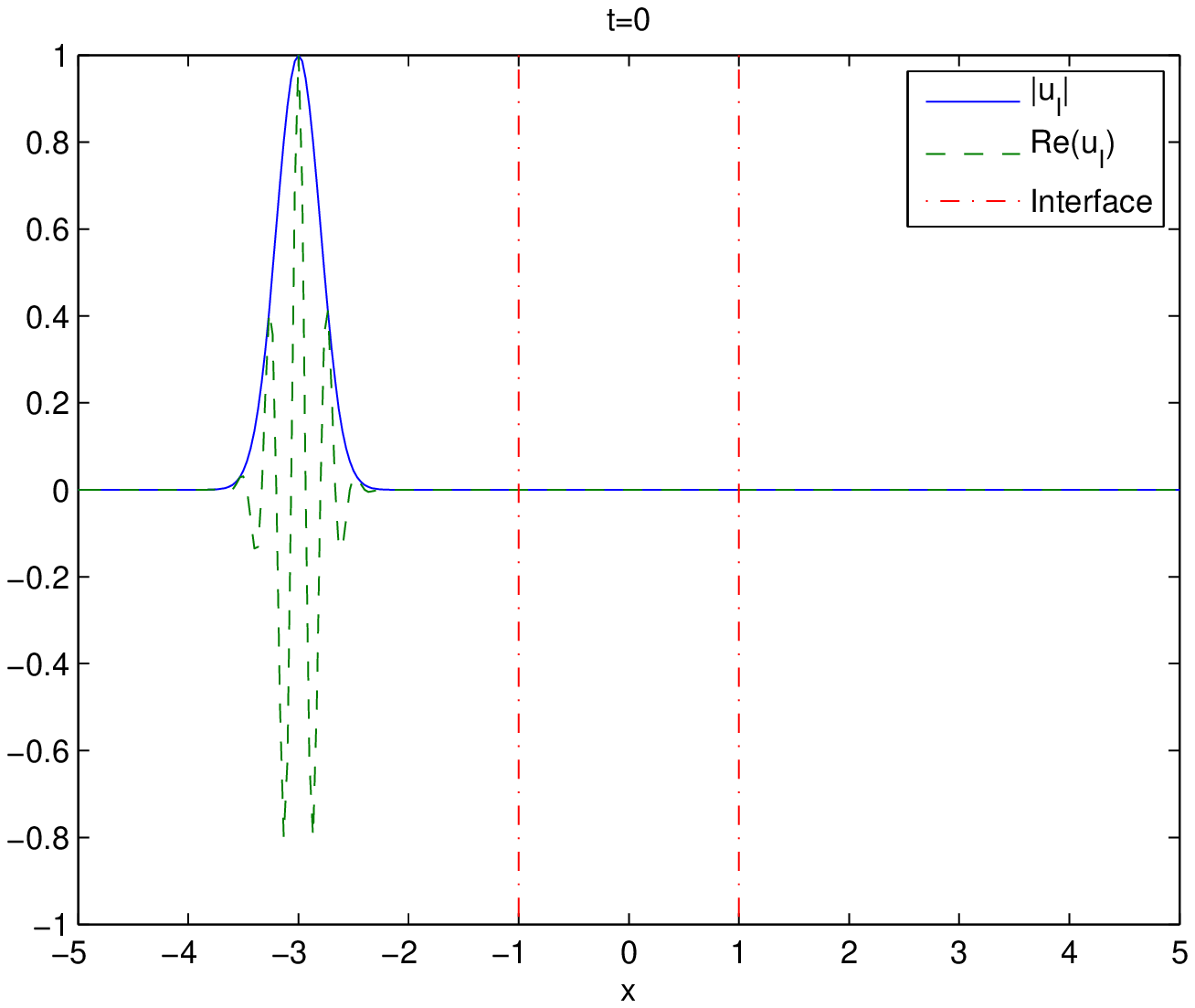}&\includegraphics[width=0.40\linewidth]{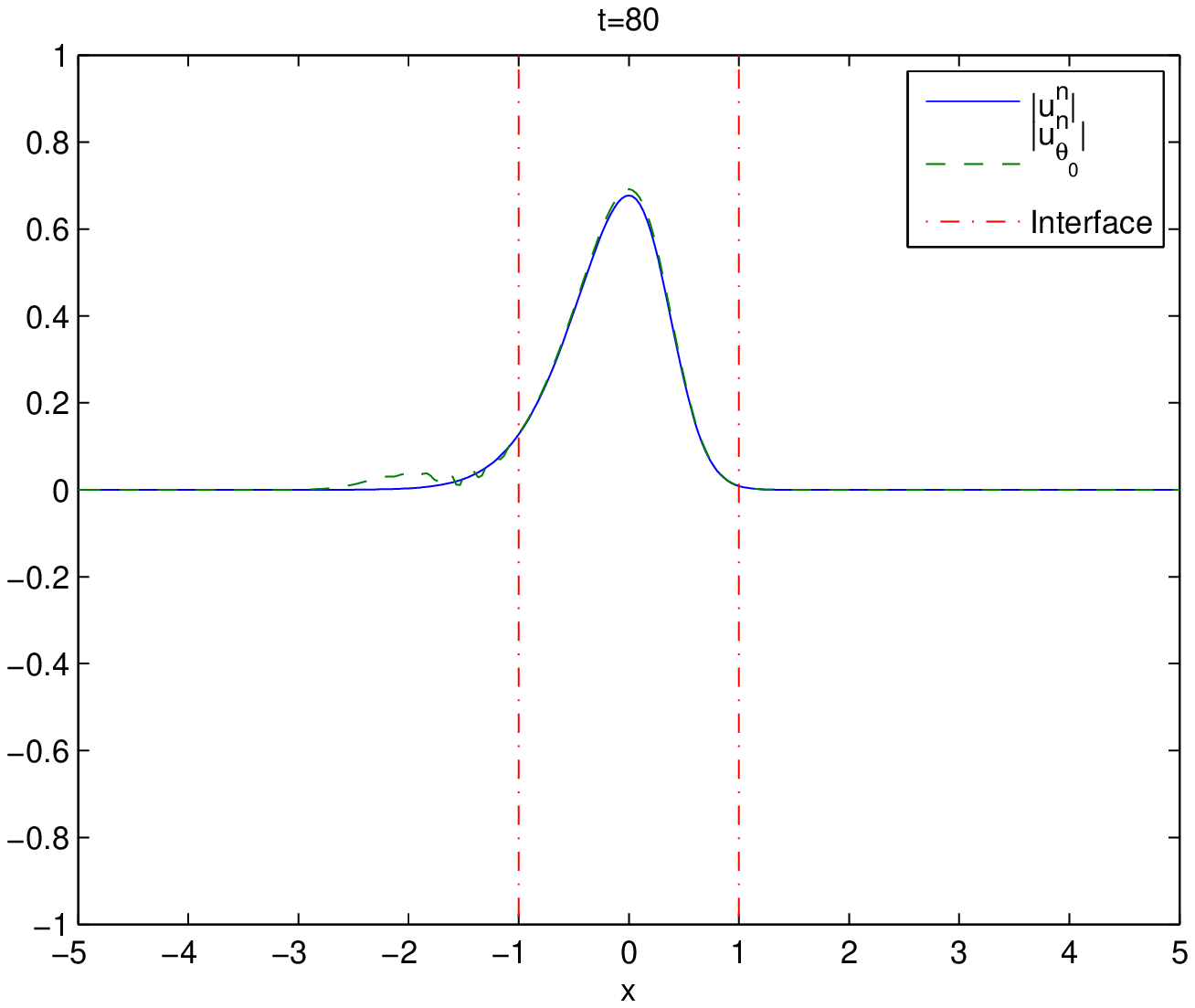}\\
\includegraphics[width=0.40\linewidth]{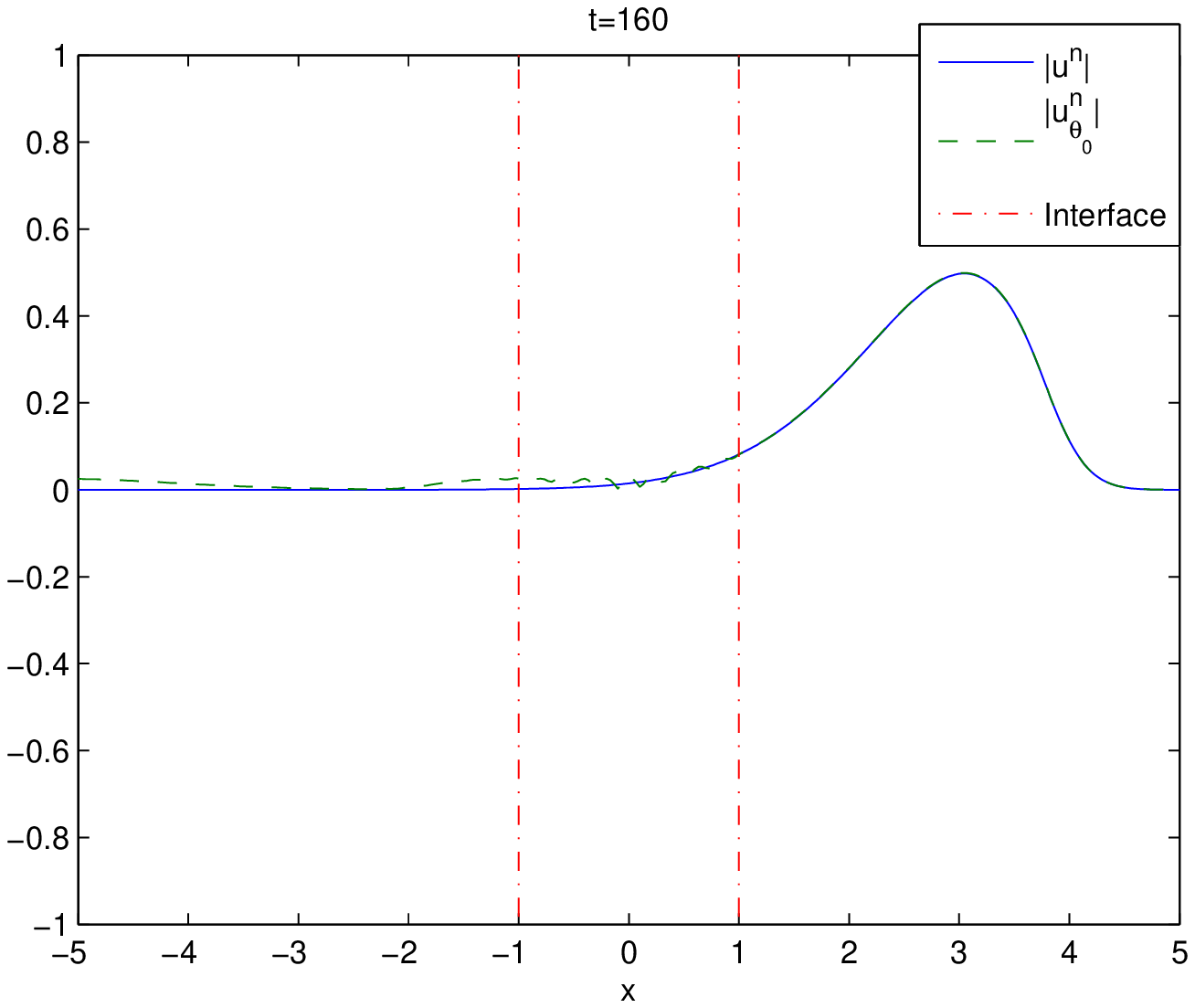}&
\includegraphics[width=0.40\linewidth]{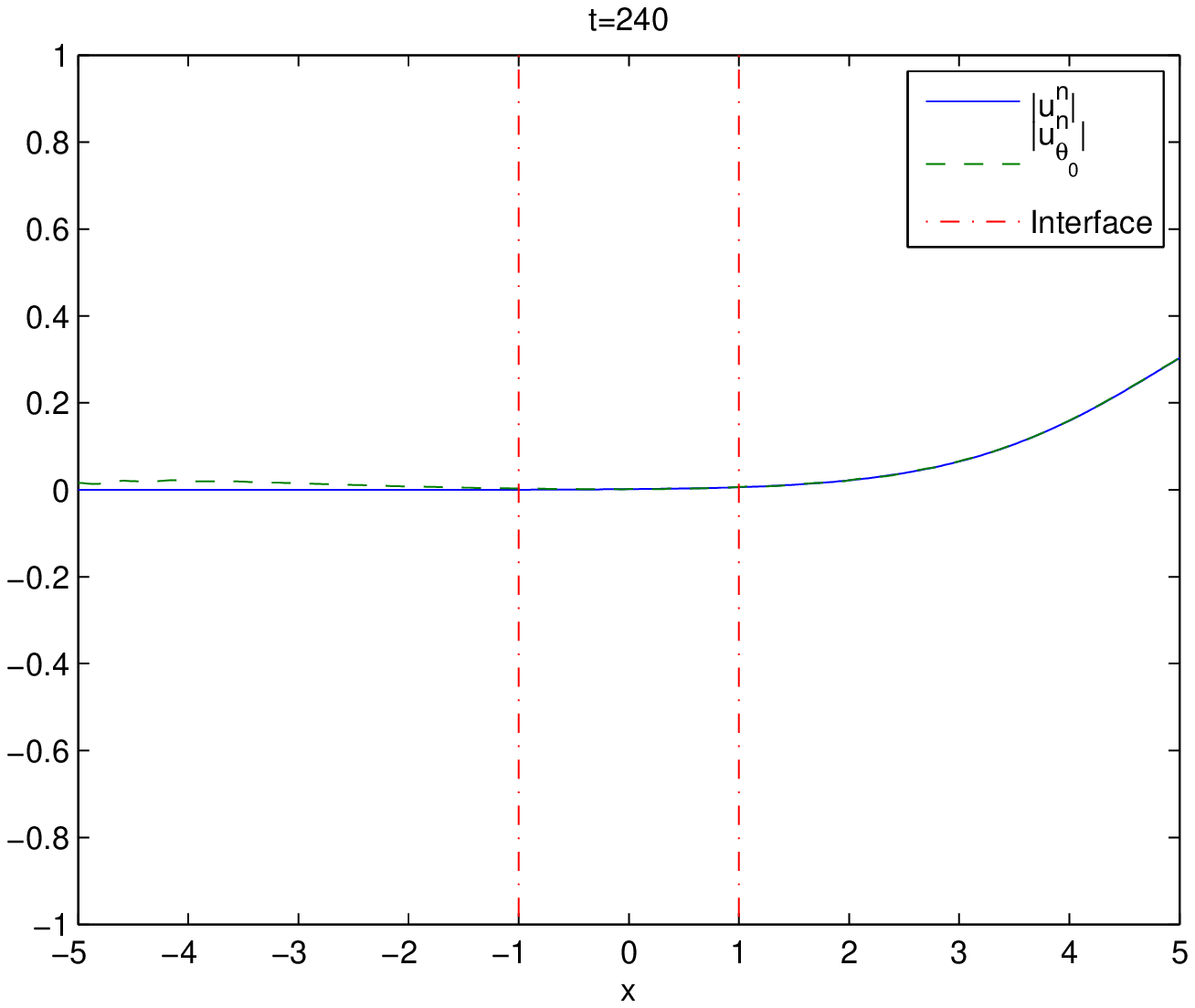}
\end{tabular}
\end{center}
\caption{Case $V=0$: modulus of the functions $u_{\theta_{0}}^{n}$ and $u^{n}$
at different time $t$ with $x_{0}=-3$ and $\theta_{0}=0.09i$}
\label{fig.evlibre}
\end{figure}

In the case of the barrier potential with $x_{0}=-3$, the solution is splitted
in two parts: a first one which passes through the barrier; and a second one
which is reflected and goes out of the domain.

\begin{figure}[ptb]
\begin{center}
\begin{tabular}{cc}
\includegraphics[width=0.40\linewidth]{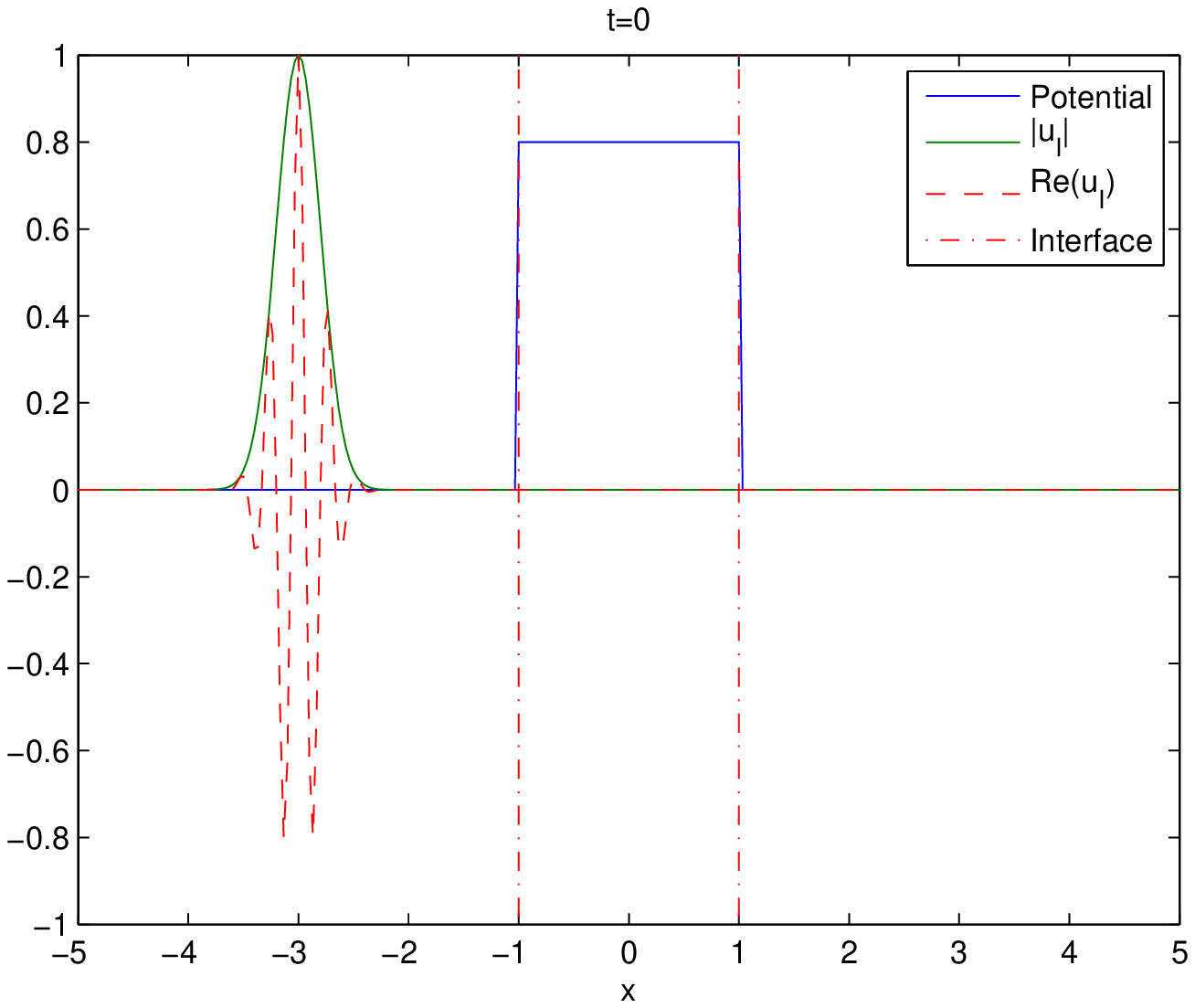}&
\includegraphics[width=0.40\linewidth]{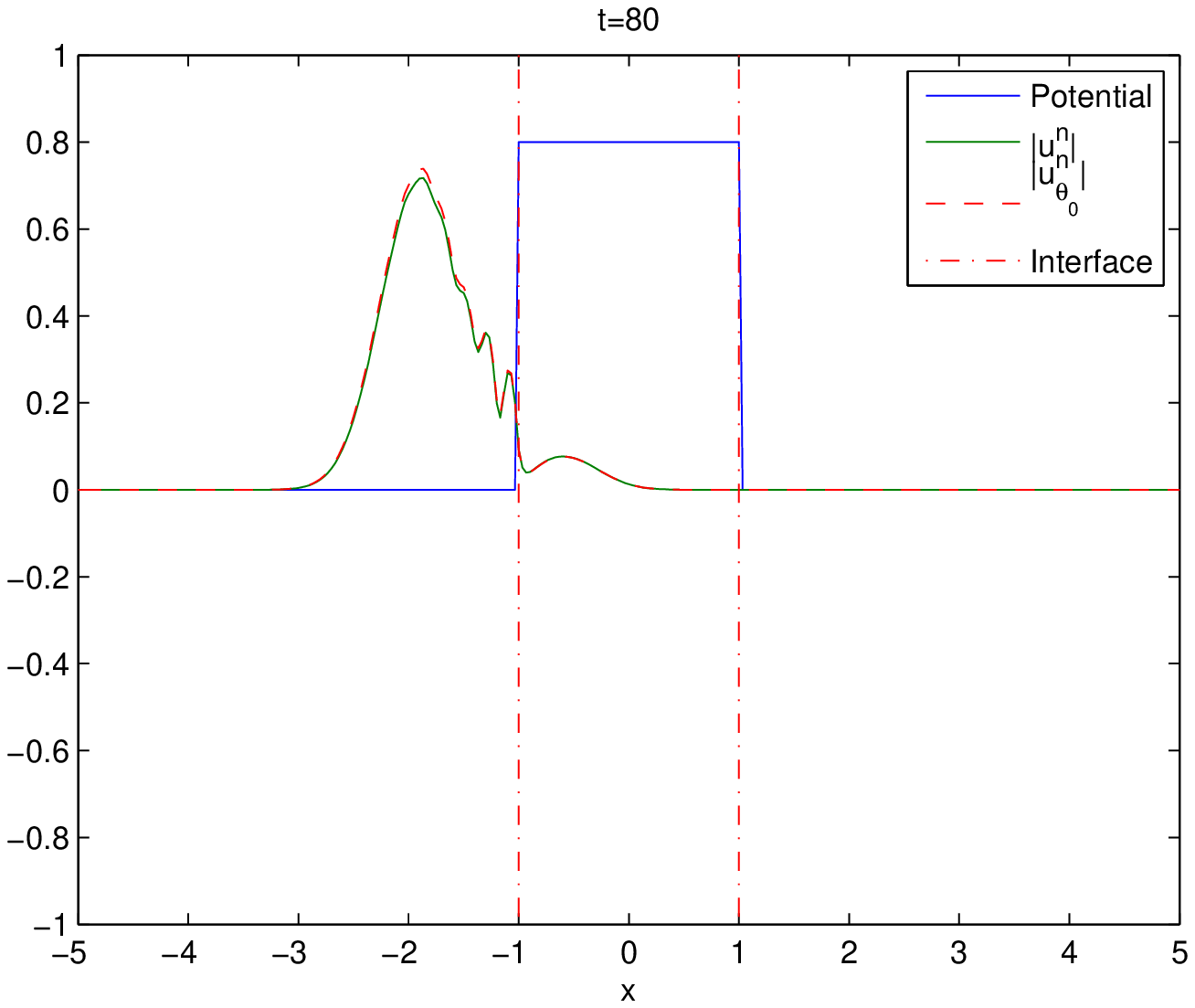}\\
\includegraphics[width=0.40\linewidth]{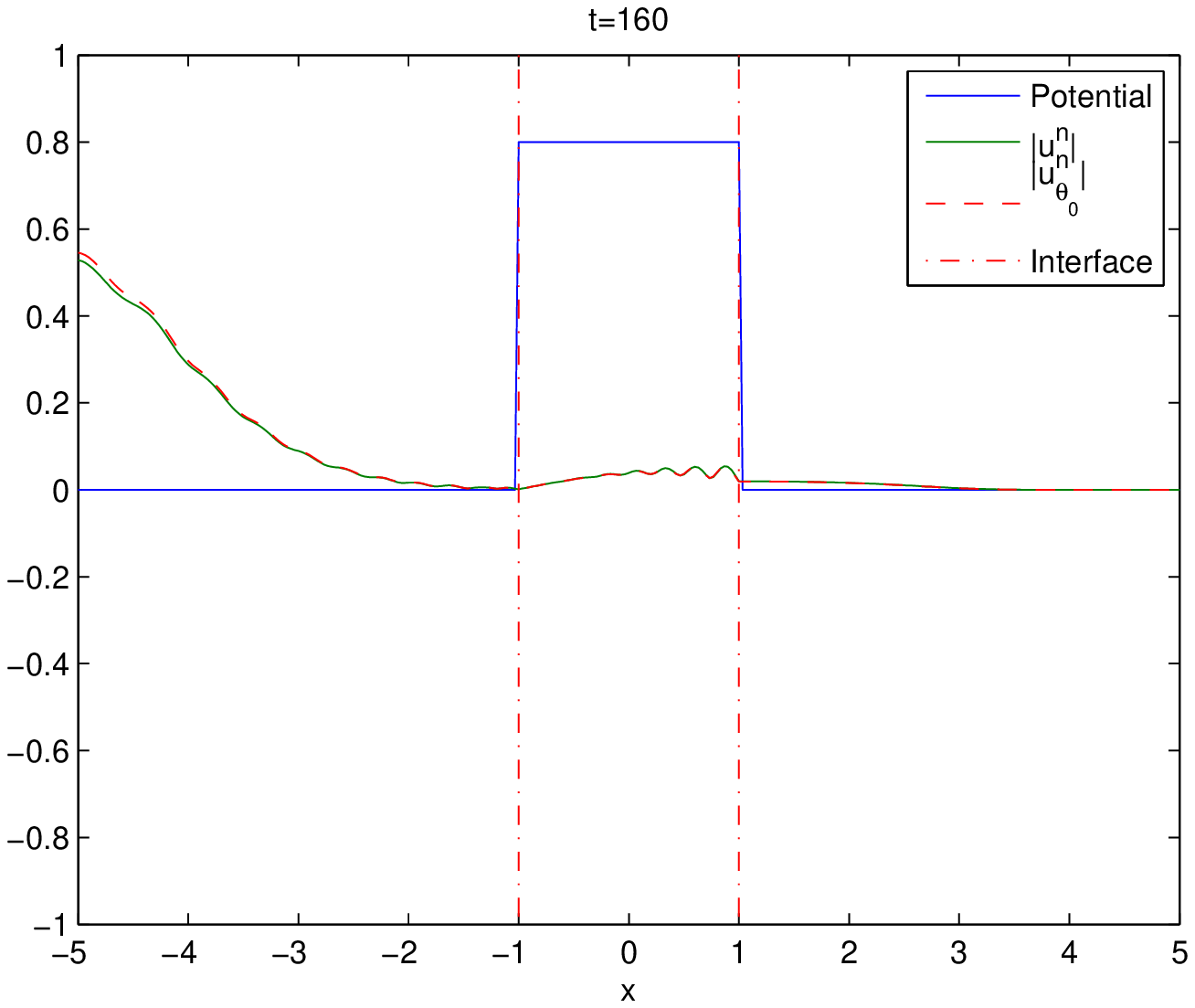}&
\includegraphics[width=0.40\linewidth]{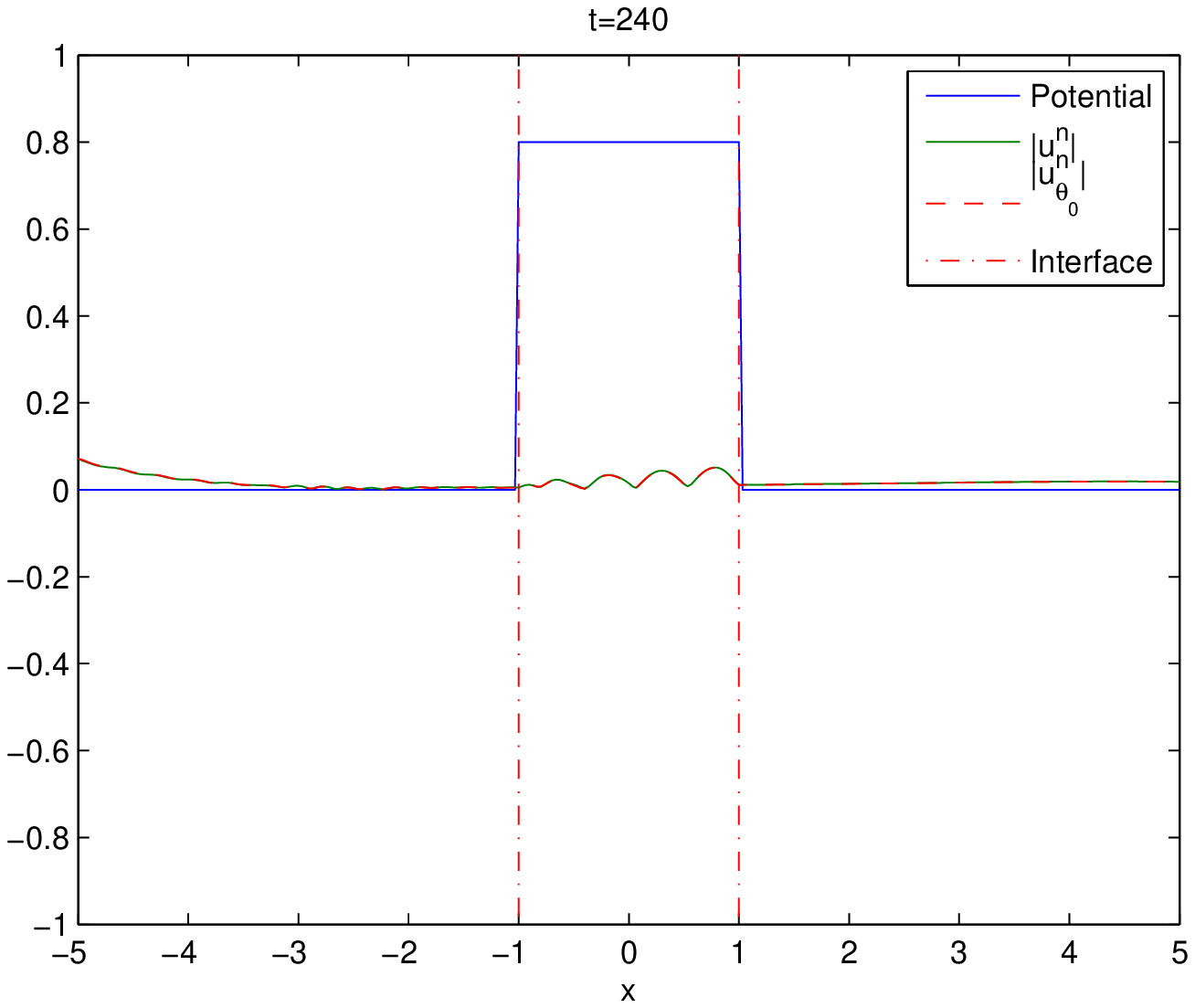}
\end{tabular}
\end{center}
\caption{Case of a barrier potential: modulus of the functions $u_{\theta_{0}
}^{n}$ and $u^{n}$ at different time $t$ with $x_{0}=-3$ and $\theta
_{0}=0.09i$}
\label{fig.evbarg}
\end{figure}

In the case of the barrier potential with $x_{0}=0$, it appears that the wave
packet is splitted in two outgoing parts: one which leaves the barrier on the
left and the second on the right. The part on the right is more important and
goes out faster, it is due to the sign of the wave vector $k$.

\begin{figure}[ptb]
\begin{center}
\begin{tabular}{cc}
\includegraphics[width=0.40\linewidth]{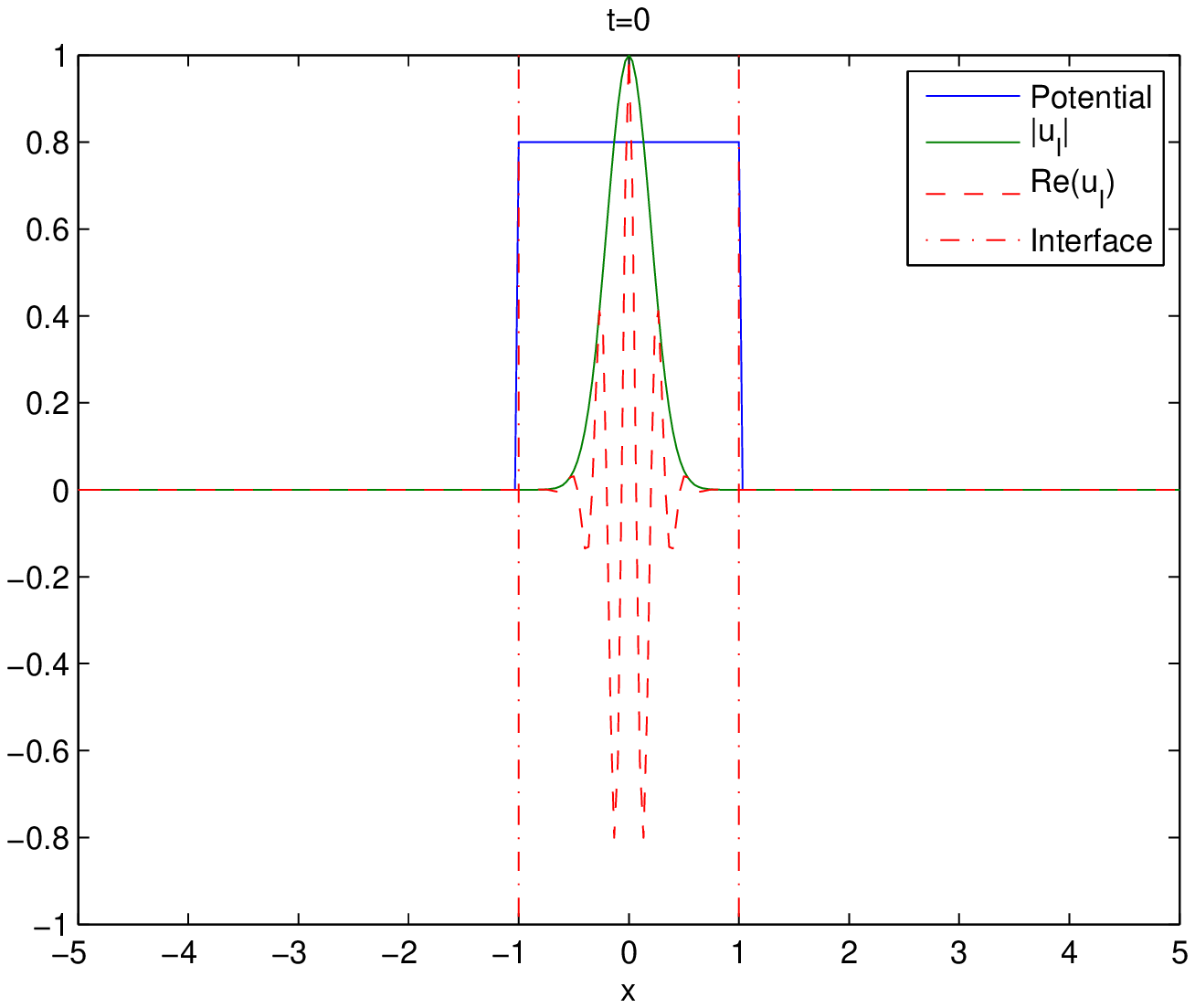}&
\includegraphics[width=0.40\linewidth]{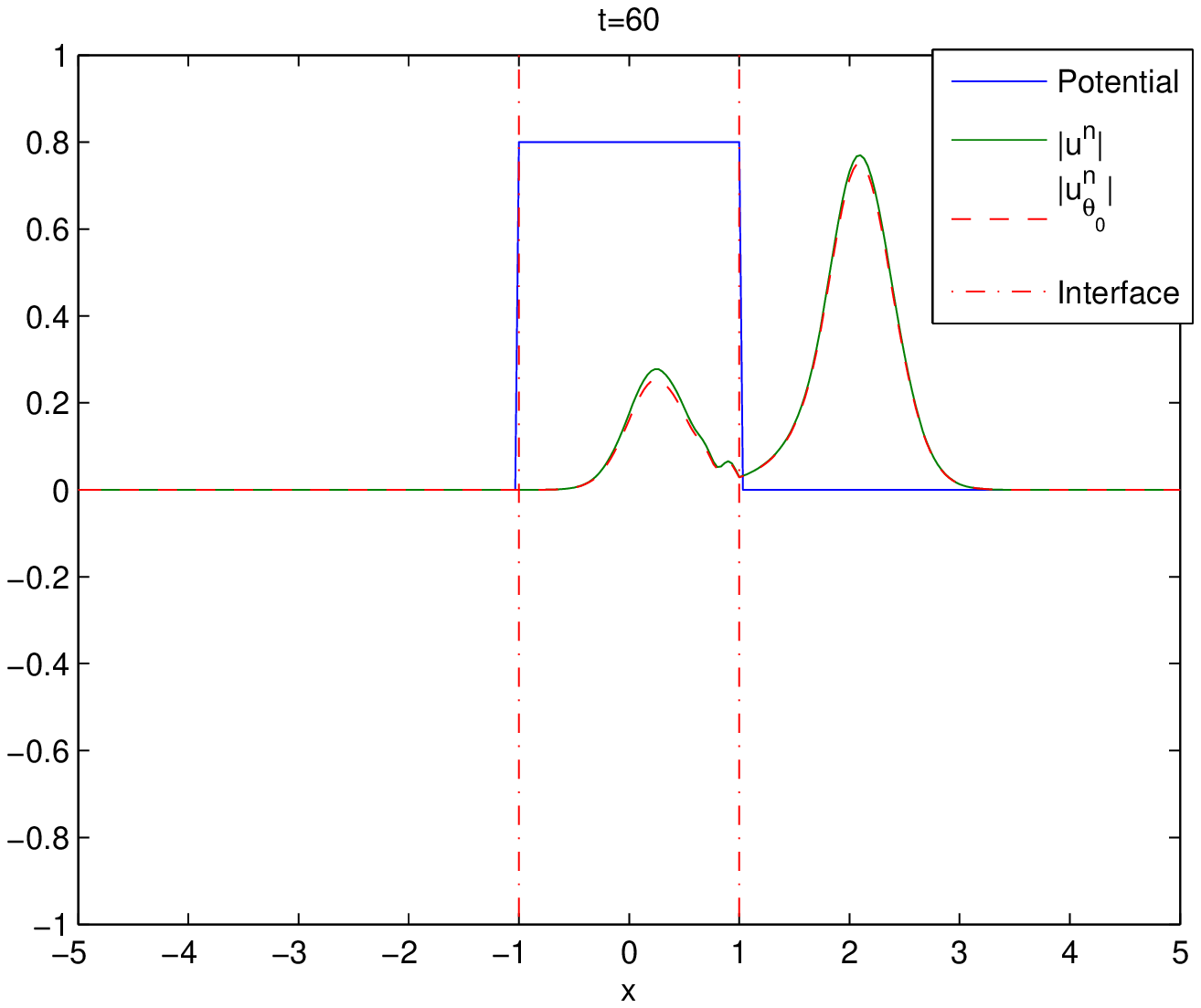}\\
\includegraphics[width=0.40\linewidth]{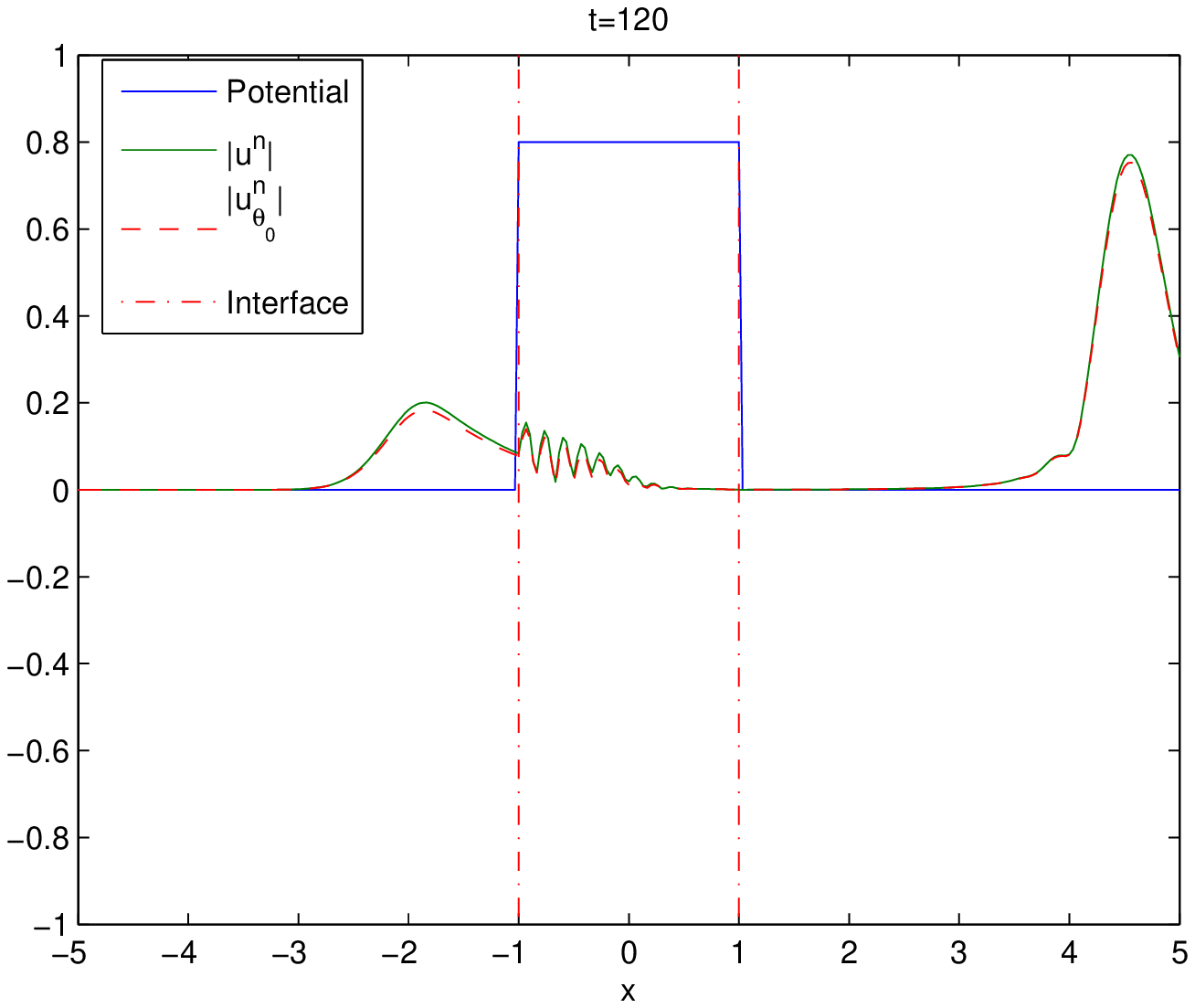}&
\includegraphics[width=0.40\linewidth]{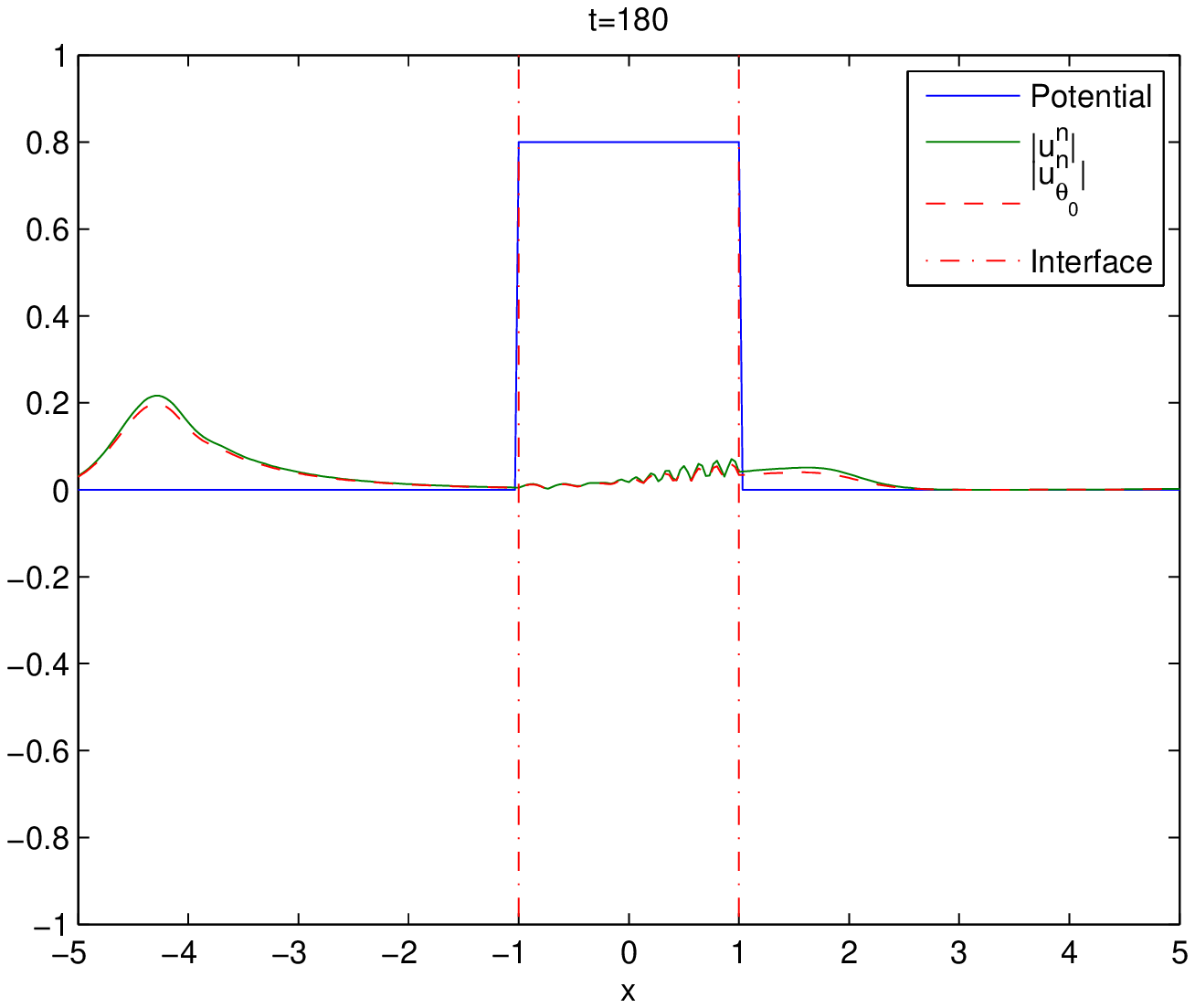}
\end{tabular}
\end{center}
\caption{Case of a barrier potential: modulus of the functions $u_{\theta_{0}
}^{n}$ and $u^{n}$ at different time $t$ with $x_{0}=0$ and $\theta_{0}
=0.09i$}
\label{fig.evbarm}
\end{figure}

In the three tests described above, although some oscillations occur when
crossing the interfaces $x=-1$ and $x=1$, the quantitative comparison gives
also good results. In particular, we represented in Figure \ref{fig.diffL2T}
the variation with respect to $\mathop{\mathrm{Im}}~\theta_{0}$ of the maximum
in time of the $L^{2}$ relative difference:
\begin{equation}
\label{eq.Dtheta0}D_{\theta_{0}}=\max_{1\leq n \leq N}100\frac{\|u_{\theta
_{0}}^{n}-u^{n}\|}{\|u_{I}\|}\,,
\end{equation}
where $N=400$ is the number of time iterations. It shows that, for every case,
the difference tends to $0$ when $\mathop{\mathrm{Im}}~\theta_{0}$ tends to
$0$. Moreover, the graphic of $D_{\theta_{0}}$ is a line which validates the
result \eqref{expansion}. We note also that the difference in the
case of a barrier potential is smaller then in the case $V=0$. This may be due to the fact that
the error coming from the interface conditions is compensated by the
exponential decay imposed by the barrier. Then, in the case of a barrier
potential, we note that the difference is more important when the initial
condition is supported in $(-1,1)$. It can be explained by the fact that the
solution crosses the two interfaces, at $x=-1$ and $x=1$, whereas, when the
solution comes from the left, only the interaction with the first interface
$x=-1$ is relevant, which is also a consequence of the exponential decay in
the barrier.

\begin{figure}[ptb]
\includegraphics[width=0.8\linewidth]{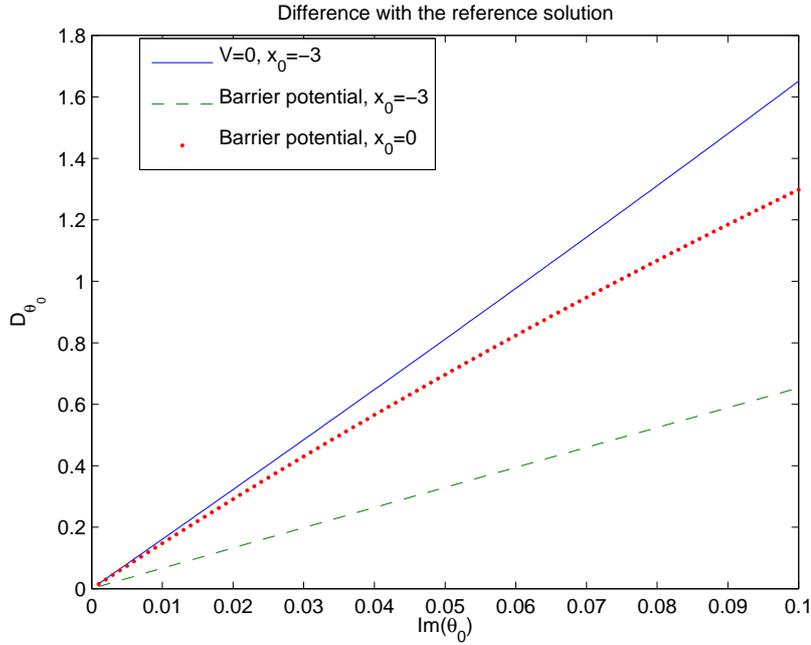}\caption{Variation
with respect to $\mathop{\mathrm{Im}}~\theta_{0}$ of the difference
$D_{\theta_{0}}$, defined in \eqref{eq.Dtheta0}, in the case $V=0$ with
$x_{0}=-3$, and in the case of a barrier potential with $x_{0}=-3$ and
$x_{0}=0$}
\label{fig.diffL2T}
\end{figure}

\section{Exterior complex scaling and local perturbations}
\label{se.extcompl}

Spectral deformations of Schr{\"o}dinger operators arising from complex
dilations form a standard tool to study resonances. This technique --
originally developed by J.M.~Combes and coauthors in \cite{AgCo}\cite{BaCo} 
for the homogeneous scaling in $L^{2}(\mathbb{R}^{n})$:
$U_{\theta}\psi(x)=e^{\frac{n\theta}{2}}\psi(e^{\theta}x)$ -- allows to relate
the resonances of the Hamiltonian $H=-\Delta+\mathcal{V}$ with the spectral
points of a non self-adjoint operator
$H(\theta)=U_{\theta}HU_{\theta}^{-1}$  with $\theta\in \C$. If the potential $\mathcal{V}$ is \emph{dilation analytic }in
the strip $\left\{  \theta\in\mathbb{C\,}\left\vert \,\left\vert
\operatorname{Im}\theta\right\vert <\alpha\right.  \right\}  $, the
poles of the meromorphic
continuation of the resolvent $\left(  H-z\right)  ^{-1}$ in the
second Riemann sheet are identified with the eigenvalues of $H(\theta)$ placed in the cone
spanned by the positive real axis and the rotated half axis $e^{-2i\operatorname{Im}\theta}\mathbb{R}_{+}$. We refer the reader to \cite{CFKS} for a summary and
we recall that many variations on this approach have been developed
since, see \cite{HiSi1}\cite{HeSj}\cite{LaMa} and \cite{HeMa} for
a short comparison of these methods. 
In particular for potentials which can be complex deformed only outside
a compact region, the exterior
complex scaling technique appeared first in \cite{SimE}
in the singular version that we reconsider here. Meanwhile regular
versions have been used in \cite{HiSi1} and extended with the so called
``black box'' formalism in \cite{SjZw}.

In this section, we consider a particular class of exterior scaling maps,
$U_{\theta}$, acting outside a compact set in 1D and introducing sharp
singularities in the domain of the corresponding deformed Hamiltonians. Let us
introduce the one-parameter family of exterior dilations
\begin{equation}
x\longrightarrow\left\{
\begin{array}
[c]{ll}
\smallskip e^{-\theta}(x-b)+b,& x>b\\
\smallskip x,& x\in(a,b)\\
e^{-\theta}(x-a)+a,& x<a\,.
\end{array}
\right. \label{map}
\end{equation}
For real values of the parameter $\theta$, the related unitary transformation
in $L^{2}(\mathbb{R})$ is
\begin{equation}
U_{\theta}\psi(x)=\left\{
\begin{array}
[c]{ll}
e^{\frac{\theta}{2}}\psi(\smallskip e^{\theta}(x-b)+b),& x>b\\
\smallskip\psi(x),& x\in(a,b)\\
e^{\frac{\theta}{2}}\psi(\smallskip e^{\theta}(x-a)+a),& x<a\,.
\end{array}
\right. \label{U_theta}
\end{equation}
Local perturbations of $H_{\theta_{0},0}^{h}(0)$ are defined by
\begin{equation}
H_{\theta_{0},\mathcal{V}}^{h}(0)=H_{\theta_{0},0}^{h}(0)+\mathcal{V}\,,\label{H_(0)+V}
\end{equation}
with $\supp\mathcal{V}\subset[a,b]$. In what follows we will assume
\begin{equation}
\mathcal{V}=\mathcal{V}_{1}+\mathcal{V}_{2},\quad \mathcal{V}_{1}\in L^{\infty}((a,b)),\quad \mathcal{V}_{2}
\in\mathcal{M}_{b}(U)\qquad\text{with }U\subset\subset(a,b)\,.\label{V_local}
\end{equation}
Under these assumptions,
\[
D(H_{\theta_{0},\mathcal{V}}^{h}(0))=\left\{  u\in H^{2}(\mathbb{R}\backslash\left\{  a,b,U\right\}  )\left\vert
\ (1-\chi)u\in D(H_{\theta_{0},0}^{h}(0)),\, -h^{2}u^{\prime\prime}+\mathcal{V}_{2}u\in L^{2}(U)\right.  \right\}\,,
\]
where $\chi\in C^{\infty}_0((a,b))$ and $\chi(x)=1$ for $x\in U$.\\
In particular, since $\mathcal{V}_{2}$ is a bounded measure, the domain
$D(H_{\theta_{0},\mathcal{V}}^{h}(0))$ is contained in $H^{1}(\mathbb{R}\backslash
\left\{  a,b\right\}  )$. The conjugated operator
\begin{equation}
H_{\theta_{0},\mathcal{V}}^{h}(\theta)=U_{\theta}H_{\theta_{0},\mathcal{V}
}^{h}(0)U_{\theta}^{-1}\label{H_(teta)+V}
\end{equation}
is defined on $D(H_{\theta_{0},\mathcal{V}}^{h}(\theta))=\left\{  u\in
L^{2}(\mathbb{R})\left\vert \,U_{\theta}^{-1}u\in D(H_{\theta_{0},\mathcal{V}}
^{h}(0))\right.  \right\}  $. The constraint $U_{\theta}^{-1}u\in
D(H_{\theta_{0},\mathcal{V}}^{h}(0))$ compels the boundary conditions
\begin{equation}
\left[
\begin{array}
[c]{l}
e^{-\frac{1}{2}\left(  \theta_{0}+\theta\right)  }u(b^{+})=u(b^{-}
);\quad\medskip e^{-\frac{3}{2}\left(  \theta_{0}+\theta\right)  }u^{\prime
}(b^{+})=u^{\prime}(b^{-})\\
e^{-\frac{1}{2}\left(  \theta_{0}+\theta\right)  }u(a^{-})=u(a^{+});\quad
e^{-\frac{3}{2}\left(  \theta_{0}+\theta\right)  }u^{\prime}(a^{-})=u^{\prime
}(a^{+})\,,
\end{array}
\right.  \label{BC_1}
\end{equation}
to hold for any $u\in D(H_{\theta_{0},\mathcal{V}}^{h}(\theta))$. Thus one has
\begin{equation}
D(H_{\theta_{0},\mathcal{V}}^{h}(\theta))=\left\{  u\in H^{2}(\mathbb{R}
\backslash\left\{  a,b,U\right\}  )\cap H^{1}(\mathbb{R}\backslash\left\{
a,b\right\}  )\left\vert \ \left(\ref{BC_1}\right)  ,\ -h^{2}u^{\prime
\prime}+\mathcal{V}_{2}u\in L^{2}(U)\right.  \right\}  .\label{H_(teta)+V_dom}
\end{equation}
The action of $H_{\theta_{0},\mathcal{V}}^{h}(\theta)$ is
\begin{equation}
H_{\theta_{0},\mathcal{V}}^{h}(\theta)u=\left[  -h^{2}\eta(x)\partial_{x}
^{2}+\mathcal{V}\right]  \,u,\qquad\eta(x)=e^{- 2\theta\,
  1_{\R\setminus(a,b)}(x)}\,.\label{H_(teta)+V_def}
\end{equation}
It is worthwhile to notice that this definition can be extended to complex values of $\theta$. For $\theta\in\mathbb{C}$, the
Hamiltonian $H_{\theta_{0},\mathcal{V}}^{h}(\theta)$ identifies with a
restriction of the operator $Q(\theta)$
\begin{equation}
\left\{
\begin{array}
[c]{l}
\medskip D(Q(\theta))=\left\{  u\in H^{2}(\mathbb{R}\backslash\left\{
a,b,U\right\}  )\cap H^{1}(\mathbb{R}\backslash\left\{  a,b\right\}
)\left\vert \ -h^{2}u^{\prime\prime}+\mathcal{V}_{2}u\in L^{2}(U)\right.  \right\}\,,  \\
Q(\theta)u=\left[  -h^{2}\eta(x)\partial_{x}^{2}+\mathcal{V}\right]  \,u\,.
\end{array}
\right. \label{Q-(teta)}
\end{equation}
For particular choices of $\theta_{0}$ and $\theta$, the quantum evolution
generated by the deformed model $H_{\theta_{0},\mathcal{V}}^{h}(\theta)$ is described by
contraction maps. To fix this point, let us consider the terms\\
$\operatorname{Re}\left\langle u,iH_{\theta_{0},\mathcal{V}}^{h}(\theta)u\right\rangle
_{L^{2}(\mathbb{R})}$; for $u\in D(H_{\theta_{0},\mathcal{V}}^{h}(\theta))$, an explicit
calculation gives
\begin{align}
\operatorname{Re}\left\langle u,iH_{\theta_{0},\mathcal{V}}^{h}(\theta)u\right\rangle
_{L^{2}(\mathbb{R})}  &  =\operatorname{Re}\left\{  -ih^{2}\left(  \bar
{u}(a^{-})u^{\prime}(a^{-})-\bar{u}(b^{+})u^{\prime}(b^{+})\right)  \,\left(
e^{-2\theta}-e^{-\frac{1}{2}\left(  \bar{\theta}+\bar{\theta}_{0}\right)
}e^{-\frac{3}{2}\left(  \theta+\theta_{0}\right)  }\right)  \right\}
\nonumber\\
&  +h^{2}e^{-2\operatorname{Re}\theta}\sin\left(  2\operatorname{Im}
\theta\right)  \int_{\mathbb{R}\backslash\left(a,b\right)  }\left\vert
u^{\prime}\right\vert ^{2}dx. \label{accretivity}
\end{align}
For $\theta=\theta_{0}=i\tau$, with $\tau\in\left(  0,\frac{\pi}{2}\right)  $, the
boundary terms disappear, and the r.h.s. of (\ref{accretivity}) is positive
\begin{equation}
\operatorname{Re}\left\langle u,iH_{i\tau,\mathcal{V}}^{h}(i\tau)u\right\rangle
_{L^{2}(\mathbb{R})}=h^{2}\sin\left(  2\tau\right)  \int_{\mathbb{R}
\backslash\left(a,b\right)  }\left\vert u^{\prime}\right\vert ^{2}dx\geq0.
\label{accretivity1}
\end{equation}
\begin{lemma}
For $\tau\in\left(0,\frac{\pi}{2}\right)$, the operator $iH_{i\tau,\mathcal{V}}^{h}(i\tau)$ is the generator of a contraction semigroup.
\end{lemma}

\begin{proof}
As a consequence of \eqref{accretivity1}, $iH_{i\tau,\mathcal{V}}^{h}(i\tau)$ is accretive. Moreover, the propriety $\sigma_{ess}\left(H_{i\tau,\mathcal{V}}^{h}(i\tau)\right)=e^{-2i\tau}\R_{+}$ in Corollary \ref{cor_krein} below, implies $i\lambda_0\in\rho(H_{i\tau,\mathcal{V}}^{h}(i\tau))$ for some $\lambda_0>0$ and $\left(iH_{i\tau,\mathcal{V}}^{h}(i\tau)+\lambda_0\right)$ is surjective. Then, a standard characterization of
semigroup generators (\cite{ReSi2}, Theorem X.48) leads to the result.
\end{proof}

\subsection{Krein formula and analyticity of the resolvent}

In order to get an expression of the adjoint operator of $H_{\theta_{0},\mathcal{V}}^{h}(\theta)$, we introduce the following operator with two-parameters boundary conditions
\begin{equation}
D(\mathcal{Q}_{\theta_{1},\theta_{2}}(\theta))=\left\{  u\in D(Q(\theta))\left\vert \ \left[
\begin{array}
[c]{l}
e^{-\frac{1}{2}\left(  \theta_{1}+\theta\right)  }u(b^{+})=u(b^{-}
);\quad\medskip e^{-\frac{1}{2}\left(  \theta_{2}+3\theta\right)  }u^{\prime
}(b^{+})=u^{\prime}(b^{-})\\
e^{-\frac{1}{2}\left(  \theta_{1}+\theta\right)  }u(a^{-})=u(a^{+});\quad
e^{-\frac{1}{2}\left(  \theta_{2}+3\theta\right)  }u^{\prime}(a^{-}
)=u^{\prime}(a^{+})
\end{array}
\right.\right.  \right\}\,,\label{B_operator_domain}
\end{equation}
\begin{equation}
\mathcal{Q}_{\theta_{1},\theta_{2}}(\theta)u=\left[-h^{2}\eta(x)\partial_{x}^{2}+\mathcal{V}\right]u,\qquad\eta(x)=e^{- 2\theta\,
  1_{\R\setminus(a,b)}(x)}\,,\label{B_operator}
\end{equation}
where $Q(\theta)$ is defined in \eqref{Q-(teta)}. Indeed, by direct computation\\
i) $\mathcal{Q}_{\theta_{1},\theta_{2}}(\theta)$ identifies with the original model $H_{\theta_{0},\mathcal{V}}^{h}(\theta)$ for the choice of parameters: $\theta_{1}=\theta_{0}$ and $\theta_{2}=3\theta_{0}$
\begin{equation}
H_{\theta_{0},\mathcal{V}}^{h}(\theta)=\mathcal{Q}_{\theta_{0},3\theta_{0}}(\theta),\label{original}
\end{equation}
ii) the adjoint operator $\left(  \mathcal{Q}_{\theta_{1},\theta_{2}}(\theta)\right)^{\ast}$ is given by
\begin{equation}
\left(\mathcal{Q}_{\theta_{1},\theta_{2}}(\theta)\right)^{\ast}=\mathcal{Q}_{-\bar{\theta}_{2},-\bar{\theta}_{1}}(\bar{\theta}).\label{adjoint}
\end{equation}
Like $H_{\theta_{0},\mathcal{V}}^{h}(\theta)$, the Hamiltonian $\mathcal{Q}_{\theta_{1},\theta_{2}}(\theta)$ is a restriction of the operator $Q(\theta)$. In this context, we fix a boundary value triple $\left\{  \Gamma_{j=1,2}^{\theta},\mathbb{C}^{4}\right\}  $ with $\Gamma_{j}^{\theta}:D(Q(\theta))\rightarrow
\mathbb{C}^{4}$
\begin{equation}
\Gamma_{1}^{\theta}\psi=h^{2}\left(
\begin{array}
[c]{c}
-e^{-\frac{3}{2}\theta}\psi^{\prime}(b^{+})\\
-\psi(b^{-})\\
\psi(a^{+})\\
e^{-\frac{3}{2}\theta}\psi^{\prime}(a^{-})
\end{array}
\right)  ;\qquad\Gamma_{2}^{\theta}\psi=\left(
\begin{array}
[c]{c}
e^{-\frac{\theta}{2}}\psi(b^{+})\\
\psi^{\prime}(b^{-})\\
\psi^{\prime}(a^{+})\\
e^{-\frac{\theta}{2}}\psi(a^{-})
\end{array}
\right)\,,  \label{Boundary_transf(teta)}
\end{equation}
and $\left(  \Gamma_{1}^{\theta},\Gamma_{2}^{\theta}\right)  :D(Q(\theta)^{\ast
})\rightarrow\mathbb{C}^{4}\times\mathbb{C}^{4}$ surjective. For all
$\psi,\varphi\in D(Q(\theta))$, these maps satisfy the relation
\begin{equation}
\left\langle \psi,Q(\theta)\varphi\right\rangle _{L^{2}(\mathbb{R}
)}-\left\langle Q(\bar{\theta})\psi,\varphi\right\rangle _{L^{2}(\mathbb{R}
)}=\left\langle \Gamma_{1}^{\bar{\theta}}\psi,\Gamma_{2}^{\theta}
\varphi\right\rangle _{\mathbb{C}^{4}}-\left\langle \Gamma_{2}^{\bar{\theta}
}\psi,\Gamma_{1}^{\theta}\varphi\right\rangle _{\mathbb{C}^{4}}\label{Boundary_transf(teta)_1}
\end{equation}
(for the definition of boundary triples and the construction of point interaction potentials in the self adjoint case see \cite{Pan} and \cite{AlPa}). Let $\Lambda,B\in\mathbb{C}^{4,4}$ be defined as
\begin{equation}\label{Matrix_def_B}
\Lambda(\theta_1,\theta_2)=\frac{1}{h^{2}}
\begin{pmatrix}
a(\theta_{1},\theta_{2}) & \\
& -a(-\theta_{2},-\theta_{1})
\end{pmatrix}
,\qquad B=
\begin{pmatrix}
b & \\
& b
\end{pmatrix}
,
\end{equation}
\begin{equation}
a(\theta_{1},\theta_{2})=
\begin{pmatrix}
-e^{-\frac{\theta_{2}}{2}} & 0\\
0 & -e^{\frac{\theta_{1}}{2}}
\end{pmatrix}
,\quad\quad b=
\begin{pmatrix}
0 & 1\\
1 & 0
\end{pmatrix}
,
\end{equation}
the boundary conditions in (\ref{B_operator_domain}) are equivalent to
\begin{equation}
\Lambda\Gamma_{1}^{\theta}\psi=B\Gamma_{2}^{\theta}\psi\,.\label{BC_teta}
\end{equation}

Let $H_{ND,\mathcal{V}}^{h}(\theta)$ denote the restriction of $Q(\theta)$
corresponding to the boundary conditions: $\Gamma_{1}^{\theta}\psi=0$; this
operator is explicitly given by
\begin{equation}
H_{ND,\mathcal{V}}^{h}(\theta)=-h^{2}e^{-2\theta}\Delta_{(-\infty,a)}
^{N}\oplus\left[  -h^{2}\Delta_{(a,b)}^{D}+\mathcal{V}\right]  \oplus
-h^{2}e^{-2\theta}\Delta_{(b,+\infty)}^{N}.
\label{H_Neumann-Dirichlet(teta,V)}
\end{equation}
Its spectrum is characterized as follows
\begin{equation}\label{sigma_HND}
\sigma\left(H_{ND,\mathcal{V}}^{h}(\theta)\right)=e^{-2\theta}\R_{+}\cup\sigma\left(-h^{2}\Delta_{(a,b)}^{D}+\mathcal{V}\right)\,.
\end{equation}
It is possible to write $\left(\mathcal{Q}_{\theta_1,\theta_2}(\theta)-z\right)^{-1}$ as the sum of
$\left(  H_{ND,\mathcal{V}}^{h}(\theta)-z\right)  ^{-1}$ plus finite rank
terms. Such a representation will be further used to develop the spectral
analysis of $H_{\theta_{0},\mathcal{V}}^{h}(\theta)$ where our Krein-like formula will allow explicit resolvent estimates
near the resonances. The space $\mathcal{N}_{z,\theta}=\Ker(Q(\theta)-z)$ is generated by the linear closure of the system
$\left\{  u_{i,z}\right\}  _{i=1}^{4}$ where $u_{i,z}$ are the independent
solutions to $(Q(\theta)-z)u=0$. The exterior solutions to this problem,
$u_{i,z}$, $i=1,4$, are explicitly given by
\begin{equation}
u_{1,z}(x)=1_{\left(  b,+\infty\right)  }e^{i\frac{\sqrt{ze^{2\theta}}}{h}(x-b)},\quad
u_{4,z}(x)=1_{\left(  -\infty,a\right)  }e^{-i\frac{\sqrt{ze^{2\theta}}}{h}(x-a)},
\label{u_i,z (i=1,4)}
\end{equation}
where the square root branch cut is fixed with $\operatorname{Im}\sqrt{\cdot
}>0$. This assumption implies $\Imag\sqrt{ze^{2\theta}}~>~0$
for all $z\in\mathbb{C}\backslash e^{-2i\operatorname{Im}\theta}\mathbb{R}_{+}
$. The interior solutions, $u_{i,z}$, $i=2,3$, can be defined through the
following boundary value problems 
\begin{equation}
\left\{
\begin{array}
[c]{l}
\medskip\left[  -h^{2}\partial_{x}^{2}+\mathcal{V}-z\right]  \,u_{2,z}
=0,\qquad\text{in }\left(  a,b\right)\,, \\
u_{2,z}(a)=0,\qquad u_{2,z}(b)=1\,,
\end{array}
\right.  \quad\left\{
\begin{array}
[c]{l}
\medskip\left[  -h^{2}\partial_{x}^{2}+\mathcal{V}-z\right]  \,u_{3,z}
=0,\qquad\text{in }\left(  a,b\right)\,, \\
u_{3,z}(a)=1,\qquad u_{3,z}(b)=0\,,
\end{array}
\right. \label{u_intz}
\end{equation}
with $z\in\mathbb{C}\backslash
\sigma_p(H_{ND,\mathcal{V}}^h(\theta))$. Owing to the property of the interior
Dirichlet realization in $(a,b)$, the solutions $u_{i,z}\big|_{(a,b)}$
 are unique and locally $H^2$ near the boundary. We consider the maps: $\gamma(\cdot,z,\theta)=\left(  \left.
\Gamma_{1}^{\theta}\right\vert _{\mathcal{N}_{z,\theta}}\right)  ^{-1}$, with
$\left.  \Gamma_{1}^{\theta}\right\vert _{\mathcal{N}_{z,\theta}}$ denoting the
restriction of $\Gamma_{1}^{\theta}$ onto $\mathcal{N}_{z,\theta}$, and
$q(z,\theta,\mathcal{V})=\Gamma_{2}^{\theta}\gamma(\cdot,z,\theta)$.
These form holomorphic families of linear operators for $z$ in a cut plane
$\mathbb{C}\backslash e^{-2i\operatorname{Im}\theta}\mathbb{R}_{+}$. Their
matrix form w.r.t. the standard basis $\left\{  \underline{e}_{j}\right\}
_{j=1}^{4}$ of $\mathbb{C}^4$ and the system $\left\{  u_{i,z}\right\}  _{i=1}^{4}$ is:
$\gamma_{ij}(z,\theta)=c_{i}(z,\theta)\delta_{ij}$ with
\begin{equation}
c_{1}(z,\theta)=\frac{ie^{\frac{3\theta}{2}}}{h\sqrt{ze^{2\theta}}},\ c_{2}=-\frac
{1}{h^{2}},\ c_{3}=\frac{1}{h^{2}},\ c_{4}(z,\theta)=\frac{ie^{\frac{3\theta}{2}}}{h\sqrt{ze^{2\theta}}}\,, \label{eq.c_i}
\end{equation}
and
\begin{equation}
q(z,\theta,\mathcal{V})=\frac{1}{h^{2}}
\begin{pmatrix}
\frac{ihe^{\theta}}{\sqrt{ze^{2\theta}}} &  &  & \\
& -u_{2,z}^{\prime}(b) & u_{3,z}^{\prime}(b) & \\
& -u_{2,z}^{\prime}(a) & u_{3,z}^{\prime}(a) & \\
&  &  & \frac{ihe^{\theta}}{\sqrt{ze^{2\theta}}}
\end{pmatrix}
. \label{q_zeta}
\end{equation}

\begin{lemma}
\label{Lemma 3}Let $\varphi\in L^{2}(\mathbb{R})$ and $j=1,...,4$; the
relation
\begin{equation}
\left[  \Gamma_{2}^{\theta}\left(  H_{ND,\mathcal{V}}^{h}(\theta
)-z\right)  ^{-1}\varphi\right]  _{j}=\left\langle \gamma(\underline{e}
_{j},\bar{z},\bar{\theta}),\varphi\right\rangle _{L^{2}(\mathbb{R})}
\label{resolvent_rel (teta)}
\end{equation}
holds with $\theta\in\mathbb{C}$ and $z\in\rho\,\left(  H_{ND,\mathcal{V}}^{h}(\theta)\right)  $.
\end{lemma}

\begin{proof}
Let: $f=\left(  H_{ND,\mathcal{V}}^{h}(\theta)-z\right)  ^{-1}\varphi$. This
function is in $D(H_{ND,\mathcal{V}}^{h}(\theta))$ so that:\\ $Q(\theta
)f=H_{ND,\mathcal{V}}^{h}(\theta)f$ and $\Gamma_{1}^{\theta}f=0$. The l.h.s of
(\ref{resolvent_rel (teta)}) writes as
\[
\left[  \Gamma_{2}^{\theta}\left(  H_{ND,\mathcal{V}}^{h}(\theta)-z\right)
^{-1}\varphi\right]  _{j}=\left\langle \underline{e}_{j},\,\Gamma_{2}^{\theta
}f\right\rangle _{\mathbb{C}^{4}}.
\]
Since $\underline{e}_{j}=\Gamma_{1}^{\bar{\theta}}\gamma(\underline{e}
_{j},\bar{z},\bar{\theta})$, we have
\begin{align*}
\left[  \Gamma_{2}^{\theta}\left(  H_{ND,\mathcal{V}}^{h}(\theta)-z\right)
^{-1}\varphi\right]  _{j}  &  =\left\langle \Gamma_{1}^{\bar{\theta}}
\gamma(\underline{e}_{j},\bar{z},\bar{\theta}),\,\Gamma_{2}^{\theta
}f\right\rangle _{\mathbb{C}^{4}}-\left\langle \Gamma_{2}^{\bar{\theta}}
\gamma(\underline{e}_{j},\bar{z},\bar{\theta}),\,\Gamma_{1}^{\theta
}f\right\rangle _{\mathbb{C}^{4}}\\
&  =\left\langle \gamma(\underline{e}_{j},\bar{z},\bar{\theta}),\,Q(\theta
)f\right\rangle _{L^{2}(\mathbb{R})}-\left\langle Q(\bar{\theta}
)\gamma(\underline{e}_{j},\bar{z},\bar{\theta}),\,f\right\rangle
_{L^{2}(\mathbb{R})}.
\end{align*}
By definition, $\gamma(\underline{e}_{j},\bar{z},\bar{\theta})\in
\mathcal{N}_{\bar{z},\bar{\theta}}$ and the r.h.s. writes as
\begin{align*}
\left\langle \gamma(\underline{e}_{j},\bar{z},\bar{\theta}),\,Q(\theta
)f\right\rangle _{L^{2}(\mathbb{R})}-\left\langle Q(\bar{\theta}
)\gamma(\underline{e}_{j},\bar{z},\bar{\theta}),\,f\right\rangle
_{L^{2}(\mathbb{R})}&=\left\langle \gamma(\underline{e}_{j},\bar{z},\bar
{\theta}),\,\,\left(  H_{ND,\mathcal{V}}^{h}(\theta)-z\right)  f\right\rangle
_{L^{2}(\mathbb{R})}\\
&=\left\langle \gamma(\underline{e}_{j},\bar{z},\bar
{\theta}),\,\varphi\right\rangle _{L^{2}(\mathbb{R})}.
\end{align*}

\end{proof}

\begin{proposition} The resolvent $\left(\mathcal{Q}_{\theta_{1},\theta_{2}}(\theta)-z\right)  ^{-1}$ allows the representation
\begin{align}
\left(  \mathcal{Q}_{\theta_{1},\theta_{2}}(\theta)-z\right)^{-1} &  =\left(
H_{ND,\mathcal{V}}^{h}(\theta)-z\right)  ^{-1}\nonumber\\
&  -\sum_{i,j=1}^{4}\left[\left(  Bq(z,\theta,\mathcal{V})-\Lambda\right)^{-1}B\right]_{ij}\left\langle \gamma(\underline{e}_{j},\bar{z},\bar{\theta}
),\cdot\right\rangle _{L^{2}(\mathbb{R})}\gamma(\underline{e}_{i}
,z,\theta),\label{res_formula_B}
\end{align}
and one has: $\sigma_{ess}(\mathcal{Q}_{\theta_{1},\theta_{2}}(\theta
))=\sigma_{ess}(H_{ND,\mathcal{V}}^{h}(\theta))=e^{-2\theta}\mathbb{R}
_{\mathbb{+}}$.
\end{proposition}

\begin{proof}
Let us consider the r.h.s. of this formula: the operator $\left(
H_{ND,\mathcal{V}}^{h}(\theta)-z\right)  ^{-1}$ is well defined for
$z\in\mathbb{C}\backslash\sigma(H_{ND,\mathcal{V}}^{h}(\theta))$. The vectors $\gamma(\underline{e}
_{i},z,\theta)$, $i=1,...,4$, are given by \eqref{u_i,z (i=1,4)}, \eqref{u_intz}, \eqref{eq.c_i}, while the boundary values of $u_{i,z}^{\prime}$
-- appearing in the definition of the matrix (\ref{q_zeta}) -- are well defined whenever $z\in\mathbb{C}\backslash\sigma_{p}(H_{ND,\mathcal{V}}^{h}(\theta))$. Therefore, the r.h.s. of
(\ref{res_formula_B}) makes sense for $z\in\mathbb{C}\backslash\left\{\sigma(H_{ND,\mathcal{V}}^{h}(\theta))\cup \mathcal{T}_0\right\}
$ where $\mathcal{T}_0$ is the (at most) discrete set, described by the transcendental
equation
\begin{equation}
\det\left(  Bq(z,\theta,\mathcal{V})-\Lambda(\theta_1,\theta_2)\right)  =0.\label{eigenvalue_eq_0}
\end{equation}
It is worthwhile to notice that $\mathbb{C}\backslash\left\{\sigma(H_{ND,\mathcal{V}}^{h}(\theta))\cup \mathcal{T}_0\right\}$ is not empty. Let us introduce the map $R_{z}(\varphi)$ defined for $\varphi\in
L^{2}(\mathbb{R})$ by
\[
R_{z}(\varphi)=\phi-\psi,
\]
\[
\phi=\left(  H_{ND,\mathcal{V}}^{h}(\theta)-z\right)  ^{-1}\varphi,
\]
\[
\psi=\sum_{i,j=1}^{4}\left[\left(  Bq(z,\theta,\mathcal{V})-\Lambda\right)^{-1}B\right]_{ij}\left\langle \gamma(\underline{e}_{j},\bar{z},\bar{\theta}
),\varphi\right\rangle _{L^{2}(\mathbb{R})}\gamma(\underline{e}_{i}
,z,\theta),
\]
with $q(z,\theta,\mathcal{V})$ given in (\ref{q_zeta}) and $z\in\rho(H_{ND,\mathcal{V}
}^{h}(\theta))\backslash \mathcal{T}_0$. In what follows we show that:\\
$R_{z}(\varphi)~=~\left(\mathcal{Q}_{\theta_{1},\theta_{2}}(\theta)-z\right)^{-1}\varphi$. Since $H_{ND,\mathcal{V}}^{h}(\theta)\subset Q(\theta)$ and $\gamma
(\underline{e}_{i},z,\theta)\in\Ker(Q(\theta)-z)$, one has: 
\[\left(  H_{ND,\mathcal{V}}^{h}(\theta)-z\right)  ^{-1}\varphi,\,\gamma
(\underline{e}_{i},z,\theta)\in D(Q(\theta))\,.
\]
This implies
$R_{z}(\varphi)\in D(Q(\theta))$. To simplify the presentation, we will temporarily use the notation $q=q(z,\theta,\mathcal{V})$. Being $\phi\in D(H_{ND,\mathcal{V}}^{h}
(\theta))$, we have: $\Gamma_{1}^{\theta}\phi=0$ and the following relation
holds
\begin{equation}
\left(  Bq-\Lambda\right)  \Gamma_{1}^{\theta}\left(  \phi
-\psi\right)  =-\left(  B\Gamma_{2}^{\theta}\gamma(\cdot
,z,\theta)-\Lambda\right)  \Gamma_{1}^{\theta}\psi=\left(  -B\Gamma
_{2}^{\theta}+\Lambda\Gamma_{1}^{\theta}\right)  \psi,\label{one}
\end{equation}
where $\psi\in\mathcal{N}_{z,\theta}$ and 
\[
\gamma(\cdot,z,\theta)\Gamma_{1}^{\theta}\vert_{\mathcal{N}_{z,\theta}}=1
\]
have been used. At the same time, the $n$-th component of the vector at the l.h.s.
can be expressed as
\begin{multline*}
\left[  \smallskip\left(  Bq-\Lambda\right)  \Gamma_{1}^{\theta
}\left(  \phi-\psi\right)  \right]  _{n}=\left[  -\left(
Bq-\Lambda\right)  \Gamma_{1}^{\theta}\psi\right]  _{n}\\
=-\sum_{i,j=1}^{4}\left[  \left(Bq-\Lambda\right)\Gamma_{1}^{\theta}\gamma(\underline{e}_{i},z,\theta)\right]_n  \,\left[\left(  Bq-\Lambda\right)^{-1}B\right]_{ij}
\left\langle \gamma(\underline{e}_{j},\bar{z},\bar{\theta}),\varphi
\right\rangle _{L^{2}(\mathbb{R})}.
\end{multline*}
Recalling that $\Gamma_{1}^{\theta}\gamma(\underline{e}_{i},z,\theta)=e_i$, we get
\begin{multline*}
\left[  \smallskip\left(  Bq-\Lambda\right)  \Gamma_{1}^{\theta
}\left(  \phi-\psi\right)  \right]  _{n}=\\
-\sum_{i,j=1}^{4}\left(  Bq-\Lambda\right)_{ni}\,\left[\left(
Bq-\Lambda\right)^{-1}B\right]_{ij}\left\langle \gamma(\underline
{e}_{j},\bar{z},\bar{\theta}),\varphi\right\rangle _{L^{2}(\mathbb{R})}
=-\sum_{j=1}^{4}B_{nj}\left\langle \gamma(\underline{e}_{j},\bar{z}
,\bar{\theta}),\varphi\right\rangle _{L^{2}(\mathbb{R})}.
\end{multline*}
Taking into account the result of the Lemma \ref{Lemma 3}, this relation
writes as
\begin{equation}
\left(  Bq-\Lambda\right)  \Gamma_{1}^{\theta}\left(  \phi
-\psi\right)  =-B\Gamma_{2}^{\theta}\phi.\label{two}
\end{equation}
Combining (\ref{one}) and (\ref{two}), one has: $-\Lambda\Gamma_{1}^{\theta}
\psi=B\Gamma_{2}^{\theta}\left(  \phi-\psi\right)  $, and,
adding the null term $\Lambda\Gamma_{1}^{\theta}\phi$ at the l.h.s.,
\[
\Lambda\Gamma_{1}^{\theta}R_{z}(\varphi)=B\Gamma_{2}^{\theta}R_{z}(\varphi),
\]
which, according to (\ref{BC_teta}), is the boundary condition characterizing
$\mathcal{Q}_{\theta_{1},\theta_{2}}(\theta)$ as a restriction of $Q(\theta)$. Then we
have: $R_{z}(\varphi)\in D\left(\mathcal{Q}_{\theta_{1},\theta_{2}}(\theta)\right)$. Furthermore,
\begin{equation}
\left( \mathcal{Q}_{\theta_{1},\theta_{2}}(\theta)-z\right)  R_{z}(\varphi)=\left(
Q(\theta)-z\right)  R_{z}(\varphi)=\varphi-\left(  Q(\theta)-z\right)
\psi=\varphi,\label{resolvent}
\end{equation}
where $\left(  Q(\theta)-z\right)  \gamma(\underline{e}_{i},z,\theta)=0$ has
been used. This leads to the surjectivity of the operator $\left(
  \mathcal{Q}_{\theta_{1},\theta_{2}}(\theta)-z\right)$ for any
$z\in\mathbb{C}\backslash\left\{\sigma(H_{ND,\mathcal{V}}^{h}(\theta))\cup
  \mathcal{T}_0\right\}$. The injectivity is obtained using the
adjoint of $\left(
  \mathcal{Q}_{\theta_{1},\theta_{2}}(\theta)-z\right)$. Indeed, the
equality \eqref{adjoint} implies 
\[
\left(\mathcal{Q}_{\theta_{1},\theta_{2}}(\theta)-z\right)^* =(\mathcal{Q}_{-\bar{\theta}_{2},-\bar{\theta}_{1}}(\bar{\theta})-\bar{z})\,,
\]
which, from the result above, appears to be surjective for all $z$ such that $\bar{z} \in \mathbb{C}\backslash\left\{\sigma(H_{ND,\mathcal{V}}^{h}(\bar{\theta}))\cup \widetilde{\mathcal{T}}_0\right\}$, where $\widetilde{\mathcal{T}}_0$ is the discrete set of the solutions to \eqref{eigenvalue_eq_0} when replacing: $\theta=\bar{\theta}$, $\theta_1=-\bar{\theta}_{2}$ and $\theta_2=-\bar{\theta}_{1}$. As a consequence of \eqref{sigma_HND}, we have
\[
\bar{z}\in\sigma\left(H_{ND,\mathcal{V}}^{h}(\bar{\theta})\right) \Leftrightarrow z\in\sigma\left(H_{ND,\mathcal{V}}^{h}(\theta)\right)\,.
\]
It follows that for any $z\in\mathbb{C}\backslash\left\{\sigma(H_{ND,\mathcal{V}}^{h}(\theta))\cup \mathcal{T}\right\}$, where $\mathcal{T} = \mathcal{T}_0 \cup \left\{z \, \textrm{ s.t. } \, \bar{z} \in \widetilde{\mathcal{T}}_0 \right\}$, the operator $\left(  \mathcal{Q}_{\theta_{1},\theta_{2}}(\theta)-z\right)$ is surjective and
\[
\Ker\left(\mathcal{Q}_{\theta_{1},\theta_{2}}(\theta)-z\right)=\left[\textrm{Ran}\,\left(\mathcal{Q}_{\theta_{1},\theta_{2}}(\theta)-z\right)^{*}\right]^{\perp}=\{0\}.
\]
Wet get that $\forall z\in\mathbb{C}\backslash\left\{\sigma(H_{ND,\mathcal{V}}^{h}(\theta))\cup \mathcal{T}\right\}$, $\left(\mathcal{Q}_{\theta_{1},\theta_{2}}(\theta)-z\right)$ is invertible and $\left(\mathcal{Q}_{\theta_{1},\theta_{2}}(\theta)-z\right)^{-1}=R_z$. Moreover, for such a complex $z$, the difference $R_{z}-\left(H_{ND,\mathcal{V}}^{h}(\theta)-z\right)^{-1}$ is compact. Then, we conclude that
\[
\sigma_{ess}(\mathcal{Q}_{\theta_{1},\theta_{2}}(\theta))=\sigma_{ess}(H_{ND,\mathcal{V}}^{h}(\theta))=e^{-2\theta}\mathbb{R}_{\mathbb{+}}
\]
(for this point, we refer to \cite{ReSi4}, Sec.~XIII.4, Lemma~3 and the strong
spectral mapping theorem), and the equality \eqref{res_formula_B} holds as an identity of meromorphic functions on $\C\setminus e^{-2\theta}\R_{+}$\,.
\end{proof}

As a direct consequence of the previous proposition and the identification \eqref{original}, the representation of the resolvent $\left(H_{\theta_{0},\mathcal{V}}^{h}(\theta)-z\right)^{-1}$ is obtained by replacing the matrix $\Lambda$ in \eqref{res_formula_B} by the matrix $A=\Lambda(\theta_{0},3\theta_{0})$ where the matrix $\Lambda(\theta_{1},\theta_{2})$ is given in \eqref{Matrix_def_B}. It follows that for the matrices
\begin{equation}
A=\frac{1}{h^{2}}
\begin{pmatrix}
-e^{-\frac{3}{2}\theta_{0}}& & &\\
& -e^{\frac{\theta_{0}}{2}}&&\\
&&e^{\frac{\theta_{0}}{2}} &\\
&&& e^{-\frac{3}{2}\theta_{0}}
\end{pmatrix}
,\qquad B=
\begin{pmatrix}
\begin{tabular}{cc}
0 & 1\\
1 & 0
\end{tabular} & \\
& \begin{tabular}{cc}
0 & 1\\
1 & 0
\end{tabular}
\end{pmatrix}
,\label{Matrix_def}
\end{equation}
the result below holds.
\begin{corollary}\label{cor_krein}
Let $\mathcal{V}$ and $A$, $B$ be defined as in (\ref{V_local}) and \eqref{Matrix_def}. The resolvent $\left(  H_{\theta_{0},\mathcal{V}}^{h}
(\theta)-z\right)  ^{-1}$ allows the representation
\begin{align}
\left(  H_{\theta_{0},\mathcal{V}}^{h}(\theta)-z\right)  ^{-1} &  =\left(
H_{ND,\mathcal{V}}^{h}(\theta)-z\right)  ^{-1}\nonumber\\
&  -\sum_{i,j=1}^{4}\left[\left(  Bq(z,\theta,\mathcal{V})-A\right)^{-1}
B\right]_{ij}\left\langle \gamma(\underline{e}_{j},\bar{z},\bar{\theta}
),\cdot\right\rangle _{L^{2}(\mathbb{R})}\gamma(\underline{e}_{i}
,z,\theta),\label{res_formula_gen}
\end{align}
and one has: $\sigma_{ess}(H_{\theta_{0},\mathcal{V}}^{h}(\theta
))=\sigma_{ess}(H_{ND,\mathcal{V}}^{h}(\theta))=e^{-2\theta}\mathbb{R}
_{\mathbb{+}}$.
\end{corollary}
The analyticity of the resolvent $\left(  H_{\theta_{0},\mathcal{V}}^{h}(\theta)-z\right)^{-1}$ w.r.t. $\theta$ is an important point in the theory of resonances. The former is obtained in the next proposition as a consequence of the formula \eqref{res_formula_gen}, the latter is developed in the next section. 
\begin{proposition}\label{prop_analyticite_res}Consider $\theta_0\in\C$, $\mathcal{V}$ fulfilling the conditions (\ref{V_local}), and let $S_{\alpha}=\left\{  \theta\in\mathbb{C}\,\left\vert \,\left\vert
\operatorname{Im}\theta\right\vert \leq \alpha\right.  \right\}$ for positive $\alpha<\frac{\pi}{4}$. Then, there exists an open subset $\mathcal{O}\subset\left\{z\in\mathbb{C}\,\left\vert\Real z < 0\right.\right\}$ such that $\forall \theta\in S_{\alpha}$, $\mathcal{O}\subset\rho\left( H_{\theta_{0},\mathcal{V}}^{h}(\theta)\right)$. Moreover, $\forall z\in\mathcal{O}$, the map: $\theta\mapsto\left(H_{\theta_{0},\mathcal{V}}^{h}(\theta)-z\right)^{-1}$ is a bounded operator-valued analytic map on the strip $S_{\alpha}$.
\end{proposition}
Before starting the proof, we note that this result implies that the $\theta$ dependent family of operators $H_{\theta_{0},\mathcal{V}}^{h}(\theta)$ is analytic in the sense of Kato in the strip $S_{\alpha}$ (definition in \cite{ReSi4}).

\begin{proof}
The equality \eqref{res_formula_gen} holding as an identity of meromorphic functions on $\C\setminus e^{-2\theta}\R_{+}$, the poles on the l.h.s. identifies with those on the r.h.s.\,. For $\theta\in S_{\alpha}$, the characterization \eqref{sigma_HND} implies that the map $z\mapsto\left(  H_{ND,\mathcal{V}}^{h}(\theta)-z\right)  ^{-1}$ is analytic when $\Real z < 0$. It results
\[
\sigma\left( H_{\theta_{0},\mathcal{V}}^{h}(\theta)\right)\cap\left\{z\in\mathbb{C}\,\left\vert\Real z < 0\right.\right\} = \mathcal{T}_0\cap\left\{z\in\mathbb{C}\,\left\vert\Real z < 0\right.\right\}\,,
\]
where
\[
\mathcal{T}_0 = \left\{z\in\mathbb{C}\,\left\vert\det\left(  Bq(z,\theta,\mathcal{V})-A\right)  =0 \right.\right\}\,.
\]
Noting that for $\Real z < 0$ and $\theta\in S_{\alpha}$ we have $\sqrt{ze^{2\theta}}=\sqrt{z}e^{\theta}$, and therefore $q(z,\theta,\mathcal{V})=q(z,0,\mathcal{V})$, we get: $\forall \theta\in S_{\alpha}$  
\[
\sigma\left( H_{\theta_{0},\mathcal{V}}^{h}(\theta)\right)\cap\left\{z\in\mathbb{C}\,\left\vert\Real z < 0\right.\right\} = \left\{z \in\mathbb{C}\,\left\vert \, \Real z < 0 \, \textrm{ and } \, \det\left(  Bq(z,0,\mathcal{V})-A\right)  =0 \right.\right\}\,,
\]
which is a discrete set independent of $\theta$. This implies that there exists an open subset $\mathcal{O}\subset\left\{z\in\mathbb{C}\,\left\vert\Real z < 0\right.\right\}$ such that $\forall \theta\in S_{\alpha}$, $\mathcal{O}\subset\rho\left( H_{\theta_{0},\mathcal{V}}^{h}(\theta)\right)$.

In what follows, we fix $z\in\mathcal{O}$. The equation \eqref{res_formula_gen} gives, for any $\theta \in S_{\alpha}$
\begin{align}
\left(  H_{\theta_{0},\mathcal{V}}^{h}(\theta)-z\right)  ^{-1} &  =\left(
H_{ND,\mathcal{V}}^{h}(\theta)-z\right)  ^{-1}\nonumber\\
&  -\sum_{i,j=1}^{4}\left[\left(  Bq(z,0,\mathcal{V})-A\right)^{-1}
B\right]_{ij}\left\langle \gamma(\underline{e}_{j},\bar{z},\bar{\theta}
),\cdot\right\rangle _{L^{2}(\mathbb{R})}\gamma(\underline{e}_{i}
,z,\theta)\,,\label{res_formula_gen_0}
\end{align}
where we used $q(z,\theta,\mathcal{V})=q(z,0,\mathcal{V})$, and we want to study the analyticity of the r.h.s. with respect to $\theta$.

Let us start with the operator $H_{ND,\mathcal{V}}^{h}(\theta)$: $\forall\theta \in S_{\alpha}$, it is a closed operator with non empty resolvent set. Moreover, $D\left(H_{ND,\mathcal{V}}^{h}(\theta)\right)$ does not depend on $\theta$ and $\forall\psi\in D\left(H_{ND,\mathcal{V}}^{h}(\theta)\right)$, $\forall f\in L^2(\R)$ the map
\[
\theta \mapsto \left\langle f, H_{ND,\mathcal{V}}^{h}(\theta)\psi\right\rangle_{L^2(\R)}
\]
is analytic on $S_{\alpha}$. This means that $H_{ND,\mathcal{V}}^{h}(\theta)$ is analytic of type (A) following the definition in \cite{ReSi4}. In addition, since $\Real z < 0$ when $z\in\mathcal{O}$, \eqref{sigma_HND} implies $z\in\rho\left(H_{ND,\mathcal{V}}^{h}(\theta)\right)$ for all $\theta \in S_{\alpha}$. Then, it results from the analyticity of type (A) propriety that the map $\theta \mapsto\left(H_{ND,\mathcal{V}}^{h}(\theta)-z\right)^{-1}$ is analytic on $S_{\alpha}$.

Concerning the finite rank part in \eqref{res_formula_gen_0}, for any $z$ with $\Real z < 0$, the functions $\gamma(\underline{e}_{i},z,\theta)$, $i=2,3$, given by \eqref{u_intz}\eqref{eq.c_i}, do no depend on $\theta$, and the functions 
\[
\gamma(\underline{e}_{1},z,\theta)=\frac{ie^{\frac{3\theta}{2}}}{h\sqrt{ze^{2\theta}}}1_{\left(  b,+\infty\right)  }e^{i\frac{\sqrt{ze^{2\theta}}(x-b)}{h}}\,,
\quad
\gamma(\underline{e}_{4},z,\theta)=\frac{ie^{\frac{3\theta}{2}}}{h\sqrt{ze^{2\theta}}}1_{\left(  -\infty,a\right)  }e^{-i\frac{\sqrt{ze^{2\theta}}(x-a)}{h}}
\]  
are such that $\forall f\in L^2(\R)$, $\theta\mapsto\left\langle f,\gamma(\underline{e}_{i},z,\theta)\right\rangle_{L^2(\R)}$ is analytic on $S_{\alpha}$. It follows that $\theta\mapsto\gamma(\underline{e}_{i},z,\theta)$ is a $L^2(\R)$-valued analytic function on $S_{\alpha}$. This propriety holding for any $z$ with $\Real z < 0$, we have also that $\theta\mapsto\overline{\gamma(\underline{e}_{i},\bar{z},\bar{\theta})}$ is a $L^2(\R)$-valued analytic function on $S_{\alpha}$. Therefore, for $i,j=1,...,4$, the operator with kernel $\gamma(\underline{e}_{i},z,\theta)\otimes\overline{\gamma(\underline{e}_{i},\bar{z},\bar{\theta})}$ is analytic w.r.t. $\theta$ on the strip $S_{\alpha}$. It allows to conclude that $\theta\mapsto\left(H_{\theta_{0},\mathcal{V}}^{h}(\theta)-z\right)^{-1}$ is a bounded operator-valued analytic map on the strip $S_{\alpha}$.
\end{proof}

\subsection{Resonances}

Next consider: $H_{\theta_{0},\mathcal{V}}^{h}=H_{\theta_{0},0}^{h}
+\mathcal{V}$, with $\mathcal{V}$ fulfilling the assumptions (\ref{V_local}).
Local perturbations of $H_{\theta_{0},0}^{h}$ can generate resonance poles for
the associated resolvent operator. These can be detected through the
deformation technique by means of an exterior complex scaling of the type
introduced in (\ref{U_theta}). To fix this point, let us introduce the set of
functions
\begin{equation}
\mathcal{A}=\left\{  u\,\left\vert \,u(x)=p(x)e^{-\beta x^{2}},\ \beta
>0\right.  \right\}  , \label{A}
\end{equation}
where $x\in\mathbb{R}$ and $p$ is any polynomial. The action of $U_{\theta}$
on the elements of $\mathcal{A}$ is
\begin{equation}
U_{\theta}\,u(x)=\left\{
\begin{array}
[c]{ll}
\medskip e^{\frac{\theta}{2}}p(e^{\theta}(x-b)+b)\,e^{-\beta\left(
e^{\theta}(x-b)+b\right)  ^{2}},& x>b\\
\medskip p(x)e^{-\beta x^{2}},& x\in(a,b)\\
e^{\frac{\theta}{2}}p(\smallskip e^{\theta}(x-a)+a)\,e^{-\beta\left(
e^{\theta}(x-a)+a\right)  ^{2}},& x<a\,.
\end{array}
\right. \label{A_teta}
\end{equation}
If $\operatorname{Re}(e^{2\theta}x^{2})>\epsilon x^{2}$ for some $\epsilon>0$,
the function $U_{\theta}\,u$ belongs to $L^{2}(\mathbb{R})$. In particular, for all positive $\alpha<\frac{\pi}{4}$, the map $\theta\mapsto U_{\theta}\,u$ is a $L^{2}$-valued analytic map on the strip $S_{\alpha}=\left\{  \theta\in\mathbb{C}\,\left\vert \,\left\vert
\operatorname{Im}\theta\right\vert \leq \alpha\right.  \right\}$.
According to the presentation of \cite{HiSi2}, the \emph{quantum resonances} of $H_{\theta
_{0},\mathcal{V}}^{h}$ are the poles of the meromorphic continuations of the
function
\begin{equation}
z\mapsto F(z,0)=\left\langle  u,\left(  H_{\theta
_{0},\mathcal{V}}^{h}-z\right)  ^{-1}v\right\rangle_{L^{2}(\mathbb{R})},\qquad
u,v\in\mathcal{A}\,, \label{F_0}
\end{equation}
from $\left\{  z\in\mathbb{C},\ \operatorname{Im}z>0\right\}  $ to $\left\{
z\in\mathbb{C},\ \operatorname{Im}z<0\right\}  $.

\begin{proposition}\label{prop.agco}
Let $\theta_{0}\in\C$ and $u,v\in\mathcal{A}$. Consider the map:
$z\mapsto F(z,0)$, a.e. defined in $\left\{\mathbb{C\,}\left\vert \,\operatorname{Im}z>0\right.  \right\}
$.\medskip\newline1) The map $z\mapsto F(z,0)$ has a meromorphic extension into the sector $\left\{\arg z\in\left(-\frac{\pi}{2},0\right)  \right\}$ of the second Riemann sheet.\medskip\newline2) The
poles of the continuation of $F(z,0)$ into the cone
\[
K_{\tau}=\left\{  \arg z\in\left(  -2\tau,0\right)  \right\}
,\qquad\text{with }\tau<\frac{\pi}{4},
\]
are eigenvalues of the operators $H_{\theta_{0},\mathcal{V}}^{h}(\theta)$ with
$\tau \leq \operatorname{Im}\theta<\frac{\pi}{4}$.
\end{proposition}

\begin{proof}
1) The proof adapts the ideas underlying the complex scaling method (see \cite{CFKS}) to the particular class exterior scaling maps $U_{\theta}$. 

Consider the strip $S_{\alpha}=\left\{  \theta\in\mathbb{C}\,\left\vert \,\left\vert\operatorname{Im}\theta\right\vert \leq \alpha\right.  \right\}$ for a positive $\alpha<\frac{\pi}{4}$. Then, consider the corresponding set $\mathcal{O}\subset\C$ given in Proposition~\ref{prop_analyticite_res} and fix $z\in\mathcal{O}$. According to Proposition \ref{prop_analyticite_res}, and to the properties of $U_{\theta}u$, $u\in\mathcal{A}$, the function
\[
F(z,\theta)=\left\langle U_{\bar{\theta}}u,\left(  H_{\theta_{0},\mathcal{V}}^{h}(\theta)-z\right)  ^{-1}U_{\theta}v\right\rangle_{L^{2}(\mathbb{R})}\label{F_teta}
\]
is analytic in the variable $\theta\in S_\alpha$. When $\theta\in\mathbb{R}$, the exterior scaling $U_{\theta}$
is an unitary map and one has
\[
F(z,\theta)=F(z,0), \quad \forall \theta \in S_\alpha\cap\R\,.
\]
Since $F(z,\theta)$ is holomorphic in $\theta$ and constant
on the real line, this is a constant function in the whole strip $S_{\alpha}$
and
\[
F(z,\theta)=F(z,0), \quad \forall \theta \in S_\alpha\,.
\]
Now, fix $\theta \in S_{\alpha}$ such that $\Imag \theta > 0$. It follows from Corollary~\ref{cor_krein} that $\sigma_{ess}(H_{\theta_{0},\mathcal{V}}^{h}(\theta))=e^{-2i\Imag\theta}\R_{+}$ and the map $z\mapsto F(z,\theta)$ is meromorphic on $\C\setminus e^{-2i\Imag\theta}\R_{+}$ such that
\begin{equation}\label{extension}
F(z,\theta)=F(z,0), \quad \forall z \in \mathcal{O}\,.
\end{equation}
Since $z\mapsto F(z,0)$ is meromorphic on $\C\setminus\R_{+}$, the equality \eqref{extension} holds as an identity of meromorphic functions $\forall z$ with $\Imag z >0$. We conclude that $F(z,\theta)$ defines a meromorphic extension of $F(z,0)$ from the set $\left\{z\in\C\left\vert \Imag z > 0  \right.\right\}$ to the sector $\left\{\arg z\in\left(-2\Imag \theta,0\right)  \right\}  $.

2) Consider $0<\tau<\frac{\pi}{4}$. From the previous point, when $\theta$ varies in $\tau \leq \operatorname{Im}\theta <\frac{\pi}{4}$, the maps $z\mapsto F(z,\theta)$ coincide in $K_{\tau}$ with meromorphic extensions of $F(z,0)$. Therefore, the poles of $z\mapsto F(z,\theta)$ in $K_{\tau}$ do not depend on $\theta$ and correspond to the poles of the meromorphic extension of $F(z,0)$ in $K_{\tau}$. The vectors $U_{\theta}u$, $u\in\mathcal{A}$, being dense in $L^{2}(\mathbb{R})$, the poles of $z\mapsto F(z,\theta)$ corresponds to eigenvalues of $H_{\theta_{0},\mathcal{V}}^{h}(\theta)$.
\end{proof}

\subsection{A time dependent model}

Consider the non-autonomous model $H_{\theta_{0},\mathcal{V}(t)}^{h}
(\theta_{0})$, where $\mathcal{V}(t)$ is a family of self-adjoint potentials
composed by $\mathcal{V}=\mathcal{V}_{1}+\mathcal{V}_{2}$
\begin{equation}
\mathcal{V}_{1}(t)\in C^{0}\left(  \mathbb{R}_{+};L^{\infty}((a,b))\right)  ,\qquad
\mathcal{V}_{2}(t)=\sum_{j=1}^{n}\alpha_{j}(t)\delta(x-c_{j})\,,\label{v_2_delta}
\end{equation}
with $\left\{  c_{j}\right\}  \subset\left(  a,b\right)  $, $\alpha_{j}(t)\in
C^{1}(\mathbb{R}_{+};\mathbb{R})$. According to the specific feature of point
perturbations, the domain's definition, given in (\ref{H_(teta)+V_dom}), can
be rephrased as
\begin{multline}
D(H_{\theta_{0},\mathcal{V}(t)}^{h}(\theta_{0}))=\left\{  u\in H^{2}
(\mathbb{R}\backslash\left\{  a,b,c_{1},...,c_{n}\right\}  )\cap H^{1}(\mathbb{R}
\backslash\left\{  a,b\right\}  )\left\vert \right.\right.\\ 
\left.h^{2}\left[  u^{\prime}(c_{j}^{+})-u^{\prime}(c_{j}^{-})\right]
=\alpha_{j}(t)u(c_{j}) \,\textrm{ and }\, \eqref{BC_1}\,\textrm{ holds} \right\}\,,\label{H_(teta)+V(t)_dom}
\end{multline}
where time dependent boundary conditions appear in the interaction points
$c_{j}$.

Most of the techniques employed in the analysis the Cauchy problem
\begin{equation}
\left\{
\begin{array}
[c]{l}
\medskip i\partial_{t}u=H(t)u\\
u_{t=0}=u_{0}
\end{array}
\right.  \label{Cauchy0}
\end{equation}
for non-autonomous Hamiltonians, $H(t)$, require, as condition, that the
operator's domain is independent of the time (we mainly refer to the
Yoshida's and Kato's results (\cite{Kat1}\cite{Kat2}\cite{Yos});
 an extensive presentation of the subject
can be found given in \cite{Fat}). In the particular case of $H_{\theta
_{0},\mathcal{V}(t)}^{h}(\theta_{0})$, one can explicitly construct a family
of unitary maps $V_{t,t_{0}}$ such that
\begin{equation}
V_{t,t_{0}}H_{\theta_{0},\mathcal{V}(t)}^{h}(\theta_{0})V_{t,t_{0}}^{-1}
\label{H(t)_conjugated0}
\end{equation}
has a constant domain. To fix this point, let us introduce a time dependent
real vector field $g(x,t)$ and assume
\begin{equation}
\left\{
\begin{array}
[c]{l}
\medskip\text{i)}\ g(\cdot,t)\in C^{1}(\mathbb{R}_{+};C_{0}^{\infty
}(\mathbb{R}))\,,\\
\medskip\text{ii) }g(c_{j},t)=0,\quad j=1,n,\ \forall t\,,\\
\text{iii) } \supp g(\cdot,t)=\cup_{j=1}^{n}I_{c_{j}},\quad\text{with: }I_{c_{j}
}=\left(  c_{j}-\epsilon_{j},c_{j}+\epsilon_{j}\right)  \text{ and }\cap
_{j}I_{c_{j}}=\emptyset\,.
\end{array}
\right.  \label{Assumptions1}
\end{equation}
For $\epsilon_{j}$ small, $g(\cdot,t)$ has support strictly included in
$\left(  a,b\right)  $ and localized around the interaction points $x=c_{j}$.
According to i), $g(\cdot,t)$ satisfy a global Lipschitz condition uniformly in
time: then the dynamical system
\begin{equation}
\left\{
\begin{array}
[c]{l}
\medskip\dot{y}_{t}=g(y_{t},t)\\
y_{t_{0}}=x
\end{array}
\right.  \label{loc_dil}
\end{equation}
admits a unique global solution continuously depending on time and Cauchy
data $\left\{  t_{0},x\right\}$. Using the notation: $y_{t}
=F(t,t_{0},x)$, one has
\begin{equation}
\left\vert F(t,t_{0},x_{1})-F(t,t_{0},x_{2})\right\vert \leq e^{M\left\vert
t-t_{0}\right\vert }\left\vert x_{1}-x_{2}\right\vert ,\qquad\text{with:
}M=\sup_{x\in\mathbb{R},\,t\geq0}\partial_{y}g(y,t)\text{.}
\end{equation}
Consider the variations of $F(t,t_{0},x)$ w.r.t. $x$: $z_{t}=\partial
_{x}F(t,t_{0},x)$. From (\ref{loc_dil}), one has
\begin{equation}
\left\{
\begin{array}
[c]{l}
\medskip\dot{z}_{t}=\partial_{1}g(F(t,t_{0},x),t)\,z_{t}\\
z_{t_{0}}=1\,,
\end{array}
\right.  \label{loc_dil_variation}
\end{equation}
with $\partial_{1}g$ denoting the derivative w.r.t. the first variable. The
solution to this problem is positive and explicitly given by
\begin{equation}
\partial_{x}F(t,t_{0},x)=e^{\int_{t_{0}}^{t}\partial_{1}g(F(s,t_{0}
,x),s)ds}>0\qquad\forall x\in\mathbb{R}\,. \label{loc_dil_variation1}
\end{equation}
According to (\ref{loc_dil_variation1}), $F(t,t_{0},\cdot)$ is a $C^{\infty}
$-diffeomorphism and the map $x\mapsto F(t,t_{0},x)$ is a
time-dependent local dilation around the points $c_{j}$. In particular, it
follows from the assumptions i) and ii) that
\begin{equation}
F(t,t_{0},c_{j})=c_{j},\quad F(t,t_{0},x)=x\text{ for all }x\in\mathbb{R}
\backslash\supp g\,, \label{loc_dil_property1}
\end{equation}
and
\begin{equation}
F(t,t_{0},F(t_{0},\bar{t},x))=F(t,\bar{t},x)\,. \label{loc_dil_property2}
\end{equation}
The unitary transformation associated with the change of variable $x\rightarrow
F(t,t_{0},x)$ is
\begin{equation}
\left\{
\begin{array}
[c]{l}
\medskip\left(  V_{t_{0},t}u\right)  (x)=\left(  \partial_{x}F(t,t_{0}
,x)\right)  ^{\frac{1}{2}}u(F(t,t_{0},x))\\
V_{t,t_{0}}=V_{t_{0},t}^{-1}\,.
\end{array}
\right. \label{V_t_t0}
\end{equation}
Regarded as a function of time, $t\mapsto V_{t_{0},t}$ is a strongly
continuous differentiable map and one has
\begin{equation}
\partial_{t}V_{t_{0},t}=V_{t_{0},t}\,\left[  \frac{1}{2}\left(  \partial
_{y}g\right)  +g\partial_{y}\right]\,.  \label{V_t_t0_variation}
\end{equation}
The form of the conjugated Hamiltonian $V_{t,t_{0}}H_{\theta_{0}
,\mathcal{V}(t)}^{h}(\theta_{0})V_{t_{0},t}$ follows by direct computation
\begin{multline}\label{H(t)_conjugated}
V_{t,t_{0}}H_{\theta_{0},\mathcal{V}(t)}^{h}(\theta_{0})V_{t_{0},t}
=-h^{2}\eta(y)\left[  \partial_{y}b^{2}\partial_{y}+\left(  \partial
_{y}ab\right)  -a^{2}\right]\\ +\mathcal{V}_{1}(F(t_{0},t,y),t)+\sum_{j=1}^{n}\alpha
_{j}(t)\,b(c_{j},t)\,\delta(x-c_{j}), 
\end{multline}
\begin{equation}
b(y,t)=e^{\int_{t_{0}}^{t}\partial_{1}g\left(  F(s,t,y),s\right)  \,ds},\qquad
a(y,t)=\frac{1}{2}b^{\frac{1}{2}}(y,t)\int_{t_{0}}^{t}\partial_{1}^{2}g\left(
F(s,t,y),s\right)  b(y,s)\,ds\,. \label{H(t)_conjugated_ab}
\end{equation}
After conjugation, the boundary conditions in the operator's domain change as
\begin{equation}
h^{2}b^{2}(c_{j},t)\,\left[  u^{\prime}(c_{j}^{+})-u^{\prime}(c_{j}
^{-})\right]  =b(c_{j},t)\,\alpha_{j}(t)\,u(c_{j})\,. \label{H(t)_conjugated_BC}
\end{equation}
Assume, for all $j=1...n$
\begin{equation}
\text{iv) }\alpha_{j}(0)\neq0;\text{\quad v) }\alpha_{j}(0)\alpha
_{j}(t)>0;\quad\text{vi) }\alpha_{j}(t)\in C^{1}(\mathbb{R}_{+};\mathbb{R})\,.
\label{Assumptions2}
\end{equation}
One can determine (infinitely many) $g(\cdot,t)\in C^{0}\left(  \mathbb{R}
_{+};C_{0}^{\infty}(\mathbb{R})\right)  $ such that
\begin{equation}
\partial_{1}g(c_{j},t)=\frac{\dot{\alpha}_{j}(t)}{\alpha_{j}(t)}\,.
\label{Assumptions3}
\end{equation}
This implies
\begin{equation}
b(c_{j},t)=e^{\int_{t_{0}}^{t}\partial_{1}g\left(  F(s,t,c_{j}),s\right)
\,ds}=e^{\int_{t_{0}}^{t}\partial_{1}g\left(  c_{j},s\right)  \,ds}
=\frac{\alpha_{j}(t)}{\alpha_{j}(0)} \label{H(t)_conjugated_BC1}
\end{equation}
and (\ref{H(t)_conjugated_BC}) can be written in the time-independent form
\begin{equation}
h^{2}\left[  u^{\prime}(c_{j}^{+})-u^{\prime}(c_{j}^{-})\right]  =\alpha
_{j}(0)\,u(c_{j})\,. \label{H(t)_conjugated_BC2}
\end{equation}
Thus, the Hamiltonians $V_{t,t_{0}}H_{\theta_{0},\mathcal{V}(t)}^{h}
(\theta_{0})V_{t_{0},t}$ have common domain given by
\begin{equation}
Y=\left\{  u\in H^{2}(\mathbb{R}\backslash\left\{  a,b,c_{j}\right\}  )\cap
H^{1}(\mathbb{R}\backslash\left\{  a,b\right\}  )\left\vert \ \left[
\ref{BC_1}\right]  ,\ \left[  \ref{H(t)_conjugated_BC2}\right]  \right.
\right\}  . \label{Y}
\end{equation}

Consider the time evolution problem for $H_{\theta_{0},\mathcal{V}(t)}
^{h}(\theta_{0})$
\begin{equation}
\left\{
\begin{array}
[c]{l}
\medskip i\partial_{t}u=H_{\theta_{0},\mathcal{V}(t)}^{h}(\theta_{0})u\\
u_{t_{0}}=u_{0}\,.
\end{array}
\right.  \label{Cauchy1}
\end{equation}
Setting $v=V_{t,t_{0}}u$, and
\begin{equation}
A(t)=V_{t,t_{0}}H_{\theta_{0},\mathcal{V}(t)}^{h}(\theta_{0})V_{t_{0}
,t}-i\left[  \frac{1}{2}\left(  \partial_{y}g\right)  +g\partial_{y}\right]\,,
\label{A(t)}
\end{equation}
with $D(A(t))=Y$, one has: $v_{t_{0}}=u_{t_{0}}$,
\begin{equation}
\left\{
\begin{array}
[c]{l}
\medskip\partial_{t}v=-iA(t)v\\
v_{t_{0}}=u_{0}\,.
\end{array}
\right.  . \label{Cauchy2}
\end{equation}

\begin{proposition}
\label{pr.Kato}
Let $\mathcal{V}(t)=\mathcal{V}_{1}(t)+\mathcal{V}_{2}(t)$ be defined with the conditions
(\ref{v_2_delta}) and (\ref{Assumptions2}).
Assume: $\theta=\theta_{0}=i\tau$, with $\tau\in\left(  0,\frac{\pi}
{2}\right)  $. There exists an unique family of operators $U_{t,s}$, $0\leq
s\leq t$, with the following properties:
\medskip
\newline$\mathbf{a})$
$U_{t,s}$ is strongly continuous in $L^{2}(\mathbb{R})$ w.r.t. the variables
$s$ and $t$ and fulfills the conditions: $U_{s,s}=Id$, $U_{t,s}\circ
U_{s,r}=U_{t,r}$ for $r\leq s\leq t$ and $\left\Vert U_{t,s}\right\Vert \leq1$
for any $s$ and $t$, $s\leq t$.
\medskip
\newline$\mathbf{b})$ For $u_{s}\in D(H_{\theta
_{0},\mathcal{V}(s)}^{h}(\theta_{0}))$, one has: $U_{t,s}u_{s}\in$
$D(H_{\theta_{0},\mathcal{V}(t)}^{h}(\theta_{0}))$ for all $t\geq s$.In
particular, for fixed $t$, $U_{t,s}$ is strongly continuous w.r.t. $s$ in the
norm of $H^{2}(\mathbb{R}\backslash\left\{  a,b,c_{j}\right\}  )\cap
H^{1}(\mathbb{R}\backslash\left\{  a,b\right\}  )$. While, for fixed $s$,
$U_{t,s}$ is strongly continuous w.r.t. $t$ in the same norm, except possibly
countable values of $t\geq s$.
\medskip
\newline$\mathbf{c})$ For fixed $s$ and
$u_{s}\in D(H_{\theta_{0},\mathcal{V}(s)}^{h}(\theta_{0}))$, the derivative
$\frac{d}{dt}U_{t,s}u_{s}$ exists and is strongly continuous in $L^{2}
(\mathbb{R})$ except, possibly, countable values $t\geq s$. With similar
exceptions, one has: $\frac{d}{dt}U_{t,s}u_{s}=-iH_{\theta_{0},\mathcal{V}
(t)}^{h}(\theta_{0})U_{t,s}u_{s}$.
 \medskip
\newline$\mathbf{d})$ Additionally, if
$\mathcal{V}_{1}(t)\in C^{1}\left(  \mathbb{R}_{+},L^{\infty}((a,b))\right)  \cap
C^{0}\left(  \mathbb{R}_{+},W^{1,\infty}(\mathbb{R})\right)  $, \ and
$\alpha_{j}(t)\in C^{2}(\mathbb{R}_{+};\mathbb{R})$, then the
conclusion of point $(c)$ holds 
for all $t\geq s$
without exceptions.
\end{proposition}

\begin{proof}
Since the Cauchy problems (\ref{Cauchy1}) and (\ref{Cauchy2}) are related by
the time-differentiable map $V_{t,t_{0}}$, it is enough to prove the result in
the case of $A(t)$. Let assume the conditions (\ref{Assumptions1}) and
(\ref{H(t)_conjugated_BC1}) to hold, and start to consider the properties of
this operator.\bigskip\newline I) As already noticed (see relation
(\ref{accretivity1})), $iH_{\theta_{0},\mathcal{V}(t)}^{h}(\theta_{0})$ is an accretive
operator. This property extends to $V_{t,t_{0}}\left(  iH_{\theta
_{0},\mathcal{V}(t)}^{h}(\theta_{0})\right)  V_{t_{0},t}$, which is unitarily
equivalent to an accretive operator, and to $A(t)$, since, as a straightforward
computation shows, the contribution $i\left[  \frac{1}{2}\left(  \partial
_{y}g\right)  +g\partial_{y}\right]  $ is self-adjoint.
The spectral profile of $A(t)$ essentially follows from the properties of
$H_{\theta_{0},\mathcal{V}(t)}^{h}(\theta_{0})$. Indeed, we notice that:
$\sigma\left(  V_{t,t_{0}}H_{\theta_{0},\mathcal{V}(t)}^{h}(\theta
_{0})V_{t_{0},t}\right)  =\sigma\left(  H_{\theta_{0},\mathcal{V}(t)}
^{h}(\theta_{0})\right)  $, since the two operators are unitarily equivalent.
Moreover, the term $i\left[  \frac{1}{2}\left(  \partial_{y}g\right)
+g\partial_{y}\right]  $ is relatively compact w.r.t. $V_{t,t_{0}}
H_{\theta_{0},\mathcal{V}(t)}^{h}(\theta_{0})V_{t_{0},t}$, since it has a
lower differential order (see definition (\ref{A(t)})). Then, $\sigma
_{ess}(A(t))=\sigma_{ess}(H_{\theta_{0},\mathcal{V}(t)}^{h}(\theta
_{0}))=e^{-2i\operatorname{Im}\theta_{0}}\mathbb{R}_{+}$, as it follows from
Corollary \ref{cor_krein}, and $\left(  A(t)-z\right)  ^{-1}$ is a
meromorphic function of $z$ in $\mathbb{C}\backslash e^{-2i\operatorname{Im}
\theta_{0}}\mathbb{R}_{+}$. This result yields: $\sigma_{ess}
(iA(t))=e^{i(\frac{\pi}{2}-2\operatorname{Im}\theta_{0})}\mathbb{R}_{+}$ and:
$\rho(iA(t))\cap\mathbb{R}_{-}\neq\emptyset$.
As a consequence of the above, one has: i) $iA(t)$ is accretive; 2)
$-\lambda_{0}\in\rho(iA(t))$ for some $\lambda_{0}>0$. Then, for any fixed
$t$, $iA(t)$ is the generator of a contraction semigroup, $e^{-isA(t)}$, on
$L^{2}(\mathbb{R})$ (see \cite{ReSi2}, Th. X.48). In particular, the operator's
domain $D(iA(t))=Y$, defined in (\ref{Y}), is invariant by the action of
$e^{-isA(t)}\Longrightarrow$ $Y$ is $iA(t)-$\emph{admissible} for any $t$.
Moreover, as $iA(t)$ is the generator of a contraction semigroup, from the
Hille-Yoshida's theorem it follows that
\begin{equation}\label{Hille-Yoshida}
\left\{
\begin{array}[c]{l}
\mathbb{R}_{-}\subset\rho(iA(t))\,,\\[1.65mm]
\left\Vert \left(  iA(t)+\lambda\right)  ^{-1}\right\Vert \leq\frac
{1}{\lambda}\text{ \ for all }\lambda>0\,.
\end{array}
\right.
\end{equation}
In particular, for any finite collection of values $0\leq t_{1}\leq...\leq
t_{k}$, one has
\begin{equation}
\left\Vert
{\textstyle\prod\nolimits_{j=1}^{k}}
\left(  iA(t_{j})+\lambda\right)  ^{-1}\right\Vert \leq\frac{1}{\lambda^{k}}\,.\label{stability}
\end{equation}
This relation implies that $iA(t)$ is \emph{stable} (with coefficients $M=1$
and $\beta=0$, according to the definition given in \cite{Kat2}).
\bigskip
\newline II) Consider the action of $A(t)$ on its domain. Here, the Hilbert
space $Y$ is provided with the norm $H^{2}(\mathbb{R}\backslash\left\{
a,b,c_{1},...,c_{n}\right\}  )\cap H^{1}(\mathbb{R}\backslash\left\{  a,b\right\}  )$.
For $u\in Y$, we use the decomposition: $u=1_{\mathbb{R}\backslash\left(
a,b\right)  }u+1_{\left(  a,b\right)  }u=u_{ext}+u_{in}$. Recalling that the
functions $g(\cdot,t)$, $b(\cdot,t)$, $a(\cdot,t)$ are supported inside
$\left(  a,b\right)  $, from the definition (\ref{H(t)_conjugated}),
(\ref{A(t)}), one has
\begin{equation}
\left(  A(t)-A(s)\right)  u=\left(  A(t)-A(s)\right)  u_{in}\,,
\label{A(t)_continuity}
\end{equation}
with
\begin{align}
A(t)u_{in} &  =\left[-h^{2}\partial_{y}b^{2}\partial_{y}+\sum_{j=1}
^{n}\alpha_{j}(t)\,b(c_{j},t)\,\delta(x-c_{j})\right]  u_{in}\nonumber\\
&  +\left[  -h^{2}\left( \left(  \partial_{y}ab\right)  -a^{2}\right)
+\mathcal{V}_{1}(F(t_{0},t,y),t)-i\left( \frac{1}{2}\left(  \partial_{y}g\right)
+g\partial_{y}\right) \right]  u_{in}\,.\label{A(t)_continuity1}
\end{align}
Taking into account the boundary conditions (\ref{H(t)_conjugated_BC2}), the
second order term in (\ref{A(t)_continuity1}) writes as
\[
-h^{2}\left(  \partial_{y}b^{2}\partial_{y}\right)  u_{in}=-h^{2}\left(
\partial_{y}b^{2}\right)  \,\partial_{y}u_{in}-h^{2}b^{2}\,\Delta_{\left(
a,b\right)  \backslash\left\{  c_{j}\right\}  }u_{in}-\sum_{j=1}^{n}\alpha
_{j}(0)\,b^{2}(c_{j},t)\,u_{in}(c_{j})\,\delta(x-c_{j})\,,
\]
where, according to (\ref{H(t)_conjugated_BC1}),
\[
\sum_{j=1}^{n}\alpha_{j}(0)\,b^{2}(c_{j},t)\,u_{in}(c_{j})\delta(x-c_{j}
)=\sum_{j=1}^{n}\alpha_{j}(t)\,b(c_{j},t)u_{in}(c_{j})\,\delta(x-c_{j})\,.
\]
From these relations, it follows
\begin{align*}
\medskip A(t)u_{in} &  =-h^{2}b^{2}\,\Delta_{\left(  a,b\right)
\backslash\left\{  c_{j}\right\}  }u_{in}-h^{2}\left[  \left(  \partial
_{y}b^{2}\right)  \,\partial_{y}+\left(  \partial_{y}ab\right)  -a^{2}\right]
u_{in}\\
&  +\mathcal{V}_{1}(F(t_{0},t,y),t)u_{in}-i\left[  \frac{1}{2}\left(  \partial
_{y}g\right)  +g\partial_{y}\right]  u_{in}\,.
\end{align*}
Therefore, we have
\begin{align*}
\medskip\left(  A(t)-A(s)\right)  u &  =-h^{2}\left[  \left(  b_{t}^{2}
-b_{s}^{2}\right)  \,\Delta_{\left(  a,b\right)  \backslash\left\{
c_{j}\right\}  }+\left(  \partial_{1}b_{t}^{2}-\partial_{1}b_{s}^{2}\right)
\,\partial_{y}+\left(  \partial_{1}a_{t}b_{t}-\partial_{1}a_{s}b_{s}\right)
-\left(  a_{t}^{2}-a_{s}^{2}\right)  \right]  u_{in}\\
&  +\left(  \mathcal{V}_{1}(F(t_{0},t,y),t)-\mathcal{V}_{1}(F(t_{0},s,y),s)\right)  u_{in}
-i\left[  \frac{1}{2}\left(  \partial_{1}g_{t}-\partial_{1}g_{s}\right)
+\left(  g_{t}-g_{s}\right)  \partial_{y}\right]  u_{in}\,,
\end{align*}
and
\begin{equation}
\left\Vert \left(  A(t)-A(s)\right)  u\right\Vert _{L^{2}(\mathbb{R})}\leq
C_{h}\,M(t,s)\,\left\Vert u\right\Vert _{H^{2}(\mathbb{R}\backslash\left\{
a,b,c_{j}\right\}  )\cap H^{1}(\mathbb{R}\backslash\left\{  a,b\right\})}\,,\label{A(t)_continuity2}
\end{equation}
where $C_{h}$ is a positive constant depending on $h$, while
\begin{align*}
\medskip M(t,s) &  =\left\Vert \mathcal{V}_{1}(t)-\mathcal{V}_{1}(s)\right\Vert _{L^{\infty}
((a,b))}+\left\Vert a_{t}^{2}-a_{s}^{2}\right\Vert _{C^{0}([a,b])}+\left\Vert
\partial_{1}a_{t}b_{t}-\partial_{1}a_{s}b_{s}\right\Vert _{L^{\infty}((a,b))}\\
&  +\sum_{l=0,1}\left(  \left\Vert \partial_{1}^{l}b_{t}^{2}-\partial_{1}
^{l}b_{s}^{2}\right\Vert _{L^{\infty}((a,b))}+\left\Vert \partial_{1}^{l}
g_{t}-\partial_{1}^{l}g_{s}\right\Vert _{L^{\infty}((a,b))}\right)\,.
\end{align*}
From our assumptions, $t\mapsto g_{t},b_{t},a_{t}$ are continuous
$C^{\infty}$-valued maps, and $t\mapsto \mathcal{V}_{1}\in C^{0}(\mathbb{R}
_{+},L^{\infty})$; this yields: $\lim_{s\rightarrow t}M(t,s)=0$, and, due to
(\ref{A(t)_continuity2}),
\[
\lim_{s\rightarrow t}\left\Vert \left(  A(t)-A(s)\right)  u\right\Vert
_{L^{2}(\mathbb{R})}=0,\qquad\forall u\in Y\,.
\]
This gives the continuity of $t\mapsto A(t)$ in the $\mathcal{L}
(Y,L^{2}(\mathbb{R}))$-operator norm.\bigskip\newline III) Consider the map:
$S(t)=\left(  iA(t)+\lambda_{0}\right)  $ for $\lambda_{0}>0$. According to
(\ref{Hille-Yoshida}), $S(t)$ defines a family of isomorphisms of $Y$ to
$L^{2}(\mathbb{R})$. Moreover, making use of the relations
(\ref{H(t)_conjugated}), (\ref{A(t)}) and the condition
I\ref{H(t)_conjugated_BC1}), the time derivative of $S(t)$ writes as
\begin{align*}
\medskip\frac{d}{dt}S(t)  & =-h^{2}\left[  \partial_{y}\left(  2b\partial
_{t}b\right)  \partial_{y}+\left(  \partial_{y,t}ab\right)  -2a\partial
_{t}a\right]  +\sum_{j=1}^{n}\frac{2\alpha_{j}(t)\,\alpha_{j}^{\prime}
(t)}{\alpha_{j}(0)}\,\delta(x-c_{j})\\
& +\partial_{t}\mathcal{V}_{1}(F(t_{0},t,y),t)+\left(  \partial_{t}F(t_{0},t,y)\right)
\,\partial_{1}\mathcal{V}_{1}(F(t_{0},t,y),t)-i\left[  \frac{1}{2}\left(  \partial
_{y,t}g\right)  +\left(  \partial_{t}g\right)  \partial_{y}\right]\,,
\end{align*}
where the term $\partial_{t}F(t_{0},t,y)$ can be written in the form
\[
\partial_{t}F(t_{0},t,y)=-b^{-1}(y,t)\,g(F(t_{0},t,y),t)\,,
\]
as it follows by using: $F(t,t_{0},F(t_{0},t,y))=y$ and the definition of
$b(y,t)$. Since the variations of the functions $a,b,\mathcal{V}_{1},g$ w.r.t. both the
variables $t$ and $y$ are supported on $\left(  a,b\right)  $, each
contribution to $\frac{d}{dt}S(t)$ acts only inside this interval. Therefore,
$u\in D\left(  \frac{d}{dt}S(t)\right)  $ is not expected to fulfill any
interface condition at $x=a,b$, while in the interaction points $c_{j}$, one
has
\[
h^{2}2b(c_{j},t)\partial_{t}b(c_{j},t)\,\left[  u^{\prime}(c_{j}
^{+})-u^{\prime}(c_{j}^{-})\right]  =\frac{2\alpha_{j}(t)\,\alpha_{j}^{\prime
}(t)}{\alpha_{j}(0)}\,u(c_{j})\,,
\]
which, according to (\ref{H(t)_conjugated_BC1}), reduces to
(\ref{H(t)_conjugated_BC2}). It follows that: $Y\subset D\left(  \frac{d}
{dt}S(t)\right)  $, and one can consider the action of $\frac{d}{dt}S(t)$ on
$Y$. Proceeding as in point II and using the stronger assumptions:
$\mathcal{V}_{1}(t)\in C^{1}\left(  \mathbb{R}_{+},L^{\infty}((a,b))\right)  \cap
C^{0}\left(  \mathbb{R}_{+},W^{1,\infty}(\mathbb{R})\right)  $, $\alpha
_{j}(t)\in C^{2}(\mathbb{R}_{+};\mathbb{R})$, one shows that $\frac{d}
{dt}S(t)$ is strongly continuous from $Y$ to $L^{2}(\mathbb{R})$
.\bigskip\newline
Finally, the points I and II resume as follows: i) the Hilbert space $Y$ is
$A(t)$-admissible for all $t$, ii) $A(t)$ define a stable family operators
$t$-continuous in the $\mathcal{L}(Y,L^{2}(\mathbb{R}))$-operator norm. Thus,
the Theorem 5.2 in \cite{Kat2} applies and provides a strongly continuous
dynamical system $\mathcal{U}_{t,t_{0}}$, $t_{0}\leq t$, for the Cauchy problem
(\ref{Cauchy2}) with $\left\Vert \mathcal{U}_{t,t_{0}}\right\Vert \leq1$ and
such that: 
\begin{description}
\item[1)] For fixed $t_{0}$, the family $\mathcal{U}_{t,t_{0}}$ defines an
a.e. strongly continuous flow in $Y$; for fixed $t$, $\mathcal{U}_{t,t_{0}}$
is strongly continuous w.r.t. $t_{0}$ in $Y$.
\item[2)] For $y\in Y$, the derivative
$\frac{d}{dt}\mathcal{U}_{t,t_{0}}y$ is a.e. strongly continuous in
$L^{2}(\mathbb{R})$ and one has: $\frac{d}{dt}\mathcal{U}_{t,t_{0}}
u_{s}=-iA(t)\mathcal{U}_{t,t_{0}}u_{s}$. The conclusions $(a)$, $(b)$ and
$(c)$ of the statement follows by using the equivalence of the problems
(\ref{Cauchy1}) and (\ref{Cauchy2}) through the maps $V_{t,t_{0}}$. 
\end{description}
Moreover, assuming 
\[
\mathcal{V}_{1}(t)\in C^{1}\left(  \mathbb{R}_{+},L^{\infty}((a,b))\right)  \cap C^{0}\left(  \mathbb{R}_{+},W^{1,\infty}(\mathbb{R})\right), \quad 
\alpha_{j}(t)\in C^{2}(\mathbb{R}_{+};\mathbb{R})
\]
it follows from point III that: $S(t)$ is a family of isomorphisms from $Y$ to
$L^{2}(\mathbb{R})$, with
\[
S(t)A(t)S^{-1}(t)=A(t)
\]
and such that $S(t)$ is strongly differentiable. In this case, Yoshida's
Theorem applies (see Theorem 6.1 and Remark 6.2 in \cite{Kat2}) and the
conclusion $(d)$ of the statement follows.
\end{proof}

\section{Exponential decay estimates}
\label{se.expdecay}

With this section, with start the analysis of the parameter dependent quantities as $h\rightarrow 0$. The exponential decay is specified with a good control of the prefactors which behave like $\frac{1}{h^N}$. These estimates are written for potentials with limited regularity assumptions in order to hold for the modelling of quantum wells in a semi-classical island with non-linear effect. Some preliminary estimates are reviewed in Appendix~\ref{app.A}.

\subsection{Exponential decay for  the Dirichlet problem}
\label{se.expdir}

Consider $V\in L^{\infty}((a,b);\R)$, and $W^{h}=W_1^h+W_2^h$ an
$h$-dependent  real-valued potential with:
\begin{equation}
  \label{eq.hyp0}
c1_{[a,b]}\leq  V\,,\quad
\|V\|_{L^{\infty}}\leq \frac{1}{c}\,,\quad
\|W_1^h\|_{L^{\infty}}\leq \frac{1}{c}\,, \quad \quad \|W_2^h\|_{\mathcal{M}_b}\leq \frac{1}{c}h\,,
\end{equation}
where $\|\mu\|_{\mathcal{M}_b}$ denotes the total variation of the measure $\mu\in\mathcal{M}_b((a,b))$. The constant $c$ denotes a fixed positive value that can be chosen small when it is required by the analysis. We suppose that $W_1^h$ and $W_2^h$ are supported in the domain
$U_h=\{x\in(a,b); \, d(x,U)\leq h\}$ where $U$ is a fixed compact subset of $(a,b)$.\\ 
After introducing the differential operators on $(a,b)$
$$
\tilde{P}^{h}:=-h^{2}\Delta +V\quad\text{and}\quad 
P^{h}:=-h^{2}\Delta + V-W^{h}\,,
$$
two Dirichlet Hamiltonians are considered
\begin{eqnarray}
\label{eq.Dirfilled}
 \tilde{H}_{D}^{h}:= -h^{2}\Delta +V=\tilde{P}^{h}\,,
&&\text{with}\quad
  D(\tilde{H}_{D}^{h})=H^{2}((a,b))\cap H^{1}_{0}((a,b))
\,,\\
\label{eq.Dirwell}
H_{D}^{h}:= -h^{2}\Delta +V-W^{h}=P^{h}\,,
&&\text{with}\quad
  D(H_{D}^{h})=\left\{u\in H^{1}_{0}((a,b)), P^{h}u\in L^{2}((a,b))\right\}\,.
\end{eqnarray}
For a real energy $\lambda\in \R$ we consider the Agmon degenerate
distance associated with $V$
$$
d_{Ag}(x,y, V,\lambda)=\int_{x}^{y}\sqrt{(V(t)-\lambda)_{+}}~dt\,,\quad
x\leq y\,.
$$
And an other tool that will be useful here is the $h$-dependent $H^k$ norm
\begin{equation}
  \label{eq.normsobh}
\|u\|_{H^{k,h}}^{2}=\sum_{\alpha\leq k}\|(h\partial_{x})^{\alpha}u\|_{L^{2}}^{2}.
\end{equation}
\begin{proposition}
\label{pr.Direxp}
\textbf{i)} Consider $f = f_1+f_2$ with $f_1\in L^2((a,b))$ and $f_2\in \mathcal{M}_b((a,b))$. If $V-\Real z\geq c$ \vskip 0.1cm
with $\|V\|_{L^{\infty}}\leq \frac{1}{c}$, then any solution $u\in H^{1}_{0}((a,b))$ to
$(\tilde{P}^{h}-z)u=f$ satisfies
\begin{equation}
  \label{eq.expdecHtil}
h^{1/2}\sup_{x\in [a,b]}|e^{\frac{\varphi(x)}{h}}u(x)|+
\|he^{\frac{\varphi}{h}}u'\|_{L^{2}}+\|e^{\frac{\varphi}{h}}u\|_{L^{2}}
\leq 
\frac{C_{a,b,c}}{h}\left(\|f_1\|_{L^2}+\frac{1}{h^{\frac{1}{2}}}\|f_2\|_{\mathcal{M}_b}\right)\,,
\end{equation}
with $\varphi(x)=d_{Ag}(x,K,V,\Real z)$ and $K\supset\supp f_1\cup\supp f_2$\,.\\
\textbf{ii)} Consider $f\in L^2((a,b))$. If $V-\Real z\geq c$ with 
$\|V\|_{L^{\infty}}\leq \frac{1}{c}$, then any solution $u\in D(H_{D}^{h})$ to 
$(H_{D}^{h}-z)u=f$ satisfies 
\begin{equation}
  \label{eq.expdecH}
h^{1/2}\sup_{x\in [a,b]}|e^{\frac{\varphi(x)}{h}}u(x)|+
\|he^{\frac{\varphi}{h}}u'\|_{L^{2}}+\|e^{\frac{\varphi}{h}}u\|_{L^{2}}
\leq 
\frac{C_{a,b,c}}{h}\left(\|u\|_{L^2}+\|f\|_{L^2}\right)\,,
\end{equation}
with $\varphi(x)=d(x,K'\cup U, V,\Real
z)$ and $K'\supset\supp f$\,. Especially when $z\notin\sigma(H^h_D)$, we have:
\[
h^{1/2}\sup_{x\in [a,b]}|e^{\frac{\varphi(x)}{h}}u(x)|+
\|he^{\frac{\varphi}{h}}u'\|_{L^{2}}+\|e^{\frac{\varphi}{h}}u\|_{L^{2}}
\leq 
\frac{C_{a,b,c}}{h}\left(\frac{1}{d(z,\sigma(H^h_D))}+1\right)\|f\|_{L^2}\,,
\]
and, when $z=E^{h}$ is an eigenvalue of $H_{D}^{h}$, the related
normalized eigenvector satisfies
$$
h^{1/2}\sup_{x\in [a,b]}|e^{\frac{\varphi(x)}{h}}u(x)|+
\|he^{\frac{\varphi}{h}}u'\|_{L^{2}}+\|e^{\frac{\varphi}{h}}u\|_{L^{2}}
\leq \frac{C_{a,b,c}}{h}\,,
$$
with $\varphi(x)=d_{Ag}(x,U,V,E^{h})$\,.
\end{proposition}
\begin{remark}
  The negative exponents of $h$ in the upper bounds are not the
  optimal ones. Some care especially has
  to be taken while modifying $\varphi$ or while commuting
  $h\partial_{x}$ with $e^{\frac{\varphi}{h}}$. This presentation is
  the most flexible one for our purpose.
\end{remark}
\noindent\textbf{Proof:} \textbf{i)} Our assumptions imply that the functions
$\varphi(x)=d_{Ag}(x,K,V,\Real z)$ and $\varphi_{h}(x)=d_{Ag}(x,
K_{h},V-h, \Real z)$, with $K_h=\left\{x\in (a,b),\; d(x,K)\leq h\right\}$, satisfy
$$
|\varphi(x)-\varphi_{h}(x)|\leq \kappa_{a,b,c}h
\quad{i.e.} \quad
\left(\frac{e^{\frac{\varphi}{h}}}{e^{\frac{\varphi_{h}}{h}}}\right)^{\pm
1}\leq e^{\kappa_{a,b,c}}\,,
$$
for some uniform constant $\kappa_{a,b,c}$\,. Hence the function
$\varphi$ can be replaced by $\varphi_{h}$ in the proof.\\
Lemma~\ref{le.agmonie} applied with $\alpha=a$,
$\beta=b$, $u_{1}=u_{2}=u$, $\varphi=\varphi_{h}$ and
$v=e^{\frac{\varphi_{h}}{h}}u$
 implies 
$$
\Real \int_{a}^{b}\bar v f \geq \int_{a}^{b}|hv'|^{2}+\int_{a}^{b}h|v|^{2}\,.
$$
Hence we get 
$$
\|v\|_{H^{1,h}}^{2}\leq \frac{1}{h}\left(\|f_1\|_{L^{2}}\|v\|_{L^{2}} + \|f_2\|_{\mathcal{M}_b}\|v\|_{L^{\infty}}\right)\,.
$$
The Gagliardo-Nirenberg estimate $\sup_{x\in[a,b]}|v(x)|\leq
C_{b-a}\|v'\|^{1/2}_{L^{2}((a,b))}\|v\|^{1/2}_{L^{2}((a,b))}$ implies:
$$
\|v\|_{H^{1,h}}^{2}\leq \frac{1}{h}\left(\|f_1\|_{L^{2}}\|v\|_{H^{1,h}} + \frac{C_{b-a}}{h^{\frac{1}{2}}}\|f_2\|_{\mathcal{M}_b}\|v\|_{H^{1,h}}\right)\,.
$$
This combined
with the equivalence of $\|v\|_{H^{1,h}}$ with
$\|he^{\frac{\varphi_{h}}{h}}u'\|_{L^{2}}+\|u\|_{L^{2}}$
leads finally to \eqref{eq.expdecHtil}.\\
\textbf{ii)} We follow the ideas of \cite{DiSj} which consists in
putting the possibly negative term of the energy estimate in the left
hand-side. Hence the equation  $(H_{D}^{h}-z)u=f$ is simply rewritten
$$
(\tilde{H}_{D}^{h}-z)u=f+W_1^{h}u+W_2^{h}u\,,
$$
and it suffices to estimate $\|W_2^{h}u\|_{\mathcal{M}_b}$\,. The Gagliardo-Nirenberg estimate gives
$$
\|W_2^{h}u\|_{\mathcal{M}_b}\leq C_{b-a}\|W_2^{h}\|_{\mathcal{M}_b}\|u'\|_{L^{2}}^{1/2}\|u\|_{L^{2}}^{1/2}\,.
$$
Applying Lemma~\ref{le.agmonie} with $\alpha=a$,
$\beta=b$, $u_{1}=u_{2}=u$, $\varphi=0$ leads to
$$
\|hu'\|_{L^{2}}^{2}+ c \|u\|_{L^{2}}^{2}\leq \left|\int_a^bf\overline{u}\right|+\left|\int_a^bW_1^h|u|^2\right|+\left|\int_a^bW_2^h|u|^2\right|\,.
$$
Apply a second time the Gagliardo-Nirenberg estimate
for 
$$
\|u\|_{H^{1,h}} \leq C_{a,b,c}\left(\|f\|_{L^2} + \|W_1^h\|_{L^{\infty}}\|u\|_{L^2}+\frac{1}{h}\|W_2^h\|_{\mathcal{M}_b}\|u\|_{L^2}\right)
$$
gives
$$
\|W_2^{h}u\|_{\mathcal{M}_b}\leq \frac{C_{a,b,c}'}{h^{\frac{1}{2}}}\|W_2^{h}\|_{\mathcal{M}_b}\left(\|f\|_{L^2} + \|W_1^h\|_{L^{\infty}}\|u\|_{L^2}+\frac{1}{h}\|W_2^h\|_{\mathcal{M}_b}\|u\|_{L^2}\right)\,.
$$
Combined with the results of \textbf{i)} applied 
with $f$ replaced by $f+W^{h}u$, this yields:
\begin{multline*}
h^{1/2}\sup_{x\in [a,b]}|e^{\frac{\varphi_h(x)}{h}}u(x)|+
\|he^{\frac{\varphi_h}{h}}u'\|_{L^{2}}+\|e^{\frac{\varphi_h}{h}}u\|_{L^{2}}
\leq \frac{C_{a,b,c}''}{h}(\|f\|_{L^2} 
+
\|W_1^h\|_{L^{\infty}}\|u\|_{L^2}
\\ \quad\quad 
+\frac{1}{h}\|W_2^{h}\|_{\mathcal{M}_b}\|f\|_{L^2}
+
\frac{1}{h}\|W_2^{h}\|_{\mathcal{M}_b}\|W_1^h\|_{L^{\infty}}\|u\|_{L^2}+\frac{1}{h^2}\|W_2^{h}\|_{\mathcal{M}_b}^2\|u\|_{L^2})\,,
\end{multline*}
where $\varphi_{h}(x)=d_{Ag}(x,K_h,V,\Real z)$ and $K_h=\left\{x\in (a,b),\; d(x,K'\cup U)<h\right\}$. With the assumptions on $W_1^h$ and $W_2^h$ and replacing $\varphi_{h}(x)$ by $\varphi(x)=d_{Ag}(x,K'\cup U,V,\Real z)$, we obtain \eqref{eq.expdecH}.
\qed

\subsection{Reduced boundary problem for generalized eigenfunctions}
\label{se.absbc}

We shall consider the boundary value problem
\begin{equation}
\label{eq.absbc}
\left\{
  \begin{array}[c]{l}
	(\tilde{P}^{h}-z)u=f\,,\\
      \left[ h\partial_{x}+i \zeta^{1/2} \right]u(a)=\ell_{a}\,,\\
      \left[ h\partial_{x}-i \zeta^{1/2} \right]u(b)=\ell_{b}\,,
  \end{array}
\right.
\end{equation}
with $\Imag z$ and $\Imag \zeta$ small enough w.r.t $h > 0$ and
specified later. Here $z^{1/2}$ denotes the complex square root with the determination $\arg z\in\left[-\frac{\pi}{2},\frac{3\pi}{2}\right)$.\\
The case $f\equiv 0$ occurs while studying 
the generalized eigenfunctions of $H_{V,\theta_{0}}(0)$ or their variation w.r.t $\theta_{0}$.
The case $\ell_{a}=\ell_{b}=0$ is concerned with the resolvent
estimates for  the non
self-adjoint Hamiltonians
\begin{align}
  \tilde{H}_{\zeta}^{h}:=\tilde{P^{h}}\,,
& ~ \text{with} ~ D(\tilde H_{\zeta}^{h})=\left\{u\in H^{2}((a,b))\,,\;
[h\partial_{x}+i \zeta^{1/2}]u(a)= 0\,, \right.\nonumber\\
&\hspace{7cm}
\left.[h\partial_{x}-i
\zeta^{1/2}]u(b)=0\right\}\,,
\label{eq.deftildHzet}
\\
  H_{\zeta}^{h}:=P^{h}\,,
& ~ \text{with} ~ D(H_{\zeta}^{h})=\left\{u\in H^{1}((a,b))\,,\;
P^{h}u\in L^{2}((a,b))\,,\;
\right.\nonumber\\
&\hspace{3cm}
\left.
[h\partial_{x}+i \zeta^{1/2}]u(a)= 0\,,\;
[h\partial_{x}-i
\zeta^{1/2}]u(b)=0\right\}\,.
\label{eq.defHzet}
\end{align}
\begin{lemma}
\label{le.estimbc}
Assume $V-\Real z\geq c$ with $\|V\|_{L^{\infty}}\leq
\frac{1}{c}$ and  $|\Imag \zeta^{1/2}|\leq \frac{h}{\kappa_{b-a}}$ for
$\kappa_{b-a}$ large enough according to $b-a$. Let $K$ be a
compact subset of $[a,b]$ and set  $\varphi=d_{Ag}(x,K, V, \Real
z)$. Then any solution $u\in L^2((a,b))$ to the boundary value problem
\eqref{eq.absbc} with $f\equiv 0$ satisfies:
$$
h^{1/2}\sup_{x\in [a,b]}|e^{\pm\frac{\varphi(x)}{h}}u(x)|+
\|he^{\pm\frac{\varphi}{h}}u'\|_{L^{2}}+\|e^{\pm\frac{\varphi}{h}}u\|_{L^{2}}
\leq 
\frac{C_{a,b,c}}{h^{1/2}}[|\ell_{a}|e^{\pm\frac{\varphi(a)}{h}}+|\ell_{b}|e^{\pm\frac{\varphi(b)}{h}}]\,.
$$
\end{lemma}
\noindent\textbf{Proof:} Again the function $\varphi$ is replaced by
$\varphi_{h}(x)=d_{Ag}(x,K,V-h, \Real z)$\,. Applying
Lemma~\ref{le.agmonie} with $\alpha=a$, $\beta=b$, $u_{1}=u_{2}=u$ and
$v=e^{\frac{\varphi_{h}}{h}}u$ implies
\begin{multline*}
0\geq\int_{a}^{b}|hv'|^{2}+\int_{a}^{b}h|v|^{2}+ h\Real(-i\zeta^{1/2})
\left[|v|^{2}(a)+|v|^{2}(b)\right]\\
+ h\Real\left[\overline{v(a)}(e^{\pm\frac{\varphi_{h}(a)}{h}}\ell_{a})
-
\overline{v(b)}(e^{\pm\frac{\varphi_{h}(b)}{h}}\ell_{b})
\right]\,.
\end{multline*}
With the Gagliardo-Nirenberg estimate, we get
$$
\|hv'\|^{2}_{L^{2}}+h\|v\|^{2}_{L^{2}}-2C_{b-a}^2|\Imag
\zeta^{1/2}|\|hv'\|_{L^{2}}\|v\|_{L^{2}}
\leq
C_{b-a}\|hv'\|_{L^{2}}^{1/2}\|v\|_{L^{2}}^{1/2} h^{1/2}\left[e^{\pm\frac{\varphi_{h}(a)}{h}}|\ell_{a}|
+
e^{\pm\frac{\varphi_{h}(b)}{h}}|\ell_{b}|
\right]\,,
$$
which implies
$$
\left(\|hv'\|^{2}_{L^{2}}+\|v\|^{2}_{L^{2}}\right)^{1/2}
\leq \frac{C_{a,b,c}}{h^{\frac{1}{2}}}
\left[e^{\pm\frac{\varphi_{h}(a)}{h}}|\ell_{a}|
+
e^{\pm\frac{\varphi_{h}(b)}{h}}|\ell_{b}|
\right]\,,
$$
provided $\kappa_{b-a}$ is large enough according to the Gagliardo-Nirenberg constant
$C_{b-a}$\,.
Rewriting the inequality with the uniform equivalence 
$\|he^{\pm\frac{\varphi_h}{h}}u'\|_{L^{2}}+\|e^{\pm\frac{\varphi_h}{h}}u\|_{L^{2}}$
with
$\|v\|_{H^{1,h}}=\left(\|hv'\|^{2}_{L^{2}}+\|v\|_{L^{2}}^{2}\right)^{1/2}$
yields the result.
\qed

The generalized eigenfunction $\tilde\psi_{-,\theta_{0}}^{h}(k,x)$, $k\in \R^{*}$,
of $H_{V,\theta_{0}}^{h}(0)$ is the
solution to 
\begin{align*}
&
(\tilde{P}^{h}-k^{2})\psi=0\quad\text{in}\;
\R\setminus\left\{a,b\right\}
\\
&
\psi(a^{+})=e^{-\frac{\theta_{0}}{2}}\psi(a^{-})\quad,\quad
\psi'(a^{+})=e^{-\frac{3\theta_{0}}{2}}\psi'(a^{-})\\
& \psi(b^{-})=e^{-\frac{\theta_{0}}{2}}\psi(b^{+})\quad,\quad
\psi'(b^{-})=e^{-\frac{3\theta_{0}}{2}}\psi'(b^{+})\\
& \psi\big|_{(-\infty,a)}=e^{i\frac{kx}{h}}+R(k)e^{-i\frac{kx}{h}}\quad,\quad
\psi\big|_{(b,+\infty)}=T(k)e^{i\frac{kx}{h}}\quad\text{for}\quad k>0\\
& \psi\big|_{(-\infty,a)}=T(k)e^{i\frac{kx}{h}}\quad,\quad
\psi\big|_{(b,+\infty)}=e^{i\frac{kx}{h}}+R(k)e^{-i\frac{kx}{h}}\quad\text{for}\quad k<0\,.
\end{align*}
This can be reformulated as the boundary value problem in $(a,b)$
\begin{equation}
\label{eq.absbcth}
\left\{
  \begin{array}[c]{lc|c}
    (\tilde{P}^{h}-k^{2})\psi=0\quad\text{in}~(a,b)\,,\\
& (k>0) & (k<0)
\\
   \left[ h\partial_{x}+i (k^{2})^{1/2}e^{-\theta_{0}} \right]\psi(a)=
&
2ik
   e^{-\frac{3\theta_{0}}{2}}e^{i\frac{ka}{h}}&0\,,\\
      \left[ h\partial_{x}-i (k^{2})^{1/2}e^{-\theta_{0}}
      \right]\psi(b)=&0   
&2ik
   e^{-\frac{3\theta_{0}}{2}}e^{i\frac{kb}{h}}\,,
  \end{array}
\right.
\end{equation}
where the choice of $z^{1/2}$ says $(k^{2})^{1/2}=|k|$ for $k\in
\R^{*}$\,.
A straightforward application of Lemma~\ref{le.estimbc} gives the next result.
\begin{proposition}
 \label{pr.geneig}
Assume $V-k^{2}\geq c$, $\|V\|_{L^{\infty}}\leq \frac{1}{c}$ and
$|\theta_{0}|\leq ch$ for $c$ small enough according to $a,b$.
The generalized eigenfunction
${\tilde\psi}_{-,\theta_{0}}^{h}(k,.)$ satisfies
\begin{eqnarray*}
  && h^{1/2}\sup_{x\in [a,b]}|e^{\frac{\varphi(x)}{h}}{\tilde\psi}^{h}_{-,\theta_{0}}(k,x)|+
\|he^{\frac{\varphi}{h}}{\tilde\psi}^{h}_{-,\theta_{0}}(k,.)'\|_{L^{2}}+\|e^{\frac{\varphi}{h}}{\tilde\psi}^{h}_{-,\theta_{0}}(k,.)\|_{L^{2}}
\leq 
\frac{C_{a,b,c}}{h^{1/2}}\,,\\
\text{with}&&
\varphi(x)=d_{Ag}(x,a,V,k^{2}) \quad\text{when}\quad k>0\,,\\
\text{and with}
&&
\varphi(x)=d_{Ag}(x,b,V,k^{2}) \quad\text{when}\quad k<0\,.
\end{eqnarray*}
\end{proposition}
With this first a priori estimate, the boundary value problem
\eqref{eq.absbcth} can be rewritten
\begin{equation*}
\left\{
  \begin{array}[c]{lc|c}
    (\tilde{P}^{h}-k^{2})\psi=0\quad\text{in}~(a,b)\,,  \\
& (k>0) & (k<0)
\\
   \left[ h\partial_{x}+i (k^{2})^{1/2} \right]\psi(a)=
&
2ik
   e^{i\frac{ka}{h}}+\mathcal{O}(\frac{|\theta_{0}|}{h})&
\mathcal{O}(\frac{e^{-\frac{d_{Ag}(a,b,V,k^{2})}{h}}|\theta_{0}|}{h})\,,\\
      \left[ h\partial_{x}-i (k^{2})^{1/2})
      \right]\psi(b)=& \mathcal{O}(\frac{e^{-\frac{d_{Ag}(a,b,V,k^{2})}{h}}|\theta_{0}|}{h})   
&2ik e^{i\frac{kb}{h}}+\mathcal{O}(\frac{|\theta_{0}|}{h})\,.
  \end{array}
\right.
\end{equation*}
Hence the difference $u={\tilde\psi}_{-,\theta_{0}}^{h}-{\tilde\psi}_{-,0}^{h}$ solves the
boundary value problem
\begin{equation*}
\left\{
  \begin{array}[c]{lc|c}
    (\tilde{P}^{h}-k^{2})u=0\quad\text{in}~(a,b)\,,\\
& (k>0) & (k<0)
\\
   \left[ h\partial_{x}+i (k^{2})^{1/2} \right]u(a)=
&
\mathcal{O}(\frac{|\theta_{0}|}{h})&\mathcal{O}(\frac{e^{-\frac{d_{Ag}(a,b,V,k^{2})}{h}}|\theta_{0}|}{h})\,,\\
      \left[ h\partial_{x}-i (k^{2})^{1/2})
      \right]u(b)=& \mathcal{O}(\frac{e^{-\frac{d_{Ag}(a,b,V,k^{2})}{h}}|\theta_{0}|}{h})   
&\mathcal{O}(\frac{|\theta_{0}|}{h})\,.
  \end{array}
\right.
\end{equation*}
Hence Lemma~\ref{le.estimbc} yields the next comparison result.
\begin{proposition}
 \label{pr.geneigdif}
Assume $V-k^{2}\geq c$, $\|V\|_{L^{\infty}}\leq \frac{1}{c}$ and
$|\theta_{0}|\leq ch$ for $c$ small enough according to $a,b$.
The difference of 
generalized eigenfunctions $u={\tilde\psi}_{-,\theta_{0}}^{h}(k,.)-{\tilde\psi}_{-,0}^{h}(k,.)$ 
satisfies
\begin{eqnarray*}
  && h^{1/2}\sup_{x\in [a,b]}|e^{\frac{\varphi(x)}{h}}u(x)|+
\|he^{\frac{\varphi}{h}}u'\|_{L^{2}}+\|e^{\frac{\varphi}{h}}u\|_{L^{2}}
\leq 
\frac{C_{a,b,c}|\theta_{0}|}{h^{3/2}}\,,\\
\text{with}&&
\varphi(x)=d_{Ag}(x,a,V,k^{2}) \quad\text{when}\quad k>0\,,\\
\text{and with}
&&
\varphi(x)=d_{Ag}(x,b,V,k^{2}) \quad\text{when}\quad k<0\,.
\end{eqnarray*}
\end{proposition}
\begin{remark}
  With an additional regularity assumption \cite{Ni2} proves in the
  case $\theta_{0}=0$ that the upper bound of $\sup_{x\in
    [a,b]}|e^{\frac{\varphi(x)}{h}}{\tilde\psi}_{-,0}^{h}(k,x)|$
is actually $\mathcal{O}(1)$ with a first order WKB
approximation. By using this result and the comparison result of
Proposition~\ref{pr.geneigdif} with a bootstrap argument or
reconsidering the complete proof of \cite{Ni2}, the estimate
$\sup_{x\in
    [a,b]}|e^{\frac{\varphi(x)}{h}}{\tilde\psi}^{h}_{-,0}(k,x)|=\mathcal{O}(1)$
  and
$\sup_{x\in
    [a,b]}|e^{\frac{\varphi(x)}{h}}({\tilde\psi}_{-,\theta_{0}}^{h}-{\tilde\psi}_{-,0}^{h})(k,x)|=\mathcal{O}(\frac{|\theta_{0}|}{h^{1/2}})$ 
could be obtained.  Here only $V\in L^{\infty}$ is assumed with a
possible loss in the $h$-exponent.
\end{remark}

\subsection{Weighted resolvent estimates}
\label{se.wresest}
We complete  the analysis of the previous subsection with results
concerned with the resolvent
$(\tilde{H}_{\zeta}^{h}-z)^{-1}$ corresponding to the boundary value
problem
\eqref{eq.absbc} with $\ell_{a}=\ell_{b}=0$\,.
\begin{proposition}
\label{pr.resestabsbc}
  Assume $V-\Real z\geq c$, $\|V\|_{L^{\infty}}\leq \frac{1}{c}$ and
$|\Imag \zeta^{1/2}|\leq \frac{h}{\kappa_{b-a}}$ with $\kappa_{b-a}$
large enough according to $b-a$. Let $K$ be a
compact subset of $[a,b]$ and set  $\varphi=d_{Ag}(x,K, V, \Real
z)$\,.
Then for $f\in L^{2}((a,b))$, the
function $u=(\tilde{H}_{\zeta}^{h}-z)^{-1}f$ satisfies
$$
h^{1/2}\sup_{x\in [a,b]}|e^{\pm\frac{\varphi(x)}{h}}u(x)|+
\|he^{\pm\frac{\varphi}{h}}u'\|_{L^{2}}+\|e^{\pm\frac{\varphi}{h}}u\|_{L^{2}}
\leq 
\frac{C_{a,b,c}}{h}\|e^{\pm\frac{\varphi}{h}}f\|_{L^{2}}\,,
$$
where $e^{\pm\frac{\varphi}{h}}f=f$ when $\supp f\subset K$\,.\\
In particular this yields
$$
\left\|e^{\pm\frac{\varphi}{h}}(\tilde{H}_{\zeta}^{h}-z)^{-1}e^{\mp\frac{\varphi}{h}}\right\|_{\mathcal{L}(L^{2})}
\leq \frac{C_{a,b,c}}{h}\,,
$$
and $z\in\rho\left(\tilde{H}_{\zeta}^{h}\right)$.
\end{proposition}
\noindent\textbf{Proof:} Again we can replace $\varphi$ by
$\varphi_{h}=d_{Ag}(x, K, V-h,\Real z)$.
Lemma~\ref{le.agmonie} with $v=e^{\pm\frac{\varphi_{h}}{h}}u$ implies
$$
\int_{a}^{b}|hv'|^{2}+ h\int_{a}^{b}|v|^{2}
+h\Real(-i\zeta^{1/2})\left[|v(a)^{2}|+|v(b)|^{2}\right]
\leq \|e^{\pm\frac{\varphi_h}{h}}f\|_{L^{2}}\|v\|_{L^{2}}\,.
$$
Absorbing the boundary term with the help of the Gagliardo-Nirenberg
inequality like in the proof of Lemma~\ref{le.estimbc} 
and taking $\kappa_{b-a}$ large enough yields the result.
\qed
\begin{proposition}
\label{pr.diffestabsbc}
  Assume $V-\Real z\geq c$, $\|V\|_{L^{\infty}}\leq \frac{1}{c}$, $|\Imag z^{1/2}|\leq \frac{h}{\kappa_{b-a}}$ and $|\Imag \zeta^{1/2}|~\leq~\frac{h}{\kappa_{b-a}}$ with $\kappa_{b-a}$
large enough according to $b-a$. Let $K$ be a
compact subset of $[a,b]$ and set  $\varphi=d_{Ag}(x,K, V, \Real
z)$\,.
Then for $f\in L^{2}((a,b))$, the
difference $w=(\tilde{H}_z^{h}-z)^{-1}f-(\tilde{H}_{\zeta}^{h}-z)^{-1}f$ satisfies
$$
h^{1/2}\sup_{x\in [a,b]}|e^{\pm\frac{\varphi(x)}{h}}w(x)|+
\|he^{\pm\frac{\varphi}{h}}w'\|_{L^{2}}+\|e^{\pm\frac{\varphi}{h}}w\|_{L^{2}}
\leq 
\frac{C_{a,b,c}|z^{1/2}-\zeta^{1/2}|}{h^{2}}\|e^{\pm\frac{\varphi}{h}}f\|_{L^{2}}\,,
$$
where $e^{\pm\frac{\varphi}{h}}f=f$ when $\supp f\subset K$\,.\\
In particular this yields
$$
\left\|e^{\pm\frac{\varphi}{h}}\left[
(\tilde{H}_z^{h}-z)^{-1}
-(\tilde{H}_{\zeta}^{h}-z)^{-1}
\right]
e^{\mp\frac{\varphi}{h}}\right\|_{\mathcal{L}(L^{2})}
\leq \frac{C_{a,b,c}|z^{1/2}-\zeta^{1/2}|}{h^{2}}\,.
$$
\end{proposition}
\noindent\textbf{Proof:} The function
$(\tilde{H}_{\zeta}^{h}-z)^{-1}f$ solves \eqref{eq.absbc} with $\ell_{a}=\ell_{b}=0$. Therefore, if we set $u=(\tilde{H}_z^{h}-z)^{-1}f$ and $v=(\tilde{H}_{\zeta}^{h}-z)^{-1}f$, the function $w=u-v$ verifies:
\[
\left\{
  \begin{array}[c]{l}
	(\tilde{P}^{h}-z)w=0\,,\\
      \left[ h\partial_{x}+i z^{1/2} \right]w(a)= -i(z^{1/2}-\zeta^{1/2})v(a)\,,\\
      \left[ h\partial_{x}-i z^{1/2} \right]w(b)= i(z^{1/2}-\zeta^{1/2})v(b)\,.
  \end{array}
\right.
\]
Then it follows from Lemma \ref{le.estimbc} that:
$$
h^{1/2}\sup_{x\in [a,b]}|e^{\pm\frac{\varphi(x)}{h}}w(x)|+
\|he^{\pm\frac{\varphi}{h}}w'\|_{L^{2}}+\|e^{\pm\frac{\varphi}{h}}w\|_{L^{2}}
\leq 
\frac{C_{a,b,c}}{h}|z^{1/2}-\zeta^{1/2}|2h^{1/2}\sup_{x\in [a,b]}|e^{\pm\frac{\varphi(x)}{h}}v(x)|\,,
$$
and we can apply Proposition \ref{pr.resestabsbc} to the function $v$ to get the result.\qed
 \section{Accurate analysis of resonances}\label{se.grush}
In this section, we use the approach of Helffer-Sj{\"o}strand relying on the introduction of a Grushin problem (see \cite{HeSj}\cite{SjZw}). This section ends with a rewriting of the Fermi Golden rule \eqref{eq.FGint} for the modified Hamitonian $H_{\theta_{0},V-W^{h}}^{h}$. 
\subsection{Resonances}\label{sec.res}
Resonances for $H_{\theta_{0},V-W^{h}}^{h}=H_{\theta_{0},V-W^{h}}^{h}(0)$ are eigenvalues of $H_{\theta_{0},V-W^{h}}^{h}(i\tau)$ for a suitable choice of $\tau$ according to the resonances to be revealed. Associated eigenfunctions are the $g_r\in L^2(\R)$ functions satisfying
\begin{equation}
\label{eq.pb_res_R}
H_{\theta_{0},V-W^{h}}^{h}(i\tau)g_r=z_rg_r\,,
\end{equation}
with $\arg(z_r)\in (-2\tau,0)$. Alternatively, $f_r=U_{-i\tau}g_r$ satisfies
\[
H_{\theta_{0},V-W^{h}}^{h}f_r=z_rf_r\,,
\]
with $f_r\in
L^2\left((a+e^{i\tau}\R_-)\cup(a,b)\cup(b+e^{i\tau}\R_+)\right)$. We
refer to \cite{HeMa} for a general comparison of the two
approaches. Accordingly, we recover the definition of Gamow resonant
functions with no incoming data and slowly exponentially increasing
outgoing waves.\\
Equivalently working with $g_{r}$, the condition $g_{r}\in L^{2}(\R)$
imposes the exponential modes in the exterior domain:
\begin{equation}
\label{eq.mod_res}
g_r(x)=\left\{\begin{array}[c]{ll}g_+e^{i\frac{ z_r^{1/2} e^{i\tau} }{h}(x-b)},& x>b\\[1.65mm]
g_{r,int}(x),& x\in(a,b)\\[1.65mm]
g_-e^{-i\frac{z_r^{1/2}e^{i\tau}}{h}(x-a)},& x<a\,,
\end{array}
\right.
\end{equation}
where we recall that $z^{1/2}$ denotes the complex square root with the determination $\arg z\in\left[-\frac{\pi}{2},\frac{3\pi}{2}\right)$. According to the definition of $D(H_{\theta_{0},V-W^{h}}^{h}(\theta))$, this function verifies the following boundary conditions
\[
\left(h\partial_x-iz_r^{1/2} e^{i\tau}\right)g_r(b^+)=0 \, \Rightarrow \, \left(h\partial_x-iz_r^{1/2} e^{-\theta_0}\right)g_r(b^-)=0\,,
\]
and
\[
\left(h\partial_x+iz_r^{1/2} e^{i\tau}\right)g_r(a^-)=0 \, \Rightarrow \, \left(h\partial_x+iz_r^{1/2} e^{-\theta_0}\right)g_r(a^+)=0
\]
(with $\left(z_re^{-2\theta_0}\right)^{1/2}=z_r^{1/2}e^{-\theta_0}$). It follows that the interior part of the solution satisfies the non-linear eigenvalue problem 
\begin{equation}
\label{eq.pb_res_int}
H_{z_re^{-2\theta_0}}^h\,g_{r,int}~=~z_rg_{r,int}
\end{equation}
(see definition \eqref{eq.defHzet}). Conversely, given $z_r$ in the sector: $\arg(z_r)\in (-2\tau,0)$ for which $g_{r,int}$ fulfilling \eqref{eq.pb_res_int}, it is possible to define suitable coefficients $g_+$ and $g_-$ such that the function $g_r$ given by \eqref{eq.mod_res} is in $D(H_{\theta_{0},V-W^{h}}^{h}(\theta))$ and solves the equation \eqref{eq.pb_res_R}. This allows to identify resonances with the poles of $\left(H_{ze^{-2\theta_0}}^h-z\right)^{-1}$.\\
It is worthwhile to notice that this technique extends to the first Riemann sheet: in this case the poles of $\left(H_{ze^{-2\theta_0}}^h-z\right)^{-1}$ correspond to proper eigenvalues of $H_{\theta_{0},V-W^{h}}^{h}$ provided that $\arg(z)~\leq~\frac{3\pi}{2}~-~2\tau$. To this concern, the following spectral characterization holds.
\begin{lemma}\label{le.nonsp}
Let $H_{i\tau,V-W^{h}}^{h}(0)$ be defined as in \eqref{H_(0)+V} with $\tau\in(0,\frac{\pi}{4})$, then
\[
\sigma_p\left(H_{i\tau,V-W^{h}}^{h}(0)\right)\cap\{\Imag z > 0\}=\emptyset\,.
\]
\end{lemma}
\noindent\textbf{Proof:} According to the Proposition~\ref{prop.agco}, the points in $\sigma_p\left(H_{\theta_{0},V-W^{h}}^{h}(0)\right)\cap\{\Imag z > 0\}$ coincides with the eigenvalues of $H_{\theta_{0},V-W^{h}}^{h}(\theta)$, in $\Imag z > 0$, for any choice of $\theta$ with: $\Imag \theta\in(0,\frac{\pi}{4})$ . In particular, for $\theta=\theta_0=i\tau$, it follows from (\ref{accretivity1}) that the operator $H_{i\tau,V-W^{h}}^{h}(i\tau)$ is accretive. In this case, $\sigma_p\left(H_{i\tau,V-W^{h}}^{h}(i\tau)\right)\cap\{\Imag z > 0\}=\emptyset$.
\qed

\subsection{The Grushin problem for resonances}
In the previous section we got some accurate estimates for the
variation w.r.t $\theta_{0}$ of the generalized eigenfunctions of the filled well Hamiltonian $H^{h}_{\theta_{0},V}$. 
Here the resonances for the full Hamiltonians
$H_{\theta_{0},V-W^{h}}^{h}$ and $H_{0,V-W^{h}}^{h}$ are considered. 
After reducing the problem to the interval $[a,b]$, we introduce like in \cite{HeSj}\cite{BNP1}\cite{BNP2}  the Grushin problem modelled from the Dirichlet
operator with the potential $V-W^{h}$ for the boundary value operator
$H_{\zeta}^{h}-z$ with $\zeta=z$ or $\zeta=ze^{-2\theta_{0}}$
according to \eqref{eq.defHzet}.

We assume that a cluster of eigenvalues $\lambda_{1}^{h},\ldots,\lambda_{\ell}^{h}$ of the Dirichlet operator 
\[
H_{D}^{h}=-h^{2}\Delta+V-W^{h}
\]
exists such that
\begin{eqnarray}
\label{eq.hyp1}
&&
d(\lambda^0,\sigma(H_{D}^{h})\setminus\left\{\lambda_{1}^{h},\ldots,\lambda_{\ell}^{h}\right\})\geq
c\,,
\\
\label{eq.hyp2}
&&
c\leq \lambda^0 \leq \inf_{x\in(a,b)} V(x)-c \leq \|V\|_{L^{\infty}}\leq \frac{1}{c}\,,\\
\label{eq.hyp3}
&& \max_{1\leq j\leq\ell} |\lambda_{j}^{h}-\lambda^0|\leq \frac{1}{c}h.
\end{eqnarray}
The domain $\omega_{ch}$ will be a neighborhood of
$\left\{\lambda_{1}^{h},\ldots, \lambda_{\ell}^{h}\right\}$ such that 
\begin{equation}
  \label{eq.defomh}
\omega_{ch}\subset \left\{z\in\C,\;
  d(z,\left\{\lambda_{1}^{h},\ldots,\lambda_{\ell}^{h}\right\})\leq ch\right\}\,.
\end{equation}
\begin{remark}
  Notice that these assumptions do not forbid $h$-dependent
  $\lambda^{0}$ with $|\lambda^{0}(h)-\lambda^{0}|\leq \frac{1}{c}h$ since, in this case, it suffices to replace $V$ by
  $V-\lambda^{0}(h)+\lambda^{0}$\,.
\end{remark}
Normalized eigenvectors associated with the  $\lambda_{j}^{h}$ are
denoted by $\Phi_{j}^{h}$ and the total spectral projector 
is
$$
\Pi^{h}=\sum_{j=1}^{\ell}|\Phi_{j}^{h}\rangle\langle \Phi_{j}^{h}|\,.
$$
We also introduce the bounded operators
\begin{eqnarray*}
  R_{0}^{-}&:& \C^{\ell}\to L^{2}((a,b))\\
   &&
u^{-} = \begin{pmatrix}
     u_{1}\\\vdots\\ u_{\ell}
   \end{pmatrix}
\mapsto 
R_{0}^{-}u^{-}=\sum_{j=1}^{\ell}u_{j}\Phi_{j}^{h}\,,\\
\\
\text{and}\quad R_{0}^{+}
&:& L^{2}((a,b))\to \C^{\ell}\\
&&u\mapsto R_{0}^{+}u=
\begin{pmatrix}
  \langle \Phi_{1}^{h}\,,\, u\rangle\\
\vdots\\
  \langle \Phi_{\ell}^{h}\,,\, u\rangle
\end{pmatrix}
\,.
\end{eqnarray*}
For $z\in\omega_{ch}$, the matricial operator $
\begin{pmatrix}
  H_{D}^{h}-z & R_{0}^{-}\\
  R_{0}^{+}& 0
\end{pmatrix}:D(H_{D}^{h})\times \C^{\ell}\to L^{2}((a,b))\times \C^{\ell}
$
is invertible with the inverse 
\begin{eqnarray*}
\begin{pmatrix}
  H_{D}^{h}-z & R_{0}^{-}\\
  R_{0}^{+}& 0
\end{pmatrix}^{-1}
=
\begin{pmatrix}
  E_{0}(z)& E_{0}^{+}\\
E_{0}^{-}& E_{0}^{-+}(z)
\end{pmatrix}\,,\\
E_{0}(z)=(H_{D}^{h}-z)^{-1}(1-\Pi^{h})\,,\quad\quad
E_{0}^{+}v^{+}=\sum_{j=1}^{\ell}v_{j}\Phi_{j}^{h}\,,\\
E_{0}^{-}v=
\begin{pmatrix}
  \langle \Phi_{1}^{h}\,,\,v\rangle\\
\vdots\\
\langle \Phi_{\ell}^{h}\,,\,v\rangle
\end{pmatrix}\,,
\quad\quad
E_{0}^{-+}(z)v^{+}=\diag(z-\lambda_{j}^{h})v^{+}\,.
\end{eqnarray*}
\textbf{Notations:}
      We set 
$$
\mathcal{H}_{D}(z)=
\begin{pmatrix}
  H_{D}^{h}-z & R_{0}^{-}\\
  R_{0}^{+}& 0
\end{pmatrix}
\quad\text{and}\quad
\mathcal{E}_{D}(z)=
\begin{pmatrix}
  E_{0}(z)& E_{0}^{+}\\
E_{0}^{-}& E_{0}^{-+}(z)
\end{pmatrix}\,.
$$
The problem $(H_{\zeta}^{h}-z)u=f$ is studied after introducing the
matricial operator
\begin{equation}\label{eq.matop}
\mathcal{H}_{\zeta}(z):=
\begin{pmatrix}
  H^{h}_{\zeta}-z & \chi_{h} R_{0}^{-}\\
R_{0}^{+}& 0
\end{pmatrix}
\,,
\end{equation}
where the function $\chi_{h}\in \mathcal{C}^{\infty}_{0}((a,b))$
satisfies
$$
\|(h\partial_{x})^{\alpha}\chi_{h}\|_{L^{\infty}((a,b))}\leq C_{\alpha}\,,\quad \alpha \in \N
\quad\text{and}\quad
\chi_{h}(x)\equiv 1 ~\text{if}~d(x,\left\{a,b\right\})\geq h\,.
$$
Another cut-off function $\psi\in \mathcal{C}^{\infty}_{0}((a,b))$ will
be used with a smaller support.
By introducing the positive quantity
$$
S_{0}:=d_{Ag}(\left\{a,b\right\},U,V,\lambda^{0}),
$$
the cut-off $\psi$ is chosen such that for some $\eta>0$ independent of
$h>0$ but to be specified later
$$
\psi(x)=\left\{
  \begin{array}[c]{ll}
    0 &\text{if}\; d_{Ag}(x,U,V,\lambda^{0})>
    \frac{S_{0}+\eta}{2}\\
1 & \text{if}\; d_{Ag}(x,U,V,\lambda^{0})< \frac{S_{0}-\eta}{2}\,.
  \end{array}
\right.
$$
When $\eta>0$ and $h>0$ are small enough 
$$U\subset\subset
\left\{\psi\equiv 1\right\}\subset \supp \psi\subset\subset
\left\{\chi_{h}\equiv 1\right\}\,.
$$
For $z,\zeta\in \omega_{ch}$, consider the approximate inverse
$$
\mathcal{F}_{\zeta}(z)=
\begin{pmatrix}
  \chi_{h}E_{0}\psi + (1-\rho_h)(\tilde{H}_{\zeta}^{h}-z)^{-1}(1-\psi) &
  \chi_{h}E_{0}^{+}\\
 E_{0}^{-}\psi & E_{0}^{-+}
\end{pmatrix}\,,
$$
where the function $\rho_{h}\in \mathcal{C}^{\infty}_{0}(U_{2h})$
satisfies
$$
\|(h\partial_{x})^{\alpha}\rho_{h}\|_{L^{\infty}((a,b))}\leq C_{\alpha}\,,\quad \alpha \in \N
\quad\text{and}\quad
\rho_{h}\equiv 1 ~\text{on}~U_{\frac{5h}{4}}\,,
$$
after recalling  $U_{t}=\left\{x\in (a,b),\;d(x,U)\leq t\right\}$\,. In particular this implies $W^h(1-\rho_h)=0$.\\
A direct calculation gives
\begin{eqnarray*}
&&\mathcal{H}_{\zeta}(z)\mathcal{F}_{\zeta}(z)=
\begin{pmatrix}
  A & B\\
C& D
\end{pmatrix}
\,,\\
\text{with}&&
A=
1+\left[(h\partial_x)^2,\rho_{h}\right](\tilde{H}_{\zeta}^{h}-z)^{-1}(1-\psi)-\left[(h\partial_x)^2,\chi_{h}\right]E_{0}\psi\,,\\
&&
B= (H_{\zeta}^{h}-z)\chi_{h} E_{0}^{+}+ \chi_{h} R_{0}^{-}E_{0}^{-+}\,,\\
&&
C= R_{0}^{+}(\chi_{h} E_{0}\psi +
(1-\rho_h)(\tilde{H}_{\zeta}^{h}-z)^{-1}(1-\psi))\,,\\
&&
D= R_{0}^{+}\chi_{h} E_{0}^{+}\,,
\end{eqnarray*}
where we have used $H_{\zeta}^{h}\chi_{h}=H_{D}^{h}\chi_{h}=P^{h}\chi_{h}$\,.
\begin{proposition}
  Assume the conditions \eqref{eq.hyp1}\eqref{eq.hyp2}\eqref{eq.hyp3}
  and suppose $z,\zeta\in \omega_{ch}$\,. The matricial operator
  $\mathcal{F}_{\zeta}(z)$ is an approximate inverse of
  $\mathcal{H}_{\zeta}(z)$:
  \begin{equation}
    \label{eq.defK}   \mathcal{H}_{\zeta}(z)\mathcal{F}_{\zeta}(z)= 1+
    \mathcal{K}_{\zeta}(z)\quad\text{and}\quad
 \mathcal{F}_{\zeta}(z)\mathcal{H}_{\zeta}(z)=1+ \mathcal{K}_{\zeta}'(z)\,,
\end{equation}
for $h>0$ small enough and after adjusting
  the parameter $\eta>0$ so that
  $\|\mathcal{K}_{\zeta}(z)\|+\|\mathcal{K}_{\zeta}'(z)\|<1$ according
  to: 
\begin{equation}
  \label{eq.roughS02}
\|\mathcal{K}_{\zeta}(z)\|+\| \mathcal{K}_{\zeta}'(z)\|\leq C_{a,b,c}e^{-\frac{S_{0}-C_{a,b,c}\eta}{2h}}\,.
  \end{equation}
More precisely the remainder term equals
$$
\mathcal{K}_{\zeta}(z)=
\begin{pmatrix}
\left[(h\partial_x)^2,\rho_{h}\right](\tilde{H}_{\zeta}^{h}-z)^{-1}(1-\psi)-\left[(h\partial_x)^2,\chi_{h}\right]E_{0}\psi
&
-[(h\partial_x)^2, \chi_{h}]E_{0}^{+}
\\
R_{0}^{+}(\chi_{h}-1)E_{0}\psi
+R_{0}^{+}(1-\rho_h)(\tilde{H}_{\zeta}^{h}-z)^{-1}(1-\psi)
&
R_{0}^{+}(\chi_{h}-1)E_{0}^{+}
\end{pmatrix}
$$
and is estimated by
$$
K_{\zeta}(z)=
\begin{pmatrix}
  \mathcal{O}(e^{-\frac{S_{0}-C_{a,b,c}\eta}{2h}})
&
 \mathcal{O}(\frac{e^{-\frac{S_{0}}{h}}}{h})\\
 \mathcal{O}(e^{-\frac{S_{0}-C_{a,b,c}\eta}{2h}})&\mathcal{O}(\frac{e^{-\frac{2S_{0}}{h}}}{h^2})
\end{pmatrix}\,.
$$
\end{proposition}
\noindent\textbf{Proof:} Set
$
\mathcal{K}_{\zeta}(z)=
\begin{pmatrix}
  K_{11}& K_{12}\\
K_{21}& K_{22}
\end{pmatrix}
$ and remember the expressions of 
$A,B,C,D$ in 
$\mathcal{H}_{\zeta}(z)\mathcal{F}_{\zeta}(z)=
\begin{pmatrix}
  A&B\\C&D
\end{pmatrix}
$\,.\\
The first coefficient $K_{11}$ is simply $A-1$ according to the above
definition.\\
The coefficient $K_{12}=B$ is computed  by making use of
$H_{\zeta}^{h}\chi_{h}=H_{D}^{h}\chi_{h}$
 and of the relation $(H_{D}^h-z)E_{0}^++R_{0}^{-}E_{0}^{-+}=0$ coming from 
$\mathcal{H}_{D}(z)\mathcal{E}_{D}(z)=1$\,.\\
The coefficient $K_{2,1}=C$ is computed after using the relation
$R_{0}^{+}E_{0}=0$ coming from
$\mathcal{H}_{D}(z)\mathcal{E}_{D}(z)=1$\,.\\
The coefficient $K_{22}=D-1$ is computed after using
$R_{0}^{+}E_{0}^{+}=1$\,.\\
\textbf{Estimate of $K_{11}$:} For the first term, remark the identity:
\begin{equation}
  \label{eq.comrhoh}
[(h\partial_x)^2, \rho_{h}]=2(h\rho_{h}')(h\partial_{x})+(h^{2}\rho_{h}'')\,,
\end{equation}
where the coefficients $h\rho_{h}'$ and $h^{2}\rho_{h}''$ are uniformly
bounded and supported in $U_{2h}$. Then, owing to
 Proposition~\ref{pr.resestabsbc}, it is estimated with:
\begin{equation}\label{prem_K11}
\|\left[(h\partial_x)^2,\rho_{h}\right](\tilde{H}_{\zeta}^{h}-z)^{-1}(1-\psi)\|\leq C_{a,b,c}e^{-\frac{S_{0}-C_{a,b,c}\eta}{2h}}\,.
\end{equation}
For the second term, we have the identity \eqref{eq.comrhoh}, where $\rho_h$ is replaced by $\chi_h$ and the coefficients $h\chi_{h}'$ and $h^{2}\chi_{h}''$ are uniformly
bounded and supported in $\left\{x\in (a,b),\;
  d(x,\left\{a,b\right\})< h\right\}$\,.\\
By introducing a circle $\gamma_{0}=\left\{z'\in\C\,,\;
  |z'-\lambda^{0}|= \frac{2}{c}h\right\}$\,, the formula
$$
E_{0}(z)=(H_{D}^{h}-z)^{-1}(1-\Pi^{h})=-\frac{1}{2i\pi}\int_{\gamma_{0}}\frac{1}{z'-z}\frac{1}{(z'-H_{D}^{h})}~dz'\,,
$$
and Proposition~\ref{pr.Direxp}-ii) imply
\begin{equation}
  \label{eq.estimE0}
 \|[(h\partial_x)^2, \chi_{h}]E_{0}\psi\|\leq \frac{C_{a,b,c}e^{-\frac{S_{0}-C_{a,b,c}\eta}{2h}}}{h^3}\,. 
\end{equation}
\noindent\textbf{Estimate of $K_{21}$:} The cut-off $(\chi_{h}-1)$
is supported in $\left\{x\in(a,b),\; d(x,\left\{a,b\right\}) < h\right\}$\,. Meanwhile we verify with the same argument as for
\eqref{eq.estimE0} that the function $u=E_{0}\psi f$ satisfies
$$
\|e^{\frac{\varphi}{h}}u\|\leq \frac{C_{a,b,c}\|f\|}{h^3}\,,
$$
for $\varphi=d_{Ag}(x,\supp \psi, V,\lambda^{0})$\,. The operator
$R_{0}^{+}$ is the finite rank operator defined by taking the scalar
product
with $\Phi_{j}^{h}$, $j=1,\ldots,
\ell$\,. With the
exponential decay of the eigenfunctions $\Phi_{j}^{h}$, $j=1,\ldots,
\ell$, stated in Proposition~\ref{pr.Direxp}, we get:
$$
\|R_{0}^{+}(1-\chi_{h})E_{0}\psi\|\leq C_{a,b,c}e^{-\frac{3S_{0}-C_{a,b,c}\eta}{2h}}\,.
$$
The second term is estimated like the first one of $K_{11}$ while
replacing $\left[(h\partial_x)^2,\rho_{h}\right]$ with $R_{0}^{+}(1-\rho_h)$:
\[
\|R_{0}^{+}(1-\rho_h)(\tilde{H}_{\zeta}^{h}-z)^{-1}(1-\psi)\|\leq C_{a,b,c}e^{-\frac{S_{0}-C_{a,b,c}\eta}{2h}}\,.
\]
\noindent\textbf{Estimate of $K_{12}$ and $K_{22}$:} The operator
$E_{0}^{+}$ is defined by
$E_{0}^{+}v^{+}=\sum_{j=1}^{\ell}v_{j}\Phi_{j}^{h}$ and the
exponential decay of the eigenfunctions $\Phi_{j}^{h}$, 
$j=1,\ldots,\ell$, stated in Proposition~\ref{pr.Direxp} 
with the relation \eqref{eq.comrhoh}, where
$\rho_h$ is replaced by $\chi_h$, yields 
\begin{equation}\label{eq_K12}
\|K_{12}\|=\|[(h\partial_x)^2, \chi_{h}]E_{0}^{+}\|\leq
C_{a,b,c}\frac{e^{-\frac{S_{0}}{h}}}{h}\,.
\end{equation}
For $K_{22}$ we use additionally the exponential decay of the
$\Phi_{j}^{h}$, $j=1,\ldots,\ell$ contained in $R_{0}^{+}$ and we get
\begin{equation}\label{eq_K22}
\|K_{22}\|=\|R_{0}^{+}(1-\chi_{h})E_{0}^{+}\|\leq
C_{a,b,c}\frac{e^{-\frac{2S_{0}}{h}}}{h^2}\,.
\end{equation}
\textbf{Left and Right inverse:} When $h$ and $\eta$ are small enough the
previous analysis says that
$\mathcal{F}_{\zeta}(z)(1+\mathcal{K}_{\zeta}(z))^{-1}$ is a
right-inverse of $\mathcal{H}_{\zeta}(z)$ and $\mathcal{H}_{\zeta}(z)$ is surjective.\\
From the definitions \eqref{eq.deftildHzet} and \eqref{eq.defHzet}, we have:
\begin{equation}\label{Hzetdemi}
\tilde H_{\zeta}^{h}=\tilde H^{h}\left(\zeta^{\frac{1}{2}}\right) \quad \textrm{ and } \quad H_{\zeta}^{h}= H^{h}\left(\zeta^{\frac{1}{2}}\right)\,.
\end{equation}
After two integrations by part, we get $\left(H^{h}\left(\zeta^{\frac{1}{2}}\right)\right)^{*}=H^{h}\left(-(\bar\zeta)^{\frac{1}{2}}\right)$. With
$(R_{0}^{+})^{*}=R_{0}^{-}$ and with the notations induced by \eqref{eq.matop} and \eqref{Hzetdemi}, we obtain
$\left[\mathcal{H}(\zeta^{\frac{1}{2}},z)\right]^{*}=\mathcal{H}(-(\bar\zeta)^{\frac{1}{2}},\bar z)$.
The analysis performed to obtain \eqref{eq.roughS02} for $\mathcal{K}(\zeta^{\frac{1}{2}},z)$ can be adapted in the case of $\mathcal{K}(-(\bar\zeta)^{\frac{1}{2}},\bar z)$: this yields the surjectivity of $\mathcal{H}(-(\bar\zeta)^{\frac{1}{2}},\bar z)$. Since
\[
\Ker\left(\mathcal{H}(\zeta^{\frac{1}{2}},z)\right)=\left[\textrm{Ran}\,\left(\mathcal{H}(\zeta^{\frac{1}{2}},z)\right)^{*}\right]^{\perp}=\{0\}\,,
\]
the injectivity of $\mathcal{H}(\zeta^{\frac{1}{2}},z)$ follows.
\qed

\noindent\textbf{Notation:} When $h>0$ is small enough, we set
\begin{equation}\label{eq.inv_grush}
\mathcal{E}_{\zeta}(z)=
\begin{pmatrix}
  E& E^{+}\\
E^{-}& E^{-+}
\end{pmatrix}
=\mathcal{H}_{\zeta}(z)^{-1}\,.
\end{equation}
The Schur complement formula
\begin{equation}\label{eq.schur_comp}
(H_{\zeta}^h-z)^{-1}=E-E^{+}(E^{-+})^{-1}E^{-}
\end{equation}
recalls that $(H_{\zeta}^h-z)$ is invertible if and only if the
$\ell\times\ell$ square matrix
$E^{-+}$ is invertible. An accurate calculation of this matrix allows
to identify the poles of $(H_{\zeta}^{h}-z)^{-1}$\,.\\
The final result comes from a higher order estimate after taking the
Neumann series 
$$
(1+\mathcal{K}_{\zeta}(z))^{-1}
=1-\mathcal{K}_{\zeta}(z)+\mathcal{K}_{\zeta}(z)^{2}-\mathcal{K}_{\zeta}(z)^{3}+\mathcal{K}_{\zeta}(z)^{4}+
\mathcal{O}(e^{-\frac{5S_{0}-C_{a,b,c}\eta}{2h}})\,.
$$
\begin{proposition}
  \label{pr.aaa}
 Assume the conditions \eqref{eq.hyp1}\eqref{eq.hyp2}\eqref{eq.hyp3}
  and suppose $z,\zeta\in \omega_{ch}$\,.
Then 
$$
E^{-+}=E^{-+}_{0}+\mathcal{O}(e^{-\frac{2S_{0}}{h}}h^{-3})\,,
$$
and
\begin{multline*}
E^{-+}=E^{-+}_{0}-E_{0}^{-}[(h\partial_x)^2, \rho_{h}](\tilde{H}_{\zeta}^{h}-z)^{-1}
[(h\partial_x)^2, \chi_{h}]E_{0}^{+}
-E_{0}^{-+}R_{0}^{+}(\chi_{h}-1)E_{0}^{+}
\\
-E_{0}^{-+}R_{0}^{+}(1-\rho_h)(\tilde{H}_{\zeta}^{h}-z)^{-1}[(h\partial_x)^2,
\chi_{h}]
E_{0}^{+}
+ \mathcal{O}(e^{-\frac{5S_{0}-C_{a,b,c}\eta}{2h}})  \,.
\end{multline*}
\end{proposition}
\textbf{Proof:} We compute first the coefficients $K_{12}^{(2)}$ and
$K_{22}^{(2)}$ where
$\mathcal{K}_{\zeta}(z)^{n}=
\begin{pmatrix}
  K_{11}^{(n)}& K_{12}^{(n)}\\
K_{21}^{(n)}& K_{22}^{(n)}
\end{pmatrix}$\,.\\
\noindent{$\mathbf{K_{12}^{(2)}=K_{11}K_{12}+ K_{12}K_{22}:}$} Due to the
support condition when $\eta>0$ is chosen small enough and $h>0$ is
small enough, the first term equals:
$$
K_{11}K_{12}=-[(h\partial_x)^2, \rho_{h}](\tilde{H}_{\zeta}^{h}-z)^{-1}[(h\partial_x)^2, \chi_{h}]E_{0}^{+}\,,
$$
and with the same argument as for \eqref{prem_K11} we get:
\begin{equation}\label{est.K11K12}
\|K_{11}K_{12}\| \leq \frac{C_{a,b,c}}{h^{2}}e^{-\frac{2S_{0}}{h}}\,,
\end{equation}
where some additional exponential decay comes from the eigenfunctions appearing in $E_{0}^{+}$ and the support of the derivatives of $\chi_{h}$. From the equations \eqref{eq_K12} and \eqref{eq_K22}, the second term satisfies
$$
\|K_{12}K_{22}\| \leq \|K_{12}\|\|K_{22}\| \leq \frac{C_{a,b,c}}{h^{3}}e^{-\frac{3S_{0}}{h}}\,.
$$

\noindent{$\mathbf{K_{22}^{(2)}=K_{21}K_{12}+ K_{22}^2:}$}  The
first term equals
$$
K_{21}K_{12}= -R_{0}^{+}(1-\rho_h)(\tilde{H}_{\zeta}^{h}-z)^{-1}[(h\partial_x)^2, \chi_{h}]E_{0}^{+}\,,
$$
and as it was done for \eqref{est.K11K12}, we obtain:
\[
\|K_{21}K_{12}\| \leq \frac{C_{a,b,c}}{h^{3}}e^{-\frac{2S_{0}}{h}}\,.
\]
Then, from equation \eqref{eq_K22}, the second term verifies
$$
\|K_{22}^{2}\|\leq\|K_{22}\|^2\leq\frac{C_{a,b,c}}{h^{4}}e^{-\frac{4S_{0}}{h}}\,.
$$
\noindent\textbf{Estimate of $K_{12}^{(3)}$, $K_{22}^{(3)}$, $K_{12}^{(4)}$ and $K_{22}^{(4)}$:} A direct computation gives for $i=1$ and $2$:
\[
K_{i2}^{(n+1)} = K_{i1}K_{12}^{(n)}+K_{i2}K_{22}^{(n)}\,,
\]
and \eqref{eq.roughS02} implies:
\[
\|K_{i2}^{(n+1)}\| \leq \|\mathcal{K}_{\zeta}(z)\|\left(\|K_{12}^{(n)}\|+\|K_{22}^{(n)}\|\right)\leq C_{a,b,c}e^{-\frac{S_{0}-C_{a,b,c}\eta}{2h}}\left(\|K_{12}^{(n)}\|+\|K_{22}^{(n)}\|\right)\,.
\]
Moreover, we have obtained:
\[
K_{12}^{(2)} = \mathcal{O}(e^{-\frac{2S_{0}-C_{a,b,c}\eta}{h}}) \quad \textrm{ and } \quad K_{22}^{(2)} = \mathcal{O}(e^{-\frac{2S_{0}-C_{a,b,c}\eta}{h}})\,,
\]
therefore we have:
\begin{eqnarray*}
&&K_{12}^{(3)} = \mathcal{O}(e^{-\frac{5S_{0}-C_{a,b,c}\eta}{2h}}),
\quad K_{22}^{(3)} =
\mathcal{O}(e^{-\frac{5S_{0}-C_{a,b,c}\eta}{2h}}),\\
&& K_{12}^{(4)} = \mathcal{O}(e^{-\frac{5S_{0}-C_{a,b,c}\eta}{2h}}), \quad K_{22}^{(4)} = \mathcal{O}(e^{-\frac{5S_{0}-C_{a,b,c}\eta}{2h}})\,.
\end{eqnarray*}
\noindent\textbf{Computing $E^{-+}$:} We have
\[
E^{-+}-E^{-+}_{0}=  E_{0}^{-}\psi[-K_{12}+ K_{12}^{(2)}- K_{12}^{(3)}+ K_{12}^{(4)}]+
E_{0}^{-+}[-K_{22}+ K_{22}^{(2)}-K_{22}^{(3)}+K_{22}^{(4)}]+
\mathcal{O}(e^{-\frac{5S_{0}-C_{abc}\eta}{2h}})\,.
\]
Since the operators $E_{0}^{-}\psi$ and $E_{0}^{-+}$ are uniformly bounded, it follows from $E_{0}^{-}\psi K_{12}=0$ and $K_{22}=\mathcal{O}(e^{-\frac{2S_{0}}{h}}h^{-2})$ that:
\[
E^{-+}-E^{-+}_{0}=\mathcal{O}(e^{-\frac{2S_{0}}{h}}h^{-3})
\]
and
\begin{multline*}
E^{-+}-E^{-+}_{0}=  E_{0}^{-}\psi K_{11}K_{12} - E^{-+}_{0}K_{22} + E^{-+}_{0}K_{21}K_{12} + \mathcal{O}(e^{-\frac{5S_{0}-c_{abc}\eta}{2h}})\\
=-E_{0}^{-}[(h\partial_x)^2, \rho_{h}](\tilde{H}_{\zeta}^{h}-z)^{-1}[(h\partial_x)^2,\chi_{h}]E_{0}^{+}
-
E_{0}^{-+}R_{0}^{+}(\chi_{h}-1)E_{0}^{+}\\
-
E_{0}^{-+}R_{0}^{+}(1-\rho_h)(\tilde{H}_{\zeta}^{h}-z)^{-1}
[(h\partial_x)^2, \chi_{h}] E_{0}^{+}
+ \mathcal{O}(e^{-\frac{5S_{0}-C_{abc}\eta}{2h}})\,.
\end{multline*}
\qed

\subsection{Localization of the resonances}

In what follows we discuss the problem of resonances for the operator
$H_{\theta_{0},V-W^{h}}^{h}(0)$.
 Using \eqref{eq.schur_comp} and the detecting method introduced in Subsection~\ref{sec.res}, these coincides with the singularities of the matrix $\left(E^{-+}(z,ze^{-2\theta_{0}}z)\right)^{-1}$ in a sector $\arg(z)\in (-2\tau,0)$ for a suitable $\tau$. Here the symbol $E^{-+}(z,\zeta)$ actually denotes the $(z,\zeta)$-dependent matrix defined in \eqref{eq.inv_grush}.\\ 
The comparison of the Schur complements $E^{-+}$ and $E^{-+}_{0}$
stated in Proposition~\ref{pr.aaa},
 allows to state the following localization result on the resonances
 of the operator $H_{V-W^{h},\theta_{0}}^{h}(0)$ and to estimate accurately
 their variations w.r.t $\theta_{0}$\,.\\
\begin{proposition}
\label{prop.locresdir}
Assume the conditions \eqref{eq.hyp1}\eqref{eq.hyp2}\eqref{eq.hyp3}
and fix $\theta_0$ such that $|\theta_0|\leq \frac{c^2h}{8}$. Then for
$h>0$ small enough,  the operator
$H_{\theta_{0},V-W^{h}}^{h}(0)$ has exactly $\ell$ resonances
$\left\{z_{1}^{h}(\theta_0),\ldots, z_{\ell}^{h}(\theta_0)\right\}$ in
$\omega_{\frac{ch}{2}}$, possibly counted with multiplicities,
with the estimate
\[
z_{j}^{h}(\theta_0) - \lambda_{j}^{h} = \mathcal{O}\left(\frac{e^{-\frac{2S_{0}}{h}}}{h^{3}}\right)\,,
\]
after the proper labelling with respect to $j\in\left\{1,\ldots,
  \ell\right\}$\,.\\
In particular, when 
\begin{equation}
  \label{eq.splilj}
\lim_{h\to 0}h^{3}e^{\frac{2S_{0}}{h}}\min_{j\neq j'}|\lambda_{j}^{h}-\lambda_{j'}^{h}|=+\infty\,,
\end{equation}
there exists $T_{a,b,c}>1$, such that every disc
$D_{j,h}(T)=\left\{z\in\C\,,|z-\lambda_{j}^{h}|
\leq T\frac{e^{-\frac{2S_{0}}{h}}}{h^{3}} \right\}$ contains exactly
one resonance $z_{j}(\theta_{0})$ when $T$ is fixed so that $T\geq
T_{a,b,c}$ and  $h>0$ is small enough. 
\end{proposition}
\noindent\textbf{Proof:} We look for the
points where the matrix $E^{-+}(z,ze^{-2\theta_0})$ is not invertible. When 
$z\in\omega_{\frac{ch}{2}}$, then $|z|\leq \frac{2}{c}$  when $h$ is
small enough and  the two points $z$ and $\zeta=ze^{-2\theta_0}=z+z(e^{2\theta_{0}}-1)$
belong to $\omega_{ch}$. Thus, Proposition \ref{pr.aaa} gives
\begin{equation}
\label{eq.appEpmth}
\|E^{-+}(z,ze^{-2\theta_0})-E^{-+}_{0}(z)\|_{\infty}\leq M_{a,b,c}\frac{e^{-\frac{2S_{0}}{h}}}{h^{3}}\,,
\end{equation}
where the equivalent norm $\displaystyle \|a_{ij}\|_{\infty}=\max_{i,j}|a_{ij}|$ is used\,.\\
Let $\displaystyle\Omega_{h}=\left\{z\in\mathbb{C};~\min_{1\leq
    j\leq\ell}|z-\lambda_j^h| < 2 \ell
    M_{a,b,c} 
\frac{e^{-\frac{2S_{0}}{h}}}{h^{3}}\right\}$ and suppose
$z\notin\Omega_h$\,. Then the coefficients $E_{ij}^{-+}$ of $E^{-+}$ are such
that for all $i\in\{1,\ldots,\ell\}$,
$\displaystyle|E_{ii}^{-+}|>\sum_{\substack{j=1\\j\ne i}}^{\ell}|E_{ij}^{-+}|$
and $E^{-+}$ is invertible by Gershgorin circle theorem.\\
 To conclude
the proof, we have to compare the number of resonances to the number
of Dirichlet eigenvalues in each connected component $\Omega_{j,h}$ of
$\Omega_{h}$ ($\Omega_{j,h}=\Omega_{j',h}$ is not forbidden). 
Defining $E^{-+}(t)= E^{-+}_{0} + t(E^{-+}-E^{-+}_{0})$
for $0\leq t\leq 1$, the number $N(t)$ of points in
$\Omega_{j,h}$ such that $E^{-+}(t)$ is not invertible, is constant on
$[0,1]$. Actually, note first that
\eqref{eq.appEpmth} implies that  for all $t\in[0,1]$, $E^{-+}(t)$ is
invertible when $z\in\partial\Omega_{j,h}$, using an argument similar to
the one used for $E^{-+}$ outside $\Omega_{h}$. Therefore, for any
$t_0\in[0,1]$ the analyticity of $E^{-+}(t_0)$ with respect to $z$
implies 
\[
\inf_{z\in\partial\Omega_{j,h}}|\det E^{-+}(t_0)|>0\,,
\]     
and for $\delta$ small enough, the estimate:
\begin{multline*}
|\det E^{-+}(t_0+\delta)-\det E^{-+}(t_0)| =
|\det(E^{-+}(t_0)+\delta(E^{-+}-E^{-+}_{0}))-\det E^{-+}(t_0)| =
|\delta||R(t_0,\delta)|\\<\inf_{z\in\partial\Omega_{j,h}}|\det
E^{-+}(t_0)|\leq|\det E^{-+}(t_0)| 
\end{multline*}
holds for all  $z\in\partial\Omega_{j,h}$ after noticing that 
the function $|R(t_0,\delta)|$ is a bounded polynomial of
$t_0$, $\delta$ and of the coefficients of $E^{-+}_{0}$ and
$E^{-+}$. The functions  $\det(E^{-+}(t_{0}))$ and $\det
E^{-+}(t_{0}+\delta)$
 are holomorphic functions of 
$z\in \omega_{\frac{ch}{2}}$ such that
$\sup_{z\in \partial\Omega_{j,h}}|\frac{\det E^{-+}(t_{0}+\delta)}{\det
  E^{-+}(t_{0})}-1|<1$. Thus, Rouch{\'e}'s theorem implies
$N(t_0+\delta)=N(t_0)$. The function $N(t)$ is continuous on $[0,1]$
with integer values. It is constant.\\
Assuming $e^{\frac{2S_{0}}{h}}h^{3}
|\lambda_{j}^{h}-\lambda_{j'}^{h}|\to +\infty$ for all pair of
distinct $j,j'$, implies $\Omega_{j,h}\subset D_{j,h}(R)$  for all the
$j$'s with $D_{j,h}(R)\cap D_{j',h}(R)=\emptyset$ if $j\neq j'$ when $R\geq 2\ell M_{a,b,c}$  and
$h$ is small enough. This yields the last statement.
\qed 
\begin{remark}
In the above proposition the term resonances is used for the eigenvalues of the operator $H_{ze^{-2\theta_{0}}}^{h}$, which in principle may still have a positive imaginary part. In the particular case of $\theta_{0}=i\tau$, $\tau\in(0,\frac{\pi}{4})$, the result of Lemma~\ref{le.nonsp} implies that these eigenvalues must lay in the lower half complex plane. On the other hand, the result of next proposition and the lower bound on $\left|\Imag z_j^h(0)\right|$ (see Proposition ~\ref{pr.FGR}) implicitely yields: $\Imag z_j^h(\theta_0)<0$ on a suitable range of $|\theta_0|$. Under each of such conditions the points $z_j^h(\theta_0)$ corresponds to resonances of the operator $H_{\theta_{0},V-W^{h}}^{h}(0)$ as defined in Proposition \ref{prop.agco}. 
\end{remark}

 The next Proposition  localizes the resonances
 $z_{j}^{h}(\theta_0)$ of $H_{\theta_{0},V-W^{h}}^{h}(0)$ with respect to
 the resonances $z_{j}^{h}:=z_{j}^{h}(0)$ of $H_{0,V-W^{h}}^{h}(0)$ by making
 use of the comparison between $E^{-+}(z,ze^{-2\theta_{0}})$ and $E^{-+}(z,z)$.

\begin{proposition}
\label{prop.locresth}
 Assume the conditions \eqref{eq.hyp1}\eqref{eq.hyp2}\eqref{eq.hyp3}
 and $e^{-\frac{S_{0}}{4h}}\leq |\theta_0|\leq \frac{c^2h}{8}$\,.
 Then for $h>0$ small enough, 
the matrices $E^{-+}$ of Proposition~\ref{pr.aaa} associated with 
$\zeta=z$ and $\zeta=ze^{-2\theta}$ satisfy
\begin{equation}
\label{eq.Et-E0}
\sup_{z\in \omega_{\frac{ch}{2}}}|E^{-+}(z,ze^{-2\theta_{0}})-E^{-+}(z,z)|= 
\mathcal{O}\left(|\theta_0|\frac{e^{-\frac{2S_{0}}{h}}}{h^{3}}\right)\,.
\end{equation}
If additionally \eqref{eq.splilj} is assumed  the variation of the
resonances around $\lambda^{0}$ for $e^{-\frac{S_{0}}{4h}}\leq |\theta_{0}|\leq
\frac{c^{2}h}{8}$ and $\theta_{0}=0$ is estimated by
\[
\max_{j\in
  \left\{1,\ldots,\ell\right\}}|z_{j}^{h}(\theta_0)-z_{j}^{h}|= \mathcal{O}\left(|\theta_0|\frac{e^{-\frac{2S_{0}}{h}}}{h^{3}}\right)\,.
\]
\end{proposition}
\noindent\textbf{Proof:} For 
$z\in\omega_{\frac{ch}{2}}$ and $\theta_0$ such that $|\theta_0|\leq\frac{c^2h}{8}$, the Proposition~\ref{pr.aaa} implies that: 
\begin{align*}
E^{-+}(z,ze^{-2\theta_{0}}) - E^{-+}(z,z) =& E_{0}^{-}[(h\partial_x)^2, \rho_{h}]D[(h\partial_x)^2, \chi_{h}]E_{0}^{+} \\
&+E_{0}^{-+}R_{0}^{+}(1-\rho_h)D[(h\partial_x)^2,
\chi_{h}]E_{0}^{+} + \mathcal{O}(e^{-\frac{5S_{0}-C_{a,b,c}\eta}{2h}})\\
=& I + II + \mathcal{O}(e^{-\frac{5S_{0}-C_{a,b,c}\eta}{2h}})\,,
\end{align*}
where
$D=(\tilde{H}_{z}^{h}-z)^{-1}-(\tilde{H}_{ze^{-2\theta_0}}^{h}-z)^{-1}$. The
operator $E_{0}^{-}$ being bounded, the first term $I$ is estimated as
we did for \eqref{est.K11K12} where $(\tilde{H}_{\zeta}^{h}-z)^{-1}$
is replaced by $D$ and we use Proposition~\ref{pr.diffestabsbc}
instead of Proposition~\ref{pr.resestabsbc}. This leads to: 
\[
||I||\leq C_{a,b,c}|z^{1/2}-(ze^{-2\theta_0})^{1/2}|\frac{e^{-\frac{2S_0}{h}}}{h^3}\leq C_{a,b,c}|\theta_0|\frac{e^{-\frac{2S_0}{h}}}{h^3}\,.
\]
For $II$, using the exponential decay given by the operator $R_{0}^{+}$, we get:
\[
||II||\leq C_{a,b,c}|\theta_0|\frac{e^{-\frac{2S_0}{h}}}{h^3}\,.
\]
The assumption $e^{-\frac{S_{0}}{4h}}\leq |\theta_{0}|$ ensures that
the remainder $\mathcal{O}(e^{-\frac{5S_{0}-C_{a,b,c}\eta}{2h}})$ is
absorbed by $|\theta_{0}|e^{-\frac{2S_{0}}{h}}h^{-3}$ as $h\to
0$\,. We have proved \eqref{eq.Et-E0}.\\
When \eqref{eq.splilj} is verified, Proposition~\ref{prop.locresdir}
says that every disc $D_{j,h}(T)=\{z\in \C,\,|z-\lambda_{j}^{h}|<
  Te^{-\frac{2S_{0}}{h}}h^{-3}\}$ for any $T\geq T_{a,b,c}$, 
 contains exactly one resonance $z_{j}^{h}(\theta_{0})$, and in particular one resonance $z_{j}^{h}$
when $\theta_{0}=0$.
Hence the matrix $E^{-+}(z,z)$ has only simple poles and its inverse
is the meromorphic function 
\begin{equation}\label{eq.devE0}
(E^{-+}(z,z))^{-1} = \sum_{j=1}^{\ell}\frac{A_j^h}{z-z_j^h} +
F^h(z)\,,\quad z\in \omega_{\frac{ch}{2}}\,.
\end{equation}
The matrix $A_{j}^{h}$ is nothing but the residue
$$
A_{j}^{h}=\frac{1}{2i\pi}\int_{\partial_{D_{j,h}(T)}}(E^{-+}(z,z))^{-1}~dz\,,
$$
while the function $F^{h}(z)$ is a holomorphic function estimated via
the maximum principle by
\begin{equation}
  \label{eq.holbord}
\sup_{z\in \omega_{\frac{ch}{4}}}|F^{h}(z)|
\leq \sup_{z\in \partial \omega_{\frac{ch}{4}}}\left[|(E^{-+}(z,z))^{-1}|+
  \sum_{j=1}^{\ell}\frac{|A_{j}^{h}|}{|z-z_{j}^{h}|}\right]\,.
\end{equation}
The estimate of Proposition~\ref{pr.aaa} says
$$
|E^{-+}(z,z)-E_{0}^{-+}(z)|\leq C_{a,b,c}\frac{e^{-\frac{2S_{0}}{h}}}{h^{3}}\,,
$$
while we know $|(E_{0}^{-+}(z))^{-1}|
\leq \max_{1\leq j\leq
  \ell}|z-\lambda_{j}^{h}|^{-1}$\,.
%%% and $\max_{1\leq j\leq \ell} |z_{j}^{h}-\lambda_{j}^{h}|\leq
%%%T_{a,b,c}\frac{e^{-\frac{2S_{0}}{h}}}{h^{3}}$\,.
After writing 
\begin{equation}
  \label{eq.secresEmp}
(E^{-+})^{-1}=\left[1+(E_{0}^{-+})^{-1}(E^{-+}-E_{0}^{-+})\right]^{-1}(E_{0}^{-+})^{-1}\,,
\end{equation}
we get for $T>\max\{T_{a,b,c},2C_{a,b,c}\}$
$$
\sup_{z\in \partial D_{j,h}(T)}|(E^{-+}(z,z))^{-1}|\leq 
\frac{h^{3}e^{\frac{2S_{0}}{h}}}{T[1-\frac{C_{a,b,c}}{T}]}
\leq \frac{2h^{3}e^{\frac{2S_{0}}{h}}}{T}\,,
$$
and finally the uniform bound for the residues 
$$
\max_{1\leq j\leq \ell}|A_{j}^{h}|\leq 2\,.
$$
The holomorphic part $F^{h}(z)$ is then estimated with
\eqref{eq.holbord}. Actually, the first term is estimated with the help of 
\eqref{eq.secresEmp} while the second term is treated with  the above
estimate of $A_{j}^{h}$ and by making use of 
$\max_{1\leq j\leq \ell} |z_{j}^{h}-\lambda_{j}^{h}|\leq
T_{a,b,c}\frac{e^{-\frac{2S_{0}}{h}}}{h^{3}}$:
$$
\sup_{z\in \omega_{\frac{ch}{4}}}|F^{h}(z)|\leq \frac{C'_{a,b,c}}{h}\,.
$$
For all $z\in \omega_{\frac{ch}{4}}\setminus \left\{z_{1}^{h},\ldots,
  z_{\ell}^{h}\right\}$,
 the inverse of  $E^{-+}(z,z)$ is thus estimated by
$$
|(E^{-+}(z,z))^{-1}|\leq \sum_{j=1}^{\ell}\frac{2}{|z-z_{j}^{h}|}+ \frac{C'_{a,b,c}}{h}\,.
$$
We now write for $z\not\in \omega_{\frac{ch}{4}}\setminus \left\{z_{1}^{h},\ldots,
  z_{\ell}^{h}\right\}$
$$
E^{-+}(z,ze^{-2\theta_{0}})= E^{-+}(z,z)\left[1+
  (E^{-+}(z,z))^{-1}(E^{-+}(z,ze^{-2\theta_{0}})-E^{-+}(z,z))\right]\,.
$$ 
Due to the estimate  \eqref{eq.Et-E0} the condition
$$
\min_{j\in\left\{1,\ldots,\ell\right\}}|z-z_{j}^{h}|\geq
T\frac{e^{-\frac{2S_{0}}{h}}|\theta_{0}|}{h^{3}}
$$
implies
$$
\left|(E^{-+}(z,z))^{-1}(E^{-+}(z,ze^{-2\theta_{0}})-E^{-+}(z,z))\right|\leq 
\left[2\ell
\frac{h^{3}e^{\frac{2S_{0}}{h}}}{T|\theta_{0}|}+C_{a,b,c}'h^{-1}\right]C_{a,b,c}''
\frac{|\theta_{0}|e^{-\frac{2S_{0}}{h}}}{h^{3}}\,,
$$
where the right-hand side is smaller than $1$ if $T\geq 4\ell C_{a,b,c}''$
and  $h>0$ is small enough. Outside $\cup_{j=1}^{\ell}\left\{z\in \C,
  |z-z_{j}^{h}|\leq T|\theta_{0}|e^{-\frac{2S_{0}}{h}}h^{-3}\right\}$,
$E^{-+}(z,ze^{-2\theta_{0}})$ is invertible.
For such a $T$ we have proved
$$
\max_{j\in \left\{1,\ldots,\ell\right\}}|z_{j}^{h}(\theta_0)-z_{j}^{h}|\leq
T\frac{e^{-\frac{2S_{0}}{h}}|\theta_{0}|}{h^{3}}\,. 
$$
\qed

\subsection{A Fermi-Golden rule}

In \cite{BNP2}, a  Fermi Golden rule
for the imaginary parts of resonances $\Gamma_{j}=-\Imag z_{j}^{h}$ in
the case $\theta_{0}=0$ has been introduced. It plays a major role in
the analysis of the nonlinear effects studied in
\cite{BNP}\cite{BNP1}\cite{BNP2}\cite{Ni1}\cite{Ni2}\cite{BFN}
for it expresses accurately how the tunnel effect between the resonant
state and the incoming waves  is balanced between the left and
right-hand sides.
By assuming 
\begin{equation}
  \label{eq.splilj2}
\lim_{h\to 0}e^{\frac{S_{U}}{h}}\min_{1\leq j<
  j'\leq
  \ell}|\lambda_{j}^{h}-\lambda_{j'}^{h}|=+\infty
\quad\text{with}\quad S_{U}<\frac{S_{0}}{8}\,,
\end{equation}
which is stronger than \eqref{eq.splilj}, the energy range of
$\lambda\in \R$ associated
with the resonance $z_{j}^{h}$ is given by
$|\lambda-z_{j}^{h}|\leq
e^{-\frac{S_{U}}{h}}$\,.
When 
$\tilde{\psi}_{-,0}^{h}(\pm\sqrt{\lambda},.)$ denote the generalized
eigenfunctions of the filled well Hamiltonian $H_{0,V}^{h}(0)$ at energy
$\lambda$ defined in section~\ref{se.absbc} and
$\Phi_{j}^{h}$, $j\in \left\{1,\ldots, \ell\right\}$, denote  the
normalized eigenfunctions of the Dirichlet Hamiltonian $H_{D}^{h}$ 
given in \eqref{eq.Dirwell}, the formula 
\begin{equation}
  \label{eq.FGR}
  \Gamma_{j}^{h}+o(\Gamma_{j}^{h})=
\frac{|\langle
  W^h\tilde{\psi}_{-,0}^{h}(\sqrt{\lambda},.),\Phi_j^h\rangle|^2+|\langle
  W^h\tilde{\psi}_{-,0}^{h}(-\sqrt{\lambda},.),\Phi_j^h\rangle|^2}{4h\sqrt{\lambda}}
\geq \frac{e^{-\frac{2 S_{0}}{h}}}{C}\,,
\end{equation}
for all $\lambda\in \R$ such that $|\lambda-z_{j}^{h}|\leq
e^{-\frac{S_{U}}{h}}$,
has been proved under additional assumptions about the localization
of the $\Phi_{j}^{h}$ within $\supp W^{h}$\,. We refer to
Proposition~7.9 in \cite{BNP2} and to the subsequent explicit
computations in Sections~7 and Section~8 of \cite{BNP2} for the
details
and in particular for the lower bound.

We shall assume that the formula \eqref{eq.FGR} is true when $\theta_{0}=0$ and
check that it remains true when $\theta_{0}\neq 0$ is small enough.
We shall use the notation
$$
\Gamma_{j}^{h}=-\Imag z_{j}^{h}\quad \text{and}\quad 
\Gamma_{j}^{h}(\theta_{0})=-\Imag z_{j}^{h}(\theta_{0})\,.
$$
\begin{proposition}
\label{pr.FGR}
Assume the conditions
\eqref{eq.hyp1}\eqref{eq.hyp2}\eqref{eq.hyp3}\eqref{eq.splilj2}\eqref{eq.FGR}
and
take $\theta_{0}$ such that $e^{-\frac{S_{0}}{4h}}\leq
|\theta_{0}|\leq \frac{c^{2}h}{8}$ then the Fermi
Golden Rule
\begin{equation}\label{eq.ROFth}
\Gamma_j^h(\theta_0)
= \frac{|\langle
  W^h\tilde\psi_{-,\theta_{0}}^{h}(\sqrt{\lambda},.),\Phi_j^h\rangle|^2+|\langle
  W^h
\tilde\psi_{-,\theta_{0}}^{h}(-\sqrt{\lambda},.),\Phi_j^h\rangle|^2}{4h\sqrt{\lambda}}+o\left(\Gamma_j^h\right)+\mathcal{O}\left(|\theta_0|\frac{e^{-\frac{2S_{0}}{h}}}{h^{5}}\right)
\end{equation}
holds for all $\lambda\in \R$ such that $|\lambda-z_{j}^{h}|\leq
e^{-\frac{S_{U}}{h}}$.\\
In particular, when $\lim_{h\to 0}h^{-5}\theta_{0}=0$ we have
\begin{equation}
  \label{eq.FGRt}
    \Gamma_{j}^{h}(\theta_{0})+o(\Gamma_{j}^{h}(\theta_{0}))=
\frac{|\langle
  W^h\tilde{\psi}_{-,\theta_{0}}^{h}(\sqrt{\lambda},.),\Phi_j^h\rangle|^2+|\langle
  W^h\tilde{\psi}_{-,\theta_{0}}^{h}(-\sqrt{\lambda},.),\Phi_j^h\rangle|^2}{4h\sqrt{\lambda}}
\geq \frac{e^{-\frac{2 S_{0}}{h}}}{C}\,.
\end{equation}
\end{proposition}
\textbf{Proof:} 
Proposition~\ref{prop.locresth} gives:
\begin{equation}\label{eq.appGa}
|\Gamma_j^h(\theta_0)-\Gamma_j^h| \leq C_{a,b,c}|\theta_0|\frac{e^{-\frac{2S_{0}}{h}}}{h^{3}}\,.
\end{equation}
Let $f(\theta_0)=\langle W^h\psi_{-,\theta_{0}}^{h}(\sqrt{\lambda},.),\Phi_j^h\rangle$. The pointwise and $L^{2}$  weighted estimates of $u={\tilde\psi}_{-,\theta_{0}}^{h}(\sqrt{\lambda},.)-{\tilde\psi}_{-,0}^{h}(\sqrt{\lambda},.)$ stated in Proposition~\ref{pr.geneigdif}
 with $\varphi(x)=d_{Ag}(x,a,V,\lambda)$ say
$$
 h^{1/2}\sup_{x\in [a,b]}|e^{\frac{\varphi(x)}{h}}u(x)|+
\|e^{\frac{\varphi}{h}}u\|_{L^{2}}
\leq 
\frac{C_{a,b,c}'|\theta_{0}|}{h^{3/2}}\,,
$$
while Proposition~\ref{pr.geneig} gives
$$
 h^{1/2}\sup_{x\in
   [a,b]}|e^{\frac{\varphi(x)}{h}}\tilde{\psi}_{-,\theta_{0}}^{h}(\sqrt{\lambda},
 x)|+
\|e^{\frac{\varphi}{h}}\tilde{\psi}_{-,\theta_{0}}^{h}(\sqrt{\lambda})\|_{L^{2}}
\leq 
\frac{C_{a,b,c}'}{h^{1/2}}\,,
$$
with the same estimate for $\theta_{0}=0$\,.
Moreover, the exponential decay of $\Phi_{j}^{h}$ stated in Proposition~\ref{pr.Direxp}
can be written as
$$
h^{1/2}\sup_{x\in [a,b]}|e^{\frac{\varphi_{j}(x)}{h}}\Phi_{j}^{h}(x)|+
\|e^{\frac{\varphi_{j}}{h}}\Phi_{j}^{h}\|_{L^{2}}
\leq \frac{C_{a,b,c}'}{h}\,,
$$
with $\varphi_{j}(x)=d_{Ag}(x,U,V,\lambda_j^{h})$. Recalling that
$$
\|W_{1}^{h}\|_{L^{\infty}}\leq
\frac{1}{c}\,,\quad\|W_{2}^{h}\|_{\mathcal{M}_{b}}\leq
\frac{h}{c}\,,\quad
\supp W^{h}\subset \left\{d(x, U)\leq h\right\}\,,
$$
and $S_{0}= d_{Ag}(U,\left\{a,b\right\},V,\lambda^0)\leq d_{Ag}(x,\left\{a,b\right\},V,\lambda)+ \mathcal{O}(h)$ when $x\in U^h$ and $|\lambda-z_{j}^{h}|\leq e^{-\frac{S_{U}}{h}}$\,.
Hence we get with our assumptions
\begin{eqnarray*}
&&
|f(\theta_0)-f(0)|
\leq C_{a,b,c}''|\theta_0|\frac{e^{-\frac{S_0}{h}}}{h^{5/2}}
\\
\text{and}
&&
|f(\theta_0)| \leq C_{a,b,c}''\frac{e^{-\frac{S_0}{h}}}{h^{3/2}}\,.
\end{eqnarray*}
We obtain 
\begin{equation}\label{eq.appf}
\left|\frac{|f(\theta_0)|^2}{4h\sqrt{\lambda}}-\frac{|f(0)|^2}{4h\sqrt{\lambda}}\right|
= \mathcal{O}\left(\frac{e^{-\frac{2S_{0}}{h}}|\theta_{0}|}{h^{5}}\right)\,.
\end{equation}
\qed
\begin{remark}
In \cite{BNP2}, the Fermi Golden Rule \eqref{eq.FGR} has been studied
with $W^{h}\in L^{\infty}((a,b))$. Nevertheless it can be proved in
cases when the singular part $W_{2}^{h}$ does not vanish by a direct
analysis like for example when $W_{2}^{h}=h \delta_{c}$. 
The presentation of Proposition~\ref{pr.FGR} shows that the stability
result w.r.t to $\theta_{0}$ holds   in this more general framework and
leaves the possibility of further applications.
\end{remark}

\section{Accurate resolvent estimates for the whole space problem}
\label{se.wholeline}

In the previous sections~\ref{se.expdecay} and \ref{se.grush}, we got
accurate resolvent estimates with respect to 
$h>0$ for the problem reduced to the interval $(a,b)$. We  use here
this information in order to derive accurate resolvent estimates  for 
$(H_{\theta_{0},V-W^{h}}(\theta_{0})-z)^{-1}$ when
$\theta_{0}=i h^{N_{0}}$, $N_{0}>1$,  which are essential in
the justification of the adiabatic evolution.

\subsection{Localization of the spectrum}
\label{se.locsp}
 The  results of Corollary~\ref{cor_krein}, Proposition~\ref{prop.agco}, Proposition~\ref{prop.locresdir} and section~\ref{sec.res}  can be summarized with the
 corresponding assumptions.
\begin{proposition}
\label{pr.summ}
Assume that $V\in L^{\infty}((a,b);\R)$ and $W^{h}=W^{h}_{1}+W^{h}_{2}\in
\mathcal{M}_{b}((a,b))$  are real valued with the hypothesis \eqref{eq.hyp0},
$$
c1_{(a,b)} \leq V\quad, \quad \|V\|_{L^{\infty}}\leq \frac{1}{c}\quad,
\quad \|W_{1}^{h}\|\leq \frac{1}{c}\quad,\quad \|W^{h}_{2}\|\leq \frac{h}{c}\,,
$$
 $W_1^h$ and $W_2^h$ are supported in the domain
$U_h=\{x\in(a,b); \, d(x,U)\leq h\}$ where $U$ is a fixed compact subset of $(a,b)$\,.
Assume also $\theta_{0}=ih^{N_{0}}$ with $N_{0}>1 $ and $h<
h_{0}$\,.
Then:
\begin{description}
\item[a)] 
$\sigma_{ess}(H^{h}_{\theta_{0},V-W^{h}}(\theta_{0}))=\sigma_{ess}(H^{h}_{ND,V-W^{h}}(\theta_{0}))=e^{-2\theta_{0}}\R_{+}$\,;\\
\item[b)] The equality \eqref{res_formula_gen} written with $\theta=\theta_{0}$
\begin{multline}
\label{res_formula_gen_rap}
\left(  H_{\theta_{0},V-W^h}^{h}(\theta_{0})-z\right)  ^{-1}  =\left(  H_{ND,V-W^h}
^{h}(\theta_{0})-z\right)  ^{-1}\\
-\sum_{i,j,k=1}^{4}\left(  Bq(z,\theta_0,V-W^h)-A\right)  _{ij}^{-1}B_{jk}\left\langle
\gamma(\underline{e}_{k},\bar{z},\bar{\theta_{0}}),\cdot\right\rangle
_{L^{2}(\mathbb{R})}\gamma(\underline{e}_{i},z,\theta_{0}) 
\end{multline} 
 holds as an identity of meromorphic functions on $\C\setminus
 e^{-2\theta_{0}}\R_{+}$\,.
\item[c)] In $\Sigma_{h}=\left\{z\in\C, c\leq \Real{z}\leq \inf_{x\in(a,b)}V(x)-c, |\arg z|< 2 |\Imag \theta_{0}| \right\}$, the (discrete) spectrum of $H^{h}_{\theta_{0},
  V-W^{h}}(\theta_{0})$ is made of eigenvalues $z^{h}$ which satisfy
$\Ker(H_{z^{h}e^{-2\theta_{0}}}^{h}-z^{h})\neq \left\{0\right\}$ where
$H_{\zeta}^{h}$ is the operator defined in \eqref{eq.defHzet}\,.
When the spectrum of the Dirichlet Hamiltonian $\sigma(H_{D}^{h})$ is
made of clusters such that \eqref{eq.hyp1} and \eqref{eq.hyp3} are
satisfied, then there exists $\kappa_{a,b,c}>0$ such that 
$$
\forall z^{h}\in \Sigma_{h}\cap
\sigma(H_{\theta_{0},V-W^{h}}^{h}(\theta_{0}))\,,\quad
d(z^{h}, \sigma (H_{D}^{h}))\leq e^{-\frac{\kappa_{a,b,c}}{h}}\,.
$$
\end{description}
\end{proposition}

The resolvent of $H_{\theta_{0},V-W^{h}}^{h}(\theta_{0})$ will be studied
in  a domain surrounding  a single cluster of eigenvalues $z^{h}$, i.e. with
$\lim_{h\to 0}z^{h}=\lambda^{0}$,  with a distance to
$\sigma(H_{\theta_{0}, V-W^{h}}^{h}(\theta_{0}))$ bounded from
below by $\frac{1}{C_{a,b,c}}h^{N_{0}}$\,.
\begin{definition}
\label{de.Gh}
  The complex domain $G_{h}(\lambda^{0})$ is chosen accordingly to the constants
  $a,b,c$ involved in the assumptions on $V$, $W^{h}$, the hypotheses
\eqref{eq.hyp0}\eqref{eq.hyp1}\eqref{eq.hyp2}\eqref{eq.hyp3} and to
  $\theta_{0}=ih^{N_{0}}$, $N_{0}>1$:
\begin{equation}
    \label{eq.Gh}
G_{h}(\lambda^{0}):= 
\left\{z\in\C,\; |\Real{z}-\lambda^{0}|\leq \Xi_{a,b,c} h
  ,\;
 |\arg z|\leq h^{N_{0}}\; \textrm{ and }
\; d(z,\sigma(H_{D}^{h}))\geq\frac{h^{N_{0}}}{\Xi_{a,b,c}} \right\}\,,
\end{equation}
for some constant $\Xi_{a,b,c}>0$ chosen large enough.
\end{definition}

\subsection{Estimates of the finite rank part}
We study here the finite rank part of \eqref{res_formula_gen_rap}:
\begin{equation}
  \label{eq.Upsi}
  \Upsilon^{h}(z, V-W^{h}) =
\sum_{i,j,k=1}^{4}\left(  Bq(z,\theta_0,V-W^h)-A\right)  _{ij}^{-1}B_{jk}\left\langle
\gamma(\underline{e}_{k},\bar{z},\bar{\theta_{0}}),\cdot\right\rangle
_{L^{2}(\mathbb{R})}\gamma(\underline{e}_{i},z,\theta_{0})\,,
\end{equation}
and its variations between the case $W^{h}=0$ and $W^{h}\neq 0$\,.
Every factor will be considered separately. 
Hence it is convenient to
keep the notation $\gamma_{V-W^{h}}(\underline{e}_{i},z,\theta)$ for  the total potential 
$V-W^{h}$ and $\gamma_{V}(\underline{e}_{i},z,\theta)$ when
$W^{h}\equiv 0$\,.
\begin{proposition}
  \label{pr.estupsi}
Assume the hypotheses \eqref{eq.hyp0}
\eqref{eq.hyp1}\eqref{eq.hyp2}\eqref{eq.hyp3}. Take
$\theta_{0}=ih^{N_{0}}$, $N_{0}>1$, and let $G_{h}(\lambda^{0})$ be the set defined
by 
\eqref{eq.Gh}. 
The matrices $(Bq(z,\theta_0,V-W^{h})-A)^{-1}$ and $(Bq(z,\theta_0,V)-A)^{-1}$ verify
the uniform estimate
\begin{equation}
  \label{eq.estimBqA}
\forall z\in G_{h}(\lambda^{0})\,,\quad|(Bq(z,\theta_0,V-W^{h})-A)^{-1}|+ |(Bq(z,\theta_0,V)-A)^{-1}|\leq C_{a,b,c}h\,,
\end{equation}
in any fixed matricial norm.\\
The functions $\gamma_{\mathcal{V}}$
satisfy the uniform estimates 
\begin{eqnarray}
\label{eq.gsupp}
&&
\supp \gamma_{\mathcal{V}}(\underline{e}_{2,3}, z,\theta_{0})= \supp
\gamma_{\mathcal{V}}(\underline{e}_{2,3},\bar z,\overline{\theta_{0}})\subset
[a,b]\,,
\\
\label{eq.gsupp1}
&&
\supp \gamma_{\mathcal{V}}(\underline{e}_{1}, z,\theta_{0})
=
\supp \gamma_{\mathcal{V}}(\underline{e}_{1}, \bar z,\overline{\theta_{0}})
\subset
[b,+\infty)\,,\\
\label{eq.gsupp4}
&&
\supp \gamma_{\mathcal{V}}(\underline{e}_{4}, z,\theta_{0})
=
\supp \gamma_{\mathcal{V}}(\underline{e}_{4}, \bar z,\overline{\theta_{0}})
\subset
(-\infty,a]\,,\\
  \label{eq.estimg23}
 && \max_{i=2,3}\|\gamma_{\mathcal{V}}(\underline{e}_{i}, z,
  \theta_{0})\|_{H^{1,h}((a,b))}
+ \|\gamma_{\mathcal{V}}(\underline{e}_{i}, \bar z,
  \overline{\theta_{0}})\|_{H^{1,h}((a,b))} \leq \frac{C_{a,b,c}}{h^{3/2}}\,,\\
 && \label{eq.estimg14}
  \max_{i=1,4}\|\gamma_{\mathcal{V}}(\underline{e}_{i}, z,
  \theta_{0})\|_{H^{1,h}(\R\setminus [a,b])}
+ \|\gamma_{\mathcal{V}}(\underline{e}_{i}, \bar z\,,
  \overline{\theta_{0}})\|_{H^{1,h}(\R\setminus [a,b])} \leq \frac{C_{a,b,c}}{h^{(N_{0}+1)/2}}
\end{eqnarray}
holds when $z\in G_{h}(\lambda^{0})$ and $\mathcal{V}=V-W^{h}$ or $\mathcal{V}=V$\,.\\
Moreover the differences $(Bq(z,\theta_0,V-W^{h})-A)^{-1}-(Bq(z,\theta_0,V)-A)^{-1}$
and $\gamma_{V-W^{h}}-\gamma_{V}$ are estimated
by 
\begin{eqnarray}
\label{eq.diffBA}
&&|(Bq(z,\theta_0,V-W^{h})-A)^{-1}-(Bq(z,\theta_0,V)-A)^{-1}|\leq
C_{a,b,c}e^{-\frac{S_{0}}{2h}}\,,
\\
\label{eq.diffg23}
 &&\max_{i=2,3}\|\gamma_{V-W^{h}}(\underline{e}_{i}, z,
  \theta_{0})-\gamma_{V}(\underline{e}_{i}, z,
  \theta_{0})\|_{H^{1,h}((a,b))}\leq
  C_{a,b,c}e^{-\frac{S_{0}}{2h}}\,,\\
\nonumber
\text{with}&& S_{0}=d_{Ag}(\left\{a,b\right\},U,V,\lambda^{0})\,,
\\
 &&
 \label{eq.diffg14}
  \gamma_{V-W^{h}}(\underline{e}_{i},
  z,\theta_{0})=\gamma_{V}(\underline{e}_{i},z,\theta_{0})\,,
\quad
\text{for}\quad i=1,4\,,
\end{eqnarray}
when $z\in G_{h}(\lambda^{0})$, with the same result for $(\bar z, \overline{\theta_{0}})$\,.
\end{proposition} 
The matrix $Bq(z,\theta_0,\mathcal{V})-A$ is expressed with boundary values and
we will need the next lemma.
\begin{lemma}
\label{lem.estBordDef}
Assume \eqref{eq.hyp2},   $\theta_{0}=ih^{N_{0}}$, $N_{0}>1$, and set
$$
\varphi_{b}(x)=d_{Ag}(x,b,V,\lambda^{0})
\,,\quad 
\varphi_{a}(x)=d_{Ag}(x,a,V,\lambda^{0})
\,,\quad S_{0}= d_{Ag}(\left\{a,b\right\},U,V,\lambda^{0})\,.
$$
 For any $z\in G_h$, the solution $\tilde{u}_{2}$ (resp $\tilde{u_{3}}$) to 
\[
\left\{\begin{array}[c]{l}(\tilde{P}^h-z)\tilde{u}_{2,3}=(-h^{2}\Delta +
    V-z)\tilde{u}_{2,3}= 0\\
\tilde{u}_2(a)=0, \quad \tilde{u}_2(b)=1
\quad
(\text{resp.}~\tilde{u}_3(a)=1, \quad \tilde{u}_3(b)=0)\,,
\end{array}
\right.
\]
verifies:
\begin{eqnarray}
\label{eq.estu2.1}
&& \|\tilde{u}_{2}\|_{H^{1,h}((a,b))}\leq C_{a,b,c}h^{1/2}\,,
\quad
\|e^{\frac{\varphi_{b}}{h}}\tilde{u}_2\|_{L^{2}((a,b))}+
\|e^{\frac{\varphi_{b}}{h}}h\tilde{u}_2'\|_{L^{2}((a,b))}\leq
\frac{C_{a,b,c}}{h^{1/2}}\,,
\\
\label{eq.estu3.1}
\text{resp.}&&
\|\tilde{u}_{3}\|_{H^{1,h}((a,b))}\leq C_{a,b,c}h^{1/2}\,,\quad
\|e^{\frac{\varphi_{a}}{h}}\tilde{u}_{3}\|_{L^{2}((a,b))}+
\|e^{\frac{\varphi_{a}}{h}}h\tilde{u}_{3}'\|_{L^{2}((a,b))}\leq
\frac{C_{a,b,c}}{h^{1/2}}\,,
\\
\label{eq.estu2.2}
&&
|\tilde{u}_2'(b)|\leq\frac{C_{a,b,c}}{h}\,,
\quad |\Imag \tilde{u}_2'(b)|\leq C_{a,b,c} h^{N_0-1}\,,
\quad
|\tilde{u}_2'(a)|\leq\frac{C_{a,b,c}}{h^{2}}e^{-\frac{2S_{0}}{h}}\,,\\
\label{eq.estu3.2}
\text{resp.}
&&
|\tilde{u}_3'(a)|\leq\frac{C_{a,b,c}}{h}\,,\quad |\Imag \tilde{u}_3'(a)|\leq C_{a,b,c}
 h^{N_0-1}\,,\quad
|\tilde{u}_3'(b)|\leq\frac{C_{a,b,c}}{h^{2}}e^{-\frac{2S_{0}}{h}}\,.
\end{eqnarray}
For any $z\in G_{h}(\lambda^{0})$, the solution $u_{2}$
(resp. $u_{3}$) to 
\eqref{u_intz}  rewritten as
\[
\left\{\begin{array}[c]{l}(P^h-z)u_{2,3}=(-h^{2}\Delta +
    V-W^{h}-z)u_{2,3}= 0\\
u_2(a)=0, \quad u_2(b)=1 \quad
(\text{resp.}~u_3(a)=1, \quad u_3(b)=0)\,,
\end{array}
\right.
\]
 can be compared with $\tilde{u}_{2}$ (resp. $\tilde{u}_{3}$) according
 to 
 \begin{eqnarray}
   \label{eq.comutu}
   \max_{i\in\left\{2,3\right\}}\|u_{i}-\tilde{u}_{i}\|_{H^{1,h}}+|u_{i}'(a)-\tilde{u}_{i}'(a)|
+ |u_{i}'(b)-\tilde{u}_{i}'(b)|\leq C_{a,b,c}e^{-\frac{3S_{0}}{4h}}\,.
 \end{eqnarray}
\end{lemma}
\noindent\textbf{Proof:} 
Let us first focus on the case $W^{h}=0$.
It suffices to study the case of $\tilde{u}_{2}$, the
result for $\tilde{u}_{3}$ being deduced by symmetry.\\
Consider a real valued function $u_0\in C^{\infty}$ 
such that $u_0(b)=1$, $\supp u_0 \cap[a,b]\subset[b-h,b]$ and:
\[
\|(h\partial_{x})^{\alpha}u_0\|\leq C_{\alpha}\,,\quad \alpha \in \N\,,
\] 
and set $\tilde{u}_2=u_0+\tilde{v}$ where $\tilde{v}$ solves
$(\tilde{H}_D^h-z)\tilde{v}=f$ 
with $f=(h\partial_x)^2u_0-(V-z)u_0$\,. We keep the notation
\eqref{eq.Dirfilled} for the Dirichlet Hamiltonian associated with $\tilde{P}^{h}$\,.\\
Owing to $V-\Real z\geq c$,  the variational formulation of
$(\tilde{H}_{D}^{h}-z)\tilde{u}=f$ with  $\|f\|_{L^2}\leq Ch^{1/2}$
provides  $\|\tilde{v}\|_{H^{1,h}}\leq Ch^{1/2}$ which yields the
first estimate of \eqref{eq.estu2.1}\,.
With the equation $-h^{2}\tilde{v}''=-(V-z)\tilde{v}+f$ and applying
Lemma~\ref{lem.estBord} to $\tilde{v}(b-x)$, we get
$|\tilde{v}'(b)|\leq \frac{C h^{1/2}}{h^{3/2}}$ and hence the first
estimate of \eqref{eq.estu2.2}. 
The second estimate of \eqref{eq.estu2.2} comes similarly from $|\Imag
\tilde{v}'(b)|\leq Ch^{N_{0}-1}$\,. It suffices to write that
$w=\tilde{v}-\overline{\tilde{v}}$ solves
$$
(\tilde{H}_{D}^{h}-z)w= (z-\bar z)u_{0}+ (z-\bar z)\overline{\tilde{v}}=g\,,
$$
where the right-hand side is estimated by $\|g\|_{L^{2}}\leq C
h^{1/2}|\Imag z|\leq C h^{N_{0}+1/2}$\,. The estimate for $|w'(b)|$
follows the same arguments as for $|\tilde{v}'(b)|$ with
$h^{N_{0}+1/2}$ instead of $h^{1/2}$.\\
The second estimate of \eqref{eq.estu2.1} is a direct application of
Proposition~\ref{pr.Direxp}-i) to $(\tilde{H}_{D}^{h}-z)\tilde{v}=f$
while noticing that $d_{Ag}(x,b,V,\lambda^{0})$ can take the place
of $d_{Ag}(x,\supp~f, V, \Real z)$  because $|\Real z
-\lambda^{0}|=\mathcal{O}(h)$, $\supp f\subset \supp u_{0}\subset [b-h,b]$ and 
$d_{Ag}(.,.,V,\lambda^{0})$ is uniformly Lipschitz on $[a,b]^{2}$\,.
On the interval $[a,a+h]$ the second derivative of $\tilde{v}$ satisfies
$-h^{2}\tilde{v}''=-(V-z)\tilde{v}$ so that
$\|\tilde{u}_{2}(a+.)\|_{H^{2,h}((0,h))}\leq
\frac{Ce^{-\frac{\varphi_{b}(a)}{h}}}{h^{1/2}}$ and we apply again
Lemma~\ref{lem.estBord}.\\

The difference $v=u_{2}-\tilde{u}_{2}$ solves the Dirichlet problem
\[
\left\{\begin{array}[c]{l}(P^h-z)v=W^{h}\tilde{u}_{2}\\
v(a)=v(b)=0\,,
\end{array}
\right.
\]
which means
$v=(H_{D}^{h}-z)^{-1}(W^{h}\tilde{u}_{2})$.
It suffices to apply Lemma~\ref{le.appDir} with
$\|W^{h}\tilde{u}_{2}\|_{H^{-1}}\leq  C_{a,b}\|W^{h}\tilde{u}_{2}\|_{\mathcal{M}_{b}}\leq
C_{a,b,c}e^{-\frac{S_{0}}{h}}$\,.
With $z\in G_{h}(\lambda^{0})$, this gives
$\|v\|_{H^{1,h}}\leq
\frac{C'_{a,b,c}e^{-\frac{S_{0}}{h}}}{h^{N_{0}+1}}$\,.
In $[a,a+h]$ or in $[b-h,b]$, the equation for $v$ is simply 
$h^{2}v''=(V-z)v$ and we use again Lemma~\ref{lem.estBord} to conclude\,.
\qed

\noindent\textbf{Proof of Proposition~\ref{pr.estupsi}:}\\
\noindent\textbf{a)} First consider the estimate of $(Bq-A)^{-1}$ in 
\eqref{eq.estimBqA} and
its variation in \eqref{eq.diffBA}. The explicit form of the matrix
$\left(Bq(z,\theta_0,V-W^h)-A\right)$, using the definitions of $A,B$ and
$q(z,\theta_0,V)$ given respectively in \eqref{Matrix_def} and \eqref{q_zeta}, can be written
\[
\left(Bq(z,\theta_0,V-W^{h})-A\right) = M(z) + R(z)\,,
\]
with
\[
M(z)=\frac{1}{h^{2}}
\begin{pmatrix}
e^{-\frac{3\theta_0}{2}} & -u_2'(b) &  & \\
\frac{ihe^{\theta_0}}{\sqrt{ze^{2\theta_0}}}& e^{\frac{\theta_0}{2}}  &  & \\
&  & -e^{\frac{\theta_0}{2}} & \frac{ihe^{\theta_0}}{\sqrt{ze^{2\theta_0}}}\\
&  & u_3'(a) & -e^{-\frac{3\theta_0}{2}}
\end{pmatrix}, \quad\quad\quad
R(z)=\frac{1}{h^{2}}
\begin{pmatrix}
& &  u_3'(b) & 0\\
&&0&0\\
0&0&&\\
0&-u_2'(a)&& 
\end{pmatrix}\,,
\]
where $u_{2,3}$ has to be replaced by $\tilde{u}_{2,3}$ when
$W^{h}\equiv 0$\,.
According to \eqref{eq.estu2.2}\eqref{eq.estu3.2} and \eqref{eq.comutu}
\begin{equation}\label{eq.estR}
\Vert R(z) \Vert \leq C_{a,b,c}e^{-\frac{S_{0}}{2h}}\,.
\end{equation}
The inverse matrix $\left(Bq(z,\theta_0,V-W^{h})-A\right)^{-1}$ is formally given by:
\begin{equation}\label{eq.Bq-Am1}
\left(Bq(z,\theta_0,V-W^{h})-A\right)^{-1}=M(z)^{-1}\left(1+R(z)M(z)^{-1}\right)^{-1}\,.
\end{equation}
An explicit computation gives
\begin{eqnarray*}
&&
M(z)^{-1}=h^{2}
\begin{pmatrix}
\frac{1}{\Delta^+}
\begin{pmatrix}e^{\frac{\theta_0}{2}}& u_2'(b)\\
-\frac{ihe^{\theta_0}}{\sqrt{ze^{2\theta_0}}}& e^{-\frac{3\theta_0}{2}}\end{pmatrix} &  \\
& \frac{1}{\Delta^-}
\begin{pmatrix}-e^{-\frac{3\theta_0}{2}}  & -\frac{ihe^{\theta_0}}{\sqrt{ze^{2\theta_0}}}\\
-u_3'(a) & -e^{\frac{\theta_0}{2}}\end{pmatrix}
\end{pmatrix}\,,\\
\text{with}
&&
\left\{\begin{array}{l}
\Delta^+=e^{-\theta_0}+\frac{ihe^{\theta_0}}{\sqrt{ze^{2\theta_0}}}u_2'(b)\,,
\\\Delta^-=e^{-\theta_0}-\frac{ihe^{\theta_0}}{\sqrt{ze^{2\theta_0}}}u_3'(a)\,. \end{array}\right.
\end{eqnarray*}
Using $|\Imag u_2'(b)|+ |\Imag u_3'(a)| \leq
C_{a,b,c}h^{N_{0}-1}$ due to
\eqref{eq.estu2.2}\eqref{eq.estu3.2} and \eqref{eq.comutu},
leads to the lower bounds for $\Delta^+$ and $\Delta^-$. With our choice of the branch cut: $\frac{e^{\theta_0}}{\sqrt{ze^{2\theta_0}}}=\pm\frac{1}{\sqrt{z}}$, depending on $\arg z$, and one has
\begin{multline*}
\vert\Delta^+\vert\geq \Real \Delta^+ \geq \cos\left(\Imag\theta_0\right) - \frac{h}{|z|}\left\vert\Imag u_2'(b)\Real \sqrt{z} - \Real u_2'(b)\Imag \sqrt{z}\right\vert\\\geq \cos\left(\Imag\theta_0\right)-C\frac{h}{|z|}\left(h^{N_0-1}\sqrt{|z|}+\frac{1}{h}|\Imag \sqrt{z}|\right)\,.
\end{multline*}
For $z\in G_h$, we have $|\Imag \sqrt{z}|\leq Ch^{N_0}$ and one finally gets:
$\vert\Delta^+\vert \geq \frac{1}{2}$. The same lower bound holds for
$\Delta^-$.
 Thus $M(z)$ is invertible with:
 $M(z)^{-1}=\mathcal{O}\left(h\right)$
 in any fixed matrix norm. 
From relations \eqref{eq.estR} and \eqref{eq.Bq-Am1}, it follows:
\begin{equation}\label{est.Bq-Am1}
\left(Bq(z,\theta_0,V-W^{h})-A\right)^{-1} =
M(z)^{-1}\left(1+\mathcal{O}\left(\frac{e^{-\frac{3S_{0}}{4h}}}{h}\right)\right)=
\mathcal{O}\left(h\right)\,,
\end{equation}
which is a rewriting of \eqref{eq.estimBqA}. The estimate
\eqref{eq.diffBA} of the
difference is due to the exponentially small size of
$|u_{2,3}'(a,b)-\tilde{u}_{2,3}'(a,b)|$ 
stated in \eqref{eq.comutu}.\\
\noindent\textbf{b)} We shall now consider the estimates for
$\gamma_{\mathcal{V}}(e_{i})$ with $i=2,3$, $\mathcal{V}=V$ or
$\mathcal{V}=V-W^{h}$\,. Actually it suffices to remember the
 equations \eqref{u_i,z (i=1,4)}, \eqref{u_intz} and the definition of the coefficients  \eqref{eq.c_i}
$$
\gamma_{V-W^{h}}(\underline{e}_{i}, z,\theta_{0})= c_{i}u_{i}=
\pm \frac{1}{h^{2}}u_{i}\,,\quad i=2,3\,,
$$
where the functions $u_{2,3}$ do not depend on $\theta_{0}$ and are
given by \eqref{u_intz}. It suffices to use
\eqref{eq.estu2.1} for \eqref{eq.estimg23} and \eqref{eq.comutu} for
\eqref{eq.diffg23}. Changing $\theta_{0}$ into $\overline{\theta_{0}}$
has no effect and $\bar z\in G_{h}(\lambda^{0})$ when $z\in G_{h}(\lambda^{0})$\,.\\
\noindent\textbf{c)} The functions
$\gamma_{\mathcal{V}}(\underline{e}_{1,4},z,\theta_{0})$ do not depend
on the potential $\mathcal{V}$:
\[
\gamma(\underline{e}_{1},z,\theta_{0})=\frac{ie^{\frac{3\theta_0}{2}}}{h\sqrt{ze^{2\theta_{0}}}}
1_{\left(  b,+\infty\right)  }e^{i\frac{\sqrt{ze^{2\theta_0}}(x-b)}{h}}\,,
\quad
\gamma(\underline{e}_{4},z,\theta_{0})=\frac{ie^{\frac{3\theta_0}{2}}}{h\sqrt{ze^{2\theta_{0}}}}1_{\left(  -\infty,a\right)  }e^{-i\frac{\sqrt{ze^{2\theta_0}}(x-a)}{h}}\,,
\] 
from which \eqref{eq.diffg14} follows while \eqref{eq.estimg14} comes from
\[
\|\gamma(\underline{e}_{i},z,\theta_{0})\|_{H^{1,h}(\R\setminus[a,b])}^2\leq
\frac{C}{h\Imag(\sqrt{z e^{2\theta_0}})}
\leq \frac{C_{a,b,c}}{h^{N_{0}+1}}\,,\quad \text{when}\; z\in G_{h}(\lambda^{0})\,.
\]
\qed

\subsection{Resolvent estimates}\label{se.resestimHthe}

We gather the information given by the Krein formula
\eqref{res_formula_gen_rap} and the control of the finite rank part
$\Upsilon^{h}(z, \mathcal{V})$
given by Proposition~\ref{pr.estupsi}.
\begin{proposition}
\label{prop.resEstL2H1}
Assume the hypotheses \eqref{eq.hyp0}
\eqref{eq.hyp1}\eqref{eq.hyp2}\eqref{eq.hyp3}. Take
$\theta_{0}=ih^{N_{0}}$, $N_{0}>1$, and let $G_{h}(\lambda^{0})$ be the set defined
by 
\eqref{eq.Gh}. \\
\noindent\textbf{a)} For $\mathcal{V}=V$ or $\mathcal{V}=V-W^{h}$, the resolvent
$(H_{\theta_{0},\mathcal{V}}(\theta_{0})-z)^{-1}$ is estimated by
\begin{eqnarray}
\label{eq.estresL2H1}
&&
\forall z\in G_{h}(\lambda^{0})\,,\quad
  \|(H_{\theta_{0},\mathcal{V}}^{h}(\theta_{0})-z)^{-1}\|_{\mathcal{L}(L^{2}(\R)
    ; H^{1}(\R\setminus\left\{a,b\right\}))}\leq
  \frac{C_{a,b,c}}{h^{N_{0}+2}}\,,
\\
&&
\label{eq.estresH1H1}
 \forall z\in G_{h}(\lambda^{0})\,,\quad 
\|(H_{\theta_{0},V-W^{h}}^{h}(\theta_{0})-z)^{-1}\chi
 \|_{\mathcal{L}(H^{-1}((a,b)) ; H^{1}(\R\setminus\left\{a,b\right\})}\leq
  \frac{C_{a,b,c,\chi}}{h^{N_{0}+3}}\,,
\end{eqnarray}
 for any fixed 
$\chi\in \mathcal{C}^{\infty}_{0}((a,b))$\,.\\
\noindent\textbf{b)} The difference of the resolvents equals
\begin{multline}
  \label{eq.diffres}
(H_{\theta_{0},V-W^{h}}^{h}(\theta_{0})-z)^{-1}
-
(H_{\theta_{0},V}^{h}(\theta_{0})-z)^{-1}
=
(H_{ND,V-W^{h}}^{h}(\theta_{0})-z)^{-1}
\\
-
(H_{ND,V}^{h}(\theta_{0})-z)^{-1}
+ R\Upsilon^{h}(z)\,,
\end{multline}
with
\begin{eqnarray}
  \label{eq.estimresteL2}
&&\forall z\in G_{h}(\lambda^{0})\,,\quad
  \|R\Upsilon^{h}(z)\|_{\mathcal{L}(L^{2}(\R)
    ; H^{1}(\R\setminus\left\{a,b\right\}))}\leq
  C_{a,b,c}e^{-\frac{S_{0}}{4h}}\,,
\\
\label{eq.estimresteH1}
&&\forall z\in G_{h}(\lambda^{0})\,,\quad
  \|R\Upsilon^{h}(z)\chi\|_{\mathcal{L}(H^{-1}((a,b))
    ; H^{1}(\R\setminus\left\{a,b\right\}))}\leq
  C_{a,b,c}e^{-\frac{S_{0}}{4h}}\,,
\end{eqnarray}
 for any fixed 
$\chi\in \mathcal{C}^{\infty}_{0}((a,b))$\,.
\end{proposition}
\noindent\textbf{Proof:} \textbf{a)} The formula
\eqref{res_formula_gen_rap} says
for $\mathcal{V}=V$ or $\mathcal{V}=V-W^{h}$:
\begin{eqnarray*}
\left(  H_{\theta_{0},\mathcal{V}}^{h}(\theta_{0})-z\right)  ^{-1}  
&=&\left(  H_{ND,\mathcal{V}}
^{h}(\theta_{0})-z\right)  ^{-1}\\
&&\qquad
-\sum_{i,j,k=1}^{4}\left(  Bq(z,\theta_0,\mathcal{V})-A\right)  _{ij}^{-1}B_{jk}\left\langle
\gamma(\underline{e}_{k},\bar{z},\bar{\theta_{0}}),\cdot\right\rangle
_{L^{2}(\mathbb{R})}\gamma(\underline{e}_{i},z,\theta_{0})
\\
&=&\left(  H_{ND,\mathcal{V}}
^{h}(\theta_{0})-z\right)  ^{-1} - \Upsilon^{h}(z,\mathcal{V})\,.
\end{eqnarray*}

By Proposition~\ref{pr.estupsi}, actually by
\eqref{eq.estimBqA}\eqref{eq.estimg14} and \eqref{eq.estimg23}, 
the term $\Upsilon^{h}(z,\mathcal{V})$ satisfies with $N_{0}>1$
\begin{eqnarray*}
  &&
\|\Upsilon^{h}(z,\mathcal{V})\|_{\mathcal{L}(L^{2}(\R)
    ; H^{1}(\R\setminus\left\{a,b\right\}))}
\leq \frac{C_{a,b,c}h}{\min\left\{h^{N_{0}+1}, h^{N_{0}/2+2}, h^{3}\right\}}\times\frac{1}{h}\leq
\frac{C_{a,b,c}}{h^{N_{0}+2}}\,,
\\
&&
\|\Upsilon^{h}(z,\mathcal{V})\chi\|_{\mathcal{L}(H^{-1}((a,b))
    ; H^{1}(\R\setminus\left\{a,b\right\}))}
\leq \frac{C_{a,b,c,\chi}h}{\min\left\{h^{N_{0}+1}, h^{N_{0}/2+2}, h^{3}\right\}}\times\frac{1}{h^{2}}\leq
\frac{C_{a,b,c,\chi}}{h^{N_{0}+3}}\,,
\end{eqnarray*}
for any fixed $\chi\in \mathcal{C}^{\infty}_{0}((a,b))$\,.

It remains to estimate the first term. The worst case is for
$\mathcal{V}=V-W^{h}$ since $G_{h}(\lambda^{0})$ lies around elements of
$\sigma(H_{D}^{h}=-h^{2}\Delta_{D}+ V-W^{h})$:
\[
\left(  H_{ND,V-W^h}^{h}(\theta_{0})-z\right)  ^{-1} =
e^{2\theta_0}\left(-h^{2}\Delta_{\R\setminus[a,b]}^{N}-ze^{2\theta_0}\right)
^{-1} \oplus\left(H^{h}_{D}-z\right)^{-1}\,.
\]
According to Lemma~\ref{le.appDir} with $\|f\|_{H^{-1,h}((a,b))}\leq
\frac{1}{h}\|f\|_{H^{-1}}$ and $\|u\|_{H^{1}((a,b))}\leq
\frac{1}{h}\|u\|_{H^{1,h}((a,b))}$, the inequality
\[
\left\Vert\left(H_{D}^{h}-z\right)^{-1}
\right\Vert_{\mathcal{L}(H^{-1}((a,b)), H^{1}_{0}((a,b)))} \leq
\frac{C_{a,b,c}}{h^{2}}\left(\frac{1}{d(z,\sigma(H_{D}^{h}))}+1\right)\leq \frac{C_{a,b,c}}{h^{N_{0}+2}}\,,
\]
holds for $z\in G_{h}(\lambda^{0})$\,.\\
The resolvent of the Neumann Laplacian 
$\left(-h^{2}\Delta_{\R\setminus[a,b]}^{N}-\zeta\right)^{-1}$
can be written
\begin{eqnarray*}
&&
\left(-h^{2}\Delta_{\R\setminus[a,b]}^{N}-\zeta\right)^{-1}
=
\left(-h^{2}\Delta_{\R\setminus[a,b]}^{N}+1\right)^{-1}
\left[1+ 
(1+\zeta)\left(-h^{2}\Delta_{\R\setminus[a,b]}^{N}-\zeta\right)^{-1}\right]\,,
\\
\text{with}
&&
\|\left(-h^{2}\Delta_{\R\setminus[a,b]}^{N}+1\right)^{-1}\|_{\mathcal{L}(L^{2}(\R\setminus[a,b]);
  H^{1}(\R\setminus[a,b]))}
\leq \frac{1}{h}\,,
\end{eqnarray*}
it is estimated by
$$
\|
\left(-h^{2}\Delta_{\R\setminus[a,b]}^{N}-\zeta\right)^{-1}
\|_{\mathcal{L}(L^{2}(\R\setminus[a,b]);
  H^{1}(\R\setminus[a,b]))}
\leq \frac{1}{h}\left[1+\frac{1+|\zeta|}{d(\zeta,\R_{+})}\right]\,.
$$
For $z\in G_{h}(\lambda^{0})$
 and $\zeta=ze^{2\theta_{0}}$, the distance $d(ze^{2\theta_{0}}, \R_{+})$, is bounded
from below by $Ch^{N_{0}}$. Hence we get
$$
\|\left(-h^{2}\Delta_{\R\setminus[a,b]}^{N}-ze^{2\theta_0}\right)
\|_{\mathcal{L}(L^{2}(\R\setminus[a,b]);
  H^{1}(\R\setminus[a,b]))}\leq \frac{C_{a,b,c}}{h^{N_{0}+1}}\,.
$$
Putting all together gives \eqref{eq.estresL2H1} and \eqref{eq.estresH1H1}\,.\\
\noindent\textbf{b)} For the difference of resolvents, it suffices to
notice that
$$
R\Upsilon^{h}(z)=\Upsilon^{h}(z,V)-\Upsilon^{h}(z,V-W^{h})\,.
$$
Hence it is the difference of trilinear quantities of which every
factor is estimated by $\frac{1}{h^{3N_{0}}}$ with variations bounded
by $\frac{C_{a,b,c}}{h}e^{-\frac{S_{0}}{2h}}$. This ends the proof.
\qed

\section{Adiabatic evolution}\label{se.adiabev}

We consider now a time-dependent real valued potential $V(t)-W^{h}(t)
= V(t)-W^{h}_{1}(t)-W^{h}_{2}(t)$
supported in $[a,b]$ with
\begin{equation}
  \label{eq.wt}
W^{h}_{1}(x,t)= \sum_{j_{1}=1}^{M_{1}}w_{j_{1}}(\frac{x-x_{j_{1},1}}{h},t)\,,\quad
W^{h}_{2}(t)=
\sum_{j_{2}=1}^{M_{2}}\alpha_{j_{2}}(t)h\delta(x-x_{j_{2},2})\,,\quad
\alpha_{j_{2}}(t)>0\,,
\end{equation}
where the $x_{j}$'s  are fixed (independent of $h$) distinct points of
$(a,b)$ and  the supports 
$\supp w_{j_{1}}$ are contained in a fixed compact set.
 The functions $V(.,t)$, $w_{j_{1}}(.,t)$, $\alpha_{j_{2}}(t)$
are possibly $h$-dependent $\mathcal{C}^{K}$ functions with a uniform
control of the derivatives and which impose a uniform control of
\eqref{eq.hyp0}\eqref{eq.hyp1}\eqref{eq.hyp2}\eqref{eq.hyp3}.
Namely  we assume that for some $\lambda^{0}(t)$ the estimates
\begin{eqnarray}
  \label{eq.hyp0t}
&&
\max_{
  \tiny{\begin{array}[t]{c}
   t\in [0,T]\\
 0\leq k \leq K\\
0\leq j_{1}\leq M_{1}\\
0\leq j_{1}\leq M_{2}
\end{array}}
}
\|\partial_{t}^{k}V(t)\|_{L^{\infty}}+
\|\partial_{t}^{k}w_{j_{1}}(t)\|_{L^{\infty}} +
|\partial_{t}^{k}\alpha_{j_{2}}(t)|
+|\partial_{t}^{k}\lambda^{0}(t)|
\leq \frac{1}{c}\,,
\\
\label{eq.hyp2t}
&&
\forall x\in[a,b]\,,\quad
c\leq \lambda^{0}(t)\leq V(x,t)-c\,,
\end{eqnarray}
hold for all $t\in [0,T]$\,.
Actually, the regularity of 
$\lambda^{0}(t)$ can be deduced from the other assumptions 
possibly by replacing it initial guess by the mean energy value
$\frac{1}{\ell}\Tr\left[H_{D}^{h}(t) \Pi_{0}(t)\right]$
with
$\Pi_{0}(t)=\frac{1}{2i\pi}\int_{|z-\lambda_{0}(t)|=\frac{2h}{c}}(z-H_{D}^{h}(t))^{-1}~dz$\,.

 This $\lambda^{0}(t)$ is moreover assumed
to be the center of a cluster of eigenvalues of the Dirichlet
Hamiltonian $H_{D}^{h}(t)=-h^{2}\Delta_{D}+V(t)-W^{h}(t)$ on $(a,b)$:
There exist
$\lambda_{1}^{h}(t),\ldots,\lambda_{\ell}^{h}(t)\in
\sigma(H_{D}^{h}(t))$ 
such that
\begin{eqnarray}
\label{eq.hyp1t}
&&
d(\lambda^0(t),\sigma(H_{D}^{h}(t))\setminus
\left\{\lambda_{1}^{h}(t),\ldots,\lambda_{\ell}^{h}(t)\right\})\geq
c\,,
\\
\label{eq.hyp3t}
&& \max_{1\leq j\leq\ell} |\lambda_{j}^{h}(t)-\lambda^0(t)|\leq \frac{h}{c}\,.
\end{eqnarray}
The operator $H^{h}_{\theta_{0}, V(t)-W^{h}(t)}(\theta_{0})$ is
studied here with
$$
\theta_{0}=ih^{N_{0}},\quad N_{0}>1\,.
$$
According to Proposition~\ref{pr.summ} and Definition~\ref{de.Gh}, 
the complex domain $G_{h}(\lambda_{0}(t))$ surrounds $\ell$
eigenvalues $z_{1}^{h}(t),\ldots, z_{\ell}^{h}(t)$ of
$H^{h}_{\theta_{0}, V(t)-W^{h}(t)}(\theta_{0})$ 
and its distance to the spectrum remains uniformly bounded from below
$$
\min_{t\in [0,T]} d(G_{h}(\lambda_{0}(t)), \sigma(H_{\theta_{0},
  V(t)-W^{h}(t)}(\theta_{0})))\geq
\frac{1}{C_{a,b,c}}h^{N_{0}}\,.
$$
The spectral projection associated with the cluster of eigenvalues 
$\left\{z_{1}^{h},\ldots, z_{\ell}^{h}\right\}$ is given by
\begin{equation}
  \label{eq.defproj}
P_{0}(t)=\frac{1}{2i\pi}\int_{\Gamma^{h}(t)}(z-H_{\theta_{0},
  V(t)-W^{h}(t)}(\theta_{0}))^{-1}~dz\,,
\end{equation}
where $\Gamma^{h}(t)$ is a contour contained in
$G_{h}(\lambda^{0}(t))$\,.\\
When $K\geq 1$, the parallel transport $\Phi_{0}(t,s)$, $t,s\in
[0,T]$,  associated with
$(P_{0}(t))_{t\in[0,T]}$ is given by
\begin{equation}
  \label{eq.PhiP0}
\left\{
  \begin{array}[c]{l}
    \partial_{t}\Phi_{0}
+    \left[P_{0},\partial_{t}P_{0}\right]\Phi_{0}=0\\
\Phi_{0}(t=s,s)=\Id\,,
  \end{array}
\right.
\end{equation}
is well  defined  and satisfies
$$
\forall s,t\in[0,T]\,,\quad
P_{0}(t)\Phi_{0}(t,s)=\Phi_{0}(t,s)P_{0}(s)\,.
$$
The time-scale is given by the parameter
$$
\varepsilon= e^{-\frac{\tau}{h}}\,,\quad \text{with}\quad
\tau>0~\text{fixed}.
$$
When the assumed regularity is large enough, $K\geq 2$,
Proposition~\ref{pr.Kato}-d).
the Cauchy problem
\begin{equation}
  \label{eq.Cauchyepsi}
  \left\{
    \begin{array}[c]{l}
      i\varepsilon \partial_{t}u= H^{h}_{\theta_{0},
        V(t)-W^{h}(t)}(\theta_{0})u\,,\quad t\geq s\,,\\
      u(t=s)= u_{s}
    \end{array}
\right.
\end{equation}
defines a dynamical system $U^{\varepsilon}(t,s)$, $0\leq s\leq t\leq T$,
of contractions on $L^{2}(\R)$\,.
\begin{theorem}
  \label{th.adiabappl} 
Assume
\eqref{eq.hyp0t}\eqref{eq.hyp1t}\eqref{eq.hyp2t}\eqref{eq.hyp3t}  
with $K\geq 2$ and  take $\theta_{0}=ih^{N_{0}}$, $N_{0}>1$\,,
$\varepsilon=e^{-\frac{\tau}{h}}$, $\tau>0$\,.
Let $P_{0}(t)$ be the spectral
projection \eqref{eq.defproj}, 
let $r$ belong to $\mathcal{C}^{0}([0,T]; L^{2}(\R))$ and let $r_{s}$
belong to $L^{2}(\R)$. For $s\in [0,T]$ take an initial data $u_{s}\in
L^{2}(\R)$ such that $P_{0}(s)u_{s}=u_{s}$.\\
Then the solutions $u^{h}$ and $v^{h}$ to the
Cauchy problems
\begin{equation}
  \label{eq.Cauchyurr}
  \left\{
    \begin{array}[c]{l}
      i\varepsilon \partial_{t}u^{h}= H^{h}_{\theta_{0},
        V(t)-W^{h}(t)}(\theta_{0})u^{h}+r(t)\,,\quad t\geq s\,,\\
      u^{h}(t=s)= u_{s}+r_{s}
    \end{array}
\right.
\end{equation}
and
\label{eq.Cauchyurrapprox}
$$
\left\{
  \begin{array}[c]{l}
    i\varepsilon\partial_{t}v^{h}=\Phi_{0}(s,t)
    P_{0}(t)(H^{h}_{\theta_{0}, V(t)-W^{h}(t)}(\theta_{0}))P_{0}(t)\Phi_{0}(t,s)
    v^{h} \,,\quad t\geq s\,\\  
v^{h}(t=s)=u_{s}\,,
  \end{array}
\right.
$$
satisfy
\begin{equation}
  \label{eq.estadiab}
\max_{t\in[s,T]}\|u^{h}(t)-\Phi_{0}(t,s)v^{h}(t)\|\leq
C_{a,b,c,\tau,T,\delta}
\left[\varepsilon^{1-\delta}\|u_{s}\|+
\|r_{s}\|+\frac{1}{\varepsilon}\max_{t\in[s,T]}\|r(t)\|\right]\,.
\end{equation}
\end{theorem}
\begin{remark}
  The estimate with the source term $r(t)$ can be improved if
  $P(t)r(t)=r(t)$ after reconsidering the proof of
  Corollary~\ref{co.nenciuadapt0} in the Appendix (possibly with a
  higher order starting approximation with $K\leq 2$).
Nevertheless the accuracy of the result may depend on the 
assumptions for $r(t)$. We prefer to postpone this kind of improvement
to a subsequent work when hypotheses for the source term
are naturally introduced.
\end{remark}
\noindent\textbf{Proof:}
\textbf{a)} When $u^{h}_{00}(t)$ denotes  the solution
to \eqref{eq.Cauchyepsi} associated with $r_{s}=0$ and $r\equiv 0$,
the contraction property of $U^{\varepsilon}(t,s)$ implies
$$
\max_{t\in [s,T]}\|u(t)-u_{00}(t)\|\leq
\|r_{s}\|+\frac{1}{\varepsilon}\max_{t\in [s,T]}\|r(t)\|\,.
$$
Hence, we can forget the remainder terms and simply prove the estimate
\eqref{eq.estadiab} when $r_{s}=0$ and $r\equiv 0$\,.\\
\noindent\textbf{b)}
We consider the operator
$A^{\varepsilon}(t)=\frac{1}{i}(H_{\theta_{0},V(t)-W^{h}(t)}(\theta_{0})-\lambda^{0}(t))$
and we notice that the domain
$$
\frac{1}{i}(G_{h}(\lambda^{0}(t))-\lambda^{0}(t))=\left\{\frac{1}{i}(z-\lambda^{0}(t))\,,\quad
z\in G_{h}(\lambda_{0}(t))\right\}
$$
contains the contour
$$
\Gamma_{\varepsilon}= \frac{1}{i}\left\{\Gamma^{h}(t)-\lambda^{0}(t)\right\}\,,
$$
which can be chosen independent of $t\in [0,T]$. 
Then  the projection
$P_{0}(t)$ is nothing but
$$
P_{0}(t)=\frac{1}{2i\pi}\int_{\Gamma_{\varepsilon}}(z-A^{\varepsilon}(t))^{-1}~dt\,.
$$
 Hence it suffices to
verify the estimates of $\partial_{t}^{k}(z-A^{\varepsilon}(t))^{-1}$
for $k\leq K+1$ and $t\in [0,T]$ in order to apply
Theorem~\ref{th.nenciu} 
and additionally the uniform boundedness of $\|P_{0}(t)\|$ and
$\|\partial_{t}P_{0}(t)\|$ in order to use its
Corollary~\ref{co.nenciuadapt0}.\\
Like in Appendix~\ref{se.adiabvar},
 we use the notation $g(\varepsilon)=\widetilde{{\cal O}}(\varepsilon^{N})$ 
 in order to summarize
$$
\forall \delta >0, \exists C_{g, \delta}>0,\quad
|g(\varepsilon)|\leq C_{g,\delta}\varepsilon^{N-\delta}\,.
$$
For $z\in \Gamma_{\varepsilon}$, the $k$-th derivative of
$(z-A^{\varepsilon}(t))^{-1}$ has the form
\begin{multline*}
\partial_{t}^{k}(z-A^{\varepsilon}(t))^{-1}
=
\sum_{\tiny
  \begin{array}[t]{c}
j_{1}+\ldots j_{m}=k
\\
j_{i}\geq 1
\end{array}
}c_{j_{1},\ldots, j_{m}}
(z-A^{\varepsilon})^{-1}[\partial_{t}^{j_{1}}(-iV+iW^{h})]
(z-A^{\varepsilon})^{-1}\ldots
\\
\ldots
[\partial_{t}^{j_{m}}(-iV+iW^{h})]
(z-A^{\varepsilon})^{-1}\,,
\end{multline*}
where the numbers $c_{j_{1},\ldots,j_{\ell}}$ are universal coefficients.\\
Remember
$$
\| \partial_{t}^{j}V(t)\|_{\mathcal{L}(L^{2}(\R))}\leq
\frac{1}{c}\,,
$$
the support condition 
$$
\partial_{t}^{j}W(t)=\chi(x) \left[\partial_{t}^{j}W(t)\right] \chi(x)
$$
with $\chi\in \mathcal{C}^{\infty}_{0}((a,b))$, which entails
$$
\|\partial_{t}^{j}W(t)\|_{\mathcal{L}(H^{1}_{0}((a,b));  H^{-1}((a,b)))}
\leq \frac{1}{c}\,.
$$
Hence the resolvent estimates of Proposition~\ref{prop.resEstL2H1}
imply
\begin{equation}
  \label{eq.estimderres}
\max_{z\in \Gamma_{\varepsilon}, k\leq K+1, t\in[0,T]}
\|\partial_{t}^{k}(z-A^{\varepsilon}(t))^{-1}\|\leq \frac{C_{a,b,c}}{h^{k(N_{0}+3)}}=\widetilde{\mathcal{O}}(\varepsilon^{0})\,.
\end{equation}
Meanwhile the length $|\Gamma_{\varepsilon}|$ is bounded by
$\mathcal{O}(1)$ and therefore the conclusions of
Theorem~\ref{th.nenciu} are valid.\\
Now comes the final points, which are the uniform boundedness of
$\|P_{0}(t)\|$ and $\|\partial_{t}P_{0}(t)\|$, in order to refer to the
more accurate version of Corollary~\ref{co.nenciuadapt0}.\\
\noindent\textbf{c)} For $P_{0}(t)$, we write
$$
P_{0}(t)=\frac{1}{2i\pi}\int_{\Gamma_{\varepsilon}}(z-A^{\varepsilon}(t))^{-1}~dz
=\frac{1}{2i\pi}\int_{\Gamma^{h}(t)}(z-H_{\theta_{0},V(t)-W^{h}(t)}(\theta_{0}))^{-1}~dz\,,
$$
and we use the formula \eqref{eq.diffres} in the form
\begin{multline}
\label{eq.diffres1}
  (H_{\theta_{0},V-W^{h}}^{h}(\theta_{0})-z)^{-1}
-
(H_{D}^{h}-z)^{-1}
=
e^{2\theta_{0}}(-h^{2}\Delta_{\R\setminus[a,b]}^{N}-ze^{2\theta_{0}})^{-1}
\\
+
(H_{\theta_{0},V}^{h}(\theta_{0})-z)^{-1}
-
(H_{ND,V}^{h}(\theta_{0})-z)^{-1}
+ R\Upsilon^{h}(z)\,.
\end{multline}
The right-hand side is the sum of three holomorphic terms in the interior
of $\Gamma^{h}(t)$ and of an exponentially small term  according to
\eqref{eq.estimresteL2}. We obtain
$$
P_{0}(t)=\frac{1}{2i\pi}\int_{\Gamma^{h}(t)}(z-H_{D}^{h})^{-1}~dz 
+ \mathcal{O}(e^{-\frac{S_{0}}{4h}})= \Pi_{0}(t)+ \mathcal{O}(e^{-\frac{S_{0}}{4h}})\,,
$$
where $\Pi_{0}(t)$ is the orthogonal spectral projector associated
with 
$\left\{\lambda_{1}^{h}(t),\ldots,
  \lambda_{\ell}^{h}(t)\right\}\subset \sigma(H_{D}^{h}(t))$ with norm
$\|\Pi_{0}(t)\|\leq 1$\,.\\
\noindent\textbf{d)} For $\partial_{t}P_{0}(t)$, we use
$$
\partial_{t}(z-H_{\theta_{0},V-W^{h}}(\theta_{0}))^{-1}=
(z-H_{\theta_{0},V-W^{h}}(\theta_{0}))^{-1}
(\partial_{t}V-\partial_{t}W^{h})(z-H_{\theta_{0},V-W^{h}}(\theta_{0}))^{-1}\,.
$$
From \eqref{eq.diffres1}, we get
\begin{eqnarray*}
&&(H_{\theta_{0},V-W^{h}}(\theta_0)-z)^{-1}1_{[a,b]}-
(H_{D}^{h}-z)^{-1}1_{[a,b]}
\\
&&\qquad
=
0+ \Upsilon^{h}(z,V)1_{[a,b]}+ R\Upsilon^{h}(z)1_{[a,b]}
\\
&&\qquad
=
\sum_{i,j=1}^{4}\sum_{k=2}^{3}\left(  Bq(z,\theta_0,V)-A\right)  _{ij}^{-1}B_{jk}\left\langle
\gamma(\underline{e}_{k},\bar{z},\bar{\theta_{0}}),\cdot\right\rangle
_{L^{2}(\mathbb{R})}\gamma(\underline{e}_{i},z,\theta_{0})
+
 R\Upsilon^{h}(z)1_{[a,b]}\,,
\end{eqnarray*}
where the first term of the right-hand side is holomophic inside
$\Gamma^{h}(t)$ and the last term is exponentially small according to
\eqref{eq.estimresteL2}
and \eqref{eq.estimresteH1}. A symmetric writing holds for
$1_{[a,b]}(z-H_{\theta_{0},V-W^{h}}(\theta_{0}))^{-1}$.  Hence the derivative
$\partial_{t}P_{0}(t)$ is the sum of several terms:
\begin{eqnarray}
\nonumber
&&
\frac{1}{2i\pi}\int_{\Gamma^{h}(t)}(z-H_{D}^{h})^{-1}
(\partial_{t}V-\partial_{t}W^{h})(z-H_{D}^{h})^{-1}~dz=\partial_{t}\Pi_{0}(t)
\\
\label{eq.derP01}
&&\hspace{2cm}
=
\frac{1}{2i\pi}\int_{|z-\lambda^{0}(t)|=c/2}(z-H_{D}^{h})^{-1}
(\partial_{t}V-\partial_{t}W^{h})(z-H_{D}^{h})^{-1}~dz\,,
\\
\nonumber
&&
-\frac{1}{2i\pi}\int_{\Gamma^{h}(t)}
\Upsilon^{h}(z,V)1_{[a,b]}
(\partial_{t}V-\partial_{t}W^{h})
(z-H_{D}^{h})~dz
\\
\label{eq.derP02}
&&\hspace{2cm}
=
-\sum_{j'=1}^{\ell}\Upsilon^{h}(\lambda_{j'}^{h},V)(\partial_{t}V-\partial_{t}W^{h})
|\Phi_{j'}^{h}\rangle\langle \Phi_{j'}^{h}|\,,
\\
\label{eq.derP03}
&&
-\frac{1}{2i\pi}\int_{\Gamma^{h}(t)}
(z-H_{\theta_{0},V-W^{h}}(\theta_{0}))^{-1}
(\partial_{t}V-\partial_{t}W^{h})
R\Upsilon^{h}(z)~dz\,,
\end{eqnarray}
plus another term symmetric to \eqref{eq.derP02}.\\
The first one \eqref{eq.derP01} is uniformly bounded because 
\begin{itemize}
\item $\|(z-H_{D}^{h})^{-1}\|_{\mathcal{L}(L^{2})}$ is uniformly
  bounded when $|z-\lambda^{0}(t)|=\frac{c}{2}$ according to
  Hypothesis~\eqref{eq.hyp1t}\eqref{eq.hyp3t} with $\|\partial_{t}V\|_{L^{\infty}}+
  \|\partial_{t}W_{1}^{h}\|_{L^{\infty}}\leq \frac{1}{c}$\,,
\item $\|(z-H_{D}^{h})^{-1}\|_{\mathcal{L}(L^{2}, H^{1,h}_{0})}$ is
  uniformly bounded when $|z-\lambda^{0}(t)|=\frac{c}{2}$ according to 
\eqref{eq.hyp1t}\eqref{eq.hyp3t} and \eqref{eq.estimH1H1Dir} in
Lemma~\ref{le.appDir} with 
$\|\partial_{t}W_{2}^{h}\|_{\mathcal{M}_b}\leq \frac{h}{c}$\,.
\end{itemize}
The last one \eqref{eq.derP03} is $\mathcal{O}(e^{-\frac{S_{0}}{8h}})$
owing to \eqref{eq.estimresteL2}\eqref{eq.estimresteH1} for
$R\Psi^{h}(z)$ and owing to \eqref{eq.estresL2H1}\eqref{eq.estresH1H1}
for $(z-H_{\theta_{0},V-W^{h}}(\theta_{0}))^{-1}$\,.\\
For the middle term \eqref{eq.derP02} and its symmetric counterpart,
 first consider for $k=2,3$ and $j'\in\left\{1,\ldots \ell\right\}$
$$
\langle \gamma(\underline{e}_{k},\bar{z},\bar{\theta_{0}})\,,\,
(\partial_{t}V-\partial_{t}W^{h})
\Phi_{j'}^{h}\rangle
=
\pm \frac{1}{h^{2}}
\langle \tilde{u}_{k}\,,\,
(\partial_{t}V-\partial_{t}W^{h})
\Phi_{j'}^{h}\rangle\,,
$$
where $\tilde{u}_{2}$ and $\tilde{u}_{3}$ are recalled in
Lemma~\ref{lem.estBordDef}. The exponential decay estimates for 
$\tilde{u}_{2,3}$ stated in \eqref{eq.estu2.1}\eqref{eq.estu3.1} and
the one for $\Phi_{j'}^{h}$ stated in Proposition~\ref{pr.Direxp}-ii)
combined with $\|\partial_{t}(V-W^{h})\|_{\mathcal{M}_{b}}\leq
\frac{1}{c}$ imply that the scalar product is smaller than
$e^{-\frac{S_{0}}{2h}}$\,. Since the other factors  of
\eqref{eq.derP02}
 are bounded by  $Ch$ or $\frac{C}{h^{N_{0}}}$,
 we conclude \eqref{eq.derP02} and its symmetric
counterpart
are smaller than $e^{-\frac{S_{0}}{4h}}$\,.
\qed

\appendix
\section{Parameter dependent elliptic estimates on the interval $[a,b]$}\label{app.A}
 We gather here elementary $h$-dependent estimates for the elliptic operator
 $-h^{2}\Delta+ \mathcal{V}$
on the interval $(a,b)$\,.

\subsection{Dirichlet problem}
\label{se.H1}

It is convenient to use the  $h$-dependent $H^k$-norms 
\begin{equation*}
\|u\|_{H^{k,h}}^{2}=\sum_{\alpha\leq k}\|(h\partial_{x})^{\alpha}u\|_{L^{2}}^{2}\,,
\end{equation*}
for $k\in \N$.
The estimates with the standard $H^{k}$, $k\in\N$, can be recovered
after
$$
\|u\|_{H^{k}}\leq \frac{1}{h^k}\|u\|_{H^{k,h}}\,.
$$
For $k=-1$, the $h$-dependent norm on
$H^{-1}((a,b))=\left[H^{1}_{0}((a,b))\right]'$ is 
$$
\|f\|_{H^{-1,h}((a,b))}=\sup_{u\in H^{1}_{0}((a,b))}\frac{|\langle
  f\,,\, u\rangle|}{\|u\|_{H^{1,h}((a,b))}}\,,
$$
with now $\|f\|_{H^{-1,h}((a,b))}\leq \frac{1}{h}\|f\|_{H^{-1}((a,b))}$\,. We will note $H^{1,h}_{0}((a,b))$ the space $H^{1}_{0}((a,b))$ equipped with the $H^{1,h}$ norm.
\begin{lemma}
\label{lem.estBord} 
 There exists a constant $C>0$ such that
$$
\forall u\in \mathcal{C}^{\infty}([0,h])\,,\quad
|u(0)|\leq \frac{C}{h^{1/2}}\|u\|_{H^{1,h}((0,h))},
\quad resp.~
|u'(0)| \leq \frac{C}{h^{3/2}}\|u\|_{H^{2,h}((0,h))}\,,
$$
and the inequality extends to $H^{1}((0,h))$ (resp. $H^{2}((0,h))$)\,.\, 
\end{lemma}

\preuve
The second estimate is simply  a consequence of the first one after
replacing $u$ with $hu'$.\\
The first estimate is simply the usual estimate
$|v(0)|\leq C\|v\|_{H^{1}((0,1))}$ applied with $v(x)=u(hx)$\,.
\qed

\begin{lemma}
\label{le.appDir}
Let $\mathcal{V}_{1}\in L^{\infty}((a,b))$ and $\mathcal{V}_{2}\in
\mathcal{M}_{b}((a,b))$   be real valued with $\supp
\mathcal{V}_{2}\subset\subset (a,b)$ and
$$
\|\mathcal{V}_{1}\|_{L^{\infty}}\leq \frac{1}{c}\quad,\quad
\|\mathcal{V}_{2}\|_{\mathcal{M}_{b}}\leq \frac{h}{c}\,.
$$
Then the Dirichlet Hamiltonian $H_{D}^{h}=-h^{2}\Delta+
\mathcal{V}_{1}+\mathcal{V}_{2}$ defined with the form domain
$H^{1}_{0}((a,b))$, satisfies the resolvent estimate
\begin{equation}
  \label{eq.estimH1H1Dir}
\forall z\not\in \sigma(H_{D}^{h})\,,\quad
\|(H_{D}^{h}-z)^{-1}\|_{\mathcal{L}(H^{-1,h}((a,b)); H^{1,h}_{0}((a,b)))}
\leq C_{a,b,c}\left[1+|z|\right]^{2}\left(1+\frac{1}{d(z,\sigma(H_{D}^{h}))}\right)\,.
\end{equation}
When $f\in L^{2}((a,b))$ and $z\not\in \sigma(H_{D}^{h})$ the traces $u'(a)$ and $u'(b)$ of $u=(H_{D}^{h}-z)^{-1}f$ are well defined with
\begin{equation}
  \label{eq.estimuaub}
|u'(a)|+|u'(b)|\leq \frac{C'_{a,b,c}\left[1+|z|\right]^{2}}{h^{5/2}}\left(1+\frac{1}{d(z,\sigma(H_{D}^{h}))}\right)\|f\|_{L^{2}}\,.
\end{equation}
\end{lemma}
\noindent\textbf{Proof:}
From the Gagliardo-Nirenberg estimate 
$|u(x)|^2\leq \frac{C_{a,b}}{h}\|hu'\|_{L^{2}}\|u\|_{L^{2}}$ with 
$\|\mathcal{V}_{2}\|_{\mathcal{M}_{b}}\leq \frac{h}{c}$, the term
$\langle u\,,\, \mathcal{V}_{2}u\rangle$ in the variational
formulation is bounded by
$$
|\langle u\,,\, \mathcal{V}_{2}u\rangle|\leq 
\frac{C_{a,b}}{c}\|hu'\|_{L^{2}}\|u\|_{L^{2}}\leq
\frac{1}{2}\|hu'\|^{2}_{L^{2}}
+
\frac{C_{a,b}^{2}}{2c^{2}}\|u\|_{L^{2}}^{2}\,.
$$
 With
$$
\forall u\in H^{1}_{0}((a,b))\,,\quad
\langle u\,, H_{D}^{h}u\rangle+C\|u\|^{2}_{L^{2}} 
\geq \frac{1}{2}\|hu'\|^{2}_{L^{2}}+ \left(C-\|\mathcal{V}_{1}\|_{L^{\infty}}-
\frac{C_{a,b}^{2}}{2c^{2}}
\right)\|u\|^{2}_{L^{2}}\,,
$$
the operator $H_{D}^{h}+C$ is bounded from below by
$-\frac{h^{2}}{2}\Delta_{D}+\frac{C}{2}$ when
$C\geq \frac{2}{c}+ \frac{C_{a,b}^{2}}{c^{2}}\geq
\|\mathcal{V}_{1}\|_{L^{\infty}}+ \frac{C_{a,b}^{2}}{c^{2}}$\,.
Lax-Milgram theorem then says
$$
\|(H_{D}^{h}+C)^{-1}f\|_{H^{1,h}_{0}((a,b))}\leq C_{a,b,c}\|f\|_{H^{-1,h}((a,b))}\,.
$$
From the iterated first
resolvent formula,
$$
(H_{D}^{h}-z)^{-1}= (H_{D}^{h}+C)^{-1}+ (C+z)(H_{D}^{h}+C)^{-2}
+ 
(C+z)^{2}(H_{D}^{h}+C)^{-1}(H_{D}^{h}-z)^{-1}(H_{D}^{h}+C)^{-1}\,,
$$
we deduce  \eqref{eq.estimH1H1Dir}.\\
It contains also the estimate
 $\|(H_{D}^{h}+C)^{-1}\|_{\mathcal{L}(L^{2}((a,b));
  H_{0}^{1,h}((a,b)))}\leq \frac{C_{a,b,c}}{h}$ and
with
$$
(H_{D}^{h}-z)^{-1}= (H_{D}^{h}+C)^{-1}+ (C+z)(H_{D}^{h}+C)^{-1}(H_{D}^{h}-z)^{-1}\,,
$$
this yields
$$
\|(H_{D}^{h}-z)^{-1}\|_{\mathcal{L}(L^{2}((a,b));
  H_{0}^{1,h}((a,b)))}\leq \frac{C_{a,b,c}\left[1+|z|\right]}{h}\left(1+\frac{1}{d(z,\sigma(H_{D}^{h}))}\right)\,.
$$
When $u=(H_{D}^{h}-z)^{-1}f$ with $f\in L^{2}((a,b))$, 
writing the equation in $[a,a+h]$ and $[b-h,b]$ in the form
$-h^{2}u''=f-(\mathcal{V}_{1}-z)u$ implies
$$
\|u\|_{H^{2,h}((a,a+h)\cup (b-h,b))}\leq\frac{C_{a,bc}[1+|z|]^{2}}{h}\left(1+\frac{1}{d(z,\sigma(H_{D}^{h}))}\right)\|f\|_{L^{2}}\,.
$$
Lemma~\ref{lem.estBord} is applied to $u(a+.)$ and $u(b-.)$ in order
to get \eqref{eq.estimuaub}.
\qed

\subsection{Agmon estimate}
\label{se.Agmestim}

The next estimate is the usual energy estimate with exponential
weights (see \cite{Agm}\cite{HeLN})\,.
 \begin{lemma}
\label{le.agmonie} Let $(\alpha,\beta)$ be  an open
   interval, $V\in L^\infty((\alpha,\beta))$, $z\in\mathbb C$ and $\varphi\in
   W^{1,\infty}((\alpha,\beta);\R)$\,. Denote 
by $P$ the Schr{\"o}dinger operator $P:=-h^2d^2/dx^2+V.$ Then for any $u_1,u_2$
in $H^1((\alpha,\beta))$ such that $u_{1}''$ is a bounded
measure in $(\alpha,\beta)$ and locally $L^{2}$ around $\alpha$ and $\beta$, the identity
\begin{eqnarray}
\nonumber
\hspace{-.8cm} \int_\alpha^{\beta} \bar u_2 e^{2\frac \varphi  h} (P-z)u_1dx&=&
\int_\alpha^{\beta} \overline{hv'_2} hv'_1 dx+ 
\int_\alpha^{\beta}(V-z-\varphi'^2) \bar v_2 v_1 dx \\
\nonumber
&&+\int_\alpha^\beta h\varphi'(\bar v_2 v_1'-\bar v_2'v_1)dx\\
\label{eq.Agmident}
&&
+h^2\left( e^{2\frac{\varphi(\alpha)}{h}}\bar
  u_2 u'_1(\alpha)-e^{2\frac{\varphi(\beta)}{h}}\bar u_{2}u'_{1}(\beta)
\right) 
\end{eqnarray}
holds by setting $v_j:=e^{\varphi /h}u_j$ for $j=1,2$\,.
   \end{lemma}
This identity is obtained after conjugation of $hd/dx$ by  $e^{\varphi/h}$ 
and  integration by parts. The weak regularity assumptions can be
checked after regularizing individually $u_{1}$, $u_{2}$, $\varphi$ or $V$.
In \cite{Ni2} it was even considered with possible jumps of the
derivative $u_{1}'$ at $\alpha$ and $\beta$, which are here removed
by the simplifying condition that $u_{1}$ is locally $H^{2}$ around
$\alpha$ and $\beta$ (Jump conditions already occur at the ends
of our intervals).

\section{Variation on adiabatic evolutions.}
\label{se.adiabvar}

We shall consider a family of contraction semigroup generators 
$(A^{\varepsilon}(s))_{s\in [0,+\infty)}$ which fulfill the two next properties.
\begin{itemize}
\item The Cauchy problem
\begin{equation}
    \label{eq.Cauchy}
\left\{
  \begin{array}[c]{l}
    i\varepsilon \partial_{t}u_{t}= iA^{\varepsilon}(t) u_{t}\\
u_{t=0}=u_{0}
  \end{array}
\right.
\end{equation}
admits a unique strong solution with $u_{t}\in D(A^{\varepsilon}(t))$ for all $t\geq 0$ as
soon as $u_{0}\in D(A^{\varepsilon}(0))$. The corresponding dynamical system of
contractions is denoted $(S^{\varepsilon}(t,s))_{t\geq s}$ with the property
$S^{\varepsilon}(t,s)D(A^{\varepsilon}(s))\subset D(A^{\varepsilon}(t))$\,.
%%\item For all $s\in [0,+\infty)$,  the set $\sigma_{1}(s)\subset
%%   \sigma_{disc}(iA^{\varepsilon}(s))$ satisfies
%% $$
%% d(\sigma_{1}(s), \sigma(iA(s))\setminus \sigma_{1}(s))\geq 3d\,,
%% $$
%% with $d>0$ independent of $\varepsilon$. More precisely we assume that
%% there exists $z_{0}\in \C$ such that
%% $$
%% \sigma_{1}(s) \subset \left\{z\in \C\,, |z-z_{0}| \leq d\right\}
%% \quad
%% \text{and}
%% \quad
%% \sigma(iA^{\varepsilon}(s))\setminus \sigma_{1}(s)\subset 
%% \left\{z\in \C\,, |z-z_{0}| \geq 3d\right\}
%% $$
%% and we set 
%% $$
%% \Gamma_{\varepsilon}=\left\{z\in\C\,, |z-z_{0}|=d\right\}\,.
%% $$
\item The
 resolvent $(z-iA^{\varepsilon}(s))^{-1}$
 defines ${\cal C}^{K+1}([0,+\infty);{\cal L}({\cal H}))$ function 
for some $z\in\C$ and
 that the exists a contour  $\Gamma_{\varepsilon}\subset \C$
 independent of $s\in[0,T]$,
such that
$$
|\Gamma_{\varepsilon}|
+ 
\max_{z\in \Gamma_{\varepsilon}, s\in [0,T]}\|\partial_{s}^{k}(z-A^{\varepsilon}(s))^{-1}\|
\leq \frac{a_{k,\delta}}{\varepsilon^{\delta}}\,,
$$
for any $k\in \left\{0,\ldots, K+1\right\}$, $K\in \N$,  any $\delta\in (0,\delta_{0})$
and any $\varepsilon\in (0,\varepsilon_{0})$\,.
%%inutile
%% \item The projector 
%% $$
%% P_{0}^{\varepsilon}(s)
%% =\frac{1}{2i\pi}\int_{\Gamma_{\varepsilon}}(z-A^{\varepsilon}(s))^{-1}~dz
%% $$
%% has a uniformly bounded norm
%% $$
%% \exists C>0,\quad \forall \varepsilon\in (0,\varepsilon_{0}), \forall
%% s\in \R,\quad \|P_{0}^{\varepsilon}(s)\|\leq C\,.
%% $$
\end{itemize}
\noindent\textbf{Notation:} We shall use the notation
$g(\varepsilon)=\widetilde{{\cal O}}(\varepsilon^{N})$ for any $N\in\Z$
 in order to summarize
$$
\forall \delta >0, \exists C_{g, \delta}>0,\quad
|g(\varepsilon)|\leq C_{g,\delta}\varepsilon^{N-\delta}\,.
$$
For example, the previous assumption can be written
\begin{equation}
  \label{eq.estimresA}
|\Gamma_{\varepsilon}|=\tilde{\cal O}(\varepsilon^{0})
\quad
\text{and} 
\quad \max_{ k\leq K+1, z\in
  \Gamma_{\varepsilon}, s\in[0,T]}\|\partial_{s}^{k}(z-A^{\varepsilon}(s))^{-1}\|= \widetilde{{\cal O}}(\varepsilon^{0})\,.
\end{equation}
The spectral
projection $P_{0}(t)=E_{0}(t)$ is defined as 
a contour integral along
$\Gamma_{\varepsilon}$ of the resolvent
$(z-A^{\varepsilon}(t))^{-1}$\,. Correction terms $E_{j}(t)$, $1\leq
j\leq K$ are then constructed by induction.
The finite sequence $(E_{j}^{\varepsilon})_{0\leq j \leq
  K}$ is defined according to
\begin{eqnarray}
\label{eq.defE0}
  &&
E_{0}^{\varepsilon}(s)=
P_{0}^{\varepsilon}(s)=\frac{1}{2i\pi}\int_{\Gamma_{\varepsilon}}
(z-A^{\varepsilon}(s))^{-1}~dz\,,\quad Q_{0}^{\varepsilon}(s)= 1-P_{0}^{\varepsilon}(s)\,,
\\
\label{eq.defSj}
&&
S_{j}^{\varepsilon}(s)=\sum_{m=1}^{j-1}E_{m}^{\varepsilon}(s)E_{j-m}^{\varepsilon}(s)\,,\quad
\text{if}\; 2\leq j\leq K\,, \quad S_{0}^{\varepsilon}=S_{1}^{\varepsilon}=0\,,\\
\label{eq.defEj}
&&
E_{j}^{\varepsilon}(s)=\frac{i}{2\pi}\int_{\Gamma_{\varepsilon}}R^{\varepsilon}
\left\{Q_{0}^{\varepsilon}(s)\partial_{s}E_{j-1}^{\varepsilon}(s)P_{0}^{\varepsilon}(s)-
P_{0}^{\varepsilon}(s)\partial_{s}E_{j-1}^{\varepsilon}(s)Q_{0}^{\varepsilon}(s)\right\}
R^{\varepsilon}~dz
\\
&&\hspace{3cm}
+ S_{j}^{\varepsilon}(s) -
2P_{0}^{\varepsilon}(s)S_{j}^{\varepsilon}(s)P_{0}^{\varepsilon}(s)\,,\;
\text{with}~R^{\varepsilon}=(z-A^{\varepsilon}(s))^{-1}\,.
\end{eqnarray}
\begin{theorem}
\label{th.nenciu} There exists a $\mathcal{C}^{1}$-projection valued
function $(P^{\varepsilon}(s))_{s\geq 0}$ such that  the  relations and estimates
\begin{eqnarray}
\label{eq.PPP}
&& P^{\varepsilon}(s)P^{\varepsilon}(s)=P^{\varepsilon}(s)\quad,\quad
P^{\varepsilon}(s)\in \mathcal{L}(\mathcal{H},
D(A^{\varepsilon}(s)))\,,
\\
%% \label{eq.estimEj}
%% && \forall j\in \left\{0,\ldots, K\right\},\quad 
%% E_{j}^{\varepsilon}(s)=\widetilde{\mathcal{O}}(\varepsilon^{0})
%% \quad\text{in}~\mathcal{L}(\mathcal{H})\,,\\
\label{eq.defP}
&& 
  P^{\varepsilon}(s)=\sum_{j=0}^{K}\varepsilon^{j}E_{j}^{\varepsilon}(s)+
  \widetilde{\mathcal{O}}(\varepsilon^{K+1})
\quad\text{with}~E_{j}^{\varepsilon}(s)=\widetilde{\mathcal{O}}(\varepsilon^{0})
\quad\text{in}~\mathcal{L}(\mathcal{H})\,,
\\
\label{eq.estimPA}
&& 
P^{\varepsilon}(s)A^{\varepsilon}(s)=\widetilde{\mathcal{O}}(\varepsilon^{0})
\quad\text{and}\quad
A^{\varepsilon}(s)P^{\varepsilon}(s)=
\widetilde{\mathcal{O}}(\varepsilon^{0})
\quad\text{in}~\mathcal{L}(\mathcal{H})\,,
\\
\label{eq.ipasP}
&&
i\varepsilon\partial_{s}P^{\varepsilon}(s)-\left[i A^{\varepsilon}(s),
  P^{\varepsilon}(s)\right]=\widetilde{\mathcal{O}}(\varepsilon^{K+1})
\quad\text{in}~\mathcal{L}(\mathcal{H})\,,
\end{eqnarray}
hold with uniform constants with respect to $s,t\in[0,T]$
for  any fixed $T<+\infty$\,.
Moreover for  $K\geq 1$, if $u_{s}=P^{\varepsilon}(s)u_{s}$ then
$u^{\varepsilon}(t)=S^{\varepsilon}(t,s)u_{s}$ satisfies
\begin{eqnarray}
  \label{eq.StsP}
&&
 \sup_{s\leq t\leq
  T}\|u^{\varepsilon}(t)-v^{\varepsilon}(t)\|=\widetilde{\mathcal{O}}(\varepsilon^{K})\,,\\
\label{eq.vPv}
\text{with}&&
v^{\varepsilon}(t)=P^{\varepsilon}(t)v^{\varepsilon}(t)\,,\\
\label{eq.eqv}
\text{and}
&&
\left\{
\begin{array}[c]{l}
  i\varepsilon \partial_{t}v^{\varepsilon}
  -i\varepsilon(\partial_{t}P^{\varepsilon}(t))
v^{\varepsilon}=
P^{\varepsilon}(t)(iA^{\varepsilon}(t))P^{\varepsilon}(t)v^{\varepsilon}\,,\quad\text{for}~t\geq s\,,
\\
v^{\varepsilon}(t=s)=u_{s}\,.
\end{array}
\right.
\end{eqnarray}
\end{theorem}
The proof of this theorem follows the lines of \cite{Nen}.For the
sake of completeness, we check that the computations are still valid
in the non self-adjoint unbounded case (bounded self-adjoint
generators have been considered in \cite{NeRa} \cite{Sjo}) and that the 
$\widetilde{\mathcal{O}}(\varepsilon^{0})$ estimates can be propagated
in the induction process like the 
uniform constants in \cite{Nen}.  
Part of the analysis could be pushed further in the spirit of
\cite{Joy} in order to get $\mathcal{O}(e^{-\frac{c}{\varepsilon}})$
error under analyticity assumptions but the techniques developed by
A.~Joye in this article should be adapted in order to work with a
$\varepsilon$-dependent gap 
or with $\tilde{\mathcal{O}}(\varepsilon^{0})$ resolvent estimates,
maybe by including all the additional information provided by our model.

 As it is
 stated, the previous result cannot be used for $K=0$ and is not formulated as usual
with a reduced evolution on the fixed space $\textrm{Ran} P^{\varepsilon}(s)$
after introducing the parallel transport associated with the ${\cal
  C}^{1}$ family $(P^{\varepsilon}(t))_{t\geq s}$\,.
Actually both problems can be solved at the first order with an
additional uniform boundedness assumption on $E_{0}^{\varepsilon}(t)$
and $\partial_{t}E_{0}^{\varepsilon}(t)$\,. This will be 
obtained as a corollary of Theorem~\ref{th.nenciu}, used with $K=1$
before reconsidering the case $K=0$.
The parallel transport
$\Phi_{0}^{\varepsilon}(t',s')$,  associated with
$(E_{0}^{\varepsilon}(t))_{t\in[0,T]}$, 
is defined for $t',s'\in[0,T]$ by
\begin{equation}
  \label{eq.Phits}
\left\{
  \begin{array}[c]{l}
    \partial_{t'}\Phi_{0}^{\varepsilon}+
    \left[E_{0}^{\varepsilon},\partial_{t}E_{0}^{\varepsilon}\right]\Phi_{0}^{\varepsilon}=0\\
\Phi_{0}^{\varepsilon}(t'=s',s')=\Id\,,
  \end{array}
\right.
\end{equation}
and the uniform boundedness of $\Phi_{0}^{\varepsilon}(t',s')$ is
inherited from the one of $E_{0}^{\varepsilon}(t)$ and $\partial_{t}E_{0}^{\varepsilon}(t)$\,.

\begin{corollary}
  \label{co.nenciuadapt0}
With the hypotheses of Theorem~\ref{th.nenciu} with $K\geq 1$, 
assume additionally that the projector $E_{0}^{\varepsilon}(s)$ 
defined in \eqref{eq.defE0} and
its derivative $\partial_{s}E_{0}^{\varepsilon}(s)$ are uniformly 
bounded continuous functions:
$$
\exists C>0\,,\forall\varepsilon\in
(0,\varepsilon_{0})\,,\quad\max_{s\in[0,T]}\|E_{0}^{\varepsilon}(s)\|+
\|\partial_{s}E_{0}^{\varepsilon}(s)\|\leq
C\,. 
$$
Then for $K\geq 1$ and when $u_{s}=E_{0}^{\varepsilon}(s)u_{s}$, the solution
$u^{\varepsilon}(t)=S^{\varepsilon}(t,s)u_{s}$ to \eqref{eq.Cauchy} satisfies 
\begin{equation}
  \label{eq.uv0}
  \sup_{s\leq t\leq
    T}\|u^{\varepsilon}(t)-\Phi_{0}^{\varepsilon}(t,s)w^{\varepsilon}(t)\|=
\widetilde{\mathcal{O}}(\varepsilon)\|u_{s}\|\,,
\end{equation}
where $\Phi_{0}^{\varepsilon}(t',s')$ is the parallel transport
defined for $t',s'\in[0,T]$ by \eqref{eq.Phits} 
and $w^{\varepsilon}\in  E_{0}^{\varepsilon}(s)\mathcal{H}$ solves the
Cauchy problem
$$
\left\{
  \begin{array}[c]{l}
    i\varepsilon\partial_{t}w^{\varepsilon}=\Phi_{0}^{\varepsilon}(s,t)
    E_{0}^{\varepsilon}(t)(iA^{\varepsilon}(t))E_{0}^{\varepsilon}(t)\Phi_{0}^{\varepsilon}(t,s)
    w^{\varepsilon}=E_{0}^{\varepsilon}(s)\Phi_{0}^{\varepsilon}(s,t)
    (iA^{\varepsilon}(t))\Phi_{0}^{\varepsilon}(t,s)E_{0}^{\varepsilon}(s)w^{\varepsilon}\\  
w^{\varepsilon}(t=s)=u_{s}\,.
  \end{array}
\right.
$$
\end{corollary}
Theorem~\ref{th.nenciu} and Corollary~\ref{co.nenciuadapt0} are proved
in several steps. We start with uniform estimates for the $E_{j}$'s.
\begin{proposition}
  \label{pr.nenciuEj} For all $j\in \left\{0,\ldots, K\right\}$, and
  any $T\in \R_{+}$, the 
  $\mathcal{L}(\mathcal{H})$-valued functions 
 $E_{j}^{\varepsilon}$ and $S_{j}^{\varepsilon}$ satisfy:
  \begin{eqnarray}
\label{eq.estimpaEj}
    && \sum_{k=0}^{K+1-j}
    \|\partial_{s}^{k}E_{j}^{\varepsilon}(s)\|+\|\partial_{s}^{k}S_{j}^{\varepsilon}(s)\|= 
    \widetilde{\mathcal{O}}(\varepsilon^{0})\,,
\quad E_{j}(s)\in
    \mathcal{L}(\mathcal{H}, D(A^{\varepsilon}(s)))\,,\\
\label{eq.estimEjA}
&&
A^{\varepsilon}(s)E_{j}(s)=\widetilde{\mathcal{O}}(\varepsilon^{0})\quad
\text{and}\quad
E_{j}(s)A^{\varepsilon}(s)=\widetilde{\mathcal{O}}(\varepsilon^{0})
\quad \text{in}~\mathcal{L}(\mathcal{H})\,,
\\
\label{eq.relEjEm}
&& E_{j}^{\varepsilon}(s)=\sum_{m=0}^{j}E_{m}^{\varepsilon}(s)E_{j-m}^{\varepsilon}(s) \underset{\textrm{if}\, j\geq 1}{=} S_{j}^{\varepsilon}(s)+
E_{0}^{\varepsilon}(s)E_{j}^{\varepsilon}(s)+E_{j}^{\varepsilon}(s)E_{0}^{\varepsilon}(s)\,,\\
\label{eq.ipaEj}
&&
i\partial_{s}E_{j-1}^{\varepsilon}(s)=\left[iA^{\varepsilon}(s),
  E_{j}^{\varepsilon}(s)\right]\,,\quad \text{for}~j\geq 1\,,
  \end{eqnarray}
with uniform constants w.r.t.  $s\in[0,T]$\,.
\end{proposition}
\noindent\textbf{Proof:} The first
statement for $j=0$ is a consequence of the definition 
\eqref{eq.defE0} of
$E_{0}^{\varepsilon}(s)=P_{0}^{\varepsilon}(s)$ combined with the
estimates \eqref{eq.estimresA} of
$\partial_{s}^{k}(z-iA^{\varepsilon}(s))^{-1}$\,. 
 By induction
assume
that the properties are satisfied for $j\leq J<K$. The definition
\eqref{eq.defSj}
of $S_{J+1}^{\varepsilon}$ and \eqref{eq.defEj} of
$E_{J+1}^{\varepsilon}$ provide directly the first statement
\eqref{eq.estimpaEj} for $j=J+1$\,. The second statement of
\eqref{eq.estimpaEj} and the estimates \eqref{eq.estimEjA}
rely on the bound of the $\int_{\Gamma_{\varepsilon}}$-term which 
is obtained after noticing 
\begin{equation}\label{eq.not}
A^{\varepsilon}(s)R^{\varepsilon}= -1
+\frac{z}{(z-A^{\varepsilon}(s))}=
R^{\varepsilon}A^{\varepsilon}(s)
\end{equation}
 and the bound of the two other terms which is
deduced from the induction assumption for $j\leq J$\,.
Compute the commutator $[iA^{\varepsilon}(s),
S_{J+1}^{\varepsilon}(s)]$:
\begin{eqnarray*}
\left[iA^{\varepsilon},S_{J+1}^{\varepsilon}\right]
&&=\sum_{m=1}^{J}(E_{m}^{\varepsilon}\left[iA^{\varepsilon},
E_{J+1-m}^{\varepsilon}\right]+
\left[iA^{\varepsilon}, E_{m}^{\varepsilon}\right]E_{J+1-m}^{\varepsilon})\,,
\\
&&
= \sum_{m=1}^{J}E_{m}^{\varepsilon}(i\partial_{s}E_{J-m}^{\varepsilon})+
(i\partial_{s}E_{m-1}^{\varepsilon})E_{J+1-m}^{\varepsilon}\,,\quad
\text{owing to}~\eqref{eq.ipaEj}~\text{with}~j\leq J\,,
\\
&&
=
i\partial_{s}E_{J}^{\varepsilon}-E_{0}^{\varepsilon}(i\partial_{s}E_{J}^{\varepsilon})
-(i\partial_{s}E_{J}^{\varepsilon})E_{0}^{\varepsilon}
\,,\quad
\text{owing to}~\eqref{eq.relEjEm}~\text{with}~j\leq J\,,\\
&&
= 
-i\left(P_{0}^{\varepsilon}(\partial_{s}E_{J}^{\varepsilon})P_{0}^{\varepsilon}-Q_{0}^{\varepsilon}(\partial_{s}E_{J}^{\varepsilon})Q_{0}^{\varepsilon}\right)\,,
\quad
\text{owing to}~E_{0}^{\varepsilon}=P_{0}^{\varepsilon}=1-Q_{0}^{\varepsilon}\,.
\end{eqnarray*} 
With $P_{0}^{\varepsilon}P_{0}^{\varepsilon}=P_{0}^{\varepsilon}$ and
$P_{0}^{\varepsilon}Q_{0}^{\varepsilon}=Q_{0}^{\varepsilon}P_{0}^{\varepsilon}=0$,
this implies
\begin{eqnarray}
  \label{eq.APSP}
  \left[iA^{\varepsilon},
    P_{0}^{\varepsilon}S_{J+1}^{\varepsilon}P_{0}^{\varepsilon}\right] =
  -i
  P_{0}^{\varepsilon}(\partial_{s}E_{J}^{\varepsilon})P_{0}^{\varepsilon}\,,\\
\label{eq.AQSQ}
 \left[iA^{\varepsilon},
    Q_{0}^{\varepsilon}S_{J+1}^{\varepsilon}Q_{0}^{\varepsilon}\right] =i Q_{0}^{\varepsilon}(\partial_{s}E_{J}^{\varepsilon})Q_{0}^{\varepsilon}\,,\\
\label{eq.APSQ}
\left[iA^{\varepsilon},
  P_{0}^{\varepsilon}S_{J+1}^{\varepsilon}Q_{0}^{\varepsilon}\right]
=
\left[iA^{\varepsilon}, Q_{0}^{\varepsilon}S_{J+1}^{\varepsilon}P_{0}^{\varepsilon}\right]=0\,.
\end{eqnarray}
The definition of $P_{0}^{\varepsilon}$ as a spectral projection
associated with $iA^{\varepsilon}$ and \eqref{eq.APSQ}, imply
(see for example \cite{Nen}~Proposition~1)
$$
P_{0}^{\varepsilon}S_{J+1}^{\varepsilon}Q_{0}^{\varepsilon}=
Q_{0}^{\varepsilon}S_{J+1}^{\varepsilon}P_{0}^{\varepsilon}=0\,.
$$
Meanwhile the definition \eqref{eq.defEj} of $E_{J+1}^{\varepsilon}$
for $j=J+1$ implies
$$
P_{0}^{\varepsilon}E_{J+1}^{\varepsilon}P_{0}^{\varepsilon}=
-P_{0}^{\varepsilon}S_{J+1}^{\varepsilon}P_{0}^{\varepsilon}
\quad,\quad
Q_{0}^{\varepsilon}E_{J+1}^{\varepsilon}Q_{0}^{\varepsilon}=
Q_{0}^{\varepsilon}S_{J+1}^{\varepsilon}Q_{0}^{\varepsilon}\,.
$$
This yields the relation \eqref{eq.relEjEm} for $j=J+1$. Another
consequence with \eqref{eq.APSP} and \eqref{eq.AQSQ} is
\begin{equation}
  \label{eq.diagAE}
\left[iA^{\varepsilon},
  P_{0}^{\varepsilon}E_{J+1}^{\varepsilon}P_{0}^{\varepsilon}
+
 Q_{0}^{\varepsilon}E_{J+1}^{\varepsilon}Q_{0}^{\varepsilon}\right] 
=
P_{0}^{\varepsilon}(i\partial_{s}E_{J}^{\varepsilon})P_{0}^{\varepsilon}
+
 Q_{0}^{\varepsilon}(i\partial E_{J}^{\varepsilon})Q_{0}^{\varepsilon}\,.
\end{equation}
Finally compute  the off-diagonal blocks $P_{0}^{\varepsilon}[iA^{\varepsilon},
E_{J+1}^{\varepsilon}]Q_{0}^{\varepsilon}$ and
$Q_{0}^{\varepsilon}[iA^{\varepsilon},
E_{J+1}^{\varepsilon}]P_{0}^{\varepsilon}$ by using  again the definition
\eqref{eq.defEj} of $E_{J+1}^{\varepsilon}$, the relation
\eqref{eq.APSQ}, and the identity \eqref{eq.not}:
\begin{eqnarray*}
 \left[iA^{\varepsilon},
   P_{0}^{\varepsilon}E_{J+1}^{\varepsilon}Q_{0}^{\varepsilon}\right]
&=&  P_{0}^{\varepsilon}\left[iA^{\varepsilon},
    E_{J+1}^{\varepsilon}\right]Q_{0}^{\varepsilon}
=iP_{0}^{\varepsilon}(\partial_{s}E_{J}^{\varepsilon})Q_{0}^{\varepsilon}\,,\\
 \left[iA^{\varepsilon},
   Q_{0}^{\varepsilon}E_{J+1}^{\varepsilon}P_{0}^{\varepsilon}\right]
&=&
iQ_{0}^{\varepsilon}(\partial_{s}E_{J}^{\varepsilon})P_{0}^{\varepsilon}\,.
\end{eqnarray*}
Summing these last two equalities with \eqref{eq.diagAE} yields the
relation \eqref{eq.ipaEj} for $j=J+1$\,.
\qed\\
The above calculations are essentially the same as in
\cite{Nen}\cite{NeRa} and, as a consequence of Proposition \ref{pr.nenciuEj}, the sum
\begin{equation}\label{eq.Teps}
T_{\varepsilon}(s)=\sum_{j=0}^{K}\varepsilon^{j}E_{j}^{\varepsilon}(s)
\end{equation}
solves \eqref{eq.PPP}, \eqref{eq.estimPA} and \eqref{eq.ipasP} in the sense of asymptotic expensions; in particular:
\begin{equation}\label{eq.Teps2}
T_{\varepsilon}^2=T_{\varepsilon}+\widetilde{\mathcal{O}}(\varepsilon^{K+1})
\end{equation}
Here comes the main difference which is
necessary because no better estimate than
$\|P_{0}^{\varepsilon}\|=\widetilde{\mathcal{O}}(\varepsilon^{0})$
can be expected in  our non self-adjoint case.
We will need the next lemma.
\begin{lemma}
  \label{le.TP}
Assume that $T\in \mathcal{L}(\mathcal{H})$ satisfies $\|T^{2}-T\|\leq
\delta<1/4$ and $\|T\|\leq C$ then 
\begin{eqnarray*}
&&\sigma(T)\subset \left\{z\in\C\,,
|z(z-1)|\leq \delta\right\}\subset \left\{z\in \C\,,|z|\leq
c_{\delta}\right\}\cup
\left\{z\in \C\,,|z-1|\leq
c_{\delta}\right\}\,,\quad c_{\delta}<\frac{1}{2}\,,
\\
&&
\max\left\{\|(z-T)^{-1}\|, |z-1|=\frac{1}{2}\right\}\leq 2\frac{2C+1}{1-4\delta}\,,
\\
&&
\|T-P\|\leq
2\frac{2C+1}{1-4\delta}\delta\,,\quad\text{with}
\quad 
P=\frac{1}{2i\pi}\int_{|z-1|=1/2}(z-T)^{-1}~dz\,.
\end{eqnarray*}
\end{lemma}
\noindent\textbf{Proof:} If $z\in \sigma(T)$ then $z(z-1)$ belongs to
$\sigma(T(T-1))\subset \left\{z\in\C, |z(z-1)|\leq
  \frac{1}{4}\right\}$ (Remember that $|z(z-1)|=\frac{1}{4}$ means
$|Z-\frac{1}{4}|=\frac{1}{4}$
with $Z=(z-\frac{1}{2})^{2}$). 
Consider $z\in \C$ such that $|z-1|=\frac{1}{2}$, then the relation
$$
(T-z)(T-(1-z))=z(1-z)+(T^{2}-T)\,,
$$
with $|z(1-z)|\geq \frac{1}{4}$ and $\|T^{2}-T\|\leq \delta<\frac{1}{4}$, implies
$$
\|(T-z)^{-1}\|\leq \|T-(1-z)\|\|[z(1-z)+T^{2}-T]^{-1}\|\leq
\frac{C+\frac{1}{2}}{\frac{1}{4}-\delta}\quad\text{for}~|z-1|=\frac{1}{2}\,.
$$
The symmetry with respect to $z=\frac{1}{2}$ due to
$(1-T)(1-T)-(1-T)=T^{2}-T$ implies also
$$
\|(T-z)^{-1}\|\leq \frac{C+\frac{1}{2}}{\frac{1}{4}-\delta}\quad\text{for}~|z|=\frac{1}{2}\,.
$$
Compute 
\begin{eqnarray*}
  T-P &=& \frac{1}
  {2i\pi}\int_{|z-1|=\frac{1}{2}}\left[T(z-1)^{-1}-(z-T)^{-1}\right]~dz\\
&=& (T-1)P+(T^{2}-T)A_{1}\quad\text{with}\quad
A_{1}=\frac{1}{2i\pi}\int_{|z-1|=\frac{1}{2}}(T-z)^{-1}(1-z)^{-1}~dz\,.
\end{eqnarray*}
In particular this implies to
$T(1-P)=(T^{2}-T)A_{1}$
while replacing $T$ with $(1-T)$ and $P$ with $1-P$ leads to 
$$
P-T=-T(1-P)+(T^{2}-T)A_{0} \quad\text{with}\quad 
A_{0}=\frac{1}{2i\pi}\int_{|z|=\frac{1}{2}}(T-z)^{-1}z^{-1}~dz\,.
$$
We finally obtain
\begin{eqnarray}
\label{eq.TPT2}
&&
T-P=(T^{2}-T)(A_{1}-A_{0})=(A_{1}-A_{0})(T^{2}-T)\\
\nonumber
\text{and}
&&
\|T-P\|\leq \|T^{2}-T\|\left[\|A_{1}\|+\|A_{0}\|\right]
\leq \delta \times\frac{C+1/2}{\frac{1}{4}-\delta}\,.
\end{eqnarray}
\qed
\begin{proposition}
\label{pr.proj} Consider the approximate projection
$T_{\varepsilon}(s)$ defined in \eqref{eq.Teps},
then there exists a projection $P^{\varepsilon}(s)$ such that
\begin{eqnarray}
  \label{eq.PeTe}
&&
\|P^{\varepsilon}(s)-T_{\varepsilon}(s)\|=\widetilde{\mathcal{O}}(\varepsilon^{K+1})\,,\quad
P^{\varepsilon}(s)\in \mathcal{L}(\mathcal{H},
D(A^{\varepsilon}(s)))\,,\\
&&
\label{eq.PeA}
\|A^{\varepsilon}(s)P^{\varepsilon}(s)\|=\widetilde{\mathcal{O}}(\varepsilon^{0})
\quad\text{and}\quad
\|P^{\varepsilon}(s)A^{\varepsilon}(s)\|=\widetilde{\mathcal{O}}(\varepsilon^{0})\,,\\
&&
i\varepsilon \partial_{s}P^{\varepsilon}(s)=\left[iA^{\varepsilon}(s),
  P^{\varepsilon}(s)\right]+\widetilde{\mathcal{O}}(\varepsilon^{K+1})\,,\quad
\text{in}~\mathcal{L}(\mathcal{H})\,.
\end{eqnarray}
\end{proposition}
\noindent{\textbf Proof:} For $\varepsilon>0$ small enough, set
$$
P^{\varepsilon}(s)=\frac{1}{2i\pi}\int_{|z-1|=\frac 1 2}(z-T_{\varepsilon}(s))^{-1}~dz\,.
$$
Owing to \eqref{eq.Teps2}, the first statement of \eqref{eq.PeTe} is a straightforward
application of Lemma~\ref{le.TP}. The definition of
$T_{\varepsilon}(s)$ implies $A^{\varepsilon}(s)T_{\varepsilon}(s)\in
\mathcal{L}(\mathcal{H})$ and $T_{\varepsilon}(s)A^{\varepsilon}(s)\in
\mathcal{L}(\mathcal{H})$. The relation
\eqref{eq.TPT2} with $T=T_{\varepsilon}$ and $P=P^{\varepsilon}$ gives
\eqref{eq.PeA}\,.
Computing the derivative $\partial_{s}P^{\varepsilon}(s)$ with
$T_{\varepsilon}\in \mathcal{C}^{1}([0,T]; \mathcal{L}(\mathcal{H}))$
gives:
\begin{eqnarray*}
  i\varepsilon \partial_{s}P^{\varepsilon}
&=&
\frac{1}{2i\pi}\int_{|z-1|=\frac
  1
  2}(z-T_{\varepsilon})^{-1}(i\varepsilon \partial_{s}T_{\varepsilon})(z-T_{\varepsilon})^{-1}~dz
\\
&=&
\frac{1}{2i\pi}\int_{|z-1|=\frac
  1
  2}(z-T_{\varepsilon})^{-1}\left[iA^{\varepsilon},
  T_{\varepsilon}\right](z-T_{\varepsilon})^{-1}~dz
\\
&&\hspace{3cm}
+ \frac{i\varepsilon^{K+1}}{2i\pi}
\int_{|z-1|=\frac
  1
  2}(z-T_{\varepsilon})^{-1}(\partial_{s}E_{K}^{\varepsilon})(z-T_{\varepsilon})^{-1}~dz\,.
\end{eqnarray*}
The relation
$$
\left[iA^{\varepsilon}, (z-T_{\varepsilon})^{-1}\right]=
(z-T_{\varepsilon})^{-1}\left[iA^{\varepsilon}, T_{\varepsilon}\right](z-T_{\varepsilon})^{-1}
$$
allows to conclude
$$
i\varepsilon\partial_{s}P^{\varepsilon}-\left[iA^{\varepsilon},
  P^{\varepsilon}\right]=
\frac{i\varepsilon^{K+1}}{2i\pi}
\int_{|z-1|=\frac
  1
  2}(z-T_{\varepsilon})^{-1}(\partial_{s}E_{K}^{\varepsilon})(z-T_{\varepsilon})^{-1}~dz
=\widetilde{\mathcal{O}}(\varepsilon^{K+1})\,.
$$
\qed\\
\noindent\textbf{Proof of Theorem~\ref{th.nenciu}:} 
The statements \eqref{eq.PPP}\eqref{eq.defP} and
\eqref{eq.ipasP} have already been checked in
Proposition~\ref{pr.nenciuEj} and Proposition~\ref{pr.proj}. Consider
now the adiabatic evolution of $S(t,s)u_{s}$, when
$u_{s}=P^{\varepsilon}(s)u_{s}$,  stated in \eqref{eq.StsP}\eqref{eq.vPv}\eqref{eq.eqv}.
\\
We assumed that the Cauchy problem \eqref{eq.Cauchy} defines a
strongly continuous dynamical system $(S^{\varepsilon}(t,s))_{t\geq s\geq 0}$ of
contractions in $\mathcal{H}$ with $S^{\varepsilon}(t,s)D(A^{\varepsilon}(s))\subset
D(A^{\varepsilon}(t))$.
We now consider the modified operator
\begin{eqnarray*}
&&H_{AD}^{\varepsilon}(t)= iA^{\varepsilon}(t) + B^{\varepsilon}(t)\,,\\
\text{with}
&&
B^{\varepsilon}(t)=(1-2P^{\varepsilon}(t))
\left(i\varepsilon \partial_{t}P^{\varepsilon}(t)-\left[iA^{\varepsilon}(t),
  P^{\varepsilon}(t)\right]\right)\,.
\end{eqnarray*}
Since $B^{\varepsilon}(s)$ is an
$\widetilde{\mathcal{O}}(\varepsilon^{K+1})$ bounded continuous
perturbation of $iA^{\varepsilon}(s)$, the Cauchy problem
$$
\left\{
  \begin{array}[c]{l}
    i\varepsilon\partial_{t}u_{t}=(iA^{\varepsilon}(t)-B^{\varepsilon}(t))u_{t}\\
    u_{t=s}=u_{s}\,,
  \end{array}
\right.
$$
defines a strongly continuous dynamical system of bounded operators
$$
S_{AD}^{\varepsilon}(t,s)=S^{\varepsilon}(t,s)-i\varepsilon^{-1}\int_{s}^{t}S^{\varepsilon}(t,s')
(-B^{\varepsilon}(s'))S_{AD}^{\varepsilon}(s',s)~ds'\,.
$$
With the help of Gronwall lemma, it satisfies
\begin{eqnarray*}
&&
\|S_{AD}^{\varepsilon}(t,s)\|
\leq e^{\varepsilon^{-1}\int_{s}^{t}\|B^{\varepsilon}(s')\|~ds'}\leq
e^{C_{\delta,T}\varepsilon^{K-\delta}(t-s)}\,
\\
&&
\|S^{\varepsilon}(t,s)-S_{AD}^{\varepsilon}(t,s)\|
\leq 
C_{\delta,T}\varepsilon^{K-\delta}(t-s)e^{C_{\delta,T}\varepsilon^{K-\delta}(t-s)}\,,
\end{eqnarray*}
 for all $s,t$, $0\leq s\leq t\leq T$, with $\delta\in (0,1)$. 
Note that the right-hand side is bounded when 
$$
K\geq 1\,.
$$
For the comparison
 \eqref{eq.StsP} we take simply
 $v^{\varepsilon}(t)=S_{AD}^{\varepsilon}(t,s)u_{s}$.
It remains to check \eqref{eq.vPv} and \eqref{eq.eqv}.
First notice the identity
\begin{eqnarray}
\label{eq.HADfor}
H_{AD}^{\varepsilon}(t)&=&P^{\varepsilon}(t)(iA^{\varepsilon}(t))P^{\varepsilon}(t)+
(1-P^{\varepsilon}(t))(iA^{\varepsilon}(t))
(1-P^{\varepsilon}(t))\\
\nonumber
&&-i\varepsilon
\left[P^{\varepsilon}(t)\partial_{t}P^{\varepsilon}(t)+(1-P^{\varepsilon}(t))\partial_{t}
  (1-P^{\varepsilon}(t))\right]\,.
\end{eqnarray}
%%where the last term equals
%%$-i\varepsilon(2P^{\varepsilon}-1)\partial_{s}P^{\varepsilon}$. 
For our choice  $v^{\varepsilon}(t)=S_{AD}^{\varepsilon}(t,s)u_{s}$
the quantity
$P^{\varepsilon}(t)(i\varepsilon\partial_{t}v^{\varepsilon}(t))$
equals
$P^{\varepsilon}(t)(iA^{\varepsilon}(t)+B^{\varepsilon}(t))v^{\varepsilon}(t)$
since
$P^{\varepsilon}(t)A^{\varepsilon}=A^{\varepsilon}(t)P^{\varepsilon}(t)
- \left[A^{\varepsilon}(t),P^{\varepsilon}(t)\right]\in
\mathcal{L}(\mathcal{H})$\,.
Hence  $P^{\varepsilon}(t)v^{\varepsilon}(t)$ satisfies in the strong
sense the equation
\begin{eqnarray*}
(i\varepsilon\partial_{t}-H_{AD}^{\varepsilon}(t))(P^{\varepsilon}v^{\varepsilon})
&=&
i\varepsilon(\partial_{t}P^{\varepsilon})v^{\varepsilon}+P^{\varepsilon}(H_{AD}^{\varepsilon}(t)v^{\varepsilon})
- 
H_{AD}^{\varepsilon}(t)P^{\varepsilon}(t)v^{\varepsilon}
\\
&=&
i\varepsilon\left[(\partial_{t}P^{\varepsilon})-
P^{\varepsilon}(\partial_{t}P^{\varepsilon})-(\partial_{t}P^{\varepsilon})P^{\varepsilon}
\right]v^{\varepsilon}(t)=0\,.
\end{eqnarray*}
As a strong solution to $i\partial_{t}v=H_{AD}^{\varepsilon}(t)v(t)$ with the
initial data
$P^{\varepsilon}(s)u_{s}=u_{s}$,
 $P^{\varepsilon}(t)v^{\varepsilon}(t)$ has to be equal to
 $v^{\varepsilon}(t)$\,. We have proved \eqref{eq.vPv}. The equation
 \eqref{eq.eqv} is a rewriting of
 $i\varepsilon\partial_{t}v^{\varepsilon}-H_{AD}(t)v^{\varepsilon}=0$
 after recalling $\Pi(\partial\Pi)\Pi=0$ when $\Pi^{2}=\Pi$\,.
\qed

\noindent\textbf{Proof of Corollary~\ref{co.nenciuadapt0}:}
Theorem~\ref{th.nenciu} applied with $K=1$ gives the approximation
$v^{\varepsilon}(t)=P^{\varepsilon}(t)v^{\varepsilon}(t)$,
$P^{\varepsilon}(t)=E_{0}^{\varepsilon}(t)+\varepsilon
E_{1}^{\varepsilon}(t)+\tilde{\mathcal{O}}(\varepsilon^{2})$, which
solves 
$$ 
\left\{
\begin{array}[c]{l}
  i\varepsilon \partial_{t}v^{\varepsilon}
  -i\varepsilon(\partial_{t}P^{\varepsilon}(t))
v^{\varepsilon}=
P^{\varepsilon}(t)(iA^{\varepsilon}(t))P^{\varepsilon}(t)v^{\varepsilon}\,,
\quad\text{for}~t\geq s\,,
\\
v^{\varepsilon}(t=s)=u_{s}\,.
\end{array}
\right.
$$
 The relations $P^{\varepsilon}=E_{0}^{\varepsilon}+\varepsilon
E_{1}^{\varepsilon}+\widetilde{\mathcal{O}}(\varepsilon^{2})$, 
$[A^{\varepsilon},E_{0}^{\varepsilon}]=0$ and 
$E_{0}^{\varepsilon}E_{1}^{\varepsilon}E_{0}^{\varepsilon}=0$
combined with the estimates \eqref{eq.estimEjA} and \eqref{eq.PeA} lead to:
\[
P^{\varepsilon}(t)(iA^{\varepsilon}(t))P^{\varepsilon}(t) =
E_{0}^{\varepsilon}(t)(iA^{\varepsilon}(t)) E_{0}^{\varepsilon}(t)
+
\varepsilon
E_{1}^{\varepsilon}(t)(iA^{\varepsilon}(t))E_{0}^{\varepsilon}(t)
+
\varepsilon
(iA^{\varepsilon}(t))E_{0}^{\varepsilon}(t)E_{1}^{\varepsilon}(t)
+
\widetilde{\mathcal{O}}(\varepsilon^{2})\,.
\]
Then Proposition~\ref{pr.nenciuEj} provide
$$
i\varepsilon\partial_{t}P^{\varepsilon}(t)=i\varepsilon \partial_{t}E_{0}^{\varepsilon}(t)+
\widetilde{\mathcal{O}}(\varepsilon^{2})\,. 
$$
This implies
\begin{equation}
  \label{eq.vepsapp}
 i\varepsilon \partial_{t}v^{\varepsilon}
  -i\varepsilon(\partial_{t}E_{0}^{\varepsilon}(t))
v^{\varepsilon}=
E_{0}^{\varepsilon}(t)(iA^{\varepsilon}(t))E_{0}^{\varepsilon}(t)v^{\varepsilon}
+
\varepsilon E_{1}^{\varepsilon}(t)(iA^{\varepsilon}(t)E_{0}^{\varepsilon}(t))v^{\varepsilon}
+\widetilde{\mathcal{O}}(\varepsilon^{2})\,,
\end{equation}
where we used $E_{0}^{\varepsilon}v^{\varepsilon}=v^{\varepsilon}+\widetilde{\mathcal{O}}(\varepsilon)$. Consider now the adiabatic generator
\begin{eqnarray*}
&&H_{0}^{\varepsilon}(t)= iA^{\varepsilon}(t)+B_{0}^{\varepsilon}(t)\,,\\
\text{with}
&&
B_{0}^{\varepsilon}(t)=(1-2E_{0}^{\varepsilon}(t))
\left(i\varepsilon \partial_{t}E_{0}^{\varepsilon}(t)-\left[iA^{\varepsilon}(t),
  E_{0}^{\varepsilon}(t)\right]\right)
=
i\varepsilon  (1-2E_{0}^{\varepsilon}(t))
\partial_{t}E_{0}^{\varepsilon}(t)
\,.
\end{eqnarray*}
The assumed estimates on $E_{0}$ and $\partial_{t}E_{0}$ with the
Gronwall Lemma lead to the uniform bound for the associated dynamical
system $S_{0}^{\varepsilon}(t,s)$:
$$
\forall s,t,  0\leq s\leq t\leq T,\quad \|S_{0}^{\varepsilon}(t,s)\|\leq C_{T}e^{C_{T}}\,,
$$
while the formula \eqref{eq.HADfor} is valid for
$H_{0}^{\varepsilon}(t)$ after replacing $P^{\varepsilon}(t)$ with
$E_{0}^{\varepsilon}(t)$:
\begin{multline*}
H_{0}^{\varepsilon}(t)=E_{0}^{\varepsilon}(t)(iA^{\varepsilon}(t))E_{0}^{\varepsilon}(t)+
(1-E_{0}^{\varepsilon}(t))(iA^{\varepsilon}(t))
(1-E_{0}^{\varepsilon}(t))\\
-i\varepsilon
\left[E_{0}^{\varepsilon}(t)\partial_{t}E_{0}^{\varepsilon}(t)+
(1-E_{0}^{\varepsilon}(t))\partial_{t}
  (1-E_{0}^{\varepsilon}(t))\right]\,.
\end{multline*}
Now compute
\begin{equation*}
(i\varepsilon\partial_{t}-H_{0}^{\varepsilon}(t))(E_{0}^{\varepsilon}v^{\varepsilon})
=
i\varepsilon(\partial_{t}E_{0}^{\varepsilon})v^{\varepsilon}+
E_{0}^{\varepsilon}(i\varepsilon\partial_{t}v^{\varepsilon}  
)
- 
H_{0}^{\varepsilon}(t)E_{0}(t)v^{\varepsilon}\,.
\end{equation*}
With \eqref{eq.vepsapp} and $E_{0}^{\varepsilon}E_{1}^{\varepsilon}E_{0}^{\varepsilon}=0$, one gets
\[
(i\varepsilon\partial_{t}-H_{0}^{\varepsilon}(t))(E_{0}^{\varepsilon}v^{\varepsilon})
=
-i\varepsilon\left[(\partial_{t}E_{0}^{\varepsilon})-
E_{0}^{\varepsilon}(\partial_{t}E_{0}^{\varepsilon})-(\partial_{t}E_{0}^{\varepsilon})E_{0}^{\varepsilon}\right]v^{\varepsilon}(t)+\widetilde{\mathcal{O}}(\varepsilon^{2})=\widetilde{\mathcal{O}}(\varepsilon^{2})\,.
\]
The uniform estimate $\|S_{0}^{\varepsilon}(t,s)\|\leq C_{T}e^{C_{T}}$
implies
$$
\|u^{\varepsilon}(t)-S_{0}^{\varepsilon}(t,s)u_{s}\|\leq
\|u^{\varepsilon}(t)-v^{\varepsilon}(t)\|+
\|\varepsilon E_{1}^{\varepsilon}(t)v^{\varepsilon}(t)\|+
\|E_{0}^{\varepsilon}(t)v^{\varepsilon}-S_{0}^{\varepsilon}(t,s)u_{s}\| 
=\widetilde{\mathcal{O}}(\varepsilon)\,.
$$
The same argument as in the end of the proof of
Theorem~\ref{th.nenciu} says that
$v_{0}^{\varepsilon}(t)=S_{0}^{\varepsilon}(t,s)u_{s}$ with
$E_{0}^{\varepsilon}(s)u_{s}=u_{s}$ satisfies
$$
\forall t\in [s,T]\,,\quad E_{0}^{\varepsilon}(t)v_{0}^{\varepsilon}(t)=v_{0}^{\varepsilon}(t)
$$
and solves the Cauchy problem
$$
\left\{
  \begin{array}[c]{l}
    i\varepsilon \partial_{t}v_{0}^{\varepsilon}-i\varepsilon(\partial_{t}E_{0}^{\varepsilon}(t))v_{0}^{\varepsilon}(t)=
    E_{0}^{\varepsilon}(t)(iA^{\varepsilon}(t))E_{0}(t)v_{0}^{\varepsilon}(t)\,,\\
v_{0}^{\varepsilon}(t=s)=u_{s}\,.
  \end{array}
\right.
$$
The uniform boundedness of  $E_{0}^{\varepsilon}(t)$ and
$\partial_{t}E_{0}^{\varepsilon}(t)$ ensures that the solution to
\eqref{eq.Phits} is well defined for $t',s'\in [0,T]$ with the uniform
estimate
$$
\forall t',s'\in [0,T]\,,\quad 
\|\Phi_{0}^{\varepsilon}(t',s')\|\leq C'_{T}\,,
$$
with the parallel transport property
$$
\forall t',s'\in[0,T]\,,\quad
\Phi_{0}^{\varepsilon}(t',s')E_{0}^{\varepsilon}(s')=E_{0}^{\varepsilon}(t')\Phi_{0}^{\varepsilon}(t',s')\,,\quad
[\Phi_{0}^{\varepsilon}(t',s')]^{-1}= \Phi_{0}^{\varepsilon}(s',t')\,.
$$
It suffices to take
$w^{\varepsilon}(t)=\Phi_{0}^{\varepsilon}(s,t)S_{0}^{\varepsilon}(t,s)u_{s}$\,.
\qed

\noindent\textbf{Acknowledgements:} The authors were partly supported
by the ANR-project Quatrain led by F.~Mehats. This work was initiated
while the second author had a CNRS post-doc position in Rennes. The
third author discussed recently  this problem 
with G.M.~Graf, E.B.~Davies, K.~Pankrashkin
and he remembers   many other former discussions related to this topic
with  colleagues appearing in the bibliography. 

\bibliography{ref}

\begin{thebibliography}{10}

\bibitem{ASFr}
W.K. Abou~Salem and J.~Fr{\"o}hlich.
\newblock Adiabatic theorems for quantum resonances.
\newblock {\em Comm. Math. Phys.}, 273(3):651--675, 2007.

\bibitem{Agm}
S.~Agmon.
\newblock {\em Lectures on exponential decay of solutions of second-order
  elliptic equations: bounds on eigenfunctions of {$N$}-body {S}chr\"odinger
  operators}, volume~29 of {\em Mathematical Notes}.
\newblock Princeton University Press, Princeton, NJ, 1982.

\bibitem{AgCo}
J.~Aguilar and J.~M. Combes.
\newblock A class of analytic perturbations for one-body {S}chr\"odinger
  {H}amiltonians.
\newblock {\em Comm. Math. Phys.}, 22:269--279, 1971.

\bibitem{AlPa}
S.~Albeverio and K.~Pankrashkin.
\newblock A remark on {K}rein's resolvent formula and boundary conditions.
\newblock {\em J. Phys. A}, 38(22):4859--4864, 2005.

\bibitem{AABCMS}
X.~Antoine, A.~Arnold, C.~Besse, M.~Ehrhardt, and A.~Sch{\"a}dle.
\newblock A review of transparent and artificial boundary conditions techniques
  for linear and nonlinear {S}chr\"odinger equations.
\newblock {\em Commun. Comput. Phys.}, 4(4):729--796, 2008.

\bibitem{Arn}
A.~Arnold.
\newblock Mathematical concepts of open quantum boundary conditions.
\newblock {\em Transp. Theory Stat. Phys.}, 30(4-6):561--584, 2001.

\bibitem{AEGS}
J.~E. Avron, A.~Elgart, G.~M. Graf, and L.~Sadun.
\newblock Transport and dissipation in quantum pumps.
\newblock {\em J. Statist. Phys.}, 116(1-4):425--473, 2004.

\bibitem{ASY}
J.~E. Avron, R.~Seiler, and L.~G. Yaffe.
\newblock Adiabatic theorems and applications to the quantum {H}all effect.
\newblock {\em Comm. Math. Phys.}, 110(1):33--49, 1987.

\bibitem{AEGSS}
J.E. Avron, A.~Elgart, G.M. Graf, L.~Sadun, and K.~Schnee.
\newblock Adiabatic charge pumping in open quantum systems.
\newblock {\em Comm. Pure Appl. Math.}, 57(4):528--561, 2004.

\bibitem{BaCo}
E.~Balslev and J.~M. Combes.
\newblock Spectral properties of many-body {S}chr\"odinger operators with
  dilatation-analytic interactions.
\newblock {\em Comm. Math. Phys.}, 22:280--294, 1971.

\bibitem{BKNR1}
M.~Baro, H.C. Kaiser, H.~Neidhardt, and J.~Rehberg.
\newblock Dissipative {S}chr\"odinger-{P}oisson systems.
\newblock {\em J. Math. Phys.}, 45(1):21--43, 2004.

\bibitem{BKNR2}
M.~Baro, H.C. Kaiser, H.~Neidhardt, and J.~Rehberg.
\newblock A quantum transmitting {S}chr\"odinger-{P}oisson system.
\newblock {\em Rev. Math. Phys.}, 16(3):281--330, 2004.

\bibitem{BNR}
M.~Baro, H.~Neidhardt, and J.~Rehberg.
\newblock Current coupling of drift-diffusion models and
  {S}chr\"odinger-{P}oisson systems: dissipative hybrid models.
\newblock {\em SIAM J. Math. Anal.}, 37(3):941--981, 2005.

\bibitem{BMN}
J.~Behrndt, M.M. Malamud, and H.~Neidhardt.
\newblock Scattering theory for open quantum systems with finite rank coupling.
\newblock {\em Math. Phys. Anal. Geom.}, 10(4):313--358, 2007.

\bibitem{BeA}
N.~Ben~Abdallah.
\newblock On a multidimensional {S}chr{\"o}dinger-{P}oisson scattering model
  for semiconductors.
\newblock {\em J. Math. Phys.}, 41(7):4241--4261, 2000.

\bibitem{BDM}
N.~Ben~Abdallah, P.~Degond, and P.A. Markowich.
\newblock On a one-dimensional {S}chr{\"o}dinger-{P}oisson scattering model.
\newblock {\em Z. Angew. Math. Phys.}, 48(1):135--155, 1997.

\bibitem{BAPi}
N.~Ben~Abdallah and O.~Pinaud.
\newblock Multiscale simulation of transport in an open quantum system:
  resonances and {WKB} interpolation.
\newblock {\em J. Comput. Phys.}, 213(1):288--310, 2006.

\bibitem{BFN}
V.~Bonnaillie-No{\"e}l, A.~Faraj, and F.~Nier.
\newblock Simulation of resonant tunneling heterostructures: numerical
  comparison of a complete {S}chr{\"o}dinger-{P}oisson system and a reduced
  nonlinear model.
\newblock {\em J. Comput. Elec.}, 8(1):11--18, 2009.

\bibitem{BNP}
V.~Bonnaillie-No{\"e}l, F.~Nier, and Y.~Patel.
\newblock Computing the steady states for an asymptotic model of quantum
  transport in resonant heterostructures.
\newblock {\em J. Comput. Phys.}, 219(2):644--670, 2006.

\bibitem{BNP1}
V.~Bonnaillie-No{\"e}l, F.~Nier, and Y.~Patel.
\newblock Far from equilibrium steady states of 1{D}-{S}chr\"odinger-{P}oisson
  systems with quantum wells. {I}.
\newblock {\em Ann. Inst. H. Poincar\'e Anal. Non Lin\'eaire}, 25(5):937--968,
  2008.

\bibitem{BNP2}
V.~Bonnaillie-No{\"e}l, F.~Nier, and Y.~Patel.
\newblock Far from equilibrium steady states of 1{D}-{S}chr\"odinger-{P}oisson
  systems with quantum wells. {II}.
\newblock {\em J. Math. Soc. Japan}, 61(1):65--106, 2009.

\bibitem{CDNP}
H.~D. Cornean, P.~Duclos, G.~Nenciu, and R.~Purice.
\newblock Adiabatically switched-on electrical bias and the
  {L}andauer-{B}\"uttiker formula.
\newblock {\em J. Math. Phys.}, 49(10):102106, 20, 2008.

\bibitem{CFKS}
H.~L. Cycon, R.~G. Froese, W.~Kirsch, and B.~Simon.
\newblock {\em Schr\"odinger operators with application to quantum mechanics
  and global geometry}.
\newblock Texts and Monographs in Physics. Springer-Verlag, Berlin, study
  edition, 1987.

\bibitem{DiSj}
M.~Dimassi and J.~Sj{\"o}strand.
\newblock {\em Spectral asymptotics in the semi-classical limit}, volume 268 of
  {\em London Mathematical Society Lecture Note Series}.
\newblock Cambridge University Press, Cambridge, 1999.

\bibitem{ErAr}
M.~Ehrhardt and A.~Arnold.
\newblock Discrete transparent boundary conditions for the {S}chr\"odinger
  equation.
\newblock {\em Riv. Mat. Univ. Parma (6)}, 4*:57--108, 2001.
\newblock Fluid dynamic processes with inelastic interactions at the molecular
  scale (Torino, 2000).

\bibitem{Fat}
H.O. Fattorini.
\newblock {\em The {C}auchy problem}, volume~18 of {\em Encyclopedia of
  Mathematics and its Applications}.
\newblock Addison-Wesley Publishing Co., Reading, Mass., 1983.

\bibitem{GeSi}
C.~G{\'e}rard and I.~M. Sigal.
\newblock Space-time picture of semiclassical resonances.
\newblock {\em Comm. Math. Phys.}, 145(2):281--328, 1992.

\bibitem{HeLN}
B.~Helffer.
\newblock {\em Semi-classical analysis for the {S}chr\"odinger operator and
  applications}, volume 1336 of {\em Lecture Notes in Mathematics}.
\newblock Springer-Verlag, Berlin, 1988.

\bibitem{HeMa}
B.~Helffer and A.~Martinez.
\newblock Comparaison entre les diverses notions de r\'esonances.
\newblock {\em Helv. Phys. Acta}, 60(8):992--1003, 1987.

\bibitem{HeSj}
B.~Helffer and J.~Sj{\"o}strand.
\newblock {\em R\'esonances en limite semi-classique}.
\newblock Number 24-25 in M\'em. Soc. Math. France (N.S.). 1986.

\bibitem{HiSi1}
P.~D. Hislop and I.~M. Sigal.
\newblock {\em Semiclassical theory of shape resonances in quantum mechanics},
  volume 78(399) of {\em Mem. Amer. Math. Soc.}
\newblock 1989.

\bibitem{HiSi2}
P.~D. Hislop and I.~M. Sigal.
\newblock {\em Introduction to spectral theory}, volume 113 of {\em Applied
  Mathematical Sciences}.
\newblock Springer-Verlag, New York, 1996.
\newblock With applications to Schr{\"o}dinger operators.

\bibitem{JLPS}
G.~Jona-Lasinio, C.~Presilla, and J.~Sj{\"o}strand.
\newblock On {S}chr\"odinger equations with concentrated nonlinearities.
\newblock {\em Ann. Physics}, 240(1):1--21, 1995.

\bibitem{JoPf}
A.~Joye and C.E. Pfister.
\newblock Exponential estimates in adiabatic quantum evolution.
\newblock In {\em X{II}th {I}nternational {C}ongress of {M}athematical
  {P}hysics ({ICMP} '97) ({B}risbane)}, pages 309--315. Int. Press, Cambridge,
  MA, 1999.

\bibitem{Joy}
Alain Joye.
\newblock General adiabatic evolution with a gap condition.
\newblock {\em Comm. Math. Phys.}, 275(1):139--162, 2007.

\bibitem{Kat1}
T.~Kato.
\newblock Integration of the equation of evolution in a {B}anach space.
\newblock {\em J. Math. Soc. Japan}, 5:208--234, 1953.

\bibitem{Kat2}
T.~Kato.
\newblock Linear evolution equations of ``hyperbolic'' type.
\newblock {\em J. Fac. Sci. Univ. Tokyo Sect. I}, 17:241--258, 1970.

\bibitem{KlRa}
M.~Klein and Rama J.
\newblock Almost exponential decay of quantum resonance states and
  {P}aley-{W}iener type estimates in {G}evrey spaces.
\newblock preprint, mp-arc 09-64, 2009.

\bibitem{KRW}
M.~Klein, J.~Rama, and R.~W{\"u}st.
\newblock Time evolution of quantum resonance states.
\newblock {\em Asymptot. Anal.}, 51(1):1--16, 2007.

\bibitem{LaMa}
A.~Lahmar-Benbernou and A.~Martinez.
\newblock On {H}elffer-{S}j\"ostrand's theory of resonances.
\newblock {\em Int. Math. Res. Not.}, (13):697--717, 2002.

\bibitem{LKF}
S.E. Laux, A.~Kumar, and M.V. Fischetti.
\newblock Analysis of quantum ballistic electron transport in ultra-small
  semiconductor devices including space-charge effects.
\newblock {\em J. Appl. Phys.}, (95):5545--5582, 2004.

\bibitem{Loc}
P.~Lochak.
\newblock About the adiabatic stability of resonant states.
\newblock {\em Ann. Inst. H. Poincar\'e Sect. A (N.S.)}, 39(2):119--143, 1983.

\bibitem{Nen}
G.~Nenciu.
\newblock Linear adiabatic theory. {E}xponential estimates.
\newblock {\em Comm. Math. Phys.}, 152(3):479--496, 1993.

\bibitem{NeRa}
G.~Nenciu and G.~Rasche.
\newblock On the adiabatic theorem for nonselfadjoint {H}amiltonians.
\newblock {\em J. Phys. A}, 25(21):5741--5751, 1992.

\bibitem{Ni1}
F.~Nier.
\newblock The dynamics of some quantum open systems with short-range
  nonlinearities.
\newblock {\em Nonlinearity}, 11(4):1127--1172, 1998.

\bibitem{Ni2}
F.~Nier.
\newblock Accurate {WKB} approximation for a 1{D} problem with low regularity.
\newblock {\em Serdica Math. J.}, 34(1):113--126, 2008.

\bibitem{Pan}
K.~Pankrashkin.
\newblock Resolvents of self-adjoint extensions with mixed boundary conditions.
\newblock {\em Rep. Math. Phys.}, 58(2):207--221, 2006.

\bibitem{Per}
G.~Perelman.
\newblock Evolution of adiabatically perturbed resonant states.
\newblock {\em Asymptot. Anal.}, 22(3-4):177--203, 2000.

\bibitem{Pin}
Olivier Pinaud.
\newblock Transient simulations of a resonant tunneling diode.
\newblock {\em J. App. Phys.}, 92(4):1987--1994, 2002.

\bibitem{PrSj}
C.~Presilla and J.~Sj{\"o}strand.
\newblock Transport properties in resonant tunneling heterostructures.
\newblock {\em J. Math. Phys.}, 37(10):4816--4844, 1996.

\bibitem{ReSi2}
M.~Reed and B.~Simon.
\newblock {\em Methods of modern mathematical physics. {II}. {F}ourier
  analysis, self-adjointness}.
\newblock Academic Press, New York, 1975.

\bibitem{ReSi4}
M.~Reed and B.~Simon.
\newblock {\em Methods of modern mathematical physics. {IV}. {A}nalysis of
  operators}.
\newblock Academic Press, New York, 1978.

\bibitem{SimR}
B.~Simon.
\newblock Resonances and complex scaling: a rigorous overview.
\newblock {\em Int.J. Quantum Chem.}, 14(4):529--542, 1978.

\bibitem{SimE}
B.~Simon.
\newblock The definition of molecular resonance curves by the method of
  exterior complex scaling.
\newblock {\em Phys. Lett.}, 71A(2,3):211--214, 1979.

\bibitem{Sjo}
J.~Sj{\"o}strand.
\newblock Projecteurs adiabatiques du point de vue pseudodiff\'erentiel.
\newblock {\em C. R. Acad. Sci. Paris S\'er. I Math.}, 317(2):217--220, 1993.

\bibitem{SjZw}
J.~Sj{\"o}strand and M.~Zworski.
\newblock Complex scaling and the distribution of scattering poles.
\newblock {\em J. Amer. Math. Soc.}, 4(4):729--769, 1991.

\bibitem{Ski1}
E.~Skibsted.
\newblock Truncated {G}amow functions, {$\alpha$}-decay and the exponential
  law.
\newblock {\em Comm. Math. Phys.}, 104(4):591--604, 1986.

\bibitem{Ski2}
E.~Skibsted.
\newblock Truncated {G}amow functions and the exponential decay law.
\newblock {\em Ann. Inst. H. Poincar\'e Phys. Th\'eor.}, 46(2):131--153, 1987.

\bibitem{Ski3}
E.~Skibsted.
\newblock On the evolution of resonance states.
\newblock {\em J. Math. Anal. Appl.}, 141(1):27--48, 1989.

\bibitem{SoWe}
A.~Soffer and M.~I. Weinstein.
\newblock Time dependent resonance theory.
\newblock {\em Geom. Funct. Anal.}, 8(6):1086--1128, 1998.

\bibitem{Yos}
K.~Yosida.
\newblock {\em Functional analysis}.
\newblock Classics in Mathematics. Springer-Verlag, Berlin, 1995.
\newblock Reprint of the sixth (1980) edition.

\end{thebibliography}

\end{document}